\documentclass[11pt]{amsart}
\usepackage[T1]{fontenc}
\usepackage{amssymb}
\usepackage{amsmath}
\usepackage{amsthm}
\usepackage{fancyhdr}
\usepackage{mathtools}
\usepackage[british]{babel}
\usepackage{geometry}
\usepackage{algpseudocode}
\usepackage{dsfont}
\usepackage{centernot}
\usepackage{xstring}
\usepackage{colortbl}
\usepackage[section]{algorithm}
\usepackage{graphicx}
\usepackage[font={footnotesize}]{caption}
\usepackage[dvipsnames,table]{xcolor}
\usepackage{pstricks,tikz}
\usepackage{tikz,pgfplots}
\pgfplotsset{compat=1.18}
\usetikzlibrary{decorations.pathreplacing, patterns, calc, shapes.symbols}
\usepackage{xcolor}
\usepackage[h]{esvect}
\usepackage[noadjust]{cite}
\usepackage[
		bookmarksopen=true,
		bookmarksopenlevel=1,
		colorlinks=true,
		linkcolor=darkblue,
        linktoc=page,
		citecolor=darkblue,
]{hyperref}
\usepackage[shortlabels]{enumitem}
\usepackage{comment}
\setlist{nosep}

\global\long\def\TK#1{\textcolor{magenta}{\textbf{[TK:} #1\textbf{]}}}

\usepackage[foot]{amsaddr}

\title[]{Advancing the R\"{o}dl Nibble: New bounds on matchings and the list chromatic index of hypergraphs}

\author[Stephen~Gould]{Stephen~Gould}
\email{stephengould1610@hotmail.co.uk}

\author[Tom~Kelly]{Tom~Kelly$^1$}

\address{$^1$Georgia Institute of Technology}
\email{tom.kelly@gatech.edu}
\thanks{Kelly's research supported by the National Science Foundation under Grant No. DMS-2247078.}

\date{November 14, 2025}

\newtheorem{theorem}{Theorem}[section]

\newtheorem{lemma}[theorem]{Lemma}
\newtheorem{cor}[theorem]{Corollary}

\theoremstyle{definition}

\newtheorem{defin}[theorem]{Definition}

\newtheorem{observation}[theorem]{Observation}

\newtheoremstyle{claimstyle}{5pt}{5pt}{\em}{5pt}{\em}{:}{5pt}{}
\theoremstyle{claimstyle}
\newtheorem{claim}{Claim}

\newtheoremstyle{stepstyle}{10pt}{5pt}{\em}{0pt}{\em}{:}{5pt}{}
\theoremstyle{stepstyle}

\numberwithin{equation}{section}

\definecolor{darkblue}{rgb}{0,0,0.5}

\def\noproof{{\unskip\nobreak\hfill\penalty50\hskip2em\hbox{}\nobreak\hfill%
       $\square$\parfillskip=0pt\finalhyphendemerits=0\par}\goodbreak}
\def\endproof{\noproof\bigskip}

\def\noclaimproof{{\unskip\nobreak\hfill\penalty50\hskip2em\hbox{}\nobreak\hfill%
       $-$\parfillskip=0pt\finalhyphendemerits=0\par}\goodbreak}

\def\endclaimproof{\noclaimproof\medskip}

\newdimen\margin
\def\textno#1&#2\par{
   \margin=\hsize
   \advance\margin by -4\parindent
          \setbox1=\hbox{\sl#1}
   \ifdim\wd1 < \margin
      $$\box1\eqno#2$$
   \else
      \bigbreak
      \hbox to \hsize{\indent$\vcenter{\advance\hsize by -3\parindent
      \it\noindent#1}\hfil#2$}
      \bigbreak
   \fi}

\def\lateproof#1{\removelastskip\penalty55\medskip\noindent\setcounter{claim}{0}\setcounter{step}{0}{\bf Proof of #1. }} 



\usepackage{thmtools}
\usepackage{thm-restate}

\def\claimproof{\removelastskip\penalty55\medskip\noindent{\em Proof of claim: }}
\def\obsproof{\removelastskip\penalty55\medskip\noindent{\em Proof of observation: }}

\begin{document}

\newcommand{\new}[1]{\textcolor{red}{#1}}
\def\COMMENT#1{}
\def\TASK#1{}
\newcommand{\APPENDIX}[1]{}
\newcommand{\NOTAPPENDIX}[1]{#1}
\renewcommand{\APPENDIX}[1]{#1}                    
\renewcommand{\NOTAPPENDIX}[1]{}                   
\newcommand{\todo}[1]{\begin{center}\textbf{to do:} #1 \end{center}}

\def\eps{{\varepsilon}}
\newcommand{\ex}{\mathbb{E}}
\newcommand{\pr}{\mathbb{P}}
\newcommand{\cB}{\mathcal{B}}
\newcommand{\cA}{\mathcal{A}}
\newcommand{\cE}{\mathcal{E}}
\newcommand{\cS}{\mathcal{S}}
\newcommand{\cF}{\mathcal{F}}
\newcommand{\cG}{\mathcal{G}}
\newcommand{\bL}{\mathbb{L}}
\newcommand{\bF}{\mathbb{F}}
\newcommand{\bZ}{\mathbb{Z}}
\newcommand{\cH}{\mathcal{H}}
\newcommand{\cC}{\mathcal{C}}
\newcommand{\cM}{\mathcal{M}}
\newcommand{\bN}{\mathbb{N}}
\newcommand{\bR}{\mathbb{R}}
\def\O{\mathcal{O}}
\newcommand{\cP}{\mathcal{P}}
\newcommand{\cQ}{\mathcal{Q}}
\newcommand{\cR}{\mathcal{R}}
\newcommand{\cJ}{\mathcal{J}}
\newcommand{\cL}{\mathcal{L}}
\newcommand{\cK}{\mathcal{K}}
\newcommand{\cD}{\mathcal{D}}
\newcommand{\cI}{\mathcal{I}}
\newcommand{\cV}{\mathcal{V}}
\newcommand{\cT}{\mathcal{T}}
\newcommand{\cU}{\mathcal{U}}
\newcommand{\cW}{\mathcal{W}}
\newcommand{\cX}{\mathcal{X}}
\newcommand{\cY}{\mathcal{Y}}
\newcommand{\cZ}{\mathcal{Z}}
\newcommand{\1}{{\bf 1}_{n\not\equiv \delta}}
\newcommand{\eul}{{\rm e}}
\newcommand{\Erd}{Erd\H{o}s}
\newcommand{\cupdot}{\mathbin{\mathaccent\cdot\cup}}
\newcommand{\whp}{whp }
\newcommand{\bX}{\mathcal{X}}
\newcommand{\bV}{\mathcal{V}}
\newcommand{\ordsubs}[2]{(#1)_{#2}}
\newcommand{\unordsubs}[2]{\binom{#1}{#2}}
\newcommand{\ordelement}[2]{\overrightarrow{\mathbf{#1}}\left({#2}\right)}
\newcommand{\ordered}[1]{\overrightarrow{\mathbf{#1}}}
\newcommand{\reversed}[1]{\overleftarrow{\mathbf{#1}}}
\newcommand{\weighting}[1]{\mathbf{#1}}
\newcommand{\weightel}[2]{\mathbf{#1}\left({#2}\right)}
\newcommand{\unord}[1]{\mathbf{#1}}
\newcommand{\ordscript}[2]{\ordered{{#1}}_{{#2}}}
\newcommand{\revscript}[2]{\reversed{{#1}}_{{#2}}}
\newcommand{\dirK}{\overleftrightarrow{K^{\circ}_{n}}}

\newcommand{\doublesquig}{%
  \mathrel{%
    \vcenter{\offinterlineskip
      \ialign{##\cr$\rightsquigarrow$\cr\noalign{\kern-1.5pt}$\rightsquigarrow$\cr}%
    }%
  }%
}

\newcommand{\defn}{\emph}

\newcommand\restrict[1]{\raisebox{-.5ex}{$|$}_{#1}}

\newcommand{\prob}[1]{\mathrm{\mathbb{P}}\left[#1\right]}
\newcommand{\probb}[1]{\mathrm{\mathbb{P}}_{b}\left[#1\right]}
\newcommand{\expn}[1]{\mathrm{\mathbb{E}}\left[#1\right]}
\newcommand{\expnb}[1]{\mathrm{\mathbb{E}}_{b}\left[#1\right]}
\newcommand{\probstar}[1]{\mathrm{\mathbb{P}}^{*}\left[#1\right]}
\newcommand{\probxj}[1]{\mathrm{\mathbb{P}}_{x(j)}\left[#1\right]}
\newcommand{\expnxj}[1]{\mathrm{\mathbb{E}}_{x(j)}\left[#1\right]}
\def\gnp{G_{n,p}}
\def\G{\mathcal{G}}
\def\lflr{\left\lfloor}
\def\rflr{\right\rfloor}
\def\lcl{\left\lceil}
\def\rcl{\right\rceil}

\newcommand{\qbinom}[2]{\binom{#1}{#2}_{\!q}}
\newcommand{\binomdim}[2]{\binom{#1}{#2}_{\!\dim}}

\newcommand{\grass}{\mathrm{Gr}}

\newcommand{\brackets}[1]{\left(#1\right)}
\def\sm{\setminus}
\newcommand{\Set}[1]{\{#1\}}
\newcommand{\set}[2]{\{#1\,:\;#2\}}
\newcommand{\krq}[2]{K^{(#1)}_{#2}}
\newcommand{\ind}[1]{$\mathbf{S}(#1)$}
\newcommand{\indcov}[1]{$(\#)_{#1}$}
\def\In{\subseteq}

\begin{abstract}

Let $H$ be a $(k+1)$-uniform hypergraph which is nearly $D$-regular, such that any set of $i$ vertices is contained in at most $D_i$ edges of $H$ for each $i = 2, 3, \dots, k+1$.
Influential results of Pippenger and of Frankl and R\"odl show that the \textit{R\"odl Nibble} -- a probabilistic procedure which iteratively constructs a matching in small bits -- can produce an almost-perfect matching in $H$, provided $D_2$ is much smaller than $D$.
The quantitative aspects of this result were sharpened by several authors, with the previously best-known result due to Vu, whose result takes more of the codegree sequence $D_2, \dots, D_{k+1}$ into account.
We improve Vu's result, by showing the R\"odl Nibble can ``exhaust'' the full codegree sequence up to one of several natural bottlenecks, even tolerating extensive ``clustering'' of codegree values.
Up to a subpolynomial error term, we believe our result to be the optimal usage of pure nibble methodology.

We also show that our matching can be taken to be ``pseudorandom'' with respect to a set of weight functions on $V(H)$, and we use this result to derive other hypergraph matching results in partite settings, including a new bound on the list chromatic index which implies the best-known result of Molloy and Reed up to the error term, and is stronger when the hypergraph is not close to linear, i.e.\ \(D_2=\omega(1)\).
We also apply our results to obtain improved bounds on almost-spanning structures in Latin squares and designs, and the maximum diameter of a simplicial complex.
\end{abstract}

\maketitle

\section{Introduction}\label{section:intro}
A \textit{matching} in a hypergraph~\(H\) is a set~\(M\) of edges of~\(H\) such that no vertex of~\(H\) is in more than one element of~\(M\), and a matching is \emph{perfect} if every vertex of~\(H\) is in an element of~\(M\) (is \emph{covered} by~\(M\)).
The problem of finding a perfect matching in \emph{graphs} is quite well-understood; indeed, in 1965, Edmonds~\cite{E65} provided an efficient algorithm for finding a matching~\(M\) in any graph, which is \emph{maximum} (no matching with more edges exists).
For more on graph matchings, see for example the book of Lov\'asz and Plummer~\cite{LP09}.
In stark contrast, the problem of finding perfect matchings in hypergraphs of higher uniformity is known to be NP-complete, and the case of \(3\)-uniform 3-partite hypergraphs was one of Karp's original 21 NP-complete problems~\cite{K72}.
For surveys on the history and applications of hypergraph matching results, see for example~\cite{F88,Ka95,Ka97,K18survey,KKKMO,Ke24-nibble}.

The notion of hypergraph matchings captures something surprisingly general; results find wide applicability in combinatorics and related fields.
Indeed, F\"uredi~\cite{F88} commented that ``almost all combinatorial questions can be reformulated as either a matching or covering problem of a hypergraph''.
Although it is difficult to determine whether a hypergraph has a perfect matching, it is often useful to have results that guarantee a \textit{nearly perfect} matching, and it is interesting to minimize the number of vertices left uncovered by the matching.
As we discuss in more detail later, several such results have been proved with the \textit{R\"odl Nibble}, so named after R\"odl's proof \cite{R85} of the Erd\H os--Hanani Conjecture \cite{EH63} on the existence of nearly complete Steiner systems.

Our contribution in this paper is threefold. 
First, we push the R\"odl Nibble to what seems to be its theoretical limit for hypergraph matchings. We establish a nearly perfect hypergraph matching result (Theorem~\ref{theorem:maintheorem}) that guarantees a matching at least as large as any other result proved with a purely nibble-based approach, with heuristics suggesting that any matching found via nibble cannot be any larger.
Second, we show that, in addition, the matchings that we obtain can be guaranteed to satisfy certain pseudorandomness properties, and we use this result to derive other matching-based results (Theorems~\ref{theorem:withreserves} and \ref{thm:mainpartitetheorem}) that can more easily be used as a blackbox in other settings, as well as a new bound on the list chromatic index of hypergraphs (Theorem~\ref{theorem:maincolourtheorem}) which strengthens a result of Molloy and Reed \cite{MR00}.
Finally, using these results, we provide short proofs of improved bounds for several problems.
Namely, we
construct simplicial complexes of large diameter (Theorem~\ref{theorem:SimplicialComplex}, which improves the bound of Bohman and Newman \cite{BN22}), find almost-spanning rainbow directed cycles in directed versions of~\(K_n\), equivalently ``almost-Hamilton'' transversals in Latin squares (Theorem~\ref{theorem:rainbowlinks}, which improves the bound of Benzing, Pokrovskiy, and Sudakov \cite{BPS20}), find large rainbow partial triangle factors in directed and undirected complete graphs (Theorem~\ref{theorem:triangles}, Corollary~\ref{corollary:triangles}), and produce new bounds on matchings and (list) edge-colourings of $(t, r, n)$-Steiner systems and their partite analogues (Theorems~\ref{theorem:Steinercolorsmatchings}--\ref{theorem:PartiteSteinercolourmatchings}, which improve the bounds of Kang, K\"uhn, Methuku, and Osthus \cite{KKMO23} whenever $t > 2$).
We remark that the objects constructed in the first two of these applications cannot be purely described locally, and one of our main innovations (outside of Theorem~\ref{theorem:maintheorem}) is in applying the R\"odl Nibble to deduce strong quantitative results for this type of problem.
Though earlier papers (e.g.~\cite{KKKO20,GKMO21}) used the nibble in conjunction with connecting arguments to find global structures like long cycles, to our knowledge we are the first to derive results about this kind of global structure directly from a matching blackbox (Theorem~\ref{thm:mainpartitetheorem}).
\subsection{The R\"odl Nibble and previous results}
Possibly the most prominent tool for finding nearly perfect hypergraph matchings is the R\"{o}dl Nibble (also called the ``semi-random method''), which arose in the 1980s from the work of Ajtai, Koml\'os, and Szemer\'edi~\cite{AKS81}, R\"{o}dl~\cite{R85}, and others (for example~\cite{FR85,AKPSS82,KPS82}).
Put simply, in the context of hypergraph matchings (the same methodology applies directly in many other settings also), the strategy is to randomly choose a small number of edges (a ``nibble''), and use probabilistic analysis to find an outcome in which most of the chosen edges form a matching whose removal leaves the properties of the remaining hypergraph largely undamaged.
One can then iterate the nibble process, and show that the associated errors grow slowly enough that the accumulated matching covers almost all of the original vertex set, by the time we must halt.
We will discuss the R\"odl Nibble, especially in the context of asymptotic improvements, in depth, but for a deeper history of the method, see the recent survey of Kang, Kelly, K\"uhn, Methuku, and Osthus~\cite{KKKMO} or the chapter of Kelly \cite{Ke24-nibble}.

Throughout the paper, for a hypergraph~\(H\) and a vertex \(v\in V(H)\) we denote by~\(\text{deg}(v)\) (the \textit{degree} of~\(v\)) the number of edges of~\(H\) containing~\(v\), and we set \(\Delta(H)\coloneqq\max_{v}\text{deg}(v)\) and \(\delta(H)\coloneqq\min_v \text{deg}(v)\).
For a set \(U\subseteq V(H)\) we write~\(\text{codeg}(U)\) (the \textit{codegree} of~\(U\)) for the number of edges containing~\(U\), and for \(j\geq 2\) we set \(C_j(H)\coloneqq \max\text{codeg}(U)\), the maximum taken over \(j\)-sets of \(V(H)\).
A hypergraph is \((k+1)\)-\textit{uniform} if all edges have exactly~\(k+1\) vertices, and \((n,D,\eps)\)-\textit{regular} if it has~\(n\) vertices, each having degree \(\text{deg}(v)=(1\pm\eps)D\), that is, \((1-\eps)D\leq\delta(H)\leq\Delta(H)\leq(1+\eps)D\).

A centrally important notion in the semi-random method is the ``gap'' or ``space'' between the average degree~\(D\) and the (best-known upper bound~\(D_2\) on the) maximum \(2\)-codegree~\(C_2(H)\).
In some sense (if all vertex degrees are very close to~\(D\)) the ratio~\(D/D_2\) parameterizes the worst-case dependence between the vertex degrees of different vertices during a random nibble, and so, in the process of iterating the nibble, as~\(D_2\) remains fixed (say) and~\(D\) decreases as we remove an ever-increasing matching, we must halt once~\(D\) becomes too close to~\(D_2\), as after this point, we cannot effectively concentrate the vertex degrees.
Inspired by R\"odl's proof \cite{R85} of the Erd\H os--Hanani Conjecture \cite{EH63} on the existence of nearly complete Steiner systems, Frankl and R\"odl~\cite{FR85} and Pippenger (unpublished, but see~\cite[Theorem 8.4]{F88} and~\cite{PS89}) proved that, for a nearly $D$-regular hypergraph, if $D_2$ is much smaller than $D$, then there is an almost-perfect matching.
Building on results of Grable~\cite{G99} and Alon, Kim, and Spencer~\cite{AKS97} improving the size of the leftover by more efficiently using the space between~\(D\) and~\(D_2\), Kostochka and R\"{o}dl~\cite{KR97} proved the following result in 1997.
We use standard notation for hierarchies of constants throughout this paper, so that \(1/A, b\ll c\) means that for any \(c\in(0,1]\) there exist \(A_0=A_0(c), b_0=b_0(c) >0\) such that the subsequent statement holds for all \(0<b\leq b_0\) and integers \(A\geq A_0\), and hierarchies with more constants are interpreted similarly.

\begin{theorem}[Kostochka and R\"{o}dl~\cite{KR97}, 1997]\label{theorem:KR}
Suppose that \(1/D\ll1/K,1/k,\delta,\gamma<1\).\COMMENT{In particular, this enforces \(k\geq 2\).}
Let~\(H\) be a \((k+1)\)-uniform, \((n,D,\eps)\)-regular hypergraph with \(C_2(H)\leq D_2\leq D^{1-\gamma}\) and \(\eps\leq K\sqrt{\log D/D}\).
Then there is a matching~\(M\) of~\(H\) which covers all but at most~\(n(D/D_2)^{-1/k +\delta}\) vertices.
\end{theorem}
Heuristically, the leftover hypergraph after removal of the matching found by the semi-random method, if the leftover has size roughly~\(pn\), say, should resemble an induced subgraph where each vertex remains independently with probability~\(p\).
In the latter regime, the expected degree of a vertex is~\(Dp^k\), so that \(p=(D_2 /D)^{1/k}\) corresponds to average degree~\(D_2\) in the leftover.
Thus, as discussed above, the leftover given in Theorem~\ref{theorem:KR} is an essentially optimal use of the space between~\(D\) and~\(D_2\), using purely nibble methods.

In 2000, Vu~\cite{V00} made a crucial observation: Perhaps instead of treating~\(D_2\) as fixed, we could leverage the space between~\(D_2\) and~\(D_3\) (where \(C_3(H)\leq D_3\)) to examine the degradation of the 2-codegrees during a nibble, and thus show that~\(D_2\) drops during the iteration process also?
Then we could continue to iterate for longer before~\(D\) decreases to~\(D_2\),\COMMENT{which isn't to say that Vu continues his process until \(D\) hits the new \(D_2\); he doesn't, but it is still (generally) better to collapse \(D_{s-1}\) to \(D_s\). Chose not to be super precise here, I think Vu's result is slightly unwieldy/not always best, for instance for some codegree sequences the bottleneck for \(x\) will be some \(D_j/D_{j+1}\), so it's better to use Kostochka-Rodl than to naively apply Vu, but Vu is always at least as good if you take s=2, and can be better given good codegree sequences or tactical choices of \(s\) and \(D_j\) adjustments. Exploring this too deeply would just add confusion I think} as~\(D_2\) has dropped non-trivially below its initial value, and we obtain a larger matching.
Indeed, Vu realized that if we have some \(s\)-codegree\COMMENT{Now I'm leaving \(C_j(H)\leq D_j\) implicit, and sort of abusing terminology and conflating \(D_s\) with the ``\(s\)-codegree''} \(D_s\) with \(s>2\) say, such that each \(D_j / D_{j+1}\) for \(2\leq j<s\) was large enough, then we could show that each such~\(D_j\) decreased during the process, such that~\(D_2\) may be allowed to decrease much further before~\(D\) drops to the new value of~\(D_2\) (or some other ratio~\(D_j/D_{j+1}\) vanished, also stopping the process).
In a nutshell, \emph{utilizing the higher codegrees yields a larger matching}.
The following theorem represents the previous `state of the art' in terms of leftover size using purely nibble methods.
Here and throughout the paper, all logarithms are base~\(e\).
\begin{theorem}[Vu~\cite{V00}, 2000]\label{theorem:V00}
Suppose that $1/D\ll1/A\ll 1/k\leq 1$.
Let~$H$ be a $(k+1)$-uniform, \((n,D_1,\eps)\)-regular hypergraph, and suppose that there is $2\leq s\leq k+1$ and quantities~$D_{j}$ for $2\leq j\leq s$ such that $C_{j}(H)\leq D_{j}$ for all $2\leq j\leq s$, and $D_{s}\geq 1$.
Assume further that there is $x>0$ such that the following conditions hold:
\begin{enumerate}[label=\upshape(\roman*)]
\item $x^{2}\leq\frac{D_{j}}{D_{j+1}}$ for all $1\leq j\leq s-1$;\label{condition:otherratios}
\item $x^{k-s+2}\leq\frac{D_{s-1}}{D_{s}}$;\label{condition:finalratio}
\item \(\eps\leq\frac{1}{x}\).
\end{enumerate}
Then there is a matching~$M$ of~$H$ which covers all but at most~$n(\log^{A}n)/x$ vertices.\COMMENT{In fact, Vu shows \(\eps\) can be as large as \(\alpha x^{-1}\log^\beta D\) for some well chosen \(\alpha,\beta>0\), but he never actually specifies the \(\alpha, \beta\), and I feel it would just muddy the statement here. Clearly saying \(\eps\leq 1/x\) is permissible, would you agree it gets the essence of the statement across?}
\end{theorem}
We remark that Vu states~\ref{condition:otherratios} with~\(x^3\) in place of~\(x^2\), but he claims that this can be improved to~\(x^2\), and indeed, this does follow from his proof.
Notice that putting \(s=2\) in Theorem~\ref{theorem:V00} yields Theorem~\ref{theorem:KR} (and with subpolynomial error improved to polylogarithmic), and for larger~$s$, $x^{k-s+2}<x^{k}$, whence condition~\ref{condition:finalratio} gives a less strict inequality on~$x$ (assuming~$D/D_{2}$ and~$D_{s-1}/D_{s}$ are comparable in size, as they often are in applications) and a larger~$x$ yields a smaller leftover.
To illustrate Theorem~\ref{theorem:V00}, suppose~\(k\geq 4\) is even (say) and consider a large \(n\)-vertex \((k+1)\)-uniform \(D\)-regular hypergraph~\(H^{*}\) which has \(D=\Theta(n^k)\) and \(D_j=\Theta(n^{k+1-j})\) for \(2\leq j\leq  k/2 +1\eqqcolon s\).
Theorem~\ref{theorem:KR} yields leftover roughly~\(n^{1-1/k}\).
Theorem~\ref{theorem:V00} yields an improvement of leftover roughly $n^{1-2/(k+2)}$, since
condition~\ref{condition:otherratios} of Theorem~\ref{theorem:V00} imposes the condition \(x\leq \sqrt{n}\), condition~\ref{condition:finalratio} imposes \(x\leq n^{2/(k+2)}\), and we may put \(\eps=0\).

As it will be useful for the purpose of comparison, we now introduce a much more recent result from 2023 which uses techniques above and beyond ``pure nibble methodology'' to obtain a matching with slightly smaller leftover (by the \((D/D_2)^{\eta}\) factor) than Theorem~\ref{theorem:KR}.
More specifically, during each nibble, the authors of~\cite{KKMO23} tracked (in addition to the degradation of~\(D\) and the vertex degree errors) the survival of a family of subgraphs called ``augmenting stars''.
When the pure nibble methodology reached the halting point as in Theorem~\ref{theorem:KR}, they showed that the surviving augmenting stars could be used to augment the matching given by the nibble.
\begin{theorem}[Kang, K\"uhn, Methuku, and Osthus~\cite{KKMO23}, 2023]\label{theorem:KKMO}
Suppose \(1/D\ll 1/k, \gamma, \mu, \eta, \delta<1\), where \(k\geq 3\) and \(\eta<f_{(\ref{theorem:KKMO})}(k)\coloneqq\frac{k-2}{k(k^3 + k^2 -2k +2)}\).
Suppose further that \(D\geq\exp(\log^{\mu}n)\).
Let~\(H\) be a \((k+1)\)-uniform, \((n,D,\eps)\)-regular hypergraph satisfying \(C_2(H)\leq D_2\leq D^{1-\gamma}\), and suppose \(\eps\leq(D_2/D)^{1/k +\delta}\).
Then~\(H\) has a matching covering all but at most~\(n(D/D_2)^{-1/k - \eta}\) vertices.\COMMENT{In particular they say there exists \(N_0\) such that it all works, whereas we assume \(D\) sufficiently large, but \(D/2\leq (1-\eps)D\leq\text{deg}(v)\leq nD_2\) so \(n\geq D/(2D_2)\geq D^{\gamma}/2\) so our hypotheses ensure \(n\) also suff large. They assume \(k\geq 4\) but their \(k\) is our \(k+1\). The value \(\eps\) comes from their Theorem 7.1, where they assume all vertices have degree \((1\pm(D_2/D)^{1-x})D\) with \(x<1-1/k\). If \(x=1-1/k\) then this would be \((1\pm (D_2/D)^{1/k})D\), and smaller \(x\) means raising \((D_2/D)<1\) to a higher power, so tighter control, which makes sense. My translation of the hypothesis ensures that \(x\leq1- 1/k-\delta\) for some \(\delta>0\).}
\end{theorem}
We remark that the authors of~\cite{KKMO23} assume~\(n\) is sufficiently large rather than~\(D\), but that these hypotheses are equivalent given the assumption \(D\geq\exp(\log^{\mu}n)\) and the fact \(D/2\leq(1-\eps)D\leq \text{deg}(v)\leq nD_2\leq nD^{1-\gamma}\) for all~\(v\).
The matching given by Theorem~\ref{theorem:KKMO} beats Theorem~\ref{theorem:KR}, and beats Theorem~\ref{theorem:V00} in the case that \(D_2= D_s\) (for example ``linear'' hypergraphs having \(C_2(H)=1\)), but is typically beaten by Theorem~\ref{theorem:V00} when \(D_2\gg D_s\) and we have sufficient information on the ``space'' available throughout the codegree sequence.
For example, with~\(H^{*}\) as above, Theorem~\ref{theorem:KKMO} gives leftover roughly \(n^{1-1/k-f_{(\ref{theorem:KKMO})}(k)}=\omega(n^{1-2/(k+2)})\) for all \(k\geq 4\),\COMMENT{Checked this on Desmos and quickly checking the algebra. The only positive range of \(k\) where \(1/k + f_{(\ref{theorem:KKMO})}(k) > 2/(k+2)\) is around \(k\in(4/3,2)\)} so the extra information here suffices for Theorem~\ref{theorem:V00} to beat Theorem~\ref{theorem:KKMO}.
We discuss further examples in Section~\ref{section:appstatements}.
\subsection{Our main result}
The initial inspiration for our main theorem was as follows: If Theorem~\ref{theorem:V00} applied to some~\(H\) is such that \(D_{s-1}/D_{s}\) is the bottleneck for~\(x\) (as occurs for~\(H^{*}\) above), so that~\(D_{s-1}\) has dropped to become basically~\(D_s\) at the end of the process, and~\(D\) and all other~\(D_j\) have dropped some amount, why not simply now redefine~\(s\) to be~\(s-1\), and apply Theorem~\ref{theorem:V00} again?
Indeed, if the codegree sequence is well-behaved, we could repeat this process until we have \emph{leveraged all of the space in the entire codegree sequence}, dropping~\(D\) until it has become essentially the original~\(D_s\), possibly even the multiplicity of the hypergraph \(D_{k+1}\coloneqq C_{k+1}(H)\) if \(D_s\approx D_{k+1}\).

Unfortunately, there were a number of obstacles to obtaining such a result so directly.
Most crucially, the machinery in the proof of Theorem~\ref{theorem:V00} requires that \(\eps \leq 1/x\), and cannot take full advantage of a better~\(\eps\) than this;
the proof necessitates setting \(\eps\approx1/x\) at the start, and this choice plays an important role in the probabilistic analysis\COMMENT{Because Vu uses the polynomial method, and truncates his polynomials at degree 2, which means he obtains error roughly \(\theta^2 D\) on the vertex degrees. This is OK for him only because he selects \(\eps\approx1/x\approx\theta\), so it is equivalent to error \(\eps\theta D\) in the vertex degrees, which is what you actually need in order for the degrees to degrade at the right rate. In order to use Vu's proof but not insist that \(\eps\approx\theta\), you would need to truncate the polynomials at a much higher degree \(r\) such that \(\theta^r D\approx \eps\theta D\), which is exactly what we used to do, and was a great deal of work above and beyond Vu.}, so any tighter control one had initially is lost.
If we wanted to apply such a result more than once, we therefore needed to find a way to begin with permissibly (much) smaller~\(\eps\) and much more tightly control the vertex degree errors in the course of the nibble process. 
Moreover, in the above na\"ive suggested iteration of Theorem~\ref{theorem:V00}, as soon as any pair of consecutive codegrees ``cluster'', which is to say either~\(D\) becomes too close to the current~\(D_2\), or some~\(D_j\) becomes too close to~\(D_{j+1}\), then the process must stop.
One of the main new features of our result and methodology is that the latter issue is \emph{entirely bypassed}.
Our methodology is robust enough to cope with any clustering of any number of consecutive codegrees other than~\(D/D_2\), at any point in the process, even at the start of the procedure.
By comparison, Theorem~\ref{theorem:V00} yields an empty matching if any codegrees are clustered, and so the best one can do is either take a smaller~\(s\) and throw away all of the ``space'' in the codegree sequence beyond this point, or `artificially' increase the values~\(D_j\) causing the problem, improving the ratio \(D_j/D_{j+1}\) but damaging the ratio \(D_{j-1}/D_j\).
We describe an example in which such `artificial adjustments' (of \(s, (D_j)\)) in the application of Theorem~\ref{theorem:V00} yield a better result than applying the theorem na\"ively, shortly after Theorem~\ref{theorem:triangles}.

The following is our main result.
We remark that in Theorem~\ref{theorem:maintheorem} (and Theorems~\ref{theorem:withreserves}, \ref{thm:mainpartitetheorem}, \ref{theorem:maincolourtheorem}) and throughout its proof, the hypergraph is allowed to be a multihypergraph, i.e.\ have \(C_{k+1}(H)>1\).
Two copies of an edge are treated as distinct, so that~\(D\) and \(C_j(H)\) count copies.
Here and throughout the paper, for a hypergraph~\(H\), a function \(\tau\colon V(H)\rightarrow\mathbb{R}_{\geq 0}\), and a set \(U\subseteq V(H)\), we define \(\tau(U)\coloneqq\sum_{u\in U}\tau(u)\).
\begin{restatable}{theorem}{mainThm}\label{theorem:maintheorem}
Suppose \(1/D\ll 1/A \ll \gamma \ll 1/k\leq 1\),  and let~\(H\) be an \((n,D,\eps)\)-regular, \((k+1)\)-uniform hypergraph with \(\eps>0\).\COMMENT{If the hypergraph is perfectly regular, then you could put \(\eps=0\), but then \(1/\eps\) is undefined. We could instead comment that we interpret \(1/0\) to be \(+\infty\) and the `minimum' operation then treats that term as arbitrarily large, but that seems a little messier than just forbidding \(\eps=0\)... Hopefully it doesn't come across like you're unable to input a perfectly regular hypergraph}
Suppose \(D_2\geq D_3\geq\dots\geq D_{k+1}\) are numbers satisfying \(C_j(H)\leq D_j\) for each \(j\in\{2,3,\dots,k+1\}\).
Suppose~\(B\) is a number satisfying\COMMENT{It is necessary to specify that \(B\) is positive, else the \(B^{-1+\gamma}\) in \(nB^{-1+\gamma}\log^A D\) may not even be defined. \(B\geq 1\) ensures the expression \(|U|B^{-1}\) at the end of the statement is at most \(|U|\), as otherwise evidently this wouldn't be satisfied. Related note, which do you feel is better or more natural, writing \(B\leq\min(\dots)\) or \(B\coloneqq\min(\dots)\)?}
\[
1\leq B\leq\min\left\{\sqrt{\frac{D}{D_2}}, \min_{j\in\{4,5,\dots,k+1\}}\left(\frac{D}{D_j}\right)^{\frac{1}{j-1}}, \frac{1}{\eps}\right\}.
\]
Then~\(H\) has a matching~\(M\) covering all but at most~\(nB^{-1+\gamma}\log^A D\) vertices.\COMMENT{If \(k\in\{1,2\}\) so that \(\{4,5,\dots,k+1\}\) is empty, then the middle term in \(B\) is the minimum over no terms, which means that \(B=\min\{\sqrt{D/D_2}, 1/\eps\}\). Is this clear enough?}

If, in addition,~\(\cT\) is a family of functions \(\tau\colon V(H)\rightarrow\mathbb{R}_{\geq0}\) such that
\begin{enumerate}[(P1), topsep = 6pt]
\item \(\tau(V(H))\geq B\cdot\max_{v\in V(H)}\tau(v)\) for all \(\tau\in\cT\);\label{mainpseud:lowerbound}
\item \(|\{v\in V(H)\colon\tau(v)>0\}|\leq D^{\log D}\) for all \(\tau\in\cT\);\label{mainpseud:supportsize}
\item Each vertex of~\(H\) is in~\(\{v\in V(H)\colon\tau(v)>0\}\) for at most~\(D^{\log D}\) of the functions \(\tau\in\cT\),\label{mainpseud:maxinvolvement}
\end{enumerate}
then there exists an~\(M\) as above additionally satisfying \(\tau(V(H))B^{-1}\leq \tau(V(H)\setminus V(M))\leq \tau(V(H))B^{-1+\gamma}\log^A D\) for each \(\tau\in\cT\).
\end{restatable}
We now discuss the first portion of the result, regarding the size of the leftover, before returning to the latter part of the result about the family~\(\cT\) of functions in Sections~\ref{section:pseudo} and~\ref{section:bey}.
We remark that the `artificial adjustments' described above are never beneficial when applying Theorem~\ref{theorem:maintheorem}; raising any~\(D_j\) only decreases some \((D/D_j)^{1/(j-1)}\) or~\(\sqrt{D/D_2}\), further constraining~\(B\).
Theorem~\ref{theorem:maintheorem} is intended to be intuitive to apply, and gives the best result when the best upper bounds for~\(C_j(H)\) are used for $D_j$.
One can set \(D_j\coloneqq D_2\) for each \(j\geq 2\) (i.e.\ ignore any available ``space'' beyond~\(D_2\)) to see that Theorem~\ref{theorem:maintheorem} implies Theorem~\ref{theorem:KR}.
Moreover, if one uses the same \((D_j)\) for \(2\leq j\leq s\) and sets \(D_j\coloneqq D_s\) for ~\(j\geq s\), one can show that Theorem~\ref{theorem:maintheorem} implies Theorem~\ref{theorem:V00} up to the subpolynomial error term, though we leave the proof of this fact to the reader.
That is to say, up to the error term, Theorem~\ref{theorem:maintheorem} always yields a matching at least as large as those given by Theorems~\ref{theorem:KR} and~\ref{theorem:V00} (even given the optimal choice of \(s, (D_j)\) in the latter).\COMMENT{The full proof of both of these implications is in the code here, commented out with \%}

In fact, if~\(C_2(H)\) is larger than the multiplicity of the hypergraph~\(C_{k+1}(H)\), then Theorem~\ref{theorem:maintheorem} usually yields a matching which is (often considerably) larger than Theorems~\ref{theorem:KR} and~\ref{theorem:V00} (and~\ref{theorem:KKMO}).
To give a brief example (we will discuss many more in Section~\ref{section:appstatements}), if we apply Theorem~\ref{theorem:maintheorem} to~\(H^{*}\) as described earlier, we can set \(D_j\coloneqq D_{k/2 +1}=\Theta(n^{k/2})\) for all \(k/2 +1<j\leq k+1\), and \(\eps\coloneqq 1/n\) (say).\COMMENT{Here giving an example of taking a perfectly regular hypergraph so just choosing a non-bottlenecking value of \(\eps>0\)}
Then \(\sqrt{D/D_2}\approx\sqrt{n}\), \((D/D_{j})^{1/(j-1)}\approx n\) for all \(2\leq j\leq k/2 +1\), \((D/D_j)^{1/(j-1)}\approx n^{(k/2)/(j-1)}\geq n^{(k/2)/(k)}=\sqrt{n}\) for all \(k/2 +1 <j\leq k+1\), and \(1/\eps=n\), so we may take \(B\approx\sqrt{n}\), attaining leftover size approximately~\(\sqrt{n}\), a substantial improvement over the \(n^{1-1/k}, n^{1-2/(k+2)}, n^{1-1/k-f_{(\ref{theorem:KKMO})}(k)}\) from Theorems~\ref{theorem:KR},~\ref{theorem:V00},~\ref{theorem:KKMO} respectively, and independent of~\(k\).
The rough heuristic is that Theorem~\ref{theorem:KR} stops the nibble process after collapsing the ratio \(D/D_2\), Theorem~\ref{theorem:KKMO} pushes slightly past this boundary with extra techniques, Theorem~\ref{theorem:V00} instead collapses the ratio \(D_{s-1}/D_s\) if each~\(D_j/D_{j+1}\) is large enough (which is better if enough information on the codegrees is known as~\(D\) decreases further in the meantime, as discussed), and Theorem~\ref{theorem:maintheorem} continues to iterate the nibble process until all ratios are collapsed, i.e.\ \(D\) to~\(D_{k+1}\) (if the ``bottlenecks'' permit, as we discuss next), which aligns with Theorems~\ref{theorem:KR} and~\ref{theorem:V00} (and is beaten by Theorem~\ref{theorem:KKMO}) if \(D_2=D_s=D_{k+1}\) and is generally better otherwise, often substantially.

The limiting expression for~\(B\) is intended to indicate the bottleneck for the input hypergraph, and our proof methodology shows that we may nibble, decreasing~\(D\) and all~\(D_j\) towards the time at which~\(D\) has become~\(D_{k+1}\) (the multiplicity of the hypergraph, namely 1 if there are no edge copies), up until the true bottleneck for~\(H\) stops the process.
If \((D/D_{k+1})^{1/k}\) is the limiting expression, then the process was able to utilize all space in the codegree sequence (by the binomial heuristic, if \(p=(D_{k+1}/D)^{1/k}\)  and \(pn\) vertices remain, then the average degree in the leftover looks like \(Dp^k=D_{k+1}\)).
If the limiting expression is~\(1/\eps\), then the bottleneck was the initial control over the vertex degree error, and the process aborts when this error grows to essentially~\(\Theta(1)\).
Regarding the bottlenecks \((D/D_j)^{1/(j-1)}\) for~\(j<k+1\), if the leftover has size~\(pn\), then by the binomial heuristic,~\(D\) and \(D_j\) should have decreased to around~\(Dp^k\) and~\(D_j p^{k+1-j}\) respectively.
Observe that \(p=(D_j/D)^{1/(j-1)}\) yields \(Dp^k=D_j p^{k+1-j}\),\COMMENT{If \(p=(D_j/D)^{1/(j-1)}\) then we have \(Dp^k=D^{1-k/(j-1)}D_j^{k/(j-1)}\) and \(D_j p^{k+1-j}=D_j^{1+(k+1-j)/(j-1)}D^{-(k+1-j)/(j-1)}=D_j^{(j-1+k+1-j)/(j-1)}D^{(j-1-k)/(j-1)}=D_j^{k/(j-1)}D^{1-k/(j-1)}=Dp^k\)} so that when the nibble reaches leftover~\(pn\), the current value of~\(D\) has dropped to the current value of~\(D_j\), which means we can no longer concentrate the vertex degrees in a nibble and must halt.
The expression~\(\sqrt{D/D_2}\) appears because the pure nibble methodology is incapable of utilizing an initial~\(\eps\) better than~\(\sqrt{D_2 / D}\), in the sense that we cannot ask for tighter control in the vertex degrees after one nibble than this.
The ratios \((D/D_2)^{1/1}, (D/D_3)^{1/2}\) do not appear in Theorem~\ref{theorem:maintheorem} because~\(\sqrt{D/D_2}\) is already more limiting.
We remark that it does not strengthen Theorem~\ref{theorem:maintheorem} to allow the user to pick some~\(s<k+1\) and use only the codegrees up to~\(D_s\).
Indeed, since~\(D\) must remain above~\(D_s\), the best leftover that can be attained is~\(n(D_s/D)^{1/k}\), which is at least \(n(D_j/D)^{1/(j-1)}\) for all \(s\leq j\leq k+1\). 
For these reasons, we believe that, up to the \(B^{\gamma}\log^A D\) error term, Theorem~\ref{theorem:maintheorem} is the optimal usage of pure nibble methodology.
\subsection{Pseudorandom hypergraph matchings}\label{section:pseudo}
The latter portion of Theorem~\ref{theorem:maintheorem} concerning the family~\(\cT\) of functions we frequently call ``weight functions'' is an example of a ``pseudorandom'' hypergraph matching result, wherein one wishes to show the leftover hypergraph is not just small but also ``random-looking''.
This turns out to be very useful in many applications, for example if one is using an absorption strategy in parallel, as in, e.g. \cite{BK21,K14,GKLO23,GKMO21,GKKO20,CDGKO24,KKKMO23,KS25}.
Theorem~\ref{theorem:KKMO} can also be considered such an application, since the authors of~\cite{KKMO23} tracked and used the surviving ``augmenting star'' subgraphs to obtain an even larger matching.
Yet another application of this kind of pseudorandom matching result is the lower bound on gaps between prime numbers~\cite{FGKMT18} (the pseudorandom tool here is stated for ``coverings'' rather than matchings).

The first pseudorandom hypergraph matching result was given by Alon and Yuster~\cite{AY05} in 2005, who showed that roughly vertex-regular hypergraphs with reasonable gap~\(D/D_2\) admit a large matching such that, for any (permissibly quite large) pre-chosen family~\(\cF\) of (not too small) vertex subsets, the matching covers a proportion of each \(F\in\cF\) roughly in line with expectations.
Ehard, Glock, and Joos~\cite{EGJ20} strengthened this in 2020, both quantitatively and by instead considering the more general notion of a family of ``weight functions'' on the edges \(\omega\colon E(H)\rightarrow\mathbb{R}_{\geq0}\) (one obtains the \(\cF\)-notion of pseudorandomness by setting \(\omega_F(e)\coloneqq|e\cap F|\) for each \(F\in\cF\)) and even on $j$-tuples of edges.
More recently, some notions of pseudorandom matchings have been considered in, for example,~\cite{GJKKL24,DP24,JMS24}, where almost-perfect matchings are found which are ``conflict-free'', or avoiding ``forbidden configurations'', with applications including the existence of high-girth Steiner systems \cite{DP24girth} and new bounds on generalized Ramsey numbers~\cite{JM24}.

Our \(\cT\)-notion of pseudorandomness is stronger than the \(\cF\)-notion in~\cite{AY05} (for each \(F\in\cF\) set \(\tau_F(v)\coloneqq 1\) for all \(v\in F\)), but not as strong as the notions of Ehard, Glock, and Joos \cite{EGJ20}, and we do not consider ``conflict-free'' matchings or matchings avoiding ``forbidden configurations'' as in~\cite{GJKKL24,DP24,JMS24}.
We chose to study weight functions on the vertices for three reasons.
First, our primary focus is on the improvement to the size of the leftover in Theorem~\ref{theorem:maintheorem}; secondly, our notion of pseudorandomness requires a surprisingly small amount of extra work to obtain; and thirdly, this notion is sufficiently strong to ``bootstrap'' our main result to three generalizations (Theorems~\ref{theorem:withreserves}--\ref{theorem:maincolourtheorem}), which are sufficient for our applications, including a new bound for the list chromatic index of hypergraphs, which we discuss in the next subsection.
Naturally, it would be interesting to see if our result Theorem~\ref{theorem:maintheorem} on the size of the leftover can be combined with the even stronger pseudorandomness notions as in \cite{EGJ20} or used to construct ``conflict-free'' matchings / matchings avoiding ``forbidden configurations'' as in \cite{GJKKL24,DP24,JMS24}.
For further information on pseudorandom hypergraph matching results and their applicability, see \cite{KKKMO,Ke24-nibble}.

\subsection{Beyond matchings}\label{section:bey}
In this subsection, we state three new results (Theorems~\ref{theorem:withreserves}--\ref{theorem:maincolourtheorem}) which follow from Theorem~\ref{theorem:maintheorem} by using the weight functions \(\tau\in\cT\) to complete an almost-perfect matching into a perfect one, using some reserve set in an ``absorption'' argument.
(We remark that we use the term absorption more liberally here than other authors may, as in all of our results there is more ``wiggle room'' than a more narrow interpretation of absorption would permit.)
Theorems~\ref{theorem:withreserves} and~\ref{thm:mainpartitetheorem} (proven in Section~\ref{section:beyond}) are intended to make easier many applications of Theorem~\ref{theorem:maintheorem} in which one wishes to perform some parallel absorption argument, as these results essentially incorporate the absorption into their machinery.

For disjoint sets~\(X,Y\), we say that a hypergraph~\(H\) is \textit{bipartite with bipartition}~\((X,Y)\) if \(V(H)=X\cup Y\) and all edges of~\(H\) have exactly one vertex in~\(X\).
A matching~\(M\) in such a hypergraph is \(X\)-\textit{perfect} if every vertex of~\(X\) is covered by~\(M\).
Our first result in this subsection is the following theorem.
The reserve hypergraph is given, namely~\(H_2\), and Theorem~\ref{theorem:maintheorem} is applied to~\(H_1\), with the weight functions used to ensure the leftover can be absorbed by~\(H_2\).
Here and throughout the paper, a hypergraph is \((k+1)\)-\emph{bounded} if all edges have at most~\(k+1\) vertices.
\begin{restatable}{theorem}{reserves}\label{theorem:withreserves}
Suppose \(1/D\ll1/A\ll\gamma\ll1/k,1/\ell\leq 1\).
Let~\(H_1\) and~\(H_2\) be edge-disjoint hypergraphs, where~\(H_1\) is \((k+1)\)-bounded and~\(H_2\) is \((\ell+1)\)-bounded.\COMMENT{The boundedness value \(k+1\) on \(H_1\) is important as it directly impacts~\(B\). The boundedness value \(\ell+1\) on \(H_2\) is not important, but it would be bad to use \(k+1\) for both as in some applications \(H_2\) may have \(\ell>k\), and this would unnecessarily impact your constraints on~\(B\) if you just put \(\ell=k\) in this result. We still need \(\ell\) to be ``constant'' though, as we later apply Lemma~\ref{lemma:bipmopup} within~\(H_2\) and have left degrees a \(\log^{A/2}D\) factor larger than the right degrees. Really all we need is \(1/D\ll1/\ell\), i.e. we don't need \(1/A\ll\gamma\ll1/\ell\), but does the current presentation seem alright as is?}
Define \(V_i\coloneqq V(H_i)\) for each \(i\in[2]\), suppose that~\(H_2\) is bipartite with bipartition~\((X,Y)\), and suppose that \(V_1\cap V_2 = X\).
Suppose that:
\begin{enumerate}[label=\upshape(\roman*)]
\item \(\Delta(H_1)\leq D\) and \(D_2\geq D_3\geq\dots\geq D_{k+1}\) are numbers satisfying \(C_j(H_1)\leq D_j\) for each \(j\in\{2,3,\dots,k+1\}\);\label{withdeg}
\item \(\Delta_Y\) is a number satisfying \(\max_{y\in Y}\text{deg}_{H_2}(y)\leq\Delta_Y\leq D^{(\log D) - \gamma}\);\label{withmaxY}
\item \(q\) is a number satisfying \(\max_{x\in X, y\in Y}\text{codeg}_{H_2}(x,y)\leq q\);\label{withq}
\item \(B\) is a number satisfying
\[
1\leq B\leq \min\left\{\sqrt{\frac{D}{D_2}}, \min_{j\in\{4,5,\dots,k+1\}}\left(\frac{D}{D_j}\right)^{\frac{1}{j-1}},\frac{\Delta_Y}{q}\right\};
\]\label{withB}
\item \(\text{deg}_{H_1}(x)\geq(1-B^{-1})D\) for all \(x\in X\);\label{withminH1}
\item \(\text{deg}_{H_2}(x)\geq \Delta_Y B^{-1+\gamma}\log^A D\) for all \(x\in X\).\label{withminX}
\end{enumerate}
Then \(H\coloneqq H_1\cup H_2\) has an \(X\)-perfect matching.
\end{restatable}
To our knowledge, there are two results of this flavour in the literature.
The first is due to Delcourt and Postle~\cite[Theorem 2.8]{DP24res}, who showed that, assuming \(1/D\ll\alpha\ll\beta<1\), we get an \(X\)-perfect matching if \(\Delta(H_1)\leq D\), \(\text{deg}_{H_2}(y)\leq D\) for all \(y\in Y\), \(C_2(H_1)\leq D^{1-\beta}\) and we replace the right side of conditions~\ref{withminH1} and~\ref{withminX} above with~\(D(1-D^{-\beta})\) and~\(D^{1-\alpha}\) respectively.
They have no hypothesis in the style of~\ref{withq}.
They used this result to provide the most recent proof of the Existence Conjecture for designs (we return to this in Section~\ref{section:designs}).
Our additional hypotheses and the strength of our nibble blackbox (Theorem~\ref{theorem:maintheorem}) enable us to make the substantially weaker hypothesis~\ref{withminX} on the \(H_2\) \(X\)-degrees.
This turns out to be essential in our proof of Theorem~\ref{thm:mainpartitetheorem} (using Theorem~\ref{theorem:withreserves}) and its applications.
Notice that we have in some sense ``disentangled''~\(H_1\) and~\(H_2\) when comparing Theorem~\ref{theorem:withreserves} to~\cite[Theorem 2.8]{DP24res}; for example the latter result hypothesizes \(\Delta(H_1),\Delta_Y\leq D\), whereas we allow~\(\Delta_Y\) to be much smaller (or larger).
This is also important for us; in the proof of Theorem~\ref{thm:mainpartitetheorem}, we select~\(H_2\) to be a very sparse random subgraph, and using an appropriately tight bound for~\(\Delta_Y\) is required for obtaining~\ref{withminX}.

The second result in the flavour of Theorem~\ref{theorem:withreserves} that we wish to highlight is the main theorem of Joos, Mubayi, and Smith~\cite{JMS24}, which we do not state here as it requires the statement of many conditions.
They also ``disentangled''~\(H_1\) and~\(H_2\); indeed, their result is phrased as a ``tripartite'' matching theorem with a hypergraph~\(H_1\) between parts~\(P\) and~\(Q\) and a bipartite hypergraph~\(H_2\) between parts~\(P\) and~\(R\).
Here~\(P,Q,R\) essentially translate as \(X, V_1\setminus X, Y\) respectively, and there are a number of hypotheses on the individual bounds of vertex degrees in different parts (and a hypothesis in the style of our~\ref{withq}).
The theorem of Joos, Mubayi, and Smith incorporates strong notions of the resulting matching being ``conflict-free'' (Delcourt and Postle~\cite[Theorem 2.10]{DP24girth} also have such a version which they used in their proof of the High-girth Existence Conjecture), which we do not obtain in Theorem~\ref{theorem:withreserves}, but, similarly to~\cite[Theorem 2.8]{DP24res}, we greatly improve upon the degree bounds that one must satisfy.
For example, in place of their hypothesis \(\text{deg}_{H_2}(x)\geq \Delta_Y D^{-\eps^4}\) for small~\(\eps\) and all \(x\in X\), we have the hypothesis \(\text{deg}_{H_2}(x)\geq \Delta_Y B^{-1+\gamma}\log^A D\) (which can even look as small as roughly \(\Delta_Y n^{-1/2}\), for example in the proof of Theorem~\ref{theorem:rainbowlinks}).

Notice that the \((n,D,\eps)\)-regular assumption of Theorem~\ref{theorem:maintheorem} has been replaced in Theorem~\ref{theorem:withreserves} by the weaker assumption \(\Delta(H_1)\leq D\), and the uniform assumptions have been weakened to boundedness assumptions (and these weaker assumptions are essentially carried forward to Theorems~\ref{thm:mainpartitetheorem} and~\ref{theorem:maincolourtheorem}).

The following new theorem incorporates the finding of a suitable reservoir (the hypergraph~\(H_2\)) into its proof (and then applies Theorem~\ref{theorem:withreserves}), reducing the problem of finding an \(X\)-perfect matching to simply checking the bounds on the degrees and codegrees.
\begin{theorem}[Bipartite Matching Theorem]\label{thm:mainpartitetheorem}
Suppose \(1/D\ll 1/A \ll \gamma \ll 1/k\leq 1\),  and let~\(H\) be a \((k+1)\)-bounded bipartite hypergraph with bipartition $(X, Y)$. 
Suppose \(\text{deg}_H(y)\leq D\) for all \(y\in Y\), and that \(D_2\geq D_3\geq\dots\geq D_{k+1}\) are numbers satisfying \(C_j(H)\leq D_j\) for each \(j\in\{2,3,\dots,k+1\}\).
Suppose~\(B\) is a number satisfying
\[
1\leq B\leq\min\left\{\sqrt{\frac{D}{D_2}}, \min_{j\in\{4,5,\dots,k+1\}}\left(\frac{D}{D_j}\right)^{\frac{1}{j-1}}\right\}.
\]
Suppose further that \(\text{deg}_H(x)\geq (1+B^{-1+\gamma}\log^A D)D\) for all \(x\in X\).
Then~\(H\) has an $X$-perfect matching.
\end{theorem}
In the absence of any information on the codegrees, a result of Haxell~\cite{Ha01} on ``independent transversals'' implies (via consideration of line graphs) a version of Theorem~\ref{thm:mainpartitetheorem} in which the \(X\)-degrees are bounded from below by~\(2kD\).
Delcourt and Postle~\cite[Theorem 2.6]{DP24} showed that we can improve this bound to~\((1+o(1))D\) if we assume that~\(D_2 \ll D\).
A little more specifically,~\cite[Theorem 2.6]{DP24} states that, if for some positive constant~\(\beta\) we have \(D_2\leq D^{1-\beta}\), then for some~\(\alpha>0\) sufficiently small with respect to~\(\beta\) and~\(k\), the result holds if the \(Y\)-degrees are at most~\(D\) and the \(X\)-degrees are at least \((1+D^{-\alpha})D\).
By using all of the information on the codegree sequence, we improve the~\(D^{-\alpha}\) term to a~\(B^{-1+\gamma}\log^A D\) term, which can be significantly smaller (even roughly~\(n^{-1/2}\), as discussed above).
We demonstrate the usage of Theorem~\ref{thm:mainpartitetheorem} when we prove Theorems~\ref{theorem:rainbowlinks} and~\ref{theorem:SimplicialComplex} in Section~\ref{Section:Applications}.

Theorem~\ref{thm:mainpartitetheorem} also very quickly yields a new state-of-the-art result on the list chromatic index of hypergraphs (Theorem~\ref{theorem:maincolourtheorem}).
A colouring of~\(E(H)\) is \textit{proper} if each vertex is in no more than one edge of each colour.
The \textit{chromatic index}~\(\chi'(H)\) is the minimum non-negative integer~\(r\) such that there is a proper colouring of~\(E(H)\) with~\(r\) colours.
If each edge has an associated list of colours, and a proper colouring assigns each edge a colour from its list, then that colouring is \textit{acceptable}, and the \textit{list chromatic index}~\(\chi_{\ell}'(H)\) is the minimum non-negative integer~\(r\) such that there exists an acceptable colouring whenever all edges are assigned lists of at least~\(r\) colours.
\begin{theorem}\label{theorem:maincolourtheorem}
Suppose \(1/D\ll 1/A\ll\gamma\ll1/k\leq 1\), and let~\(H\) be a \((k+1)\)-bounded hypergraph with \(\Delta(H)\leq D\).
Suppose \(D_2\geq D_3\geq\dots\geq D_{k+1}\) are numbers satisfying \(C_{j}(H)\leq D_j\) for each \(j\in\{2,3,\dots,k+1\}\), and that~\(B\) is a number satisfying
\[
1\leq B\leq\min\left\{D^{\frac{1}{k+1}}, \sqrt{\frac{D}{D_2}}, \min_{j\in\{4,5,\dots,k+1\}}\left(\frac{D}{D_j}\right)^{\frac{1}{j-1}}\right\}.
\]
Then \(\chi'_{\ell}(H)\leq(1+B^{-1+\gamma}\log^{A}D)D\).
\end{theorem}
\begin{proof}
Suppose each edge \(e\in E(H)\) admits a list~\(L(e)\) of at least \((1+B^{-1+\gamma}\log^A D)D\) colours.
Construct a \((k+2)\)-bounded bipartite auxiliary hypergraph~\(\cH\) with parts \(X\coloneqq E(H)\) and \(Y\coloneqq V(H)\times\bigcup_{e\in E(H)}L(e)\), and \(E(\cH)=\{\{e\}\cup\{(v,c)\colon v\in e\}\colon e\in E(H), c\in L(e)\}\).
Then clearly \(\text{deg}_{\cH}(x)\geq (1+B^{-1+\gamma}\log^A D)D\) for all \(x\in X\) and \(\text{deg}_{\cH}(y)\leq D\) for all \(y\in Y\).
It is easy to check\COMMENT{If you take a set of 2 or more \(\cH\)-vertices and it includes 2 vertices in \(X\) then codegree 0. If 1 vertex in \(X\) then codegree \(\leq 1\) (1 only if the `vertices' you have on the right are all in this edge and the same colour). If 2 vertices on right then codeg only >0 if the same colour, but then its the same as their codeg in \(H\)} that \(C_j(\cH)\leq D_j\) for all \(2\leq j\leq k+1\), and we have \(C_{k+2}(\cH)=1\), so we may put \(D_{k+2}=1\).
Clearly the~\(B\) as in the hypothesis of the present theorem satisfies the hypotheses of Theorem~\ref{thm:mainpartitetheorem} (in particular notice \(B\leq D^{1/(k+1)}=(D/D_{k+2})^{1/(k+1)}\)), so Theorem~\ref{thm:mainpartitetheorem} yields an \(X\)-perfect matching of~\(\cH\), which corresponds to a proper colouring of~\(E(H)\) in which each edge receives a colour from its list.
\end{proof}
Before Theorem~\ref{theorem:maincolourtheorem}, the best-known result on the (non-list) chromatic index~\(\chi'(H)\) was the following statement of Kang, K\"uhn, Methuku, and Osthus~\cite[Corollary 1.6]{KKMO23}, obtained from their main matching theorem (Theorem~\ref{theorem:KKMO} in the current paper).
\begin{theorem}[Kang, K\"uhn, Methuku, and Osthus~\cite{KKMO23}]\label{KKMOcolouring}
Suppose that \(1/D\ll 1/k,\gamma,\mu,\eta,\delta<1\), where \(k\geq 2\) and \(\eta<f_{(\ref{theorem:KKMO})}(k+1)\).
Suppose further that \(D\geq\exp(\log^{\mu}n)\).
Let~\(H\) be a \((k+1)\)-uniform multihypergraph on~\(n\) vertices, \(\Delta(H)\leq D\), and \(C_2(H)\leq D_2\leq D^{1-\gamma}\).
Then \(\chi'(H)\leq(1+(D/D_2)^{-1/(k+1) - \eta})D\).
\end{theorem}
The best-known result on the list chromatic index~\(\chi'_{\ell}(H)\) was still the following theorem of Molloy and Reed~\cite[Theorem 2]{MR00} from 2000, which asymptotically improved Kahn's~\cite{Ka96} bound.

\begin{theorem}[Molloy and Reed~\cite{MR00}]\label{MRcolouring}
Suppose \(1/D\ll1/k\leq 1\), and let~\(H\) be a \((k+1)\)-uniform hypergraph with \(\Delta(H)\leq D\) and \(C_2(H)\leq D_2\).
Then \(\chi'_{\ell}(H)\leq(1+(D/D_2)^{-1/(k+1)}\log^5 D)D\).
\end{theorem}
Theorem~\ref{KKMOcolouring} beats (the non-list version of) Theorem~\ref{MRcolouring} in a similar fashion to the way in which Theorem~\ref{theorem:KKMO} beats Theorem~\ref{theorem:KR} for matchings, i.e.\ by the~\((D/D_2)^{\eta}\) factor.
If \(D_2\approx 1\), then Theorem~\ref{theorem:maincolourtheorem} yields the same result as Theorem~\ref{MRcolouring} up to the error term, and is beaten by Theorem~\ref{KKMOcolouring}.
If \(D_2\approx D_{k+1}=\omega(1)\), then the~\(B\) of Theorem~\ref{theorem:maincolourtheorem} is constrained by the minimum of~\(D^{1/(k+1)}\) and~\((D/D_2)^{1/k}\), each of which is larger than \((D/D_2)^{1/(k+1)}\), so Theorem~\ref{theorem:maincolourtheorem} beats the other two results (assuming~\(D_{k+1}\) is a large enough~\(\omega(1)\) to beat the~\(\eta\) factor).
If instead \(D_2\gg D_{k+1}\), then this only strengthens Theorem~\ref{theorem:maincolourtheorem}, without impacting the other two results.
For more dense hypergraphs, Theorem~\ref{theorem:maincolourtheorem} beats the other two results (often substantially).
Indeed, we have in general that \(D^{1/(k+1)}\geq (D/D_2)^{1/(k+1)}\), \(\sqrt{D/D_2}\geq (D/D_2)^{1/(k+1)}\) since \(k\geq 1\), and \((D/D_j)^{1/(j-1)}>(D/D_j)^{1/(k+1)}\geq (D/D_2)^{1/(k+1)}\).
We discuss more specific comparisons in the context of applications to design theory in Section~\ref{section:appstatements}.
\subsection{Applications}\label{section:appstatements}
To further demonstrate the power and usability of Theorems~\ref{theorem:maintheorem},~\ref{thm:mainpartitetheorem}, and~\ref{theorem:maincolourtheorem}, we now discuss a number of applications, most of which yield new, best-known results.
Each is proven in Section~\ref{Section:Applications}.
We note that in each of the applications of Theorems~\ref{theorem:maintheorem} and~\ref{thm:mainpartitetheorem}, the bottleneck for~\(B\) turns out to be~\((D/D_{k+1})^{1/k}\), so that the nibble process halts only once all the space in the codegree sequence has been used up, i.e.\ \(D\) has dropped to roughly the multiplicity~\(D_{k+1}\).
\subsubsection{Latin squares}
The first applications we wish to discuss are in the setting of \emph{Latin squares}, which are~\(n\times n\) arrays of~\(n\) symbols, such that each row and each column contains exactly one instance of each symbol.
The set~\(\cL_n\) of~\(n\times n\) Latin squares with symbol set~\([n]\) is in bijection with a set of colourings of a directed graph we describe now.
Let~\(\dirK\) be the digraph obtained from the complete graph~\(K_n\) by adding a directed loop at each vertex and replacing each undirected edge with two arcs; one in each direction.
A \textit{proper} \(n\)-\textit{arc colouring} of~\(\dirK\) is a colouring of~\(E(\dirK)\) using~\(n\) colours, such that each vertex is the tail of exactly one arc of each colour, and also the head of exactly one arc of each colour.
Let~\(\Phi(\dirK)\) be the set of proper \(n\)-arc colourings of~\(\dirK\) using the vertex set \(V=[n]\) and colour set \(C=[n]\).
We slightly abuse notation and write \(G\in\Phi(\dirK)\) to mean~\(G\) is a copy of~\(\dirK\) on vertex set~\(V\), equipped with a proper \(n\)-arc colouring using colour set~\(C\).
It is easy to see that \(f\colon \Phi(\dirK)\rightarrow\cL_n\) is a bijection, where~\(f\) puts the colour of arc~\(ij\) in cell~\((i,j)\).
The following is our first application.
Here, an edge-coloured graph is \textit{rainbow} if its edges have distinct colours.

\begin{theorem}\label{theorem:rainbowlinks}
Let \(\delta>0\) and suppose \(1/n\ll\delta\).
Let \(G\in\Phi(\dirK)\).
Then~\(G\) contains a rainbow directed cycle covering all but at most~\(n^{1/2 + \delta}\) vertices.
\end{theorem}
The best result in this direction before Theorem~\ref{theorem:rainbowlinks} was a result of Benzing, Pokrovskiy, and Sudakov~\cite{BPS20}, who proved that each \(G\in\Phi(\dirK)\) admits a rainbow directed cycle covering all but at most~\(O(n^{4/5})\) vertices.
Theorem~\ref{theorem:rainbowlinks} thus constitutes the new best result towards a conjecture of Gy\'arf\'as and S\'ark\"ozy~\cite{GS14} on ``cycle-free partial transversals'' in Latin squares (removing one arc from the cycle obtained by Theorem~\ref{theorem:rainbowlinks} yields a cycle-free partial transversal in the corresponding Latin square), and the optimal colourings case of a conjecture of the current authors~\cite{GK23} on almost-spanning rainbow cycles in~\(G\in\Phi(\dirK)\) (named ``almost-Hamilton transversals'' in the Latin square setting). We note also that, up to the subpolynomial error term, Theorem~\ref{theorem:rainbowlinks} matches the result of Balogh and Molla \cite{BM19}  in the undirected setting (providing a different proof), which in turn improved a result of Alon, Pokrovskiy, and Sudakov \cite{APS17}.

If we run a random greedy self-avoiding rainbow walk on \(G\in\Phi(\dirK)\) until the leftover has~\(pn\) vertices and colours, then \(p=n^{-1/2}\) corresponds to the time at which we should expect to only have a constant number of candidates for the next vertex in a self-avoiding rainbow extension of the walk.
We remark that it may be possible to prove Theorem~\ref{theorem:rainbowlinks} this way, but the proof would likely be very technical.

We prove Theorem~\ref{theorem:rainbowlinks} in Section~\ref{section:raincycles} by first randomly partitioning the vertices and colours into different `slices', then constructing an auxiliary bipartite hypergraph~\(H\) (with bipartition~\((X,Y)\), say) in which edges correspond to short rainbow directed paths of~\(G\), having say~\(\ell\) internal vertices, and an \(X\)-perfect matching is the desired cycle.
A little more specifically, we select a random set~\(W_0=\{w_1,w_2,\dots,w_m\}\) of \(G\)'s vertices, with~\(m\) a little less than \(n/(\ell+1)\), and put \(X=\{(w_i,w_{i+1})\}\), indices modulo~\(m\), with an \(H\)-edge containing~\((w_i,w_{i+1})\) encoding a rainbow path in~\(G\) with~\(\ell\) internal vertices from~\(w_i\) to~\(w_{i+1}\).
Using the properties of the random partition, we check that the degrees and codegrees are such that Theorem~\ref{thm:mainpartitetheorem} can be applied, which then yields the result.
Up to constant factors, the codegree sequence \(D, D_2, \dots, D_{2\ell+2}\) for~\(H\) is \((n^{\ell}, n^{\ell-1}, n^{\ell-1}, \dots, n, n, 1,1,1)\). 
The trick is that choosing~\(m\) just a little less than~\(n/(\ell+1)\) ensures that the degrees of the~\(X\)-vertices in~\(H\) are larger than the degrees of the \(Y\)-vertices by just enough for Theorem~\ref{thm:mainpartitetheorem} to apply.
The improved asymptotics of Theorem~\ref{thm:mainpartitetheorem} (compared to~\cite[Theorem 2.6]{DP24}, say) thus directly lead to the improved asymptotics in Theorem~\ref{theorem:rainbowlinks}.

We note that the random partition does not include the finding of a small ``reservoir'' (which we would need if were instead using Theorem \ref{theorem:maintheorem} or \ref{theorem:withreserves}, which would require significantly more work), so that all ``slices'' are large, on the order~\(n\).
That is, we do not need to develop an absorption strategy, as this is ``built in'' to the machinery of Theorem~\ref{thm:mainpartitetheorem}.
In fact, we only need the partition to ensure that each colour cannot have too many ``roles'' in a rainbow path, which helps to keep the degrees of colours in~\(H\) low.
We leave the remaining details to Section~\ref{section:raincycles}.

Instead of finding one large rainbow cycle in \(G\in\Phi(\dirK)\), one may instead search for a rainbow collection of small cycles.
We say a `partial rainbow directed triangle factor' is a rainbow collection of disjoint triangles, each exhibiting consistent directions on the arcs.
\begin{theorem}\label{theorem:triangles}
Let \(\delta>0\) and suppose \(1/n\ll\delta\).
Let \(G\in\Phi(\dirK)\).
Then~\(G\) contains a partial rainbow directed triangle factor covering all but at most \(n^{3/5+\delta}\) vertices.
\end{theorem}
Since each vertex is in roughly~\(n^2\) of these triangles, if each vertex and colour survived to the leftover independently with probability~\(p\), we would expect the number of surviving rainbow triangles containing a given vertex or colour to become roughly constant if \(p=n^{-2/5}\), matching Theorem~\ref{theorem:triangles}.
We remark that the expected total number of surviving rainbow triangles is~\(\Theta(n^3p^6)\), and so leftover approximately~\(\sqrt{n}\) may be attainable via analysis of a random rainbow-triangle removal process, but this would likely be hard due to the need to analyze the number of rainbow triangles within large vertex subsets.

The proof of Theorem~\ref{theorem:triangles} is fairly short and relies upon applying Theorem~\ref{theorem:maintheorem} to a natural auxiliary hypergraph~\(H\) in which a large matching corresponds to a large partial rainbow triangle factor of~\(G\).
The codegree sequence of this hypergraph is \((D,D_2,D_3,D_4,D_5,D_6)=(n^2,3n,3n,6,6,2)\).
Applying Theorem~\ref{theorem:KR} to~\(H\) instead would yield leftover~\(n^{4/5+\delta}\), Theorem~\ref{theorem:KKMO} would give roughly~\(n^{0.796}\), and Theorem~\ref{theorem:V00} does not directly apply with \(s>2\) due to clustering in the codegree sequence.
However, with \(s=2\), Theorem~\ref{theorem:V00} yields~\(n^{4/5}\log^A n\), and tactically damaging some ratios \(D_j/D_{j+1}\) and taking larger~\(s\) yields~\(n^{3/4}\log^A n\) (and no better).\COMMENT{Undamaged, the codegree sequence is \(n^2, n, n, 1, 1, 1\), starting with \(D=n^2\). If \(s=3\) then \(x^{k-s+2}=x^4\), so we should push \(D_2=n\) up by at least \(n^{4/5}\) if we want \(x\geq n^{1/5}\), but then \(\sqrt{D/D_2}=n^{1/5}\), so this is actually worse. If \(s=4\) then \(x^{k-s+2}=x^3\). We already have \(D_3/D_4=n\) so this gives \(x\leq n^{1/3}\). But we also need \(x^2\leq D/D_2, D_2/D_3\), and clearly the best we can do here is put \(D_2\) in the middle, i.e.\ \(D=n^2, D_2=n^{3/2}, D_3=n, D_4=1\) yielding best \(x=n^{1/4}\), as claimed. Using any larger \(s\), we are still limited by \(x^2\leq D/D_2, D/D_3\) and the fact we can't achieve any better than \(x=n^{1/4}\) from these constraints. This I think actually shows the value of what we achieved in the changed approach to proving the Chomp, where the constraints for intermediate ratios were changed to \(x\) rather than \(x^2\).}
To be a little more specific, the best one can do with Theorem~\ref{theorem:V00} is selecting \(s=4\) and adjusting \((D,D_2,D_3,D_4)=(n^2,n^{3/2},3n,6)\), which still falls short of Theorem~\ref{theorem:triangles}.

Theorem~\ref{theorem:triangles} easily yields the following undirected analogue, which provides the new best known bound towards a conjecture of M\"uyesser~\cite[Conjecture 7.3]{M23}, which posits that there is some absolute constant~\(C\) such that the~\(n^{3/5+\delta}\) in Corollary~\ref{corollary:triangles} can be replaced by~\(C\).
This conjecture can, in turn, be seen as an extension of the \(k=3\) case of the Friedlander-Gordon-Tannenbaum Conjecture~\cite{FGT81} from group theory.
\begin{cor}\label{corollary:triangles}
Let \(\delta>0\) and suppose \(1/n\ll\delta\).
Suppose that~\(K_n\) is properly edge-coloured with~\(n\) colours.\COMMENT{I'm mimicking the setting of M\"{u}yesser Conjecture 7.3. Is this what we want, rather than optimal edge colourings?}
Then there is a partial rainbow triangle factor of~\(K_n\) covering all but at most~\(n^{3/5 + \delta}\) vertices.
\end{cor}
\begin{proof}
Relabel the vertex set and colour set \(V,C=[n]\).
Use the colouring on~\(K_n\) to construct an element \(G\in\Phi(\dirK)\) by replacing each \(c\)-edge with a \(c\)-arc in both directions, and colouring the loops with the only unused colour at each vertex.
The result follows by applying Theorem~\ref{theorem:triangles} to~\(G\), then ignoring the directions.
\end{proof}
\subsubsection{Simplicial complexes}
Our next application concerns the maximum diameter of a simplicial complex.
Let~\(n\) be a positive integer.
A (pure) \textit{simplicial complex}~\(\cK\) (on ``vertex set''~\([n]\)) is a collection of non-empty subsets of~\([n]\) such that if \(U\in\cK\), then every non-empty subset of~\(U\) is in~\(\cK\).
A simplicial complex~\(\cK\) has \textit{dimension}~\(d\) if the maximum size of a set in~\(\cK\) is~\(d+1\).
Those sets \(U\in\cK\) having \(|U|=d+1\) are the \textit{facets} of~\(\cK\).
Let~\(F(\cK)\) denote the set of facets of~\(\cK\).
The \textit{dual graph}~\(G(\cK)\) of~\(\cK\) is the graph with vertex set~\(V(G(\cK))=F(\cK)\) and an edge between~\(F,F'\) if \(|F\cap F'|=d\).
If~\(G(\cK)\) is connected, then we say that~\(\cK\) is \textit{strongly connected}.
If~\(\cK\) is strongly connected, then we say that the (combinatorial) \textit{diameter} of~\(\cK\) is the graph diameter of~\(G(\cK)\).
We let~\(H_s(n,d)\) denote the maximum diameter over all \(d\)-dimensional strongly connected simplicial complexes~\(\cK\) on~\([n]\).

A simple volume argument given by Santos~\cite[Corollary 2.7]{S13} shows that \(H_s(n,d)\leq (\binom{n}{d} - (d+1))/d \leq n^d/(d\cdot d!)\), and this is still the best-known upper bound on~\(H_s(n,d)\).
The first and currently only lower bound on~\(H_s(n,d)\) obtaining \(H_s(n,d)\geq(1-o(1))n^d/(d\cdot d!)\) was given by Bohman and Newman~\cite{BN22} in 2022, who used a random greedy approach to construct a strongly connected~\(\cK\) with diameter at least \((1/(d\cdot d!) - (\log n)^{-\eps})n^d\) for \(\eps<1/d^2\) (and \(d\geq 2\)).
This improved on earlier lower bounds~\cite{S12,CS17,CN21} which were not asymptotic to the upper bound.
In Section~\ref{Section:Applications}, we apply Theorem~\ref{thm:mainpartitetheorem} to obtain the following result, significantly improving the asymptotics of the lower bound given by Bohman and Newman~\cite{BN22}.
\begin{theorem}\label{theorem:SimplicialComplex}
Suppose that \(1/n\ll\delta\ll1/d\leq 1/2\).
Then there is a pure \(d\)-dimensional simplicial complex~\(\cK\) for which~\(G(\cK)\) is a cycle of length at least \((1-n^{-1/d + \delta})n^d/(d\cdot d!)\).
In particular, 
\[
H_s(n,d)\geq(1-n^{-1/d + \delta})\frac{n^d}{d\cdot d!}.
\]
\end{theorem}
Notice the second part of the result follows by removing any one facet from~\(\cK\) (after applying the first part with~\(\delta/2\) in the role of~\(\delta\), say).
Up to the~\(n^\delta\) factor, this attains the asymptotics that Bohman and Newman~\cite{BN22}, in their closing remarks, argued heuristically may hold, but stated that ``the proof of such a statement would likely be intricate if using currently available techniques'' and that one would likely need to ``develop sophisticated self-correcting estimates'' in conjunction with the random greedy approach to prove this result.

We include the portion of the statement concerning the Hamilton cycle~\(G(\cK)\) (rather than a path) as we believe this to invoke an interesting question:
Say a complex~\(\cK\) is \textit{cyclical} if~\(G(\cK)\) is a (Hamilton) cycle, and consider the underlying cyclic arrangement of the vertices of~\([n]\) used by~\(\cK\).
In this setting, say a \(d\)-set from~\([n]\) is \textit{covered} by~\(\cK\) if it a subset of exactly one facet, or of exactly two, consecutive facets.\COMMENT{Clearly a \(d\)-set \(U\) cannot be a subset of two non-consecutive facets, as then \(G(\cK)\) would include an edge between them, and \(U\) cannot be a subset of three or more facets, as then it is a subset of two non-consecutives}
Proceeding round the cycle, each new~\([n]\)-vertex corresponds to a new facet, which introduces~\(d\) new covered \(d\)-sets (all \(d\)-subsets of the newest~\(d+1\) vertices, except the one omitting the new vertex).
The volume argument of Santos thus implies that a cyclical complex has at most $\binom{n}{d}/d$ facets.
Provided~\(d\) divides~\(\binom{n}{d}\), is there a cyclical~\(\cK\) in which every \(d\)-set in~\([n]\) is covered? Recently, Parczyk, Rathke, and Szab\'{o}~\cite{PRS25} (see also~\cite{PRS25b,GPRS25} in their paper for further progress) posed a similar question and proved Santos' upper bound for $H_s(n,2)$ is tight for all $n\neq6$.
\COMMENT{Since each facet going round the cycle introduces \(d\) new \(d\)-sets (you could say you get \(d+1\) from the first you consider, but then it'll only give you a final \(d-1\) when you wrap all the way round), and each of the~\(\binom{n}{d}\) \(d\)-sets become covered, there must be \(\binom{n}{d}/d\) facets.} 

Similarly to the proof of Theorem~\ref{theorem:rainbowlinks}, we prove Theorem~\ref{theorem:SimplicialComplex} by using probabilistic techniques to first find a good choice~\(X\) of~\(m\) \(d\)-sets from~\([n]\), where~\(m\) is a little less than \(n^d/(d\cdot d!(\ell+d))\) for a large integer~\(\ell\).
We construct an auxiliary bipartite hypergraph~\(\cH\) in which edges correspond to (the \(d\)-sets ``covered'' by) sequences of~\(\ell\) ``internal'' vertices between a pair \((X_i,X_{i+1})\in\cX\) of consecutive elements of~\(X\), and an \(\cX\)-perfect matching yields a sequence of vertices in~\([n]\)  which induces a cyclic ordering on $(d+1)$-sets covering each of the \(X\)-elements in turn, with~\(\ell\) vertices in between each consecutive pair, with some useful additional properties.
The \(\cX\)-perfect matching can be thought of as a sequence of \(m\) ``corridors'' whose union is cyclical.
The choice of~\(m\) a little less than \(n^d/(d\cdot d!(\ell+d))\) ensures that the degrees of the \(X\)-pairs in~\(\cH\) are a little larger than the degrees of the other \(d\)-sets, so that Theorem~\ref{thm:mainpartitetheorem} yields the desired matching.
We construct the facets of~\(\cK\) to be the collection of \((d+1)\)-sets appearing consecutively in the ordering, and the additional properties ensure these \(m(\ell+d)\approx n^d/(d\cdot d!)\) facets are distinct and~\(G(\cK)\) contains only the desired Hamilton cycle.
\subsubsection{Designs}\label{section:designs}
Our remaining applications of Theorem~\ref{theorem:maintheorem} (and Theorem~\ref{theorem:maincolourtheorem}) pertain to design theory.
A \((t,r,n)\)-\textit{Steiner system} is a collection~\(S\) of \(r\)-subsets of an \(n\)-set (say \([n]\coloneqq\{1,2,\dots,n\}\)), such that each \(t\)-subset of~\([n]\) is contained in exactly one element of~\(S\).
Clearly this can be viewed as an \(r\)-uniform hypergraph on vertex set~\([n]\), and so Theorems~\ref{theorem:maintheorem} and~\ref{theorem:maincolourtheorem} apply.
We first discuss the literature.
The case of \((2,3,n)\)-Steiner systems (\emph{Steiner triple systems}) has received particular attention.
Regarding matchings in Steiner triple systems, Montgomery~\cite{Mont23} recently proved in a breakthrough paper that all sufficiently large Steiner triple systems have a matching of size at least~\((n-4)/3\), confirming a longstanding conjecture of Brouwer~\cite{B81} for large~\(n\). 
For \(1/n\ll1/r<1/t\leq 1/2\) other than \(t=2, r=3\), little is known, with the best bound in general previously coming from either Theorem~\ref{theorem:V00} or \ref{theorem:KKMO}.

The topic of edge colouring Steiner systems is closely related to the ``resolvability'' of designs.
A hypergraph~\(H\) is \emph{resolvable} if \(\chi'(H)=\Delta(H)\), and if a resolvable~\(H\) is \(D\)-regular, then~\(H\) admits a decomposition into perfect matchings.
Note that \((t,r,n)\)-Steiner systems can be viewed as a class of \((\binom{n-1}{t-1}/\binom{r-1}{t-1})\)-regular hypergraphs.
Thus, the 2006 conjecture of Meszka, Nedela, and Rosa~\cite{MNR06} that if \(n>7\), then every \((2,3,n)\)-Steiner system~\(S\) has \(\chi'(S)\leq (n-1)/2 + 3\) is akin to saying that such Steiner triple systems are ``almost resolvable''.\COMMENT{Considered mentioning Ferber and Kwan's "Almost all Steiner triple systems are almost resolvable" paper. But they're approaching it differently, more directly from the resolvability angle - they prove that most STS have \((1/2 - o(1))n\) disjoint PMs. You can't obtain an asymptotic bound on the chromatic index from their result by lazily making every remaining edge its own colour class since we have \(n^2/6 - o(n)\) already coloured edges and \(n^2/6 - n/6\) edges total, so you'd need another \(\Theta(n)\) colours in this way.}
We remark that Keevash~\cite{K18} proved that for all~\(r,t\) and sufficiently large~\(n\) satisfying some necessary divisibility conditions, there exist resolvable \((t,r,n)\)-Steiner systems, answering one of the oldest questions in combinatorics in the affirmative (for large~\(n\)).
Again, not many general bounds for~\(\chi'(S)\) for such \((t,r,n)\) are known, with the best bounds previously coming from Theorem~\ref{KKMOcolouring}.
The second author \cite{Ke24-nibble} conjectured that, for fixed $r$ and $t$, every $(t, r, n)$-Steiner system has chromatic index $\binom{n+O(1)}{t-1}\big/\binom{r-1}{t-1}$ and consequently a matching of size at least $n/r - O(1)$.

We apply Theorems~\ref{theorem:maintheorem} and~\ref{theorem:maincolourtheorem} to obtain (quickly) the following bounds on matchings and the list chromatic index in general \((t,r,n)\)-Steiner systems.
\begin{theorem}\label{theorem:Steinercolorsmatchings}
Suppose \(1/n\ll1/r<1/t\leq 1/2\), and suppose \(\delta>0\) satisfies \(1/n\ll\delta\).
Let~\(S\) be a (full) \((t,r,n)\)-Steiner system, and let~\(U(S)\) denote the number of vertices uncovered by a maximum matching of~\(S\).
Set \(D\coloneqq\binom{n-1}{t-1}/\binom{r-1}{t-1}\).
\begin{enumerate}[label=\upshape(\roman*)]
\item If \(t<(r+1)/2\) then \(U(S)\leq n^{\frac{r-t}{r-1}+\delta}\), and if \(t\geq(r+1)/2\) then \(U(S)\leq n^{1/2+\delta}\);
\item If \(t< \frac{r}{2}+1\) then \(\chi'_{\ell}(S)\leq\left(1+n^{-\frac{t-1}{r}+\delta}\right)D\), and if \(t\geq \frac{r}{2}+1\) then \(\chi'_{\ell}(S)\leq(1+n^{-1/2+\delta})D\).
\end{enumerate}
\end{theorem}
We remark that Theorem~\ref{theorem:maintheorem} can be replaced in the proof of Theorem~\ref{theorem:Steinercolorsmatchings} by Theorems~\ref{theorem:KR},~\ref{theorem:V00}, or~\ref{theorem:KKMO}, which yield \(U(S)\leq n^{(r-2)/(r-1)+\delta}\), \(U(S)\leq n^{(r-t)/(r-t+1)}\log^C n\), \(U(S)\leq n^{(r-2)/(r-1)-f_{(\ref{theorem:KKMO})}(r-1)}\) (each for any~\(t\)) respectively (though we leave the verification of these bounds to the reader).\COMMENT{From proof of Theorem~\ref{theorem:Steinercolorsmatchings}, we have \(D=\binom{n-1}{t-1}/\binom{r-1}{t-1}\), \(D_2=\binom{n-2}{t-2}/\binom{r-2}{t-2}\), and uniformity \(k+1=r\). Then \(D/D_2=(n-1)/(r-1)\). Then \(n(D/D_2)^{-1/k}\approx n^{1-1/(r-1)}=n^{(r-2)/(r-1)}\). Meanwhile for Vu, we have \(D=\Theta(n^{t-1})\) and \(D_j=\Theta(n^{t-j})\) for \(2\leq j\leq t\). Take \(s=t\). Then~\(x\) is constrained by \(\sqrt{n}\) from the intermediate ratios \(D_j/D_{j+1}\), and by \(x^{r-1-t+2}\leq n\) from the final ratio, i.e.\ \(x\leq n^{1/(r-t+1)}\), which is more limiting since \(r-t+1\geq 2\), so leftover essentially \(n^{1-1/(r-t+1)}=n^{(r-t)/(r-t+1)}\).}
Similarly, Theorem~\ref{MRcolouring} yields \(\chi'_{\ell}(S)\leq(1+n^{-1/r}\log^6 n)D\) (for all~\(t\)), and Theorem~\ref{KKMOcolouring} yields \(\chi'(S)\leq(1+n^{-1/r - \eta})D\) for any~\(t\) and any \(0<\eta<f_{(\ref{theorem:KKMO})}(r)\).\COMMENT{Again \(D=\Theta(n^{t-1})\) and \(D_2=\Theta(n^{t-2})\) and \(k+1=r\). So \((D/D_2)^{1/(k+1)}=\Theta(n^{1/r})\) and the error terms deal with the rest.}
In particular, our result gives an improvement over all previously known results whenever $t > 2$.

A \textit{partite} \((t,r,n)\)-\textit{Steiner system} is an \(r\)-uniform, \(r\)-partite hypergraph where each part has size~\(n\) and each \(t\)-set of vertices in distinct parts is contained in exactly one edge.
Again, the case of \((2,3,n)\)-partite Steiner systems has received particular attention, as this case concerns \(n\times n\) Latin squares.
A matching in such a system is equivalent to a Latin square \textit{partial transversal} (a set of cells containing distinct symbols, at most one from each row and column).
Montgomery's~\cite{Mont23} breakthrough paper also shows that if~\(n\) is large and even, then~\(n\times n\) Latin squares have a partial transversal of size~\(n-1\), which resolves the famous Ryser-Brualdi-Stein Conjecture~\cite{R67,S75,BR91} for such~\(n\).
Bowtell and Montgomery~\cite{BM25} also showed that, as $n \rightarrow \infty$, almost all Latin squares are resolvable.
For other values \(r>t\geq 2\), much less is known, with, again, the previously best known bounds in general arising from either Theorem~\ref{theorem:V00} or~\ref{theorem:KKMO}.
Here, Wanless~\cite{Wa11} conjectured that every partite $(d, d+1, n)$-Steiner system has a perfect matching if~$n$ or~$d$ is even, and the second author \cite{Ke24-nibble} conjectured that, for fixed~$r,t$, every partite $(t, r, n)$-Steiner system has chromatic index~$(n + O(1))^{t-1}$ and consequently a matching of size at least~$n - O(1)$.
Regarding edge colouring, the result of Keevash~\cite{K18} also extends to the partite setting, showing that there exist resolvable partite \((t,r,n)\)-Steiner systems for large~\(n\) satisfying some necessary divisibility conditions.
However, for general Latin squares (and \((t,r,n)\)-partite Steiner systems for other \(r>t\geq 2\)), the previously best known bounds we have come again from Theorem~\ref{KKMOcolouring}.

Theorems~\ref{theorem:maintheorem} and~\ref{theorem:maincolourtheorem} can also be applied in this setting, yielding the following.
\begin{theorem}\label{theorem:PartiteSteinercolourmatchings}
Suppose \(1/n\ll1/r<1/t\leq 1/2\), and suppose \(\delta>0\) satisfies \(1/n\ll\delta\).
Let~\(S\) be a partite \((t,r,n)\)-Steiner system, and let~\(U(S)\) denote the number of vertices uncovered by a maximum matching of~\(S\).
\begin{enumerate}[label=\upshape(\roman*)]
\item If \(t<(r+1)/2\) then \(U(S)\leq n^{\frac{r-t}{r-1}+\delta}\), and if \(t\geq(r+1)/2\) then \(U(S)\leq n^{1/2+\delta}\);
\item If \(t< \frac{r}{2}+1\) then \(\chi'_{\ell}(S)\leq\left(1+n^{-\frac{t-1}{r}+\delta}\right)n^{t-1}\), and if \(t\geq\frac{r}{2}+1\) then \(\chi'_{\ell}(S)\leq(1+n^{-1/2+\delta})n^{t-1}\).
\end{enumerate}
\end{theorem}
Using Theorems~\ref{theorem:KR},~\ref{theorem:V00},~\ref{theorem:KKMO} instead yield the same bounds on~\(U(S)\) as they do for the non-partite setting.\COMMENT{From the proof of Theorem~\ref{theorem:PartiteSteinercolourmatchings}, we have \(D=n^{t-1}\) and \(D_j=n^{t-j}\) for \(j=2,...,t\) and \(k+1=r\) so all these theorems do the same thing here as they did for the non-partite case.}
Theorems~\ref{KKMOcolouring} and~\ref{MRcolouring} also yield the same error bounds on~\(\chi'(H)\) and~\(\chi'_{\ell}(H)\) respectively as they do for the non-partite setting.
So again, our result gives an improvement over all previously known results whenever $t > 2$.

We remark that a (non-partite, though similar holds in the partite setting) Steiner system can itself be viewed as a matching in a hypergraph in which vertices are \(t\)-sets of~\([n]\) and edges are those \(t\)-sets in a given \(r\)-set.\COMMENT{I haven't yet defined the notation \(\binom{L}{t}\) and would rather delay it to the prelims section, so I tried to avoid it here}
Theorem~\ref{theorem:maintheorem} can therefore be used to find large partial Steiner systems (that is, almost full subsets of a Steiner system, which is a collection of exactly \(\binom{n}{t}/\binom{r}{t}\) \(r\)-sets).
The existence of a (full) \((t,r,n)\)-Steiner system is one of the oldest questions in combinatorics, dating back to the mid-19th century, and was finally resolved for large~\(n\) in 2014 by Keevash~\cite{K14}, who showed that for such~\(n\) satisfying the necessary divisibility conditions relating to~\(t,r\), there exist \((t,r,n)\)-Steiner systems, proving further there exist \((t,r,n,\lambda)\)-\emph{Steiner systems} (each \(t\)-set is in exactly~\(\lambda\) of the \(r\)-sets), which resolves the `Existence Conjecture' for large~\(n\).
Alternative proofs have since been presented by Glock, K\"uhn, Lo, and Osthus~\cite{GKLO23} and Delcourt and Postle~\cite{DP24res}.

Though the Existence Conjecture is now resolved, we still present the following application of Theorem~\ref{theorem:maintheorem}.
Our primary motivating factor for this is a conference proceedings from Kim~\cite{K01} in 2001, in which Kim claimed the same result as Theorem~\ref{theorem:Steiner} (up to the error term~\(n^{\delta}\)),\COMMENT{In fact, his is logarithmic. It's plausible we could get logarithmic in this specific setting too with a tailor-made sequence of Chomp applications, using the specific codegree sequence afforded by the Steiner setting} though a proof never appeared.
Interestingly, Kim mentioned his proof used the R\"{o}dl Nibble, and commented that he believed ``the nibble method gives no better bound'' up to the error term.
\begin{theorem}\label{theorem:Steiner}
Suppose \(1/n\ll 1/r < 1/t \leq 1/2\), and suppose \(\delta>0\) satisfies \(1/n\ll\delta\).
There exists a partial \((t,r,n)\)-Steiner system of size at least
\[
\left(1-n^{-(r-t)/(\binom{r}{t}-1)+\delta}\right)\binom{n}{t}/\binom{r}{t}.
\]
\end{theorem}
Theorem~\ref{theorem:KR} used in place of Theorem~\ref{theorem:maintheorem} in the proof of Theorem~\ref{theorem:Steiner} would yield error term~\(n^{-1/(\binom{r}{t}-1)+\delta}\).
This is slightly improved upon by Theorem~\ref{theorem:KKMO}, by a factor of around \(n^{-f_{(\ref{theorem:KKMO})}(k)}\), where \(k=\binom{r}{t}-1\).
Again, the naturally arising auxiliary hypergraph exhibits heavy ``clustering'' of the codegree sequence that prohibits directly applying Theorem~\ref{theorem:V00} with \(s>2\), but no choice of \((D_j)_{j=2}^s\) with Theorem~\ref{theorem:V00} would yield Theorem~\ref{theorem:Steiner}.\COMMENT{I think I can illustrate this: We have \(D=n^{r-t}\) and \(D_j=n^{r-t-1}\) for \(2\leq j\leq t+1\). Spreading these values out as much as possible to get the least constraining version of Vu's hypothesis on the intermediate ratios, we would put each \(D_j=n^{1/t}\) larger than the next, and would get \(x\leq n^{1/(2t)}\). Any choice of the \((D_j)\) would have \(x\) constrained by this or worse by these intermediate ratios. Now it suffices to find any value of \(r,t\) for which our \(B=n^{(r-t)/(\binom{r}{t}-1)}\) is strictly greater than \(n^{1/(2t)}\) to see that Vu cannot reach our result in general (and in fact, I suspect the last ratio will present the worst constraint for Vu anyway so I dont think Vu's result achieves \(x=n^{1/(2t)}\) in general). \(r=10, t=9\) suffices.
We already know our result implies Vu in general too, so his result can never be the stronger of the two.}

Theorem~\ref{theorem:Steiner} is also related to a folklore conjecture (see \cite[Conjecture 1.2]{BB19}) that asymptotically almost surely, the random greedy algorithm produces a partial $(t, r, n)$-Steiner system of the same size as in Theorem~\ref{theorem:Steiner}.
For the case $t = 2$ and $r = 3$, which corresponds to the \textit{triangle removal process}, Joel Spencer offered \$200 for a proof. 
This case was solved by Bohman, Frieze, and Lubetzky \cite{BFL15} in 2015, and the conjecture was recently proved in full generality by Joos and K\"uhn \cite{JK25}.

\subsection{Organisation of the paper}
In Section~\ref{section:sketch}, we sketch the proof of Theorem~\ref{theorem:maintheorem}.
In Section~\ref{section:prelims}, we cover the notation we will use throughout the remainder of the paper and state a number of probabilistic results that will be useful in our random nibble machinery, and in the generalizations and applications.
In Section~\ref{section:beyond}, we use Theorem~\ref{theorem:maintheorem} to prove Theorems~\ref{theorem:withreserves} and~\ref{thm:mainpartitetheorem}.
Section~\ref{Section:Applications} covers the proofs of each of the applications of Theorems~\ref{theorem:maintheorem},~\ref{thm:mainpartitetheorem}, and~\ref{theorem:maincolourtheorem} discussed in Section~\ref{section:appstatements}.
Sections~\ref{Section:Nibble}--\ref{section:exhausting} collectively prove Theorem~\ref{theorem:maintheorem}, with Section~\ref{Section:Nibble} covering the analysis of a single random nibble, Section~\ref{section:chomp} showing how to iterate this nibble to a ``Chomp'' which yields a matching with vanishing leftover, and Section~\ref{section:exhausting} demonstrating how to tactically iterate many Chomps, so as to tolerate any codegree clustering and still exhaust the whole codegree sequence up to the bottleneck as in Theorem~\ref{theorem:maintheorem}.
\section{Sketch of the proof of Theorem~\ref{theorem:maintheorem}}\label{section:sketch}
The proof of Theorem~\ref{theorem:maintheorem} splits naturally into three stages: the ``Nibble'', which finds a matching covering just~\(o(n)\) vertices; the ``Chomp'', which iterates many nibbles to find a matching covering all but~\(n/x=o(n)\) vertices; and finally we iterate many chomps with ``small'' \(x=\omega(1)\) to tactically exhaust as much of the input hypergraph codegree sequence as possible.
We describe each of these stages in turn, though first we discuss some different approaches to the semi-random method, in order to motivate our strategy and parameter selections.

Our nibble statement is Lemma~\ref{lemma:nibble}, and all of Section~\ref{Section:Nibble} is dedicated to its proof.
The interplay between the parameters~\(\eps\) and~\(\theta\) is of central importance, and guides our strategy.
Here,~\(\eps\) is the error in the vertex degrees, and~\(\theta/D\) is the probability with which a single edge is chosen for the matching in a nibble. 
It is useful to distinguish between two common nibble strategies.
In the ``wasteful'' nibble (e.g.~\cite{G99,V00}), one chooses edges independently with probability \(\theta/D\), sets the matching to be the subset of these edges not intersecting any others, and simply throws away the vertices in the intersecting edges, so they do not take part in the next iteration of the nibble.
This involves roughly~\(\theta^2 n\) vertices being ``wasted'' in each step, as they are not covered by the matching nor are they available to be covered in the future, and introduces error~\(\theta^2 D\) in the vertex degrees.
In contrast, in the ``non-wasteful'' nibble (e.g.~\cite{AKS97}), vertices in chosen but intersecting edges are not thrown away.
To help further control vertex degree errors (in either regime, e.g.~\cite{AKS97,V00}), one (initially counter-intuitively) independently puts some vertices in a ``waste set''~\(W\), with different probabilities for each vertex (that turn out to be on the order~\(\eps\theta\)), chosen carefully to counteract the slightly different vertex degrees such that every vertex is exactly equally likely to survive to the next iteration (that is, be uncovered by~\(M\) \textit{and} not be put in~\(W\)).
In this regime, the vertices uncovered by~\(M\) that are also unavailable to be covered in the future are precisely those in~\(W\setminus V(M)\), of which there are around~\(\eps\theta n\), and this introduces error~\(\eps\theta D\) in the vertex degrees.

This turns out to be very convenient, as this means the parameter~\(\eps\) grows in each iteration of the nibble by a factor~\((1+\theta)\), essentially the same rate by which the parameter~\(n\) decreases, so that each ``waste set''~\(W\) has size \(\eps'\theta n'\approx\eps\theta n\).
Thus, provided one only ever takes say logarithmically many nibbles, the total number of vertices ``wasted'' over the process is only around~\(\eps\theta n\).
In order to ensure the vertex degrees are concentrated in a single random nibble, within the error as described above, the `right' hypothesis on the ratio~\(D/D_2\) turns out to be that~\(\eps^{2}\theta D/D_2\) should be at least polylogarithmic in~\(D\) (see Lemma~\ref{lemma:nibble},~\ref{hypothesis:degratio}).
As discussed in Section~\ref{section:intro}, starting the whole iteration process with~\(\eps\) as small as possible is ideal, to maximise the number of times we may iterate.
We thus wish to permit~\(\eps\) to be essentially as small as~\(\sqrt{D_2/D}\) if the host hypergraph allows, and this suggests we should set~\(\theta\) to be ``large'' (in the sense that \(1/\theta\) should be a small~\(\omega(1)\)), and indeed we later set \(\theta\coloneqq 1/\log^2 D\) in the proof of the Chomp Lemma (Lemma~\ref{lemma:chomp}).
As a result, the total number of vertices we plan to ever throw in the ``waste'' is around \(\eps\theta n \approx \eps n\).

It is therefore clear that we should adopt a non-wasteful nibble, as errors on the order~\(\theta^2 D\) are much too large to tolerate if \(\eps\ll\theta\).
To concentrate the vertex degrees within a single non-wasteful nibble, within error~\(\eps\theta D\), we turn to a recent, powerful version of Talagrand's Inequality given by Delcourt and Postle~\cite{DP24} called ``Linear Talagrand's with Exceptional Outcomes'' (Theorem~\ref{theorem:lintal} in the current paper).
Though the random variable~\(\textbf{deg}\mathbf{'}(v)\) which tracks the ``surviving degree'' of~\(v\) after the nibble is not suitable for a Talagrand-type inequality directly, we show that it may be split up into a combination of random variables which are.
The ``exceptional outcomes'' aspect of Theorem~\ref{theorem:lintal} is also needed, as there are outcomes of the random nibble in which too many random choices have too much effect on~\(\textbf{deg}\mathbf{'}(v)\), but crucially, we show that these outcomes are very unlikely.
The approach to this part of the argument is discussed further shortly before it is implemented in the proof of Lemma~\ref{lemma:concanalysis}\ref{concvertexdegree}.

Another crucial ingredient of our overall strategy is the role of the set \(J^{*}\subseteq [k]\setminus\{1\}\) first appearing in the statement of Lemma~\ref{lemma:nibble}.
In his corresponding Nibble Lemma, Vu~\cite[Lemma 5.2]{V00} hypothesises that all ratios \(D_j/D_{j+1}\) for \(2\leq j<s\) (with~\(s\) as in Theorem~\ref{theorem:V00}) are large enough that we may successfully track the expected degradation of all such~\(D_j\).
In our overall approach, we simply wish to strategically choose different \(J^{*}\subseteq[k]\setminus\{1\}\) at each timestep (nibble) for which we hypothesise that those \(j\in J^{*}\) are such that \(D_j/D_{j+1}\) is large enough to track the degradation of~\(D_j\).
We make no hypotheses nor conclusions for those \(j\notin J^{*}\).
We will describe in the later stages how to strategically choose~\(J^{*}\) at each timestep to accomplish our goals.
To summarize, we perform a non-wasteful nibble, and we seek to track the degradation of~\(D\), the degradation of~\(D_j\) for those \(j\in J^{*}\), and the growth of~\(\eps\) within an acceptably slow rate.
Theorem~\ref{theorem:lintal} suffices for all of these purposes (and also for showing the nibble is pseudorandom with respect to some well-behaved family~\(\cT\) of weight functions), with simpler concentration inequalities (Lemmas~\ref{chernoff},~\ref{lemma:martingale}) covering the other necessaries like the size of~\(W\) and the remaining vertex set.
An application of the asymmetric Lov\'{a}sz Local Lemma (Lemma~\ref{lemma:LLL}) shows there is some outcome satisfying all of the above simultaneously.

In the Chomp Lemma (Lemma~\ref{lemma:chomp}), we seek to iterate many nibbles (roughly~\(\log^3 D\)), such that the remaining hypergraph has size~\(n/x=o(n)\), where~\(x=\omega(1)\) is constrained by the codegree ratios in a similar way to the~\(x\) of Theorem~\ref{theorem:V00}.
For convenience, all of the nibbles within a chomp use the same~\(J^{*}\).
If we iterate the nibble until the leftover hypergraph has~\(n/x\) vertices, then we expect that~\(D\) has dropped to roughly~\(D/x^k\), and each~\(D_j\) for \(j\in J^{*}\) has dropped to~\(D_j/x^{k-j+1}\).
This causes a phenomenon that will be essential later:\ earlier codegrees (\(D_j\) for smaller~\(j\)) degrade faster than later ones, assuming both are in~\(J^{*}\).
Those~\(D_j\) for \(j\notin J^{*}\) (including~\(D_{k+1}\)) are unaffected; we do not track their degradation in any of the nibbles, so the best we can say for these~\(j\) is that \(C_j(H')\leq C_j(H)\leq D_j\), where \(H, H'\) are the hypergraphs before and after this chomp, respectively.
Our approach to proving the Chomp Lemma via iteration of the Nibble Lemma is relatively straightforward: simply iterate with \(\theta\coloneqq1/\log^2 D\) (always the original~\(D\), that is), \(\lflr\log^2 D\log x\rflr\) times.

We remark that the Chomp Lemma implies Theorem~\ref{theorem:V00},\COMMENT{I prove this in a COMMENT just before the statement of the Chomp} and has some important additional strengths.\COMMENT{Discussed further in a new spiel at the start of Section~\ref{section:chomp}}
The most essential differences are the freedom to choose~\(J^{*}\) and, as discussed, the Chomp Lemma ensures that any input~\(\eps\geq\sqrt{D_2/D}\) degrades at the expected rate, allowing one to input an~\(\eps\) much smaller than the required~\(1/x\) if desired and still get out something useful, potentially allowing for many future Chomps.
Our strategy of iterating the Nibble also makes better use of the ``space'' between codegrees, and so the~\(x\) of the Chomp Lemma is only constrained by~\(D_j/D_{j+1}\) when \(j,j+1\in J^{*}\) (as is effectively the case for all \(2\leq j<s-1\) in Theorem~\ref{theorem:V00}), rather than \(\sqrt{D_j / D_{j+1}}\).
This occurs because such~\(D_j, D_{j+1}\) will decrease to \(D_j/x^{k-j+1}, D_{j+1}/x^{k-j}\) respectively, so the ratio \(D_j/D_{j+1}\) becomes smaller by factor~\(x\) (so~\(x\) must initially ``fit into'' this ratio in order for the ratio to still be large enough at each timestep for the nibbles to work).
Our~\(x\) is also constrained by \(x^{k-j+1}\leq D_j/D_{j+1}\) whenever \(j\in J^{*}\), \(j+1\notin J^{*}\), as these~\(D_j\) will degrade to~\(D_j/x^{k-j+1}\) while~\(D_{j+1}\) does not move.
Note that substituting \(j=s-1\) yields condition~\ref{condition:finalratio} of Theorem~\ref{theorem:V00}.

It remains only to discuss Section~\ref{section:exhausting}, in which we iterate the Chomp Lemma, strategically choosing~\(J^{*}\) at each iteration so as to exhaust all the space in the codegree sequence of~\(H\) that we can before hitting some bottleneck.
With~\(B, \gamma\) as in Theorem~\ref{theorem:maintheorem}, we set \(x\coloneqq B^{\gamma^4}\), and seek to iterate the Chomp Lemma, always with this~\(x\), roughly~\(\gamma^{-4}\) times, as this will result in leftover size roughly \(n/x^{\gamma^{-4}}=n/B\) as desired (we return to the waste~\(W\) later).
The rough idea is that, at each timestep~\(t\), we simply check which ratios \(D_j/D_{j+1}\) for \(2\leq j\leq k\) are currently large enough to tolerate putting \(j\in J^{*}_t\) (in practice we put \(j\in J^{*}_t\) if \(D_j^{(t)}/D_{j+1}^{(t)}\geq x^k\), so this is enough space if \(j+1\in J^{*}_t\) or otherwise).
These~\(j\) become precisely the indices of~\(J^{*}_t\), so that~\(j\notin J^{*}_t\) implies that \(D_j^{(t)}/D_{j+1}^{(t)}<x^k\), and we call such~\(j\) ``super-stuck'' (at time~\(t\)).
The backbone of the proof is an algorithm we call ``MINI-CHOMP-WHERE-ABLE'', or ``MCWA'' for short, which simply identifies which~\(j\) are not super-stuck at time~\(t\) and puts these \(j\in J^{*}_t\) for an application of the Chomp Lemma with the ``mini'' value~\(x\).
Then the Chomp Lemma finds a matching (and waste set~\(W\)) of the current hypergraph whose removal sends \(n_t\rightarrow n_t/x\), \(\eps_t\rightarrow \eps_t x\), \(D^{(t)}\rightarrow D^{(t)}/x^k\), and \(D_j^{(t)}\rightarrow D_{j}^{(t)}/x^{k-j+1}\) for \(j\in J^{*}_t\), \(D_j^{(t)}\rightarrow D_j^{(t)}\) for \(j\notin J^{*}_t\), becoming the values~\(n_{t+1}\) and so on.

Before we look at the behaviour of the whole codegree sequence during MCWA, it is useful to examine the behaviour of two consecutive codegrees \((D_j^{(t)}), (D_{j+1}^{(t)})\).
We say an index~\(j\) is ``semi-stuck'' if \(D_j^{(t)}/D_{j+1}^{(t)}<B^{\gamma^3}\), which is a weaker condition than being super-stuck. 
For as long as the ratio \(D_j^{(t)}/D_{j+1}^{(t)}\) is larger than~\(x^k\), MCWA will continue putting \(j\in J^{*}_t\), so that~\(D_j^{(t)}\) continues to degrade.
In each of these timesteps, either \(j+1\in J^{*}_t\) or not.
If not, then~\(D_j^{(t)}\) ``closes the gap'' considerably.
If \(j+1\in J^{*}_t\), then \(D_j^{(t)}\rightarrow D_j^{(t)}/x^{k-j+1}\) and \(D_{j+1}^{(t)}\rightarrow D_{j+1}^{(t)}/x^{k-j}\) so that the ratio \(D_j/D_{j+1}\) will still shrink by~\(x\), i.e.\ \(D_j^{(t)}\) \textit{still closes the gap}.
Eventually, there will come a point when~\(j\) is now semi-stuck, but is still not super-stuck, and at a later timestep still,~\(j\) becomes super-stuck.
When this happens,~\(j\) is no longer put in~\(J^{*}_t\), and~\(D_j^{(t)}\) and~\(D_{j+1}^{(t)}\) remain together for as long as \(j+1\notin J^{*}_t\), but begin to drift apart when~\(j+1\) has been put in~\(J^{*}_t\) for a few iterations.
However, because we make such a small amount of ``progress'' with each Chomp (\(x\) is chosen small),~\(D_{j+1}^{(t)}\) cannot fall far enough below~\(D_{j}^{(t)}\) for~\(j\) not to be semi-stuck at any timestep before~\(j\) becomes not super-stuck, at which point~\(D_j^{(t)}\) closes the gap again.
The key behaviour is, though indices~\(j\) may frequently change their state of being super-stuck, \emph{once an index becomes semi-stuck, it remains semi-stuck}.

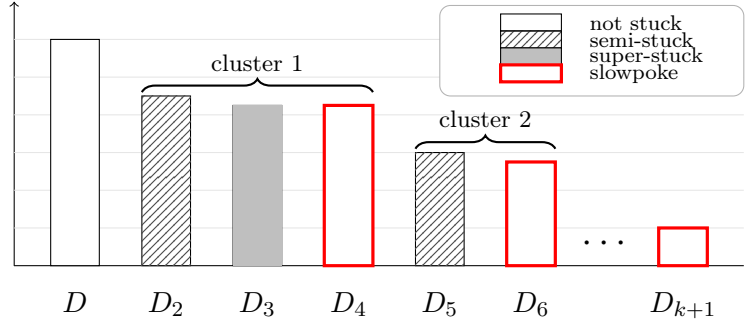
\begin{figure}
  \centering

  \begin{tikzpicture}[x=0.8cm,y=0.25cm]

    \def\barw{0.8} 

    \def\xD{0}  \def\xDtwo{1.5}  \def\xDthree{3}  \def\xDfour{4.5}
    \def\xDfive{6}  \def\xDsix{7.5}

    \def\xDots{8.75}   
    \def\xDkone{10}

    \def\hD{12} \def\hDtwo{9} \def\hDthree{8.5} \def\hDfour{8.5}
    \def\hDfive{6} \def\hDsix{5.5} \def\hDkone{2}

    \def\yTopOne{9.2}   
    \def\yTopTwo{6.2}   

    \draw[] (-1,0) -- (11,0); 
    \draw[->,line cap=round] (-1,0) -- (-1,14);
    \foreach \y in {2,4,6,8,10,12}{
      \draw[gray!20] (-1,\y) -- (11,\y);
    }

    \foreach \x/\h in {
      \xD/\hD, \xDtwo/\hDtwo, \xDthree/\hDthree,
      \xDfour/\hDfour, \xDfive/\hDfive, \xDsix/\hDsix}{
      \draw[fill=white,draw=black] (\x-\barw/2,0) rectangle (\x+\barw/2,\h);
    }

    \foreach \x/\h in {\xDtwo/\hDtwo, \xDfive/\hDfive}{
      \draw[pattern=north east lines,pattern color=black!70,draw=none]
      (\x-\barw/2,0) rectangle (\x+\barw/2,\h);
    }
    \foreach \x/\h in {\xDthree/\hDthree}{
      \draw[fill=black!25,draw=none]
      (\x-\barw/2,0) rectangle (\x+\barw/2,\h);
    }
    \foreach \x/\h in {\xDfour/\hDfour, \xDsix/\hDsix, \xDkone/\hDkone}{
      \draw[draw=red,very thick,fill=none]
      (\x-\barw/2,0) rectangle (\x+\barw/2,\h);
    }

    \node at (\xDots, 1.2) {\Large $\cdots$};

    \node[below=6pt] at (\xD,0) {$D$};
    \node[below=6pt] at (\xDtwo,0) {$D_2$};
    \node[below=6pt] at (\xDthree,0) {$D_3$};
    \node[below=6pt] at (\xDfour,0) {$D_4$};
    \node[below=6pt] at (\xDfive,0) {$D_5$};
    \node[below=6pt] at (\xDsix,0) {$D_6$};
    \node[below=6pt] at (\xDkone,0) {$D_{k+1}$};

    \draw[decorate,decoration={brace,amplitude=5pt},thick]
    (\xDtwo-\barw/2,\yTopOne) -- node[above=4pt,midway,font=\footnotesize]{cluster 1}
    (\xDfour+\barw/2,\yTopOne);

    \draw[decorate,decoration={brace,amplitude=5pt},thick]
    (\xDfive-\barw/2,\yTopTwo) -- node[above=4pt,midway,font=\footnotesize]{cluster 2}
    (\xDsix+\barw/2,\yTopTwo);

    \begin{scope}[shift={(7,12)}] 
      \draw[rounded corners,fill=white,draw=black!30] (-1.0,-3.0) rectangle (4.0,1.8);

      \def\lh{0.9} 
      \draw[fill=white,draw=black] (0,0.9-\lh/2) rectangle ++(1.0,\lh);
      \node[anchor=west] at (1.3,0.9) {\scriptsize not stuck};

      \draw[pattern=north east lines,pattern color=black!70,draw=black]
      (0,0-\lh/2) rectangle ++(1.0,\lh);
      \node[anchor=west] at (1.3,0.0) {\scriptsize semi-stuck};

      \draw[fill=black!25,draw=black] (0,-0.9-\lh/2) rectangle ++(1.0,\lh);
      \node[anchor=west] at (1.3,-0.9) {\scriptsize super-stuck};

      \draw[draw=red,very thick,fill=none] (0,-1.8-\lh/2) rectangle ++(1.0,\lh);
      \node[anchor=west] at (1.3,-1.8) {\scriptsize slowpoke};
\end{scope}
\end{tikzpicture}
\caption{An illustration of semi-stuck, super-stuck, and slowpoke indices, and clustering of the degree sequence.}
\label{fig:degree-clusters}
\end{figure}

Now suppose we begin with some~\(H\), having, say, a monotone decreasing codegree sequence \(D,D_2,\dots,D_{k+1}\), and assume \(D\gg D_2\) (see, e.g., Figure \ref{fig:degree-clusters}).
The indices \(2,3,\dots,k+1\) partition into subsets of consecutive indices such that all but the largest index in the subset is semi-stuck.
We call such subsets ``clusters'', and the largest index in each cluster is the ``slowpoke''.
(If all codegrees are meaningfully different then each cluster has size 1, and otherwise the codegree sequence exhibits at least some stepping/clustering behaviour.)
Clearly, at each timestep for as long as the clusters remain the same, MCWA will put each slowpoke (except~\(k+1\)) in~\(J^{*}_t\), so the codegree~\(D_z\) corresponding to the slowpoke~\(z\) degrades at the expected rate.
All other indices~\(j\) in the cluster are semi-stuck, and so by the above observation, continue to be semi-stuck at each timestep, and so remain in this cluster by definition.
Since these~\(j\) remain semi-stuck, each such~\(D_j^{(t)}\) remains not far above the value \(D_z^{(t)}\).
That is to say, the macro-behaviour is that \emph{the whole cluster degrades in unison, at the rate corresponding to the slowpoke}.
The slowpoke receives that name because all other indices in the cluster ``want'' their~\(D_j\) to decrease faster than the slowpoke allows, but instead must always lag behind the slowpoke.
Each cluster degrades during MCWA at a rate faster than the clusters corresponding to larger indices, and so eventually, one cluster may catch up to the next, at which point the slowpoke of the faster cluster becomes semi-stuck, and the clusters merge into one supercluster, now all degrading at the rate corresponding to the slower slowpoke, and MCWA continues.
Since concentrating the vertex degrees during a nibble is not optional (whereas tracking the degradation of some codegree is), we need~\(D/D_2\) to remain large for as long as possible.
The above behaviour of the codegree clusters during MCWA ensures that~\(D_2^{(t)}\) is always being ``pushed down'' as much as the codegree sequence allows, prolonging the time when we must halt due to \(D^{(t)}\approx D_2^{(t)}\). 

To initialize MCWA, we set \(\eps^{*}\approx B^{-1}\), so~\(H\) is \((n,D,\eps^{*})\)-regular since \(B\leq1/\eps\) by hypothesis, and we require \((\eps^{*})^2 D/D_2\) to be roughly at least~\(x^k=B^{k\gamma^4}\) (see Lemma~\ref{lemma:chomp},~\ref{chomphyp:trackingdegratio}--\ref{chomphyp:nontrackingdegratio}).
This holds because \(\eps^{*}\approx B^{-1}\geq\sqrt{D_2/D}\) by hypothesis.
Since, by the above, any amount of clustering of the codegrees \(D_2, D_3, \dots, D_{k+1}\) does not stop MCWA, there are only two reasons for MCWA to abort before completing the desired~\(\gamma^{-4}\) iterations: either~\(\eps_t\) has grown to almost~\(\Theta(1)\), or~\(D\) (which degrades with every timestep) has caught up to~\(D_2^{(t)}\).
The former issue does not occur too early because \(\eps^{*}\approx B^{-1}\) takes~\(\gamma^{-4}\) timesteps for \(\eps_t=\eps^{*}x^t=\eps^{*}B^{t\gamma^4}\) to grow to~\(\Theta(1)\).
The latter issue requires a more careful analysis of the evolution of the whole codegree sequence to understand the evolution of~\(D_2^{(t)}\) as it gets stuck behind progressively slower slowpokes, and we defer the details to Section~\ref{section:exhausting}, but the main observation is that~\(D_2^{(t)}\) is always roughly equal to \(\max_{j\geq 2} (D_j/x^{t(k-j+1)})\), as this is the largest codegree of any slowpoke at time~\(t\).
It turns out that~\(D\) will catch up to~\(D_2^{(t)}\) at time~\(t\) if \(x^t\approx(D/D_j)^{1/(j-1)}\), where~\(j\) attains the maximum expression above at time~\(t\) (i.e.\ the index~\(2\) is stuck behind this slowpoke~\(j\) at time~\(t\)), so since~\(B\) is chosen to be smaller than all these expressions, this does not occur before MCWA completes the desired number of iterations.
When MCWA is finished, the leftover hypergraph has around~\(n/x^{\gamma^{-4}}=n/B\) vertices, and the total number of vertices ever thrown in the waste~\(W\), as discussed earlier, is around \(\eps^{*}n\approx n/B\).\COMMENT{I sort of feel an inclination to flip the script here and explain that our proof really shows that if you nibble with no endgoal in mind and proceed until you are stopped by something at time \(t\) having made progress \(x^t\), you are stopped because \(x^t=\) one of the bottlenecks, and explain what that means for each bottleneck. I don't know if it's necessary though, it's clearly already at least implicit, even if that's not the form our proof takes, and the sketch already leans long...}

We remark that instead of saying the process aborts when~\(D^{(t)}\) reaches~\(D_2^{(t)}\) (the ``simplified heuristic''), it would be more accurate to say that~\((\eps_t)^2 D^{(t)}/D_2^{(t)}\) must remain~\(\omega(1)\).
This clearly fails if~\(D^{(t)}\) becomes~\(D_2^{(t)}\), but may fail sooner if~\(\eps^{*}\) is taken too small.
However, our choice of~\(\eps^{*}\) ensures that~\(\eps_t\) reaches~\(\Theta(1)\) from below at roughly the same time that \(D^{(t)}/D_2^{(t)}\) reaches~\(\Theta(1)\) from above (assuming the bottleneck for~\(B\) was not the original~\(1/\eps\), as in that case the process stops potentially long before~\(D^{(t)}\) becomes~\(D_2^{(t)}\)), so the simplified heuristic is accurate for us.
We defer the precise details to Section~\ref{section:exhausting}.

\section{Preliminaries}\label{section:prelims}
In this section, we discuss the notation we will use throughout the paper, and some results we will use in the proofs in Sections~\ref{section:beyond},~\ref{Section:Applications}, and~\ref{Section:Nibble}.
\subsection{Notation}\label{section:notation}
We write \(a=(1\pm b)c\) to mean \((1-b)c\leq a\leq (1+b)c\).
For a positive integer~\(n\), we define \([n]\coloneqq\{1,2,\dots,n\}\) and \([n]_0\coloneqq[n]\cup\{0\}\).
For a set \(A\) and an integer \(i\in[|A|]_0\), we define \(\binom{A}{i}\coloneqq\{B\subseteq A\colon |B|=i\}\).
Let~\(H\) be a hypergraph.
For a set \(Z\subseteq V(H)\), we define \(\partial_{H}(Z)\coloneqq\{e\in E(H)\colon e\supseteq Z\}\), dropping the subscript~\(H\) when~\(H\) is clear from context.
We write~\(\partial_{H}(v)\) for~\(\partial_{H}(\{v\})\).
For distinct $v,v'\in V(H)$, we denote by~$\text{dist}_{H}(v,v')$ (the \textit{distance} from~$v$ to~$v'$ in~$H$) the minimum~$i$ (or~$\infty$ if there is no such~$i$) such that there are edges~$\{e_{j}\}_{j\in[i]}$ in~$H$ for which $v\in e_{1}$, $v'\in e_{i}$, and $e_{j}\cap e_{j+1}\neq\emptyset$ for all $j\in[i-1]$.
If $v=v'$ then define $\text{dist}_{H}(v,v')=0$.
The edge set~\(E(H)\) of a \textit{multihypergraph}~\(H\) is permitted to be a multiset, with copies of an edge each being treated as a distinct member of~\(E(H)\).\COMMENT{Though I go on to make more multihypergraph definition clarifications in Section~\ref{section:multi}, I felt there needed to be at least something initially to explain what exactly is meant by multihypergraph in the statement of Lemma~\ref{lemma:nibble}, which will be read before Section~\ref{section:multi}.}
As such,~\(\partial(v)\) and~\(\partial(Z)\) are allowed to be multisets, and~\(\text{deg}(v)\) and~\(\text{codeg}(Z)\) count copies.
We use bold lettering for random variables and random sets.
We consider empty sums to be~\(0\), empty products to be~\(1\), and identify intersections of no events with the sample space. 
\subsection{Tools}\label{section:concineqs}
Throughout the paper, we make frequent use of the inequalities \(1+x\leq(1+x/n)^n \leq e^x\) for all \(n\geq 1\), \(|x|\leq n\) and \(e^x\leq 1+x+x^2\) for all \(x\in(-\infty,1.79)\).  

We will use the following version of the ``asymmetric'' Lov\'{a}sz Local Lemma (see for instance~\cite{AS92}) when we collect failure probabilities in our random nibble procedure in Section~\ref{Section:Nibble}.
We also use it to find an appropriate reservoir in the proof of Theorem~\ref{thm:mainpartitetheorem}.
Where \(A_1, A_2, \dots, A_m\) are events in a probability space, we say that a directed graph~\(G\) is a \textit{dependency digraph} for \(\{A_1, A_2, \dots, A_m\}\) if \(V(G)=\{a_i\colon i\in[m]\}\) and each~\(A_i\) is mutually independent of \(\{A_j\colon a_{i}a_{j}\notin E(G)\}\).\COMMENT{Leave it unsaid that \(a_{i}a_{j}, a_{j}a_{i}\in E(G)\) permissible?}
\begin{lemma}[Lov\'{a}sz Local Lemma]\label{lemma:LLL}
Suppose that~$G$ is a dependency digraph for~$\{A_{i}\colon i\in[m]\}$, and suppose that \(x_1,x_2,\dots,x_m\) are real numbers satisfying \(0\leq x_i < 1\) and \(\prob{A_i}\leq x_i\prod_{j\colon a_{i}a_{j}\in E(G)}(1-x_j)\) for all \(i\in[m]\).
Then \(\prob{\bigcap_{i\in[m]}\overline{A_i}}>0\).
\end{lemma}

We will need a variety of commonly used concentration inequalities, the first of these being the following versions of Chernoff's inequality (see for example~\cite{JLR11}).
\begin{lemma}[Chernoff's Inequality]\label{chernoff}
Suppose $\mathbf{X}\sim\text{Bin}(n,p)$.
\begin{enumerate}[label=\upshape(\roman*)]
\item\label{lemma:chernoffbigexp} For $0<\eps\leq3/2$, we have
\[
\prob{|\mathbf{X}-\expn{\mathbf{X}}|\geq \eps\expn{\mathbf{X}}}\leq 2\exp\left(-\frac{\eps^{2}}{3}\expn{\mathbf{X}}\right).
\]
\item\label{lemma:chernoffsmallexp}
For \(\beta>1\), we have
\[
\prob{\mathbf{X}>\beta\expn{\mathbf{X}}}\leq(\beta/e)^{-\beta\expn{\mathbf{X}}}.
\]
\end{enumerate}
\end{lemma}

We will use McDiarmid's Inequality~\cite{M89} when randomly partitioning \(G\in\Phi(\dirK)\) in Section~\ref{section:raincycles} (see Lemma~\ref{slicinglemma}).
We first need the following definition.
Let $\mathbf{X}_{1}, \mathbf{X}_{2},\dots,\mathbf{X}_{m}$ be mutually independent random variables taking values in a set~$\cX$, let $f:\cX^{m}\rightarrow\mathbb{R}$, and define $\mathbf{Y}\coloneqq f(\mathbf{X}_{1},\dots,\mathbf{X}_{m})$.
Fix $i\in[m]$.
If for all $x_{1},x_{2},\dots,x_{m},x'_{i}\in\cX$ we have
\[
|f(x_{1},\dots,x_{i-1},x_{i},x_{i+1},\dots,x_{m})-f(x_{1},\dots,x_{i-1},x'_{i},x_{i+1},\dots,x_{m})|\leq c_{i},
\]
then we say that~$\mathbf{X}_{i}$ \emph{affects}~$\mathbf{Y}$ by at most~$c_{i}$.
\begin{lemma}[McDiarmid's Inequality]\label{lemma:mcd}
Suppose \(\mathbf{X}_1, \mathbf{X}_2, \dots, \mathbf{X}_m\) are mutually independent random variables taking values in a set~\(\cX\), let \(f\colon \cX^m \rightarrow\mathbb{R}\) and set \(\mathbf{Y}\coloneqq f(\mathbf{X}_1,\mathbf{X}_2,\dots,\mathbf{X}_m)\).
Suppose that~\(\mathbf{X}_i\) affects~\(\mathbf{Y}\) by at most~\(c_i\), for each \(i\in[m]\).
Then for all \(t>0\),
\[
\prob{|\mathbf{Y}-\expn{\mathbf{Y}}|\geq t}\leq\exp\left(-\frac{2t^2}{\sum_{i=1}^m c_i^2}\right).
\]
\end{lemma}
The following result is a straightforward corollary\COMMENT{In the language of~\cite{AKS97}, it is clear that a decision tree formed of a single path of queries regarding the outcome of each of the $\mathbf{X}_{i}$ in turn (in an arbitrary order) is a strategy for finding~$\mathbf{Y}$.
Further, the only line of questioning in this strategy has total variance~$\sum_{i\in[m]}p_{i}(1-p_{i})c_{i}^{2}$, which is at most~$\sigma^{2}$ by hypothesis.
Now directly apply the ``Martingale Inequality'' in~\cite{AKS97}, with $\eps=\delta=1$ (which the authors remark is a permissible choice of these parameters).} of the ``Martingale Inequality'' from~\cite{AKS97}.
We will use this lemma to concentrate the number of vertices in the leftover after a nibble (see the proof of Lemma~\ref{lemma:concanalysis}\ref{concleftover}).
\begin{lemma}\label{lemma:martingale}
Let $\mathbf{X}_{1},\mathbf{X}_{2},\dots,\mathbf{X}_{m}$ be mutually independent indicator random variables, let $f:\{0,1\}^{m}\rightarrow\mathbb{R}$, and let $\mathbf{Y}\coloneqq f(\mathbf{X}_{1},\dots,\mathbf{X}_{m})$.
For each $i\in[m]$, set $p_{i}\coloneqq\prob{\mathbf{X}_{i}=1}$ and suppose that~$\mathbf{X}_{i}$ affects~$\mathbf{Y}$ by at most~$c_{i}$.
Let~$C \coloneqq \max\{c_i : i \in [m]\}$.
Suppose further that $\sqrt{\sum_{i\in[m]}p_{i}(1-p_{i})c_{i}^{2}}\leq\sigma$.
Then for all $0<\alpha<2\sigma/C$, we have
\[
\prob{|\mathbf{Y}-\expn{\mathbf{Y}}|>\alpha\sigma}\leq 2\exp\left(-\frac{\alpha^{2}}{4}\right).
\]
\end{lemma}
We will use the following version of Talagrand's~\cite{T95} Inequality (sometimes also called McDiarmid's Inequality, though distinct from Lemma~\ref{lemma:mcd}) when we both randomly select and randomly order \(d\)-sets of~\([n]\) in the construction of a large-diameter simplicial complex in Section~\ref{section:SC}.
\begin{lemma}\label{lemma:permutations}
Let~\(\mathbf{Z}\) be a non-negative random variable which is determined by~\(n\) independent random variables \(\mathbf{T}_1, \mathbf{T}_2, \dots, \mathbf{T}_n\) and~\(m\) independent uniformly random permutations \(\mathbf{f}_1\colon X_1\rightarrow X_1,\mathbf{f}_2\colon X_2\rightarrow X_2,\dots,\mathbf{f}_m\colon X_m\rightarrow X_m\), where each~\(X_i\) is finite and non-empty.
Assume moreover that \(\mathbf{T}=(\mathbf{T}_1,\mathbf{T}_2,\dots,\mathbf{T}_n)\) and \(\mathbf{f}=(\mathbf{f}_1,\mathbf{f}_2,\dots,\mathbf{f}_m)\) are independent.
Let \(c,r>0\) and assume the following conditions hold for any outcome \((T=(T_1,T_2,\dots,T_n), f=(f_1,f_2,\dots,f_m))\) of~\((\mathbf{T}, \mathbf{f})\) (and put \(Z=\mathbf{Z}(T,f)\)):
\begin{itemize}
\item Changing the outcome of a single~\(T_i\) changes~\(Z\) by at most~\(2c\);
\item Swapping two co-ordinates of a single~\(f_j\) changes~\(Z\) by at most~\(c\);
\item If \(Z=s\) then there is a set~\(\{T_{i_1}, T_{i_2},\dots,T_{i_k}\}\) and a set \(\{f_{j_1}, f_{j_2}, \dots, f_{j_{\ell}}\}\) such that \(k+\ell\leq rs\) and \(\mathbf{Z}\geq s\) holds whenever \(\mathbf{T}_{i_a}=T_{i_a}\) and \(\mathbf{f}_{j_b}=f_{j_b}\) holds for every \(a\in[k]\) and \(b\in[\ell]\). 
\end{itemize}
Then, for any \(t\geq 128c\sqrt{r\expn{\mathbf{Z}}}+512rc^2\), we have
\[
\prob{|\mathbf{Z}-\expn{\mathbf{Z}}|>t}\leq4\exp\left(-\frac{t^2}{32rc^2(4\expn{\mathbf{Z}}+t)}\right).
\]
\end{lemma}
The first result of this flavour was due to McDiarmid~\cite{M02}, with concentration about the \textit{median} rather than the expectation.
A version with concentration about the expectation first appeared in the book of Molloy and Reed~\cite{MR02}, though it is known that the proof contained a flaw (in~\cite[Fact 20.1]{MR02}).
An updated version, which is very similar to Lemma~\ref{lemma:permutations}, appears in a later paper of Molloy and Reed~\cite{MR10}.
Our version can be proven with the same reasoning as used in the appendix of~\cite{KP20} (which also describes the flaw in~\cite[Fact 20.1]{MR02}).

The inequality that we use to analyze the concentration of the vertex degrees and codegrees during a random nibble, as well as the degradation of a given family~\(\cT\) of weight functions (see the proof of Lemma~\ref{lemma:concanalysis}\ref{concvertexdegree}--\ref{conctestwaste}), is the following powerful form of Talagrand's inequality due to Delcourt and Postle~\cite[Theorem 4.4]{DP24}.
To state it, we first need some definitions.
\begin{defin}
Let~\(((\Omega_i, \Sigma_i, \mathbb{P}_i))_{i=1}^{n}\) be probability spaces and~\((\Omega, \Sigma, \mathbb{P})\) their product space.
Let~\(\Omega^{*}\subseteq\Omega\), and let \(\mathbf{X}\colon\Omega\rightarrow\mathbb{R}_{\geq 0}\) be a non-negative random variable.
Let \(r,d>0\).
\begin{itemize}
\item Suppose \(\omega=(\omega_1,\omega_2,\dots,\omega_n)\in\Omega\) and \(s\geq 0\).
An \((r,d)\)-\textit{certificate} for~\(\mathbf{X}, \omega, s\), and~\(\Omega^{*}\) is an index set \(I\subseteq[n]\) and a vector~\((c_i \colon i\in I)\) which satisfy \(\sum_{i\in I}c_i \leq rs\) and \(\max_{i\in I}c_i \leq d\), such that for all \(I'\subseteq I\), we have
\[
\mathbf{X}(\omega) \geq s - \sum_{i\in I'}c_i
\]
for all outcomes \(\omega'=(\omega'_1, \omega'_2, \dots, \omega'_n)\in\Omega\setminus\Omega^{*}\) such that \(\omega_i = \omega'_i\) for all \(i\in I\setminus I'\).
\item If for every outcome \(\omega\in\Omega\setminus\Omega^{*}\), there exists an \((r,d)\)-certificate for~\(\mathbf{X}, \omega, s\coloneqq\mathbf{X}(\omega)\), and~\(\Omega^{*}\), then we say~\(\mathbf{X}\) is \((r,d)\)-\textit{observable} with respect to~\(\Omega^{*}\).
\end{itemize}
\end{defin}
The set~\(\Omega^{*}\) is often called the ``exceptional outcomes'' in the literature.
The idea is a natural one; in a setting where the Lipschitz constants~\(c_i\) may be large enough to frustrate an attempt to concentrate a random variable of interest (or there is some other obstacle to being appropriately observable), we instead show that the random variable admits the desired behaviour in all but a very unlikely set of bad outcomes~\(\Omega^{*}\).
For earlier examples of Talagrand's incorporating some notion of exceptional outcomes, see for example~\cite{KP20, MR02, BJ18}.
\begin{theorem}[Linear Talagrand's with Exceptional Outcomes~\cite{DP24}]\label{theorem:lintal}
Let~\(((\Omega_i, \Sigma_i, \mathbb{P}_i))_{i=1}^{n}\) be probability spaces and~\((\Omega, \Sigma, \mathbb{P})\) their product space.
Let~\(\Omega^{*}\subseteq\Omega\), and let \(\mathbf{X}\colon\Omega\rightarrow\mathbb{R}_{\geq 0}\) be a non-negative random variable.
Let \(r,d>0\).

If~\(\mathbf{X}\) is \((r,d)\)-observable with respect to~\(\Omega^{*}\), then for \(t> 96\sqrt{rd\expn{\mathbf{X}}} + 128rd + 8\prob{\Omega^{*}}(\mathrm{sup}\mathbf{X})\), we have
\[
\prob{|\mathbf{X}-\expn{\mathbf{X}}|>t}\leq 4\exp\left(\frac{-t^2}{8rd(4\expn{\mathbf{X}}+t)}\right)+4\prob{\Omega^{*}}.
\]
\end{theorem}
The ``linear'' in ``Linear Talagrand's'' refers to the power of~\(d\) in the denominator of the exponent.
Delcourt and Postle~\cite{DP24} showed that the~\(d^2\) in many classical statements of Talagrand's (see for example~\cite{MR02}) can be improved to simply~\(d\).
In our applications in Section~\ref{Section:Nibble}, this improvement is essential.

We will use a polynomial concentration inequality of Kim and Vu~\cite{KV00} to analyze the properties of a sparse random reservoir in the proof of Theorem~\ref{thm:mainpartitetheorem}.
We first require some definitions.
If \(\mathbf{X}_1, \mathbf{X}_2, \dots, \mathbf{X}_m\) are mutually independent indicator random variables and~\(H\) is a \(k\)-bounded hypergraph having \(V(H)=[m]\) such that every edge \(e\in E(H)\) has associated positive weight~\(w(e)\), we say the \((\mathbf{X}_1,\mathbf{X}_2,\dots,\mathbf{X}_m)\)-\textit{associated polynomial}~\(\mathbf{Z}_H\) of~\(H\) is \(\mathbf{Z}_H\coloneqq\sum_{e\in E(H)}w(e)\prod_{i\in e}\mathbf{X}_i\) (dropping the \((\mathbf{X}_1,\mathbf{X}_2,\dots,\mathbf{X}_m)\) when clear from context).
For a set \(A\subseteq V(H)\), we define the \(A\)-\textit{truncated polynomial} as
\begin{equation}\label{eq:Atruncated}
\mathbf{Z}_{H_A}\coloneqq\sum_{\substack{e\in E(H)\colon\\ A\subseteq e}}w(e)\prod_{i\in e\setminus A}\mathbf{X}_i.
\end{equation}
Set \(\mathbb{E}_i[\mathbf{Z}_H]\coloneqq\max_{A\in \binom{V(H)}{i}}\expn{\mathbf{Z}_{H_A}}\), \(\mathbb{E}'[\mathbf{Z}_H]\coloneqq \max_{i\geq 0}\mathbb{E}_i[\mathbf{Z}_H]\), \(\mathbb{E}''[\mathbf{Z}_H]\coloneqq \max_{i\geq 1}\mathbb{E}_i[\mathbf{Z}_H]\).
Observe that \(\mathbb{E}_0[\mathbf{Z}_H]=\expn{\mathbf{Z}_H}\).
\begin{theorem}[Kim-Vu Polynomial Concentration \cite{KV00}]\label{theorem:kimvu}
Suppose \(\mathbf{X}_1,\mathbf{X}_2,\dots,\mathbf{X}_m\) are mutually independent indicator random variables, \(k\geq 1\) is an integer,\COMMENT{The statement in Kim-Vu's paper is a bit vague with regards to this but it's clear from their discussion beforehand, and also their proof (as it's by induction on k and they include the case \(k=1\)) that you can use \(k=1\), which actually becomes very useful for us in the proof of the bipartite matching theorem} and~\(H\) is a \(k\)-bounded hypergraph with \(V(H)=[m]\), and \(w\colon E(H)\rightarrow\mathbb{R}_{>0}\).
Consider the associated polynomial \(\mathbf{Z}_H=\sum_{e\in E(H)}w(e)\prod_{i\in e}\mathbf{X}_i\) of~\(H\).
For any \(\lambda>1\),
\[
\prob{|\mathbf{Z}_H-\expn{\mathbf{Z}_H}|\geq (8\lambda)^k \sqrt{k!\mathbb{E}'[\mathbf{Z}_H]\mathbb{E}''[\mathbf{Z}_H]}}\leq 2e^2\exp(-\lambda +(k-1)\log m).\COMMENT{Kim-Vu express the \(2e^2\) as big \(O\) in their main theorem, but they clarify at the start of their main proof (bottom of page 7) that the big \(O\) means \(2e^2\). I'm trying to avoid mixing hierarchies and asymptotic notation throughout the paper (by only using hierarchies), so I felt this was cleaner.}
\]
\end{theorem}
We will use the following result to ``mop up'' the leftover after applying Theorem~\ref{theorem:maintheorem} to find an almost-perfect matching of~\(H_1\) in the proof of Theorem~\ref{theorem:withreserves}.
\begin{lemma}\label{lemma:bipmopup}
Suppose~\(H\) is a \((k+1)\)-bounded bipartite hypergraph with bipartition \((X,Y)\) which satisfies \(\text{deg}(x)\geq 2ekD\) for all \(x\in X\) and \(\text{deg}(y)\leq D\) for all \(y\in Y\).
Then~\(H\) has an \(X\)-perfect matching.
\end{lemma}
We remark that a simple application of the Lov\'asz Local Lemma suffices to prove Lemma~\ref{lemma:bipmopup}, and indeed, such a proof appears in~\cite{Re99}.
As discussed in Section~\ref{section:bey}, Haxell~\cite{Ha01} used a topological proof to show that the bound~\(2ekD\) on~\(\text{deg}(x)\) can be replaced by simply~\(2kD\).
Notice that Theorem~\ref{thm:mainpartitetheorem} is essentially a further improvement of this bound to a~\((1+o(1))D\) term under the additional assumption we have good information on the codegree sequence.
\section{Beyond matchings}\label{section:beyond}
In this section, we apply Theorem~\ref{theorem:maintheorem} to prove Theorems~\ref{theorem:withreserves} and~\ref{thm:mainpartitetheorem}.
As there are a number of hypotheses in the statement of Theorem~\ref{theorem:withreserves}, which then also need to be checked in the proof of Theorem~\ref{thm:mainpartitetheorem}, we restate Theorem~\ref{theorem:withreserves} now for convenience.
\reserves*
The core of the proof of Theorem~\ref{theorem:withreserves} is of course applying Theorem~\ref{theorem:maintheorem} to~\(H_1\), for which we first need to `regularize'~\(H_1\) by adding edges (without increasing any of the codegrees), as those vertices in~\(V_1\setminus X\) may have degree too small.
We also need to convert the \((k+1)\)-bounded assumption on~\(H_1\) to a uniform one, which can be done easily by adding dummy vertices outside of~\(X\) before the regularization step.
To do this, we need an analogue of a lemma of Kang, K\"{u}hn, Methuku, and Osthus~\cite[Lemma 8.1]{KKMO23}.
\begin{lemma}\label{lemma:regularization}
Let~\(H_1\) be as in the hypotheses of Theorem~\ref{theorem:withreserves}.
Then there exists a \((k+1)\)-uniform superhypergraph \(H'_1\supseteq H_1\) in which all vertices have degree~\((1\pm B^{-1})D\), such that for every \(x\in X\) there is a bijection \(f_x\colon \partial_{H_1}(x)\rightarrow\partial_{H'_1}(x)\) such that \(e\subseteq f_x(e)\) and \((f_x(e)\setminus e)\cap X=\emptyset\) for each \(e\in\partial_{H_1}(x)\).
Moreover, \(C_j(H'_1)=C_j(H_1)\) for all \(j\in\{2,3,\dots,k+1\}\).
\end{lemma}
The authors of~\cite{KKMO23} sought to regularize a hypergraph in order to bootstrap their main matching result (Theorem~\ref{theorem:KKMO} in the current paper) to an edge colouring version (Theorem~\ref{KKMOcolouring} in the current paper) with a maximum degree hypothesis (rather than an `approximately regular' hypothesis).
Our situation is slightly different, primarily because we need to ensure we do not add edges to the~\(X\) vertices (so the subset of the obtained matching covering the vertices in~\(X\) exists in~\(H_1\)), we must verify that none of the codegrees~\(C_j(H)\) are increased, we must convert the bounded assumption to a uniform one, and we also permit smaller~\(k\), but this changes the regularization argument very little.
As the proof of~\cite[Lemma 8.1]{KKMO23} is short, and it is easy to make the small changes needed to prove Lemma~\ref{lemma:regularization}, we omit the details here.
The main idea is to:
For all edges which have say~\(r<k+1\) vertices, add \(k+1-r\) dummy vertices outside of~\(X\) (disjoint for each edge), and adjust the edge to include these vertices, so the resulting hypergraph, say~\(H''_1\), is \((k+1)\)-uniform; clearly the codegrees are undamaged, and we have potentially introduced some low-degree vertices, but only outside of~\(X\).
Now copy~\(H''_1\), \(k^2 D^2\) times, and for each low-degree vertex \(v\in V_1\setminus X\), put disjoint \((2,k+1,z)\)-Steiner systems on~\(v\) and its clones, for well-chosen~\(z\) which brings the degree of those vertices into the range~\([D-(k+1)(k+3),D]\).
In particular, \((k+1)(k+3)\leq\sqrt{D}\leq B^{-1}D\), using the hypothesis \(B\leq\sqrt{D/D_2}\).
Since all added edges come from (pairwise-disjoint)~\((2,k+1,z)\)-Steiner systems each using at most one vertex from any copy of~\(H_1\), this does not damage any of the codegrees.\COMMENT{Proof of Lemma~\ref{lemma:regularization} (though before the bounded -> uniform bit, but it is clear that makes no difference): Introduce a new constant~\(Z\) into the hierarchy such that \(1/D\ll1/Z\ll1/k\).
With such a~\(Z\), it is well known that for any integer~\(z\geq Z\) such that~\(z-1\) and~\(z(z-1)\) are divisible by~\(k\) and~\(k(k+1)\) respectively, there exists a \((2,k+1,z)\)-Steiner system.
Notice that for any integer \(z\geq Z\), there exists an integer \(0\leq t(z)<k(k+1)\) for which~\(z+t(z)\) satisfies the above divisibility conditions. [Suppose \(z\) satisfies \(k\mid z-1\) and \(k(k+1)\mid z(z-1)\). We seek to show \(z+k(k+1)\) then also satisfies \(k\mid z+k(k+1)-1\) and \(k(k+1)\mid (z+k(k+1))(z+k(k+1)-1)\). Well \(z+k(k+1)-1 = (z-1) + k(k+1)\), and \(k\) is a factor of both these terms. And \((z+k(k+1))(z+k(k+1)-1) = z(z-1) + zk(k+1) + k(k+1)(z+k(k+1)-1)\), and \(k(k+1)\) is a factor of all three of these terms.]
Let~\(S(z)\) denote the \(z\)-vertex hypergraph obtained from \(S=S(2,k+1,z+t(z))\) by deleting an arbitrary set of~\(t(z)\) vertices.
Notice that~\(S(z)\) satisfies \(C_j(S(z))=1\) for all \(2\leq j\leq k+1\), and, since a \((t,\ell,n)\)-Steiner system is \(\binom{n-1}{t-1}/\binom{\ell-1}{t-1}\)-regular, each vertex of~\(S(z)\) has degree [Each vertex has degree in \(S(2,k+1,z+t(z))\) exactly \((z+t(z)-1)/k\). For the lower bound, this is at least \((z-1)/k\), and then we remove \(t(z)\leq k(k+1)\) vertices. For each surviving vertex, each deleted vertex reduces the degree by exactly 1 by the original Steiner condition, so the surviving degrees are at least \((z-1)/k - k(k+1)\). For the upper bound, we simply bound the original degree from above as \((z+t(z)-1)/k\leq (z+k(k+1)-1)/k = (z-1)/k +k+1\).] \(\frac{z-1}{k}-k(k+1)\leq \text{deg}_{S(z)}(v)\leq\frac{z-1}{k}+k+1\).
For each integer \(Z\leq z\leq D\), we construct an approximately \(z\)-regular hypergraph~\(H^{(z)}\) with exactly~\(k^2D^2\) vertices [By my definition of hierarchies we have that \(D\) is integer] as follows:
Put \(\ell(z)\coloneqq\lceil kD^2/(z-k-1)\rceil\) and let \(a_1\geq\dots\geq a_{\ell(z)}\) be a sequence [Firstly it's clear that \(\ell(z)\) is a large integer, between \(\Omega(D)\) and \(O(D^2)\) depending on~\(z\). We have \(\frac{kD^2}{z-k-1} \leq \ell(z) \leq \frac{kD^2}{z-k-1}+1\) so \(\ell(z)(k(z-k-1)-1)=k\ell(z)(z-k-1) -\ell(z)\leq (\frac{kD^2}{z-k-1}+1)k(z-k-1) - \ell(z)=k^2 D^2 +k(z-k-1) - \ell(z)\) and \(k(z-k-1)-\ell(z)\leq kD - \frac{kD^2}{z-k-1}\leq kD-kD=0\) using \(z\leq D\), so \(\ell(z)(k(z-k-1)-1)\leq k^2 D^2\). On the other hand, \(\ell(z)k(z-k-1)\geq\frac{kD^2}{z-k-1}k(z-k-1)=k^2D^2\). So clearly there is such a sequence.] such that each~\(a_i\) is either~\(k(z-k-1)-1\) or~\(k(z-k-1)\), and \(\sum_{i=1}^{\ell}a_i=k^2D^2\).
Set~\(H^{(z)}\) to be the vertex-disjoint union of \(S(a_1), S(a_2), \dots, S(a_{\ell(z)})\).
It follows that each vertex of~\(H^{(z)}\) has degree at least \(\frac{k(z-k-1)-2}{k}-k(k+1)\geq z-(k+1)(k+3)\), and at most \(\frac{k(z-k-1)-1}{k}+k+1\leq z\). [\(\frac{k(z-k-1)-2}{k}-k(k+1) = z-k-1-2/k -k^2 -k \geq z-k-3-k^2-k \geq z-k^2 - 4k - 3 = z-(k+1)(k+3)\) using \(k\geq 1\), and \(\frac{k(z-k-1)-1}{k}+k+1 = z-k-1-1/k + k + 1\leq z\) is clear.]
Now construct~\(\cH\) by taking \(T\coloneqq k^2D^2\) vertex-disjoint copies \(\cH_1, \dots, \cH_T\) of~\(H_1\), identifying \(H_1=\cH_1\).
For each vertex \(v\in V(H_1)\), let~\(v^i\) denote the clone vertex of~\(v\) in~\(\cH_i\).
For each vertex~\(v\) of~\(H_1\) having \(\text{deg}_{H_1}(v)<(1-B^{-1})D\) (each of which must be in~\(V_1\setminus X\) by hypothesis), we now extend~\(\cH\) by making \(\cH[v^1,\dots,v^T]\) induce a copy of \(H^{(D-\text{deg}_{H_1}(v))}\).
(Here, clearly \(D-\text{deg}_{H_1}(v)\geq B^{-1}D\geq \sqrt{DD_2}\geq\sqrt{D}\geq Z\) by hypothesis on~\(B\).)
Then the \(\cH\)-degrees of all clones of~\(V_1\setminus X\) vertices are at least~\(D-(k+1)(k+3)\) and at most~\(D\).
Since \((k+1)(k+3)\leq \sqrt{D}\leq B^{-1}D\), we deduce that~\(\cH\) is \((T|V_1|,D,B^{-1})\)-regular, and moreover clearly satisfies \(C_j(\cH)=C_j(H_1)\leq D_j\) for all \(2\leq j\leq k+1\).}

We now show that Theorem~\ref{theorem:withreserves} follows quickly from Theorem~\ref{theorem:maintheorem}, by using weight functions~\(\tau_y\) which ensure the surviving \(H_2\)-degree of each \(y\in Y\) is averagely affected by the nibble on~\(H'_1\).
In order to spread out the~\(\tau_y\)-values so that no~\(\tau_y\) is too concentrated on a small number of vertices (i.e.\ so that~\ref{mainpseud:lowerbound} is satisfied, which is required for successful concentration of~\(\tau_y\) within a nibble), we use another cloning trick.
\lateproof{Theorem~\ref{theorem:withreserves}}
Deleting \(H_2\)-edges if necessary, we may assume that each \(x\in X\) satisfies \(\Delta_Y B^{-1+\gamma}\log^A D \leq \text{deg}_{H_2}(x)\leq D^{\log D}/k\) (here using~\ref{withmaxY}).\COMMENT{Need \(\Delta_Y B^{-1+\gamma}\log^A D\leq D^{\log D}/k\) for this range to exist. But \(1\leq B\leq\sqrt{D/D_2}\leq \sqrt{D}\) so LHS \(\leq \Delta_Y B^\gamma \log^A D\leq\Delta_Y D^{\gamma/2}\log^A D\leq\Delta_Y D^{3\gamma/4}\stackrel{(ii)}{\leq}D^{(\log D)-\gamma/4}\leq D^{\log D}/k.\)}
We may also assume that each \(y\in Y\) is not isolated in~\(H_2\).\COMMENT{Otherwise delete them}
Let \(H'_1\supseteq H_1\) be as given by Lemma~\ref{lemma:regularization}, and construct~\(H''_1\) to be~\(\lcl B\rcl\) vertex-disjoint copies \(\cH_1, \cH_2,\dots, \cH_{\lcl B\rcl}\) of~\(H'_1\), identifying \(\cH_1=H'_1\).
Clearly~\(H''_1\) has the same vertex degrees and~\(C_j\) values as~\(H'_1\).
Let~\(X_i\subseteq V(\cH_i)\) be the copy of~\(X\) in~\(\cH_i\), so that~\(H_2\) is bipartite with bipartition~\((X_1, Y)\) (we do not add edges to~\(H_2\) between~\(Y\) and any other~\(X_j\)).

For each \(y\in Y\) having \(\text{deg}_{H_2}(y)\geq Bq\), we define the weight function \(\tau_y\colon V(H''_1)\rightarrow\mathbb{R}_{\geq0}\) by putting \(\tau_y(x)\coloneqq\text{codeg}_{H_2}(x,y)\) for all \(x\in X\,(=X_1)\) and \(\tau_y(v)\coloneqq 0\) for all other \(v\in V(H''_1)\).
For those \(y\in Y\) having \(\text{deg}_{H_2}(y)<Bq\), select an \(x\in X\) having \(1\leq\text{codeg}_{H_2}(x,y)\leq q\) arbitrarily\COMMENT{Such \(x\) must exist since \(y\) not isolated, and the hypothesis on \(q\)}.
Let \(x_i\in X_i\) be the clones of~\(x\) for \(i\in\{2,\dots,\lcl B\rcl\}\).
Put \(\tau_y(x')\coloneqq\text{codeg}_{H_2}(x',y)\) for all \(x'\in X_1\setminus \{x\}\) and \(\tau_y(v)\coloneqq 0\) for all \(v\in V(H''_1)\setminus(X_1\cup\{x_2,\dots,x_{\lcl B\rcl}\})\).
Beginning with \(\tau_y(x)=\text{codeg}_{H_2}(x,y)\) and \(\tau_y(x_i)=0\) for \(i\in\{2,\dots,\lcl B\rcl\}\), inflate the~\(\tau_y\)-values (without actually adding any edges) of~\(x\) first, then each~\(x_i\) in order, each to a maximum value of~\(q\), until we have \(\tau_y(V(H''_1))=Bq\).
We seek to apply Theorem~\ref{theorem:maintheorem} to~\(H''_1\) with weight functions \(\{\tau_y\colon y\in Y\}\).
Notice each~\(\tau_y\) satisfies \(B\cdot\max_{v\in V(H''_1)}\tau_y(v)\leq Bq\leq\tau_y(V(H''_1))\leq\Delta_Y\) (here the hypothesis \(B\leq\Delta_Y/q\) ensures this range exists), so, in particular,~\ref{mainpseud:lowerbound} holds.
Clearly each~\(\tau_y\) satisfies \(|\{v\in V(H''_1)\colon\tau_y(v)>0\}|\leq\Delta_Y+B\leq D^{(\log D)-\gamma}+D\leq D^{\log D}\) so~\ref{mainpseud:supportsize} holds.
Further, each \(x\in X\) (or clone) is only in (at most)\COMMENT{For those \(y\) that already had high enough degree, the original is in the support of those \(\tau_y\), but the clones aren't} as many of the sets~\(\{v\in V(H''_1)\colon\tau_y(v)>0\}\) as the number of vertices~\(y\in Y\) that~\(x\) (the non-clone version) is in an~\(H_2\)-edge with, which is at most \(k\cdot D^{\log D}/k=D^{\log D}\), so~\ref{mainpseud:maxinvolvement} holds. 
Notice that we may thus apply Theorem~\ref{theorem:maintheorem} to~\(H''_1\) with the~\(B\) as in Theorem~\ref{theorem:withreserves}\COMMENT{An extra constraint can only make \(B\) smaller, so, fine. We already checked \(H''_1\) inherits all the values \(D, D_j, \eps=1/B\) from \(H'_1\).} (in particular, put \(\eps=1/B\)), and~\(A/2\) playing the role of~\(A\), say.

Let~\(M\) be the obtained matching of~\(H''_1\), and let~\(M'\subseteq M\) be those edges covering a vertex in \(X=X_1\).
By the properties of the bijections \(f_x\colon\partial_{H_1}(x)\rightarrow\partial_{H'_1}(x)\), for every \(M'\)-edge~\(e'\), there is an \(H_1\)-edge \(e''\subseteq e'\) covering the same vertices of~\(X\), so the set~\(M''\) of these~\(e''\) is a matching of~\(H_1\), and with \(X'\coloneqq X\setminus V(M'')\), we have \(\tau_y(X')\leq\tau(V(H''_1)\setminus V(M))\leq\tau_y(V(H''_1))B^{-1+\gamma}\log^{A/2}D\leq\Delta_Y B^{-1+\gamma}\log^{A/2}D\) for all \(y\in Y\) by Theorem~\ref{theorem:maintheorem}.
Consider \(H'_2\coloneqq H_2[X'\cup Y]\).
Notice that, since~\(H_2\) is bipartite with bipartition~\((X,Y)\), each \(y\in Y\) satisfies\COMMENT{If \(y\) is a vertex originally having high degree, so \(\tau_y\) untampered with, then the first inequality is an equality. Otherwise, we still only ever increased the \(\tau_y\)-values (and for only one \(x\in X\) for this \(y\)), i.e.\ \(\text{codeg}_{H_2}(x,y)\leq\tau_y(x)\)} \(\text{deg}_{H'_2}(y)=\sum_{x\in X'}\text{codeg}_{H_2}(x,y)\leq\tau_y(X')\leq\Delta_Y B^{-1+\gamma}\log^{A/2}D\).
Since \(\text{deg}_{H'_2}(x)=\text{deg}_{H_2}(x)\geq \Delta_Y B^{-1+\gamma}\log^A D\) for all \(x\in X\) by~\ref{withminX},~\(H'_2\) has an \(X'\)-perfect matching~\(M'''\) by Lemma~\ref{lemma:bipmopup},\COMMENT{Here is where it matters that \(\ell\) is constant (well, at most some polylog would suffice, but still)} so \(M''\cup M'''\) is an \(X\)-perfect matching of~\(H=H_1\cup H_2\), as required.
\endproof
We now use Theorem~\ref{theorem:withreserves} to prove Theorem~\ref{thm:mainpartitetheorem} by first taking a random~\(\mathbf{Y'}\subseteq Y\) of an appropriate size, and using Theorem~\ref{theorem:kimvu} and the Lov\'asz Local Lemma (Lemma~\ref{lemma:LLL}) to show there exists an outcome~\(Y'\) of~\(\mathbf{Y'}\) in which \(H_1\coloneqq H[X\cup(Y\setminus Y')]\) and \(H_2\coloneqq H[X\cup Y']\) satisfy the hypotheses of Theorem~\ref{theorem:withreserves}.
\lateproof{Theorem~\ref{thm:mainpartitetheorem}}
The result holds by Hall's Marriage Theorem if \(k=1\),\COMMENT{If it's 2-uniform, it holds without the \(B^{-1+\gamma}\log^A D\) on the left; indeed suppose \(G\) is a bipartite graph with parts \(X\) and \(Y\) and \(\text{deg}(x)\geq D\) and \(\text{deg}(y)\leq D\). Let \(S\subseteq X\). Counting the edges between \(S\) and \(N(S)\) in two ways, \(|S|D\leq \sum_{x\in S}\text{deg}(x)=e(S,N(S))\leq \sum_{y\in N(S)}\text{deg}(y)\leq |N(S)|D\). If it has some 1-edges and they appear only in \(Y\) then just ignore them; if some appear in~\(X\) then take each of these as edges in the matching and apply Hall's to what remains} so assume \(k\geq 2\) throughout.
If \(X=\emptyset\) then the result is trivial, so assume \(X\neq\emptyset\).\COMMENT{Not precluded by the hypotheses since each element of an empty set can satisfy the lower bound on the \(X\)-degrees. But if \(X\) is empty then the empty matching covers it}
If~\(H\) is \(1\)-uniform then clearly~\(\{\{x\}\colon x\in X\}\) is an \(X\)-perfect matching, and otherwise we have \(|Y|\geq D/D_2\geq B^2\).
It is easy to see that if \(B^{-1+2\gamma/3}\geq 1/\log^A D\), then the result holds by Lemma~\ref{lemma:bipmopup},\COMMENT{Firstly if \(B\leq\log D\) then clearly \(B^{-1+\gamma}\log^A D=\omega(1)\) so done. So \(B\geq\log D\). But then if \(B^{-1+2\gamma/3}\geq1/\log^A D\) then \((1+B^{-1+\gamma}\log^A D)D\geq (1+B^{\gamma/3})D=\omega(D)\) so done} so assume \(B^{1-2\gamma/3}>\log^A D\) throughout.
By deleting edges, we may assume that \(\deg_H(x)=\lceil(1+B^{-1+\gamma}\log^A D)D\rceil\) for all \(x\in X\).

Set \(p\coloneqq B^{-1+2\gamma/3}\log^A D\) (whence \(0<p<1\)).
Construct a \((k+1)\)-uniform \(H'\supseteq H\) by adding, for each edge~\(e\) having \(r<k+1\) vertices, a set~\(W_e\) of \(k+1-r\) dummy vertices (these~\(W_e\) disjoint for different~\(e\)) and adjusting~\(e\) to include~\(W_e\).
If \(|e|=k+1\) in~\(H\) then put \(W_e=\emptyset\).
Put \(Y^{*}\coloneqq Y\cup\bigcup_{e\in E(H)}W_e\) and notice this does not increase any of the values~\(C_j(H)\), and the dummy vertices~\(w\) have \(\text{deg}(w)=1\leq D\), and~\(H'\) is \((k+1)\)-uniform and bipartite with bipartition~\((X,Y^{*})\).
Construct a random set \(\mathbf{Y'}\subseteq Y^{*}\) by putting each \(y\in Y^{*}\) in~\(\mathbf{Y'}\) independently with probability~\(p\).
Put \(\mathbf{H}_2\coloneqq H'[X\cup\mathbf{Y'}]\) and \(\mathbf{H}_1\coloneqq H'[X\cup (Y^{*}\setminus \mathbf{Y'})]\).
For \(y\in Y^{*}\), set \(\textbf{deg}\mathbf{'}(y)\coloneqq |\{e\in\partial_{H'}(y)\colon ((e\cap Y^{*})\setminus\{y\})\subseteq\mathbf{Y'}\}|\), and for \(x\in X, y\in Y^{*}\), set \(\textbf{codeg}\mathbf{'}(x,y)\coloneqq|\{e\in\partial_{H'}(\{x,y\})\colon ((e\cap Y^{*})\setminus\{y\})\subseteq\mathbf{Y'}\}|\), and notice that if \(y\in\mathbf{Y'}\), then~\(\textbf{deg}\mathbf{'}(y)\) and~\(\textbf{codeg}\mathbf{'}(x,y)\) are simply~\(\text{deg}_{\mathbf{H}_2}(y)\) and~\(\text{codeg}_{\mathbf{H}_2}(\{x,y\})\) respectively.
Set\COMMENT{It turns out it does not suffice to just put \(D'=\lceil(1+B^{-1+\gamma}\log^A D)D\rceil\) for us. This is too wasteful. Indeed if you do so, then the \(H_1\)-degrees end up being \(D'(1-p)^k\geq D'(1-pk)\) ish, and \(pk=kB^{-1+2\gamma/3}\log^A D\) is clearly more than \(1/B\). I thought perhaps use \(B'\) slightly smaller than B, but it will not work. You need essentially \(pk\leq (1/B')\), but you also need that \(Dp^k\geq (B')^{-1+\gamma/3}\log^A D \cdot Dp^{k-1}\), i.e.\ you need \(p\gg(1/B')\). What I've done to circumnavigate this is really be more careful to concentrate the \(H_1\)-degrees close to their expectation. This then yields an error term which is only logarithmically larger than \(B^{-1}\), independent of the power of \(\gamma\) in the \(p\), and it now does suffice to take a slightly smaller \(B'<B\).} \(D'\coloneqq \lceil(1+B^{-1+\gamma}\log^A D)D\rceil (1-p)^k +(\log^{6k}D)DB^{-1}\) and observe that\COMMENT{\((1-p)^k=(1-kp/k)^k\geq 1-kp\) by \((1+x/n)^n\geq 1+x\) which holds for \(n\geq 1\), \(|x|\leq n\), with \(n=k\), \(x=kp\)}
\begin{equation}\label{eq:D'geqD}
D'\geq(1+B^{-1+\gamma}\log^A D)(1-kp)D\geq (1+B^{-1+\gamma}\log^A D - 2kp)D\geq D,
\end{equation}
where in the last inequality we used that \(2kp=2kB^{-1+2\gamma/3}\log^A D\leq B^{-1+\gamma}\log^A D\), since \(B^{\gamma/3}\geq 2k\).
Define \(Q\coloneqq\max_{j\in\{2,3,\dots,k+1\}}(D_j p^{k+1-j})\).
We will show each of the following:
\begin{enumerate}[label=\upshape(\roman*)]
\item \(C_j(\mathbf{H}_1)\leq D_j\) for all \(2\leq j\leq k+1\), deterministically;\label{determsequence}
\item Fix \(x\in X\). Then \(\prob{\text{deg}_{\mathbf{H}_2}(x)<\frac{1}{2}Dp^k}\leq D^{-\frac{1}{2}\log^4 D}\);\label{degxH2bound}
\item Fix \(x\in X\). Then \(\prob{\text{deg}_{\mathbf{H}_1}(x)>D'}\leq D^{-\frac{1}{2}\log^4 D}\);\label{degxH1up}
\item Fix \(x\in X\). Then \(\prob{\text{deg}_{\mathbf{H}_1}(x)<D'(1-2B^{-1}\log^{6k}D)}\leq D^{-\frac{1}{2}\log^4 D}\);\label{degxH1down}
\item Fix \(y\in Y^*\). Then \(\prob{\textbf{deg}\mathbf{'}(y)>2Dp^{k-1}}\leq D^{-\frac{1}{2}\log^4 D}\);\label{degybound}
\item Fix \(x\in X, y\in Y^*\). Then\COMMENT{This one is interesting. I'm glad this popped out because we can't always expect it to be close to its expectation of \(D_2 p^{k-1}\). Indeed, in rainbow links \(D_2 p^{k-1}\approx n^{-1/2}\). As it stands notice that \(D_{k+1}p^0\) is one of the terms, so this is always a more sensible statement. More on this later} \(\prob{\textbf{codeg}\mathbf{'}(x,y)>2Q\log^{5k}D}\leq D^{-\frac{1}{2}\log^4 D}\).\label{codegxybound}
\end{enumerate}
Indeed,~\ref{determsequence} clearly holds for any choice of~\(\mathbf{Y'}\).
Suppose for now that none of the bad events in \ref{degxH2bound}--\ref{codegxybound} occur for some outcome~\(Y'\) of~\(\mathbf{Y'}\) (and \(H_1, H_2\) of \(\mathbf{H}_1, \mathbf{H}_2\) respectively, say); that is,
\begin{enumerate}[label=(\alph*)]
\item $C_j(H_1) \leq D_j$ for all $2 \leq j \leq k + 1$; \label{bipartite-pf-H1codeg}
\item $\deg_{H_2}(x) \geq \frac{1}{2}Dp^k$ for every $x \in X$;\label{bipartite-pf-H2xdeg}
\item $(1 - 2B^{-1}\log^{6k}D)D' \leq \deg_{H_1}(x) \leq D'$ for every $x \in X$; \label{bipartite-pf-H1xdeg}
\item $\deg_{H_2}(y) \leq 2Dp^{k-1}$ for every $y \in Y'$;\label{bipartite-pf-H2ydeg}
\item $\text{codeg}_{H_2}(\{x,y\}) \leq 2Q\log^{5k}D$ for every $x \in X$ and $y \in Y^*$.\label{bipartite-pf-H2codeg}
\end{enumerate}

Set \(\Delta_{Y'}\coloneqq 2Dp^{k-1}\) (and note \(\Delta_{Y'}\leq (D')^{(\log D')-\gamma})\) and \(q\coloneqq2Q\log^{5k}D\).
We seek to apply Theorem~\ref{theorem:withreserves} with~\(D',\gamma/3, B'\coloneqq B/(2\log^{6k}D)\) playing the role of~\(D,\gamma, B\) respectively.
To that end, first observe that \ref{withdeg} of Theorem~\ref{theorem:withreserves} holds by \ref{bipartite-pf-H1codeg} and \ref{bipartite-pf-H1xdeg}, \ref{withmaxY} of Theorem~\ref{theorem:withreserves} holds by \ref{bipartite-pf-H2ydeg}, \ref{withq} of Theorem~\ref{theorem:withreserves} holds by \ref{bipartite-pf-H2codeg},  and \ref{withminH1} of Theorem~\ref{theorem:withreserves} holds by \ref{bipartite-pf-H1xdeg}.
To show \ref{withB} of Theorem~\ref{theorem:withreserves} observe that \(B'\geq 1\) and \(B'\leq B\leq\sqrt{D/D_2}\leq\sqrt{D'/D_2}\) by~(\ref{eq:D'geqD}), and similarly for each \(j\in\{4,5,\dots,k+1\}\) we have \(B'\leq B\leq (D/D_j)^{1/(j-1)}\leq (D'/D_j)^{1/(j-1)}\).
We show \(B\leq\Delta_{Y'}/q\), from which \(B'\leq \Delta_{Y'}/q\) follows.
To that end, we firstly claim that
\begin{equation}\label{eq:convenient}
\frac{Dp^{j-2}}{D_j}\geq B\log^A D\,\,\,\text{for all}\,\,j\in\{2,3,\dots,k+1\}.
\end{equation}
Indeed, for \(j\in\{2,3\}\) we have \(Dp^{j-2}/D_j \geq Dp/D_2\geq D\log^A D/(D_2 B)\geq B\log^A D\) by the hypothesis \(B\leq\sqrt{D/D_2}\), and for all other~\(j\) we have \(Dp^{j-2}/D_j\geq D\log^A D/(D_j B^{j-2}) = DB\log^A D/(D_j B^{j-1})\geq B\log^A D\) by the hypothesis \(B\leq (D/D_j)^{1/(j-1)}\), so~(\ref{eq:convenient}) holds.
Set (throughout the proof) \(J(Q)\) to be the smallest \(j\in\{2,3,\dots,k+1\}\) which attains \(Q=D_{J(Q)}p^{k+1-J(Q)}\).
Notice that
\begin{equation}\label{eq:crucialratio}
\frac{Dp^k}{Q} = \frac{Dp^{J(Q)-1}}{D_{J(Q)}}\stackrel{(\ref{eq:convenient})}{\geq}pB\log^A D\geq\log^A D.
\end{equation}
Thus \(\Delta_{Y'}/q=Dp^{k-1}/(Q\log^{5k}D)\geq pB/p=B\), as desired.
Finally, to show \ref{withminX} of Theorem~\ref{theorem:withreserves}, we use \ref{bipartite-pf-H2xdeg}; we need to check that \((1/2)Dp^{k}\geq \Delta_{Y'}(B')^{-1+\gamma/3}\log^A (D')\), for which it suffices to show \(p\geq 4(B')^{-1+\gamma/3}\log^A (D')\).
Indeed, we have \(D'\leq D^2\) and \(8(B')^{-1+\gamma/3}\log^A D= 8B^{-1+\gamma/3}\log^A D \cdot (2\log^{6k}D)^{1-\gamma/3}\leq B^{-1+2\gamma/3}\log^A D=p\), using \(B^{\gamma/3}\geq \log^{\gamma A/6}D\geq 16\log^{6k}D\), say.
We conclude that Theorem~\ref{theorem:withreserves} yields an \(X\)-perfect matching of \(H_1\cup H_2\subseteq H'\).\COMMENT{Out of interest, edges with a vertex in \(\mathbf{Y'}\) and a vertex in \(Y\setminus\mathbf{Y'}\) don't make it into \(\mathbf{H}_1\) or \(\mathbf{H}_2\).}
It is clear that this translates to a matching of~\(H\) by ignoring the dummy vertices in any edge of~\(M\).

It remains to show each of \ref{degxH2bound}--\ref{codegxybound}, which we do in turn using Theorem~\ref{theorem:kimvu}, and to show there is an outcome in which all of these bad events simultaneously do not occur, for which we use Lemma~\ref{lemma:LLL}.
For \(y\in Y^{*}\), let~\(\mathbf{I}_y\) be the indicator random variable for the event \(y\in\mathbf{Y'}\) and set~\(\mathbf{I'}_y\coloneqq 1-\mathbf{I}_y\).
Set \(D''\coloneqq \lceil(1+B^{-1+\gamma}\log^A D)D\rceil\).
\vspace{1mm}
\newline\noindent\ref{degxH2bound}: Fix \(x\in X\).
Clearly \(\expn{\text{deg}_{\mathbf{H}_2}(x)}=D''p^k\geq Dp^k\).
Construct the auxiliary hypergraph~\(\cH\) with \(V(\cH)=Y^{*}\), and the edges of~\(\cH\) are those \(k\)-sets \(U\subseteq Y^{*}\) for which \(\text{codeg}_{H'}(U\cup\{x\})>0\).
For each \(U\in E(\cH)\), put \(w(U)\coloneqq\text{codeg}_{H'}(U\cup\{x\})\).\COMMENT{Formulated this way to permit multiplicity}
We have \(\mathbf{Z}_{\cH}\coloneqq\sum_{e\in E(\cH)}w(e)\prod_{y\in e}\mathbf{I}_y=\text{deg}_{\mathbf{H}_2}(x)\).
For \(i\in[k]\) and \(W\in\binom{Y^*}{i}\) we have \(\expn{\mathbf{Z}_{\cH_W}}=\text{codeg}_{H'}(W\cup\{x\})p^{k-i}\leq D_{i+1}p^{k-i}\leq Q\),\COMMENT{Notice if you change index \(j=i+1\) then this is one of the terms in \(Q\)} so \(\mathbb{E}''[\mathbf{Z}_{\cH}]\leq Q\).
Further, \(\mathbb{E}'[\mathbf{Z}_{\cH}]\leq\max(D''p^k, Q)=D''p^k\) by~(\ref{eq:crucialratio}).
We apply Theorem~\ref{theorem:kimvu} with \(\lambda\coloneqq\log^5 D\) and use the fact that~\(\text{deg}_{\mathbf{H}_2}(x)\) is decided by at most~\(2Dk\) random variables~\(\mathbf{I}_y\) to deduce that
\[
\prob{\left|\text{deg}_{\mathbf{H}_2}(x)-\expn{\text{deg}_{\mathbf{H}_2}(x)}\right|\geq f}\leq 2e^2\exp(-\log^5 D + (k-1)\log(2Dk))\leq D^{-\frac{1}{2}\log^4 D},
\]
where \(f=(8\log^5 D)^k \sqrt{k! D''p^k Q}\leq (\log^{A/4}D)Dp^k\sqrt{Q/Dp^k}\leq (1/2)Dp^k\) by~(\ref{eq:crucialratio}).
\vspace{1mm}
\newline\noindent\ref{degxH1up}--\ref{degxH1down}: Fix \(x\in X\).
We have \(\expn{\text{deg}_{\mathbf{H}_1}(x)}=D''(1-p)^k\).
Construct the auxiliary hypergraph~\(\cH\) with \(V(\cH)\coloneqq Y^{*}\), and the edges of~\(\cH\) are those \(k\)-sets \(U\subseteq Y^{*}\) with \(\text{codeg}_{H'}(U\cup\{x\})>0\).
For each \(U\in E(\cH)\), put \(w(U)\coloneqq \text{codeg}_{H'}(U\cup\{x\})\).
We have \(\mathbf{Z}_{\cH}\coloneqq\sum_{e\in E(\cH)}w(e)\prod_{y\in e}\mathbf{I'}_y=\text{deg}_{\mathbf{H}_1}(y)\).
For \(i\in[k]\) and \(W\in\binom{Y^*}{i}\) we have \(\expn{\mathbf{Z}_{\cH_{W}}}=\text{codeg}_{H'}(W\cup\{x\})(1-p)^{k-i}\leq D_{i+1}\).
Thus, it is clear that \(\mathbb{E}'[\mathbf{Z}_{\cH}]\leq 2D\), and \(\mathbb{E}''[\mathbf{Z}_{\cH}]\leq D_2\).
We apply Theorem~\ref{theorem:kimvu} with \(\lambda\coloneqq\log^5 D\) to deduce that
\[
\prob{\left|\text{deg}_{\mathbf{H}_1}(x)-\expn{\text{deg}_{\mathbf{H}_1}(x)}\right|\geq(8\log^5 D)^k \sqrt{k!\cdot 2DD_2}}\leq D^{-\frac{1}{2}\log^4 D},
\]
and we have that \(\sqrt{DD_2}=D\sqrt{D_2/D}\leq D/B\), so with probability at least \(1-D^{-(1/2)\log^4 D}\), \(\text{deg}_{\mathbf{H}_1}(x)\) is within \((\log^{6k}D)\cdot (D/B)\) of its expectation, so that \(\text{deg}_{\mathbf{H}_1}(x)\leq D''(1-p)^k + (\log^{6k}D)DB^{-1} = D'\), and
\begin{eqnarray*}
\text{deg}_{\mathbf{H}_1}(x) & \geq & D''(1-p)^k - \frac{D\log^{6k}D}{B} = D'\frac{D''(1-p)^k - DB^{-1}\log^{6k}D}{D''(1-p)^k + DB^{-1}\log^{6k}D}\\ & = & D'(1-2(D')^{-1}DB^{-1}\log^{6k}D)\stackrel{(\ref{eq:D'geqD})}{\geq}D'(1-2B^{-1}\log^{6k}D).
\end{eqnarray*}
\vspace{1mm}
\newline\noindent\ref{degybound}: Fix \(y\in Y^*\). We have \(\expn{\textbf{deg}\mathbf{'}(y)}=\text{deg}_{H'}(y)p^{k-1}\leq Dp^{k-1}\).
Construct the auxiliary hypergraph~\(\cH\) with \(V(\cH)\coloneqq Y^*\setminus\{y\}\), and the edges of~\(\cH\) are those \((k-1)\)-sets\COMMENT{Recall we assume \(k\geq 2\). The minimum permissible uniformity in Kim-Vu is 1, and we have here uniformity \(k-1\geq 1\) so this is OK.} \(U\subseteq Y^*\setminus\{y\}\) with \(\text{codeg}_{H'}(U\cup\{y\})>0\).
For each \(U\in E(\cH)\), put \(w(U)\coloneqq\text{codeg}_{H'}(U\cup\{y\})\).
We have \(\mathbf{Z}_{\cH}\coloneqq\sum_{e\in E(\cH)}w(e)\prod_{u\in e}\mathbf{I}_u=\textbf{deg}\mathbf{'}(y)\).
For \(i\in[k-1]\) and \(W\in\binom{Y^*\setminus\{y\}}{i}\), we have \(\expn{\mathbf{Z}_{\cH_W}}\leq D_{i+1}p^{k-1-i}\leq Q/p\).
Then \(\mathbb{E}'[\mathbf{Z}_{\cH}]\leq\max(Dp^{k-1},Q/p)=Dp^{k-1}\) by~(\ref{eq:crucialratio}), and \(\mathbb{E}''[\mathbf{Z}_{\cH}]\leq Q/p\).
We apply Theorem~\ref{theorem:kimvu} with \(\lambda\coloneqq\log^5 D\) to deduce that
\[
\prob{\left|\textbf{deg}\mathbf{'}(y)-\expn{\textbf{deg}\mathbf{'}(y)}\right|\geq f}\leq D^{-\frac{1}{2}\log^4 D},
\]
where \(f=(8\log^5 D)^{k-1}\sqrt{(k-1)!Dp^{k-1}(Q/p)}\leq(\log^{A/2}D)Dp^{k-1}\sqrt{Q/(pDp^{k-1})}\), which is at most~\(Dp^{k-1}\) by~(\ref{eq:crucialratio}).
\vspace{1mm}
\newline\noindent\ref{codegxybound}: Fix \(x\in X, y\in Y^*\).
We have \(\expn{\textbf{codeg}\mathbf{'}(x,y)}=\text{codeg}_{H'}(\{x,y\})p^{k-1}\leq D_2 p^{k-1}\leq Q\).
Construct the auxiliary hypergraph~\(\cH\) with \(V(H)\coloneqq Y^*\setminus\{y\}\), and the edges of~\(\cH\) are those \((k-1)\)-sets \(U\subseteq Y^*\setminus\{y\}\) with \(\text{codeg}_{H'}(U\cup\{x,y\})>0\).
For each \(U\in E(\cH)\), put \(w(U)\coloneqq\text{codeg}_{H'}(U\cup\{x,y\})\).
We have \(\mathbf{Z}_{\cH}\coloneqq\sum_{e\in E(\cH)}w(e)\prod_{u\in e}\mathbf{I}_u=\textbf{codeg}\mathbf{'}(x,y)\).
For \(i\in[k-1]\) and \(W\in\binom{Y^*\setminus\{y\}}{i}\), we have \(\expn{\mathbf{Z}_{\cH_W}}\leq D_{i+2}p^{k-1-i}\leq Q\).\COMMENT{Change index \(j=i+2\)}
We deduce that \(\mathbb{E}'[\mathbf{Z}_{\cH}], \mathbb{E}''[\mathbf{Z}_{\cH}]\leq Q\).
Then by Theorem~\ref{theorem:kimvu} with \(\lambda\coloneqq\log^5 D\), we have
\[
\prob{\left|\textbf{codeg}\mathbf{'}(x,y)-\expn{\textbf{codeg}\mathbf{'}(x,y)}\right|\geq (8\log^5 D)^{k-1}Q\sqrt{(k-1)!}}\leq D^{-\frac{1}{2}\log^4 D},
\]
and otherwise \(\textbf{codeg}\mathbf{'}(x,y)\leq D_2 p^{k-1}+(\log^{5k}D)Q\leq2Q\log^{5k}D\).\COMMENT{Here is where something interesting happens in my opinion; there is no guarantee that \(D_2p^{k-1}\) is the largest expression \(D_jp^{k+1-j}\) in \(Q\), and indeed for rainbow links it really isn't - there we have \(D_2p^{k-1}\approx n^{-1/2}\) and \(D_{k+1}p^{k+1-(k+1)}=D_{k+1}\geq 1\). So we just bound the expectation part above by \(Q\) as well as the tail. In the other Kim-Vu applications, it was guaranteed that the expectation part \(Dp^k\) was larger than \(Q\) (or \(Dp^{k-1}\) larger than \(Q/p\))}

It remains only to show there exists an outcome in which all of the bad events in \ref{degxH2bound}--\ref{codegxybound} simultaneously do not occur.
To that end, for \(x\in X, y\in Y^*\), let \(A^{(1)}_x, A^{(2)}_x, A^{(3)}_x, A_y, A_{(x,y)}\) denote the bad events in \ref{degxH2bound}--\ref{codegxybound} respectively.
Construct an auxiliary digraph~\(G\) with a vertex for each of these events, and an arc in both directions between two distinct vertices if the corresponding events are not decided by disjoint sets of trials \(y'\in\mathbf{Y'}\).
Clearly~\(G\) is a dependency digraph.\COMMENT{If a vertex \(a\) has no out-arcs to a whole set \(U\subseteq V(G)\) then for each \(u\in U\), the event corresponding to \(a\) is decided by a disjoint set of trials to the event corresponding to \(u\), whence it is decided by a set of trials disjoint to the entire set of trials determing all the events for \(u\in U\), i.e. \(A\) mut ind of all \(U\)}
One checks easily\COMMENT{For arcs to other \(A^{(j)}_{x'}\), 3 choices for~\(j\), \(D''k\) choices for the \(y\) whose outcomes interacts with both \(A^{(i)}_x\) and \(A^{(j)}_{x'}\), then \(D\) choices for an edge containing~\(y\), which determines \(x'\). To \(A_y\), \(D''k\) choices for the \(y'\) interfering with both, then \(D(k-1)\) choices for an edge containing \(y'\) and a non-\(y'\) vertex \(y\) of that edge. To \(A_{(x,y)}\), \(D''k\) choices for the \(y'\) interfering with both, then \(D(k-1)\) choices for an edge containing \(y'\) and a non-\(y'\) vertex \(y\) of that edge (and the \(x\) is uniquely determined by this last edge)} that any vertex corresponding to an event~\(A^{(i)}_x\) sends out at most \(3D''k D + 2D''kD(k-1)\leq D^3\) arcs to other \(G\)-vertices, and similarly the other \(G\)-vertices also\COMMENT{For \(A_y\) out to \(A^{(j)}_x\), there's 3 choices for \(j\), \(D(k-1)\) choices for the interfering \(y'\), then \(D\) choices for \(x\). Out to \(A_{y'}\) there's \(D(k-1)\) choices for the interfering \(y''\), then \(D(k-1)\) for \(y'\). Out to \(A_{(x,y')}\) the same as for out to \(A_{y'}\) since \(x\) determined by the last edge. So \(\leq 3D^2(k-1)+2D^2(k-1)^2\leq D^3\). For \(A_{(x,y)}\) out, there's at most \(D_2(k-1)\leq D(k-1)\) choices for the interfering \(y'\) then proceed as usual.} have out-degree at most~\(D^3\).
Setting \(f_v\coloneqq (1/2)D^{-\log^3 D}\)  for each \(v\in V(G)\) and letting~\(A_v\) denote the bad event corresponding to~\(v\), we then have \(f_v \prod_{w\in V(G)\colon vw\notin E(G)}(1-f_w)\geq (1/2)D^{-\log^3 D}\exp(-D^{-\log^{3}D}\cdot D^3)\geq (1/4)D^{-\log^3 D}\geq D^{-(1/2)\log^4 D}\geq\prob{A_v}\), so Lemma~\ref{lemma:LLL} finishes the proof. 
\endproof
\section{Applications}\label{Section:Applications}
In this section, we use Theorems~\ref{theorem:maintheorem},~\ref{thm:mainpartitetheorem}, and~\ref{theorem:maincolourtheorem} to prove each of the results stated in Section~\ref{section:appstatements} in turn.
Of these, the most involved proofs are those of Theorems~\ref{theorem:rainbowlinks} and~\ref{theorem:SimplicialComplex}, which both require a preliminary random partition.
For brief sketches of these two proofs, refer to Section~\ref{section:intro}.
\subsection{Rainbow directed cycles}\label{section:raincycles}
The aim of this subsection is to prove Theorem~\ref{theorem:rainbowlinks}, for which we first need some definitions.
For a set~\(Z\) and non-negative real numbers \(p_0, p_1, \dots, p_m\) satisfying \(\sum_{i\in[m]_0}p_i=1\), a \((p_0, p_1 \dots, p_m)\)-\textit{partition} \((\mathbf{Z}_0, \mathbf{Z}_1, \dots, \mathbf{Z}_m)\) of~\(Z\) is defined as follows:
Set \(q_0=0\) and for \(i\in[m+1]\) set \(q_{i}=\sum_{j\in[i-1]_0}p_j\), and for each \(z\in Z\) independently, choose a uniformly random label \(\ell_z\in[0,1)\) and put~\(z\in\mathbf{Z}_i\), where \(\ell_z\in [q_{i},q_{i+1})\).
Clearly~\((\mathbf{Z}_0, \mathbf{Z}_1, \dots, \mathbf{Z}_m)\) is indeed an (ordered) partition of~\(Z\), and we have that each \(z\in Z\) is placed, independently, in~\(\mathbf{Z}_i\) with probability~\(p_i\).

Recall that~\(\Phi(\dirK)\) is the set of proper \(n\)-arc colourings of~\(\dirK\) on vertex set \(V\coloneqq[n]\) and colour set \(C\coloneqq [n]\).
For a coloured digraph \(G\in\Phi(\dirK)\) and an arc \(e\in E(G)\), we denote by~\(\phi_G(e)\) (dropping the~\(G\) when~\(G\) is clear from context) the colour of~\(e\) in~\(G\).
For a vertex \(v\in V\) and a colour subset \(Q\subseteq C\), we denote by~\(N^{+}_{G,Q}(v)\) (respectively~\(N^{-}_{G,Q}(v)\)) the set of outneighbours (inneighbours) of~\(v\) in~\(G\) via arcs with colour in~\(Q\) (dropping~\(G\) when clear from context), and we denote by~\(N^{+}_{c}(v)\) (respectively~\(N^{-}_{c}(v)\)) the unique vertex in~\(N^{+}_{\{c\}}(v)\) (\(N^{-}_{\{c\}}(v)\)).
For a colour \(c\in C\) and sets \(V_1, V_2\subseteq V\), we denote by \(E_c(V_1,V_2)\) the set of arcs in~\(G\) with colour~\(c\), tail in~\(V_1\), and head in~\(V_2\).
A \textit{directed path} is a directed graph~\(P\) having (distinct) vertices \(v_1, v_2, \dots, v_r\) and arc set \(E(P)=\{v_iv_{i+1}\colon i\in[r-1]\}\).
We sometimes denote \(P=v_1 v_2 \dots v_r\).
The \textit{tail} of~\(P\) is~\(v_1\), the \textit{head} of~\(P\) is~\(v_r\), and the \textit{internal} vertices of~\(P\) are \(v_2, v_3, \dots, v_{r-1}\).
Suppose~\(G\in\Phi(\dirK)\) and \(u,v\in V\) are distinct.
A \((u,v)\)-\textit{rainbow link} is a directed path in~\(G\) which is rainbow and has tail~\(u\) and head~\(v\).
For sets \(V_1, V_2\subseteq V\) and \(C_1, C_2, C_3\subseteq C\), we define~\(\cL_G(u,C_1,V_1,C_2,V_2,C_3,v)\) to be the set of \((u,v)\)-rainbow links \(uv_1v_2v\) in~\(G\) (dropping~\(G\) when~\(G\) is clear from context) where \(v_i\in V_i\), \(\phi(uv_1)\in C_1, \phi(v_1v_2)\in C_2, \phi(v_2v)\in C_3\).
Define~\(C_{\text{bad}}(G)\) to be the set of those colours \(c\in C\) appearing on at least~\(\sqrt{n}\) loops of~\(G\).

We now give an easy lemma which finds partitions of~\(V\) and~\(C\setminus C_{\text{bad}}(G)\) with properties that will be useful in the proof of Theorem~\ref{theorem:rainbowlinks}.
\begin{lemma}\label{slicinglemma}
Suppose that \(1/n\ll1/\ell\ll1\) with~\(\ell\) even\COMMENT{It helps to never have to do any parity case analysis}, and fix \(G\in\Phi(\dirK)\).
Set \(\alpha\coloneqq(1-n^{-1/2+1/\ell})/(\ell+1)\), \(\beta\coloneqq (1-\alpha)/\ell\), \(\eps\coloneqq n^{-1/2}\log^2 n\).
Then there exist (ordered) partitions~\((W_0, W_1, \dots, W_{\ell})\) and~\((Q_0,Q_1,\dots,Q_{\ell})\) of~\(V\) and~\(C\setminus C_{\text{bad}}\) respectively such that the following conditions all hold:
\begin{enumerate}[(F1), topsep = 6pt]
\item \(|W_0|=(1\pm\eps)\alpha n\); \(|W_i|=(1\pm\eps)\beta n\) for all \(i\in[\ell]\); and \(|Q_i|=(1\pm\eps)\frac{n}{\ell+1}\) for all \(i\in[\ell]_0\);\label{MAIN1}
\item \(|(N^{*}_{Q_i}(v)\cap W_j)\setminus\{v\}|\COMMENT{In fact, we could get away without the "setminus v" here - it doesn't change the fact that this is a binomially distributed random variable, and we only ever apply \ref{MAIN2} in the main proof to find neighbours of \(v\) in a different set \(W_i\) to the one that \(v\) is in, so necessarily can't include \(v\) itself. Still, this shows we're taking some care to reject the loops... what do you think?}=(1\pm\eps)\frac{\beta}{\ell+1}n\) for all \(*\in\{+,-\}\), \(v\in V\), \(i\in[\ell]_0\), \(j\in[\ell]\); and \(|(N^{*}_{Q_i}(v)\cap W_0)\setminus\{v\}|=(1\pm\eps)\frac{\alpha}{\ell+1}n\) for all \(*\in\{+,-\}\), \(v\in V\), \(i\in[\ell]_0\);\label{MAIN2}
\item \(|E_c(W_i,W_{i+1})|=(1\pm\eps)\beta^2 n\) for all \(c\in C\setminus C_{\text{bad}}(G)\), \(i\in[\ell-1]\); and \(|E_c(W_0,W_1)|\), \(|E_c(W_{\ell},W_0)|=(1\pm\eps)\alpha\beta n\) for all \(c\in C\setminus C_{\text{bad}}(G)\);\label{MAIN3}
\item \(|\cL(u,Q_i,W_{i+1},Q_{i+1},W_{i+2},Q_{i+2},v)|=(1\pm\eps)\frac{\beta^2}{(\ell+1)^3}n^2\) for all \(i\in[\ell-2]_0\) and distinct \(u,v\in V\).\label{MAIN4}
\end{enumerate}
\end{lemma}
\begin{proof}
Set \(p_0\coloneqq\alpha\) and \(p_i\coloneqq\beta\) for \(i\in[\ell]\) (and notice \(\sum_{i\in[\ell]_0}p_i=1\)), and \(p'_j\coloneqq1/(\ell+1)\) for \(j\in[\ell]_0\).
Let \((\mathbf{W}_0,\mathbf{W}_1,\dots,\mathbf{W}_{\ell})\) be a \((p_0,p_1,\dots,p_{\ell})\) partition of~\(V\), and let \((\mathbf{Q}_0,\mathbf{Q}_1,\dots,\mathbf{Q}_{\ell})\) be a \((p'_0,p'_1,\dots,p'_{\ell})\)-partition of~\(C\setminus C_{\text{bad}}\) (these partitions independent of each other).
We will show that these random partitions satisfy each of \ref{MAIN1}--\ref{MAIN4} individually with high probability, from which the result follows.
Notice \(|C_{\text{bad}}(G)|\leq\sqrt{n}\).
\newline\noindent\ref{MAIN1}: This clearly holds with high probability by a simple application of Lemma~\ref{chernoff}\ref{lemma:chernoffbigexp}.
We omit the details.\COMMENT{Let \(n'\coloneqq|C\setminus C_{\text{bad}}(G)|\) so that \(n-\sqrt{n}\leq n''\leq n\). Clearly each of these r.v.s is binomial with the obvious parameters (in particular \(|Q_i|\sim\text{Bin}(n'',1/(\ell+1))\)).
Apply Lemma~\ref{chernoff}\ref{lemma:chernoffbigexp} with \(\eps\coloneqq(\sqrt{n}\log n)/\expn{|\mathbf{W}_i|}=\log n/ (\beta\sqrt{n})\) yields
\[
\prob{||\mathbf{W}_i|-\expn{|\mathbf{W}_i|}|\geq \sqrt{n}\log n}\leq 2\exp\left(-\frac{\log^2 n}{3\beta}\right) =o(1)
\]
and we have \(\sqrt{n}\log n \leq n^{-1/2}\log^2 n \cdot \beta n\). For \(|\mathbf{W}_0|\) do exactly the same thing except replace \(\beta\) with \(\alpha\).
For \(|\mathbf{Q}_i|\), we have \(|\mathbf{Q}_i|\sim\text{Bin}(n',\beta)\), where \(n(1-n^{-1/2})\leq n'\leq n\). One easily ends up with the same probability bound. Then the upper tail bound still applies, and for the lower tail, \(|\mathbf{Q}_i|\geq (1/(\ell+1)) n' - \sqrt{n}\log n \geq (1/(\ell+1)) n(1-n^{-1/2}-(\ell+1)n^{-1/2}\log n)\geq (1/(\ell+1)) n(1-n^{-1/2}\log^2 n)\) as required.}
\newline\noindent\ref{MAIN2}: For fixed \(v, *, i, j\), the fact that the colouring on~\(G\) is proper implies that \(|N^{*}_{\mathbf{Q}_i}(v)\cap\mathbf{W}_0|\sim\text{Bin}(n',\alpha/(\ell+1))\) and \(|N^{*}_{\mathbf{Q}_i}(v)\cap \mathbf{W}_j|\sim\text{Bin}(n',\beta/(\ell+1))\), where \(n-\sqrt{n}\leq n'\coloneqq |C\setminus C_{\text{bad}}(G)|\leq n\), so again this property holds with high probability by Lemma~\ref{chernoff}\ref{lemma:chernoffbigexp}.\COMMENT{I've left it as implicit that a vertex is its own inneighbour/outneighbour due to the loop. 
Consider \(\mathbf{Z}=|N^{*}_{\mathbf{Q}_i}(v)\cap \mathbf{W}_0|\) (do exactly the same for \(\mathbf{W}_j\) but replace \(\alpha\) with \(\beta\)). For now allowing \(v\) itself in the set, this means \(v\) starts with \(n\) candidates to be in \(N_{\mathbf{Q}_{i}^{*}}(v)\cap \mathbf{W}_0\), but then we must remove those via bad arc colours, as these can never be in any~\(\mathbf{Q}_i\). Now each of the remaining \(n'\) vertices is independently in \(N_{\mathbf{Q}_{i}^{*}}(v)\cap \mathbf{W}_0\) with probability \(\alpha/(\ell+1)\). Put \(\eps=\sqrt{n}\log n/\expn{\mathbf{Z}}=(\ell+1)\sqrt{n}\log n/(\alpha n')\) yields
\[
\prob{|\mathbf{Z}-\expn{\mathbf{Z}}|\geq\sqrt{n}\log n}\leq 2\exp\left(-\frac{n\log^2 n}{3\expn{\mathbf{Z}}}\right) \leq 2\exp\left(-\frac{(\ell+1)\log^2 n}{3\alpha}\right)=o(1/n),
\]
and there are \(2(\ell+1) n\) events to take a UB over so the failure probability is \(o(1)\). Simple to show that \(\expn{\mathbf{Z}}+\sqrt{n}\log n \leq (1+n^{-1/2}\log^2 n)(\alpha/(\ell+1)) n\) and \(\expn{\mathbf{Z}}-\sqrt{n}\log n =(\alpha/(\ell+1)) n' -\sqrt{n}\log n \geq(\alpha/(\ell+1)) n(1-n^{-1/2}-((\ell+1)/\alpha)n^{-1/2}\log n)\geq (\alpha/(\ell+1)) n(1-(1/2)n^{-1/2}\log^2 n)\), so now also possibly remove the extra 1 vertex \(v\) itself.}
\newline\noindent\ref{MAIN3}: Let \(i\in[\ell-1]\), \(c\in C\setminus C_{\text{bad}}(G)\), and notice \(\mathbf{Z}_c\coloneqq|E_c(\mathbf{W}_i,\mathbf{W}_{i+1})|\) is a function of the mutually independent random variables~\(\{\mathbf{I}_v\}_{v\in V}\) having (say) \(\mathbf{I}_v=i\) if \(v\in\mathbf{W}_i\).
Note that \(\expn{\mathbf{Z}_c}=\beta^2 n''\), where \(n-\sqrt{n}\leq n''\leq n\) is the number of non-loop \(c\)-arcs (recall \(c\notin C_{\text{bad}}(G)\)), and each~\(\mathbf{I}_v\) affects~\(\mathbf{Z}_c\) by at most~\(1\).\COMMENT{All the \(v\) hosting a \(c\)-loop have no affect at all because that vertex cannot simultaneously be in~\(\mathbf{W}_i\) and \(\mathbf{W}_{i+1}\). Otherwise \(v\) has a \(c\)-arc in and a \(c\)-arc out.
Only one of these arcs can count towards \(\mathbf{Z}_c\) at a time; the out-arc only if \(v\in\mathbf{W}_i\) (and the out-neighbour happens to be in \(\mathbf{W}_{i+1}\)), in which case the in-arc cannot count due to the asymmetry of \(E_c(\mathbf{W}_i,\mathbf{W}_{i+1})\), and similarly the other way around. If changing \(\mathbf{I}_v\) causes one of the two arcs to no longer count, and the other to count, then there has been no change. Otherwise the largest change that can happen is that changing \(\mathbf{I}_v\) causes a counting arc to no longer count (and the other one never did) or a non-counting arc to count (and the other one never did). We still need McDiarmid's rather than Chernoff though as the events \(e\in E_c(\mathbf{W}_i,\mathbf{W}_{i+1})\) are not mutually independent for \(e\) in the colour class of \(c\) in \(G\) - two of these events are clearly dependent if the two arcs share an endpoint, which can and does happen.}
This property thus holds with high probability for~\(\mathbf{Z}_c=|E_c(\mathbf{W}_i,\mathbf{W}_{i+1})|\) by a simple application of Lemma~\ref{lemma:mcd}.
The analyses of~\(|E_c(\mathbf{W}_0,\mathbf{W}_1)|\) and~\(|E_c(\mathbf{W}_{\ell}, \mathbf{W}_0)|\) are analogous, these latter random variables having expectation~\(\alpha\beta n''\) rather than~\(\beta^2 n''\).
We omit the details.\COMMENT{Focus on~\(\mathbf{Z}_c\) (and just replace one of the \(\beta\) with \(\alpha\) for the other r.v.s).
In the application of McDiarmid's, the number \(m\) of random variables in play is the number of vertices not hosting a \(c\)-loop, i.e.\ \(n''\).
\[
\prob{|\mathbf{Z}_c - \expn{\mathbf{Z}_c}|\geq \sqrt{n}\log n}\leq\exp\left(-\frac{2n\log^2 n}{n}\right)=\exp(-2\log^2 n)=n^{-2\log n},
\]
so UB this over \((\ell-1)n'\) events (recall \(n'\) the number of good colours) still yields probability \(o(1)\). Then similarly to before \(\beta^{2}n'' \pm \sqrt{n}\log n =\beta^{2} n(1\pm n^{-1/2}\log^2 n)\).}
\newline\noindent\ref{MAIN4}: Fix distinct \(u,v\in V\), \(i\in[\ell-2]_0\). \(\mathbf{Z}\coloneqq |\cL(u,\mathbf{Q}_i, \mathbf{W}_{i+1},\mathbf{Q}_{i+1},\mathbf{W}_{i+2},\mathbf{Q}_{i+2},v)|\) is a function of the mutually independent random variables~\(\{\mathbf{I}_{z}\}_{z\in V\cup C\setminus C_{\text{bad}}(G)}\) having (say) \(\mathbf{I}_w=i\) (respectively \(\mathbf{I}_c=i\)) if \(w\in\mathbf{W}_i\) (\(c\in \mathbf{Q}_i\)).
For a fixed \(w\in V\), note that there are at most~\(n\) \((u,v)\)-rainbow links in~\(G\) with two internal vertices, the ``first'' of which being~\(w\),\COMMENT{The whole link is fixed by choosing the second internal vertex.} and similarly at most~\(n\) links with two internal vertices, the second being~\(w\).
For a fixed \(c\in C\setminus C_{\text{bad}}(G)\), there are at most~\(n\) \((u,v)\)-rainbow links with two internal vertices, in which~\(c\) is the colour of the out-arc at~\(u\)\COMMENT{The first internal vertex is fixed as the \(c\)-outneighbour of~\(u\). \(v\) is already fixed. Fixing the second internal vertex identifies the whole link, and it may or may not be rainbow.}, and similarly at most~\(n\) links with two internal vertices in which~\(c\) is the colour of the in-arc at~\(v\).
Finally, since there are~\(n\) arcs coloured~\(c\) in~\(G\), there are at most~\(n\) \((u,v)\)-rainbow links with two internal vertices with the arc not incident to~\(u\) nor~\(v\) having colour~\(c\).
We deduce that~\(\mathbf{I}_z\) affects~\(\mathbf{Z}\) by at most~\(3n\) for each \(z\in V\cup C\setminus C_{\text{bad}}(G)\).\COMMENT{In fact, as before, due to the partitioning, \(w\) cannot contribute the links in which it's the first vertex whilst also contributing the links in which it's the second vertex - it can only play the role of one of these at a time, depending on whether \(v\in\mathbf{W}_{i+1}\) or \(v\in\mathbf{W}_{i+2}\). Similarly for the colours. So in fact they each affect by at most~\(n\), but it seems simpler not to both with clarifying this. Not sure what you think?}
There are \((n-2)(n-3)\) directed paths from~\(u\) to~\(v\) in~\(G\) with two internal vertices; at most~\(3n\) of these are not rainbow,\COMMENT{Fix a colour to be repeated, so there's at most \(n\) choices of this \(c\), 3 choices of which pair of arcs share this colour, and doing so fixes the whole link.} and at most~\(3n^{3/2}\) use a colour in~\(C_{\text{bad}}(G)\),\COMMENT{Pick one of the \(\leq\sqrt{n}\) bad colours, and one of the three arc roles. If it isn't the middle arc role, then there's \(\leq n\) choices for the remaining vertex, and if it is the middle arc, there's at most \(n\) choices for that arc, deciding the link} so that\COMMENT{Each rainbow link with 2 internal vertices and all good colours has probability \(\beta^2/(\ell+1)^3\) of counting towards~\(\mathbf{Z}\). Clearly \((n-2)(n-3)\leq n^2(1+4n^{-1/2})\) and \((n-2)(n-3)-3n - 3n^{3/2}= n^{2}-3n^{3/2}-8n+6\geq n^2 - 4n^{3/2}=n^2 (1-4n^{-1/2})\).} \(\expn{\mathbf{Z}}=(1\pm 4n^{-1/2})(\beta^{2}/(\ell+1)^3)n^2\).
A simple application of Lemma~\ref{lemma:mcd} now suffices, so we omit the details.\COMMENT{\(\prob{|\mathbf{Z}-\expn{\mathbf{Z}}|\geq n^{3/2}\log n}\leq\exp(-2n^3\log^2 n/(n+n')9n^2)\leq \exp(-(1/9)\log^2 n)=n^{-(1/9)\log n}\), so UB over at most \((\ell-1) n^2\) events give total failure probability \(o(1)\). Then \(\mathbf{Z}=\expn{\mathbf{Z}}\pm n^{3/2}\log n = (\beta^{2}/(\ell+1)^3)n^2(1\pm 4n^{-1/2} \pm ((\ell+1)^3/\beta^2)n^{-1/2}\log n) =(\beta^2/(\ell+1)^3) n^2(1\pm n^{-1/2}\log^2 n)\).}
\end{proof}
Equipped with Theorem~\ref{thm:mainpartitetheorem} and Lemma~\ref{slicinglemma}, we can now prove Theorem~\ref{theorem:rainbowlinks}.
\lateproof{Theorem~\ref{theorem:rainbowlinks}}
Clearly we may assume that \(1/n\ll\delta\ll1\).
Let~\(\ell\) be the smallest even integer strictly larger than~\(3/(2\delta)\), say.
Fix \(G\in\Phi(\dirK)\), set \(\alpha\coloneqq (1-n^{-1/2+1/\ell})/(\ell+1)\), \(\beta\coloneqq (1-\alpha)/\ell\), \(\eps\coloneqq n^{-1/2}\log^2 n\).
By Lemma~\ref{slicinglemma}, there exist partitions~\((W_0,W_1,\dots,W_{\ell})\) and~\((Q_0,Q_1,\dots,Q_{\ell})\) of~\(V\) and~\(C\setminus C_{\text{bad}}(G)\) respectively which satisfy~\ref{MAIN1}--\ref{MAIN4}.
Define \(m\coloneqq|W_0|\,(=(1\pm\eps)\alpha n\) by~\ref{MAIN1}).
Arbitrarily order \(W_0=\{x_1,x_2,\dots,x_m\}\).
Construct an auxiliary bipartite hypergraph~\(H\) with parts \(X\coloneqq\{(x_i,x_{i+1})\colon i\in[m-1]\}\cup\{(x_m,x_1)\}\) and \(Y\coloneqq\bigcup_{r\in[\ell]}W_r \cup\bigcup_{s\in[\ell]_0}Q_s\), and put the edge~\(\{(x_i,x_{i+1}), w_1, w_2, \dots, w_{\ell}, q_0, q_1, \dots, q_{\ell}\}\) (indices~\(i\) modulo~\(m\)) in~\(H\) if each \(w_r\in W_r\), \(q_s\in Q_s\), and \(x_iw_1w_2\dots w_{\ell}x_{i+1}\) is an \((x_i, x_{i+1})\)-rainbow link in~\(G\) with \(\phi(x_iw_1)=q_0\), \(\phi(w_rw_{r+1})=q_r\) for \(r\in[\ell-1]\), and \(\phi(w_{\ell}x_{i+1})=q_{\ell}\).
Then~\(H\) is \((2\ell+2)\)-uniform and bipartite with bipartition~\((X,Y)\).
Notice that an \(X\)-perfect matching of~\(H\) corresponds to a rainbow directed cycle of~\(G\) on~\((\ell+1)m\) vertices, and we have \((\ell+1)m\geq (1-\eps)(\ell+1)\alpha n = (1-\eps)(1-n^{-1/2+1/\ell})n\geq n(1-2n^{-1/2+1/\ell})\geq n-n^{1/2+\delta}\) as required.\COMMENT{\(2n^{1/\ell}\leq n^{\delta}\) since \(2/\delta > \ell > 3/(2\delta) > 1/\delta\).}
It remains only to check we may apply Theorem~\ref{thm:mainpartitetheorem} to~\(H\), for which we need to bound the degrees in~\(Y\) and~\(X\), and examine the codegree sequence.
Set \(D\coloneqq (1+n^{-1/2}\log^3 n)\alpha(\beta/(\ell+1))^{\ell}n^{\ell}\).
\begin{claim}\label{degYcheck}
\(\text{deg}_{H}(y)\leq D\) for all \(y\in Y\).
\end{claim}
\claimproof
First, consider some \(q\in Q_s\) and suppose without loss of generality that \(s\leq\ell/2\).
If \(s\geq 1\) then we have \(|E_{q}(W_s, W_{s+1})|\leq(1+\eps)\beta^2 n\) by~\ref{MAIN3}.
Fixing such an arc determines the vertices \(w_s, w_{s+1}\) in a rainbow link ``using'' \(q\in Q_s\).
Applying~\ref{MAIN2} repeatedly, proceeding through inneighbours from~\(w_s\), there are at most \(((1+\eps)(\beta/(\ell+1))n)^{s-1}\) valid choices for the vertices \(w_{s-1},w_{s-2},\dots,w_1\) via arcs with colours in the appropriate sets.
Then by~\ref{MAIN2} there are at most \((1+\eps)(\alpha/(\ell+1))n\) choices for an inneighbour~\(x_i\) of~\(w_1\) in~\(W_0\) via an arc with colour in~\(Q_0\).
Altogether there are at most \((1+\eps)^{s+1}\frac{\alpha \beta^{s+1}}{(\ell+1)^s}n^{s+1}\) valid choices for the directed path \(x_iw_1w_2\dots w_sw_{s+1}\).
If instead \(s=0\) then by~\ref{MAIN3} there are at most~\((1+\eps)\alpha\beta n = (1+\eps)^{s+1}\frac{\alpha\beta^{s+1}}{(\ell+1)^s}n^{s+1}\) valid choices for the arc \(x_iw_0\) with colour~\(q\).
Whether \(s=0\) or \(s\geq 1\), the choice of~\(x_i\) now fixes the~\(x_{i+1}\in W_0\) such that \((x_i,x_{i+1})\in X\).
Applying~\ref{MAIN2} repeatedly, proceeding through inneighbours from~\(x_{i+1}\), there are at most \(((1+\eps)(\beta/(\ell+1))n)^{\ell-s-3}\) valid choices for the vertices \(w_{\ell}, w_{\ell-1}, \dots, w_{s+4}\), and finally at most~\((1+\eps)(\beta^2/(\ell+1)^3)n^2\) valid choices for the remaining vertices \(w_{s+3}, w_{s+2}\) by~\ref{MAIN4}.
Altogether, \(\text{deg}_{H}(q)\leq(1+\eps)^{\ell-1}\frac{\alpha\beta^{\ell}}{(\ell+1)^{\ell}}n^{\ell}\leq D\).\COMMENT{\((1+n^{-1/2}\log^2 n)^{\ell-1}\leq 1+n^{-1/2}\log^3 n\)}

One proceeds similarly for \(w\in W_r\) (\(r\in[\ell]\)).
For any \(r\leq \ell/2\) (say), we obtain by repeated application of~\ref{MAIN2} that there are at most~\((1+\eps)^r \frac{\alpha\beta^{r-1}}{(\ell+1)^{r}}n^r\) choices for the directed path \(x_iw_1w_2\dots w_r\), where \(w_r=w\).
Then applying~\ref{MAIN2} repeatedly from~\(x_{i+1}\) there are at most~\(((1+\eps)(\beta/(\ell+1))n)^{\ell-r-2}\) valid choices for \(w_{\ell}, w_{\ell-1}, \dots, w_{r+3}\), and finally \((1+\eps)(\beta^2/(\ell+1)^3)n^2\) choices for \(w_{r+2}, w_{r+1}\) by~\ref{MAIN4}.
Thus \(\text{deg}_H(w)\leq (1+\eps)^{\ell-1}\frac{\alpha\beta^{\ell-1}}{(\ell+1)^{\ell+1}}n^{\ell}\leq (1+n^{-1/2}\log^3 n)\frac{\alpha\beta^{\ell-1}}{(\ell+1)^{\ell+1}}n^{\ell}=\frac{D}{(\ell+1)\beta}\leq D\).\COMMENT{\((\ell+1)\beta=(\ell+1)\frac{1-\alpha}{\ell}=(\ell+1)\frac{1-\frac{1-n^{-1/2+1/\ell}}{\ell+1}}{\ell}=(\ell+1)\frac{\frac{\ell+n^{-1/2+1/\ell}}{\ell+1}}{\ell}=\frac{\ell+n^{-1/2+1/\ell}}{\ell}>1\).}
\endclaimproof
Set \(D'\coloneqq (1-n^{-1/2}\log^3 n)(\beta^{\ell}/(\ell+1)^{\ell+1})n^{\ell}\).
\begin{claim}\label{degXcheck}
\(\text{deg}_H((x_i,x_{i+1}))\geq D'\) for all \((x_i,x_{i+1})\in X\) (indices~\(i\) modulo~\(m\)).
\end{claim}
\claimproof
Fix \((x_i,x_{i+1})\in X\).
One applies~\ref{MAIN2} repeatedly, proceeding through outneighbours from~\(x_{i}\), to see there are at least \(((1-\eps)(\beta/(\ell+1))n)^{\ell-2}\) valid choices for \(w_1,w_2, \dots, w_{\ell-2}\), and finally at least~\((1-\eps)(\beta^2/(\ell+1)^3)n^2\) valid choices for \(w_{\ell-1}, w_{\ell}\) by~\ref{MAIN4}.
The claim follows.
\endclaimproof
\begin{claim}\label{codegreesequencecheck}
Let \(j=2i\in[2\ell]\) be even.
Then \(C_j(H), C_{j+1}(H)\leq n^{\ell-i}\), and \(C_{2\ell+2}(H)\leq 1\).
\end{claim}
\claimproof
Suppose \(U\in\binom{V(H)}{j}\).
We first prove \(\text{codeg}_{H}(U)\leq n^{\ell-i}\).
By construction, it suffices to only consider such~\(U\) with at most one element of~\(X\) and each of the \(W_r, Q_s\).
If~\(U\) contains an element \((x_t,x_{t+1})\in X\), then~\(U\) contains~\(j-1=2i-1\) vertices and colours from among the sets \(W_r, Q_s\), and so contains at least~\(i\) vertices or at least~\(i\) colours.
If the former, then there are\COMMENT{We've even already fixed the ``roles'' of the known vertices, and now we can just go through the unknown roles, one by one in order, choosing a vertex one at a time} at most~\(n^{\ell-i}\) choices of the other (at most)~\(\ell-i\) internal vertices, and each such choice fixes the whole rainbow link.
If the latter, then there are (at most) \(\ell+1-i\) indices, say \(s_1,s_2,\dots,s_{\ell+1-i}\in[\ell]_0\) in increasing order, such that \(U\cap Q_{s_z}=\emptyset\) for each \(z\in[\ell+1-i]\), and we claim that any of the at most~\(n^{\ell-i}\) choices of the colours \(q_{s_1},q_{s_2},\dots,q_{s_{\ell-i}}\) now determines a unique (candidate)\COMMENT{Some choices of the colours may lead to vertices not in the appropriate sets, but we're after an upper bound for the codegree so that's fine} \(H\)-edge using the elements of~\(U\).
Indeed, the endpoints \(x_t,x_{t+1}\) of the rainbow link are known, and the endpoints \(w_{s_{\ell+1-i}}, w_{s_{\ell+1-i}+1}\) (here identifying \(w_0=x_t\) and \(w_{\ell+1}=x_{t+1}\) if necessary) of the only arc with unassigned colour can be determined by following the out-arcs of the assigned colours from~\(x_t\) and the in-arcs of the assigned colours from~\(x_{t+1}\), and this determines all vertices of the link.
We deduce that \(\text{codeg}(U)\leq n^{\ell-i}\) if \(U\cap X\neq\emptyset\).

Suppose instead that \(U\in\binom{V(H)}{j}=\binom{V(H)}{2i}\) does not intersect~\(X\).
We split this possibility into three cases: either~\(U\) contains at least \(i+1\) vertices from among (different)\COMMENT{Or at least, as discussed at the start, it suffices to only consider those \(U\) with at most one element from each of the sets} sets~\(W_r\), or~\(U\) contains at least \(i+1\) colours from among (different) sets~\(Q_s\), or~\(U\) contains exactly~\(i\) of each.
In the first of these three cases, there are clearly at most~\(n^{\ell-i}\) choices of the other~\(\ell-i-1\) internal vertices (and their roles) of a link ``containing''~\(U\) and an element~\((x_t,x_{t+1})\in X\).
If instead~\(U\) contains at least~\(i+1\) colours, then let \(U'\subseteq U\) be a set of~\(i+1\) colours, and let \(s_1,s_2,\dots,s_{\ell-i}\in[\ell]_0\) be\COMMENT{\(\ell+1-(i+1)=\ell-i\)}, in increasing order, the indices~\(s\in[\ell]_0\) such that \(U'\cap Q_s=\emptyset\).
There are at most~\(n^{\ell-i}\) choices of \(\ell-i-1\) colours \(\{q_{s_z}\in Q_{s_z}\}_{z\in[\ell-i-1]}\) and a pair \((x_t,x_{t+1})\in\cX\), which, as discussed earlier, determines the whole link.\COMMENT{We have an element \((x_t,x_{t+1})\) (the endpoints of the link) and all but one colour assigned, so as in the previous paragraph, all vertices are determined by following the arcs of the assigned colours either from~\(x_t\) or from \(x_{t+1}\)}
Finally, suppose~\(U\cap X=\emptyset\) and \(U\) contains exactly~\(i\) vertices from among different~\(W_r\), and exactly~\(i\) colours from among different~\(Q_s\).
Let \(r_1,r_2,\dots,r_i\in[\ell]\) be, in increasing order, the indices~\(r\) for which~\(U\) contains a vertex~\(w_r\in W_r\), and similarly, let \(s_1,s_2,\dots,s_i\in[\ell]_0\) be, in increasing order, the indices~\(s\) for which~\(U\) contains a colour~\(q_s\in Q_s\).
There must exist an index \(s_z\) such that at least one of \(s_z,s_{z}+1\) is not in \(R\coloneqq\{r_1,\dots,r_i\}\).\COMMENT{Otherwise every arc whose colour we know has both endpoints known, but then you must know at least \(i+1\) vertices. Also it's fine if it's \(s_z=0\notin[\ell]\) or \(s_{z}+1=\ell+1\notin[\ell]\)}
If \(|\{s_z,s_{z}+1\}\cap R|=1\) then without loss of generality, \(s_z\notin R,s_{z}+1\in R\).
Then the~\(i\) vertices \(\{w_r\colon r\in R\}\) are known in any link ``containing''~\(U\), and the additional vertex~\(w_{s_{z}}\) is determined as~\(N^{-}_{q_{s_z}}(w_{s_{z}+1})\).
If \(s_z\geq1\) then there are at most~\(n^{\ell-i}\) choices of the remaining~\(\ell-i-1\) internal vertices of a rainbow link ``containing''~\(U\) and a pair~\((x_t,x_{t+1})\in X\).
If \(s_z=0\), then \(w_0=x_t\) and thus~\(x_{t+1}\) are determined, and there are at most~\(n^{\ell-i}\) choices for the remaining~\(\ell-i\) internal vertices of the link.
If \(s_z,s_{z}+1\notin R\), then there are at most~\(n\) choices for the arc~\(e_{s_z}=w_{s_z}w_{s_{z}+1}\) with colour~\(q_{s_z}\in Q_{s_z}\).\COMMENT{Here and several times throughout the proof already, I could have said ``by~\ref{MAIN3}'' for example, but actually I'm only using the weaker upper bound from the fact there's at most \(n\) arcs of any colour originally in all of \(G\)}
If \(1\leq s_z\leq \ell-1\) then there are at most~\(n^{\ell-i-1}\) choices for the remaining~\(\ell-i-2\) internal\COMMENT{\(1\leq s_z\leq\ell-1\) means both endpoints of \(e_{s_z}\) are internal. We already knew \(i\) internal vertices directly from \(U\), and since \(s_z,s_{z}+1\) are both \(\notin R\), that's 2 more internal vertices known} vertices and a pair~\((x_t,x_{t+1})\), and if \(s_z=0\) (the case \(s_z=\ell\) handled analogously) then \(w_1, x_t, x_{t+1}\) are determined and there are at most~\(n^{\ell-i-1}\) choices for the remaining~\(\ell-i-1\) internal\COMMENT{\(i\) internal vertices known directly from~\(U\). We have \(s_{z}+1=1\notin R\) so that's one more internal vertex known} vertices of the link.
We conclude that \(\text{codeg}_{H}(U)\leq n^{\ell-i}\).

The claim now follows from the simple facts \(C_{j+1}(H)\leq C_{j}(H)\) and \(C_{2\ell+2}(H)\leq C_{2\ell}(H) (\leq n^{\ell-\ell}=1)\).
\endclaimproof
Setting \(D_j\coloneqq n^{\ell-j/2}\) for even \(j\in[2\ell]\), \(D_j\coloneqq n^{\ell-(j-1)/2}\) for odd \(j\in[2\ell+1]\), and \(D_{2\ell+2}\coloneqq 1\), we have \(C_j(H)\leq D_j\) for all \(2\leq j\leq2\ell+2\) by Claim~\ref{codegreesequencecheck}.
Set \(B\coloneqq n^{1/2-1/(4\ell)}\,\,(\geq 1)\).
Set \(\zeta\coloneqq(1+n^{-1/2}\log^3 n)\alpha(\beta/(\ell+1))^{\ell}\), so that \(D=\zeta n^{\ell}\).
Then \(\sqrt{D/D_2}=\zeta^{1/2}n^{1/2}> B\), for even \(j\in[2\ell]\) we have \((D/D_j)^{1/(j-1)}=\zeta^{1/(j-1)}n^{j/(2(j-1))}>\zeta^{1/(j-1)}n^{1/2}>B\), for odd \(j\in[2\ell+1]\) we have \((D/D_j)^{1/(j-1)}=\zeta^{1/(j-1)}n^{(j-1)/(2(j-1))}>B\), and \((D/D_{2\ell+2})^{1/(2\ell+1)}=\zeta^{1/(2\ell+1)}n^{1/2 - 1/(4\ell+2)}\geq B\).
Observe moreover that
\begin{eqnarray*}
D' & = & D\frac{1-n^{-1/2}\log^3 n}{1+n^{-1/2}\log^3 n}\cdot\frac{1}{(\ell+1)\alpha}=D\frac{1-n^{-1/2}\log^3 n}{1+n^{-1/2}\log^3 n}\cdot\frac{1}{1-n^{-1/2+1/\ell}}\\ & \geq & D(1-2n^{-1/2}\log^3 n)(1+n^{-1/2+1/\ell})\geq D(1+(1/2)n^{-1/2+1/\ell}),
\end{eqnarray*}
and that, with \(1/n\ll1/A\ll\gamma\ll\delta\ll1\)\COMMENT{And the uniformity of~\(H\) yields \(k=2\ell+1\) on the order of \(1/\delta\), so \(\gamma\ll1/k\)} (and recalling the uniformity of~\(H\) is \(k+1\coloneqq(2\ell+1)+1\) and \(1/\delta<\ell<2/\delta\)), we have \(B^{-1+\gamma}\log^A D\leq n^{\gamma/2}\log^A (n^{\ell})\cdot n^{-1/2+1/(4\ell)}\leq n^{\delta/4}\cdot n^{-1/2+1/(4\ell)}\leq n^{-1/2+3/(4\ell)}\leq (1/2)n^{-1/2+1/\ell}\),\COMMENT{Used \(\delta\leq2/\ell\) so \(n^{\delta/4}\leq n^{1/(2\ell)}\)} so \(D'\geq(1+B^{-1+\gamma}\log^A D)D\), and Theorem~\ref{thm:mainpartitetheorem} finishes the proof.
\endproof
We remark that one may initially hope to avoid the need for a random partition (Lemma~\ref{slicinglemma}) of \(G\in\Phi(\dirK)\), by simply selecting vertices for~\(W_0\) randomly with probability~\(\alpha\) and defining~\(H\) in terms of rainbow links with endpoints in~\(W_0\) and internal vertices in~\(V\setminus W_0\).
However, this does not appear to work; each~\((x_i,x_{i+1})\) admits roughly \(((1-\alpha)n)^{\ell}\) rainbow links, and for each of the ``roles'' a colour or a non-\(W_0\) vertex can have, there turn out to be around \(\alpha n ((1-\alpha)n)^{\ell-1}\) rainbow links (ranging over \((x_i,x_{i+1})\in X\)) in which that vertex or colour performs the given role.
The issue is that the colours have too many roles, since \((\ell+1)\alpha n((1-\alpha)n)^{\ell-1}>((1-\alpha)n)^{\ell}\),\COMMENT{Equivalent to \((\ell+1)\alpha>(1-\alpha) \Leftrightarrow 1-n^{-1/2+1/\ell}>\frac{\ell+n^{-1/2+1/\ell}}{\ell+1}\), which is clear since LHS arbitrarily close to 1 and RHS arbitrarily close to \(\ell/(\ell+1)\).
Having \(\ell\) roles is fine since we do have \(\ell\alpha n((1-\alpha)n)^{\ell-1}<((1-\alpha)n)^{\ell}\), as this rearranges to \((1-\alpha)/\alpha > \ell\), i.e.\ \( \frac{\ell+n^{-1/2+1/\ell}}{1-n^{-1/2+1/\ell}}>\ell\), and we win by the required amount} so the full partitioning seems to be important.
\subsection{Partial rainbow directed triangle factors}\label{section:triangles}
In this subsection we prove Theorem~\ref{theorem:triangles} by removing those colours involved in too many loops or non-rainbow triangles, and then applying Theorem~\ref{theorem:maintheorem}.
\lateproof{Theorem~\ref{theorem:triangles}}
Fix \(G\in\Phi(\dirK)\).
We say a colour \(c\in C\) is \(\Delta\)-\textit{bad} (in~\(G\)) if the number~\(N(c)\) of non-rainbow directed triangles with exactly one \(c\)-arc\COMMENT{So, necessarily, the other two arcs have the same, non-\(c\) colour.} satisfies \(N(c)\geq n^{3/2}\), and set~\(C_{\text{bad}}^{\Delta}\) to be the \(\Delta\)-bad colours.
Notice there are at most~\(n^2\) non-rainbow directed triangles in~\(G\)\COMMENT{Let \(A\) be the set of pairs (vertex, colour not on loop at that vertex) and \(B\) the set of non-rainbow triangles in \(G\). Define \(f\colon A\rightarrow B\) by looking at that vertex as the middle vertex of a mono 2-path of that colour, and this identifies a single output non-rainbow triangle (just fill in the missing edge - it has some colour we don't get to choose). Then \(f\) is surjective. Indeed, if we have a non-rainbow triangle then it either has exactly 2 arcs the same colour or all 3; if the former then clearly there's exactly one preimage of this triangle, and if the latter then there's 3. Thus \(|B|\leq |A|=n(n-1)\)}.
Those whose every arc is the same colour do not contribute towards~\(N(c)\) for any~\(c\), and those with exactly two arcs of a common colour contribute one triangle towards~\(N(c)\) for one~\(c\).
We deduce that \(|C_{\text{bad}}^{\Delta}|\leq\sqrt{n}\).
We say~\(c\) is \textit{loop-bad} if~\(G\) has at least~\(\sqrt{n}\) loops coloured~\(c\), and set~\(C_{\text{bad}}^{\circ}\) to be the loop-bad colours, so clearly \(|C_{\text{bad}}^{\circ}|\leq\sqrt{n}\).
Define \(C_{\text{good}}\coloneqq C\setminus(C_{\text{bad}}^{\Delta}\cup C_{\text{bad}}^{\circ})\).
Define the auxiliary multihypergraph~\(H\) with \(V(H)\coloneqq V\cup C_{\text{good}}\), putting an edge $\{v_1, v_2, v_3, c_1, c_2, c_3\}$ in~\(H\) for each set of distinct vertices $\{v_1, v_2, v_3\}$ and colours $\{c_1, c_2, c_3\}$ and permutation~\(\sigma\) of~\(\{1,2\}\) such that the directed triangle \(v_{\sigma(1)} v_{\sigma(2)} v_3 v_{\sigma(1)}\) in~\(G\) is rainbow with colours (in some order) \(c_1,c_2,c_3\).\COMMENT{There are 2 directed triangles on each triple of vertices; namely \(v_1v_2v_3v_1\) and \(v_2v_1v_3v_2\). For each that is rainbow, it contributes an edge to \(H\). It's possible that both are rainbow with the same set of colours, which we want to include, so there may be 2 copies of some edges. This is the price I've paid for avoiding an unnecessary partitioning argument (random or otherwise), and actually I think it's nice to illustrate a usage of the multi part of Theorem~\ref{theorem:maintheorem}}
\begin{claim}\label{claim:triangleauxstats}
\(H\) is \((n',D,\eps)\)-regular, where \(n'=(1\pm n^{-1/2})2n\), \(D=n^2\), and \(\eps=10n^{-1/2}\).
Further, \(C_2(H), C_3(H)\leq 3n\), \(C_4(H),C_5(H)\leq 6\), and \(C_6(H)\leq 2\).
\end{claim}
\claimproof
The value of~\(n'\) follows from \(|C_{\text{bad}}^{\Delta}|, |C_{\text{bad}}^{\circ}|\leq\sqrt{n}\).\COMMENT{\(n'\leq n+n=2n\). \(n'\geq n + (n-2\sqrt{n})=2n-2\sqrt{n}=2n(1-n^{-1/2})\).}
Fix \(u\in V\).
Clearly there are at most~\(n\) choices for the outneighbour of~\(u\) in a triangle, and at most~\(n\) choices for the remaining vertex, so \(\text{deg}_{H}(u)\leq n^2\).\COMMENT{No funny business with the multi-edges; that's resolved by counting the choices in order beginning with the outneighbour of \(u\).}
On the other hand, notice the set~\(V_{u}^{\text{bad}}\subseteq V\setminus\{u\}\) of vertices~\(v\) for which there are at least~\(\sqrt{n}\) vertices \(w\in V\setminus\{u,v\}\) such that the directed path \(vwu\) is monochromatic must satisfy \(|V_{u}^{\text{bad}}|\leq\sqrt{n}\); otherwise there are more than~\(n\) monochromatic directed \(2\)-paths in~\(G\) with head~\(u\), a contradiction.  
Thus there are at least \(n-1-\sqrt{n}-2\sqrt{n}\) choices for an outneighbour \(v\in V\setminus{V_u^{\text{bad}}}\) of~\(u\) in~\(G\) via an arc with colour in~\(C_{\text{good}}\).
Then since~\(v\notin V_u^{\text{bad}}\), there are at least~\(n-2-\sqrt{n}\) choices for a vertex \(w\in V\setminus\{u,v\}\) such that~\(vwu\) is not monochromatic, and at most \(4\sqrt{n}+2\) of these must be rejected due to one (or both) of the arcs \(vw, wu\) having a bad colour or the colour of~\(uv\).\COMMENT{Rule out all out-neighbours of~\(v\) via arcs with any of the \(\leq 2\sqrt{n}\) bad colours or the arc with colour \(\phi(uv)\). Do the same for in-neighbors of \(u\).}
It follows that \(\text{deg}_{H}(u)=(1\pm 10n^{-1/2})n^2\).\COMMENT{\(\text{deg}_{H}(u)\geq(n-4\sqrt{n})(n-2-\sqrt{n}-4\sqrt{n}-2)\geq(n-4\sqrt{n})(n-6\sqrt{n})\geq n^2 -10n^{3/2}\)}
Now fix instead \(c\in C_{\text{good}}\).
Clearly \(\text{deg}_{H}(c)\leq n^2\).\COMMENT{There are at most~\(n\) choices for a non-loop arc~\(uv\) of colour~\(c\), and at most~\(n\) acceptable choices for the final vertex~\(w\) of the triangle, the directions of all arcs now decided.}
Since \(c\notin C_{\text{bad}}^{\circ}\), there are at least \(n-\sqrt{n}\) non-loop \(c\)-arcs~\(uv\).
Then there are at least \(n-2-4\sqrt{n}-2\geq n-5\sqrt{n}\) choices of a third vertex \(w\in V\setminus\{u,v\}\) such that \(c\notin\{\phi(vw), \phi(wu)\}\subseteq C_{\text{good}}\), so there are at least \(n^2 - 6n^{3/2}\) directed triangles in~\(G\) with exactly one \(c\)-arc and no bad colours.
Since \(c\notin C_{\text{bad}}^{\Delta}\), at most~\(n^{3/2}\) of these are not rainbow, so we conclude that \(\text{deg}_{H}(c)=(1\pm 10n^{-1/2})n^2\), as required.

It remains to find upper bounds for the codegrees.
For distinct vertices \(u,v\in V\), there are two choices for whether the arc~\(uv\) or the arc~\(vu\) will appear in the triangle, and at most~\(n\) choices for the final vertex~\(w\) of the triangle.\COMMENT{And clearly the directions of the remaining arcs are decided}
For distinct colours \(c,d\in C_{\text{good}}\), there are at most~\(n\) choices of a non-loop arc~\(uv\) of colour~\(c\), and two choices for whether~\(d\) will be the colour of the out-arc at~\(v\), or the in-arc at~\(u\), which decides the rest of the triangle.
For a vertex~\(u\) and a good colour~\(c\), either the out-arc at~\(u\) in a potential triangle has colour~\(c\), or the in arc at~\(u\), or the arc disjoint from~\(u\).
In the first two cases, there are at most~\(n\) choices for the third vertex, and in the final case, there are at most~\(n\) choices for the \(c\)-arc, and each of these choices decides the rest of the triangle.
It follows that \(C_2(H), C_3(H)\leq 3n\).
We now examine \(4\)-sets.
For any set \(S\subseteq V\cup C_{\text{good}}\) containing three \(V\)-vertices (and a colour), there are~\(2\) choices for a permutation of those vertices leading to distinct triangles (since the ``starting'' vertex does not matter).\COMMENT{Then either this colour is in those triangles or it isn't. Worst case it's in both}
For any set \(S\subseteq V\cup C_{\text{good}}\) containing three colours and a vertex~\(u\), there are~\(6\) choices for the order in which a triangle traverses the colours, beginning with the out-arc at~\(u\).
Finally, for a set~\(S\) of two vertices \(u,v\) and two colours \(c,d\), first we may choose whether the arc~\(uv\) or the arc~\(vu\) appears in the triangle.
The colour of that arc is then known, and is either \(c,d\), or some other colour.
In any of those cases, there are two options for completing the triangle (if that arc is not coloured \(c,d\), then choose the order for the colours \(c,d\) of the remaining arcs, and otherwise choose which of the two remaining arcs is the remaining known colour).\COMMENT{For a total of 4 actual choices, i.e. \(\text{codeg}(S)\leq 4\) in this case - some of what appear to be ``options'' are really just cases depending on the structure of~\(G\)}
It follows that \(C_4(H), C_5(H)\leq 6\), and \(C_6(H)\leq 2\) is clear from the construction of~\(H\).
\endclaimproof
Applying Theorem~\ref{theorem:maintheorem} to~\(H\) with \(D_2, D_3 = 3n\), \(D_4, D_5, D_6=6\) (say), and \(D=n^2, \eps = 10n^{-1/2}\), and~\(n'=(1\pm n^{-1/2})2n\) playing the role of~\(n\) (and \(\cT=\emptyset\)), notice that we may set \(B=(D/D_6)^{1/5}=(n^2 / 6)^{1/5}\), which yields the result.\COMMENT{\(\sqrt{D/D_2}=\sqrt{n/3}\), \((D/D_4)^{1/3}=(n^2 / 6)^{1/3} = \Theta(n^{2/3})\), \((D/D_5)^{1/4}=(n^2 /6)^{1/4}=\Theta(n^{1/2})\), \((D/D_6)^{1/5}=\Theta(n^{2/5})\), \(1/\eps = \Theta(n^{1/2})\). We may assume the \(\delta\) in the statement satisfies \(\delta\ll1\), since proving the statement for small \(\delta\) proves it for larger \(\delta\). Set \(\gamma=\delta\) and assume \(1/n\ll1/A\ll\gamma\) and we have \(\gamma\ll1/k\) by the assumption \(\delta\ll1\). Then there's a matching covering all but at most \(n'B^{-1+\delta}\log^A (n^2)\leq 2^{A+2}n\cdot 6n^{-2/5}B^{\delta}\log^A n\leq 2^{A+2}\cdot 6\log^A n \cdot n^{3/5}\cdot n^{2\delta/5}\leq n^{3/5+\delta}\) vertices of~\(H\), which by construction corresponds to a partial rainbow directed triangle factor of~\(G\) covering all but at most~\(n^{3/5+\delta}\) vertices of~\(G\).}
\endproof
\subsection{Large-diameter simplicial complexes}\label{section:SC}
In this subsection we prove Theorem~\ref{theorem:SimplicialComplex}.
To construct the desired complex~\(\cK\), we first need to find a subset \(X\subseteq\binom{[n]}{d}\) (and a good choice of orderings of the \(X\)-elements) with properties that will enable us to find many paths which do not ``cover'' any \(X\)-sets, except at their ends.
We sometimes abuse terminology and identify a permutation \(f\colon A\rightarrow A\) (for \(A=\{a_1,a_2,\dots,a_r\}\subseteq[n]\) with \(a_1<\dots<a_r\)) with the ordering \((f(a_1),f(a_2),\dots,f(a_r))\) of the elements of~\(A\), writing~\(f(A)\) for the latter.
\begin{lemma}\label{BNslicing}
Suppose \(1/n\ll1/\ell\ll\delta\ll1/d\leq1/2\).
Put \(p\coloneqq(1-n^{-1/d+\delta})/(d(\ell+d))\) and \(\eps\coloneqq n^{-1/2}\log^2 n\).
There exists a set \(X=\{A_1,A_2,\dots,A_m\}\subseteq\binom{[n]}{d}\) and a set of permutations \(\{f(A)\colon A\in\binom{[n]}{d}\}\) such that each of the following hold:
\begin{enumerate}[($X$1), topsep = 6pt]
\item \(m=(1\pm\eps)p\binom{n}{d}\);\label{Arightsize}
\item For any \(U\in\binom{[n]}{d}\), there are~\((1\pm\eps)n(1-p)^{d}\) vertices \(v\in[n]\setminus U\) such that \(\left(\binom{U\cup\{v\}}{d}\setminus \{U\}\right) \cap X=\emptyset\);\label{radiation}
\item For all vertex-disjoint ordered \(d\)-tuples \(\overrightarrow{U}=(u_1,u_2,\dots,u_d)\), \(\overrightarrow{U'}=(u'_1,u'_2,\dots,u'_d)\) with \(u_1,\dots,u_d,u'_1,\dots,u'_d\) distinct elements of~\([n]\), there are \((1\pm\eps)n^2 (1-p)^{d^2 + 2d-1}\) choices for an ordered pair of distinct vertices~\((v,w)\) with \(v,w\in[n]\setminus(U\cup U')\) such that all non-\((U,U')\) \(d\)-sets contained (not necessarily consecutively) among any sequence of \(d+1\) consecutive vertices from \(u_1,u_2,\dots,u_d,v,w,u'_1,u'_2,\dots,u'_d\) are not in~\(X\);\label{connection}
\item For each \(i\in [d]\), \(v\in[n]\), and each ordered sequence \((a_1,a_2,\dots,a_{i-1},a_{i+1},a_{i+2},\dots,a_d)\) of \(d-1\) vertices from \([n]\setminus\{v\}\), there are \((1\pm\eps)np(1-p)^{d-1}/d!\) choices of a vertex~\(a_i\), not among~\(v\) and the other~\(a_j\), such that \((a_1,a_2,\dots,a_d)\in\{f(A)\colon A\in X\}\) and \(\binom{\{a_1,\dots,a_d,v\}}{d}\setminus\{\{a_1,\dots,a_d\}, \{a_1,\dots,a_d,v\}\setminus\{a_i\}\}\cap X=\emptyset\).\label{GetintoAi}
\end{enumerate}
\end{lemma}
\begin{proof}
Construct a random set \(\mathbf{X}\subseteq\binom{[n]}{d}\) by selecting each \(A\in\binom{[n]}{d}\) independently for~\(\mathbf{X}\) with probability~\(p\).
Independently to this, for each \(A\in\binom{[n]}{d}\) independently, let~\(\mathbf{f}(A)\) be a uniformly random permutation of~\(A\).
We prove each of \ref{Arightsize}--\ref{GetintoAi} holds with high probability individually, from which the result will follow.
\newline\noindent\ref{Arightsize}: This follows from a simple application of Lemma~\ref{chernoff}\ref{lemma:chernoffbigexp}.
We omit the details.\COMMENT{\(|\mathbf{X}|\sim\text{Bin}(\binom{n}{d},p)\).
Put \(\eps\coloneqq\log n/\sqrt{\expn{|\mathbf{X}|}}\,\,(=\Theta(\log n \cdot n^{-d/2}))\).
Then by Lemma~\ref{chernoff}~\ref{lemma:chernoffbigexp} we have
\[
\prob{||\mathbf{X}|-\expn{|\mathbf{X}|}|\geq\sqrt{\expn{|\mathbf{X}|}}\log n}\leq2\exp(-\log^2 n/3).
\]
Then \(\binom{n}{d}p\pm\sqrt{\expn{|\mathbf{X}|}}\log n = \binom{n}{d}p(1\pm \log n (\expn{|\mathbf{X}|})^{-1/2})\). We have \((\expn{|\mathbf{X}|})^{-1/2}=\Theta(n^{-d/2})\leq \log n \cdot n^{-d/2}\leq \log n\cdot n^{-1/2}\) since \(d\geq 2\).}
\newline\noindent\ref{radiation}: Fix \(U\in\binom{[n]}{d}\).
Notice that the sought number of vertices \(v\in[n]\setminus U\) is the random variable \(\mathbf{Z}=\sum_{v\in[n]\setminus U}\mathbf{I}_v\), where~\(\mathbf{I}_v\) indicates that none of the~\(d\) different \(d\)-sets \((U\cup\{v\})\setminus\{u\}\) (ranging over \(u\in U\)) are in~\(\mathbf{X}\), so has \(\prob{\mathbf{I}_v=1}=(1-p)^{d}\).
Moreover, only those \(d\)-sets~\(Y\) intersecting~\(U\) on~\(d-1\) vertices affect any~\(\mathbf{I}_v\), and they only affect the~\(\mathbf{I}_v\) where~\(v\) is the unique vertex in~\(Y\setminus U\).
It follows that \(\mathbf{Z}\sim\text{Bin}(n-d,(1-p)^d)\), so the claim follows from Lemma~\ref{chernoff}\ref{lemma:chernoffbigexp}.
We omit the details.\COMMENT{Put \(\eps\coloneqq \log n/\sqrt{\expn{\mathbf{Z}}}\,\,(=\Theta(n^{-1/2}\log n))\). We have
\[
\prob{|\mathbf{Z}-\expn{\mathbf{Z}}|\geq \sqrt{\expn{\mathbf{Z}}}\log n}\leq2\exp(-\log^2 n/3).
\]
And \((n-d)(1-p)^d \pm\sqrt{\expn{\mathbf{Z}}}\log n=n(1-p)^d\pm d(1-p)^d\pm\sqrt{\expn{\mathbf{Z}}}\log n =n(1-p)^d\pm2\sqrt{\expn{\mathbf{Z}}}\log n=n(1-p)^d(1\pm2\sqrt{\expn{\mathbf{Z}}}\log n/(n(1-p)^d))\).
And \(2\sqrt{\expn{\mathbf{Z}}}\log n/(n(1-p)^d)\leq 2\log n/(\sqrt{n(1-p)^d})\leq n^{-1/2}\log^2 n\).}
\newline\noindent\ref{connection}: Fix such \(\overrightarrow{U},\overrightarrow{U'}\).
We use Theorem~\ref{theorem:lintal} (with \(\Omega^*=\emptyset\)).\COMMENT{Clearly we may identify the choice of~\(\mathbf{X}\) with the product space of all of the choices to put each \(d\)-set in~\(\mathbf{X}\). The permutations aren't relevant for this property.}
Let~\(\mathbf{Z}\) denote the number of acceptable choices of~\((v,w)\) and suppose that \(\mathbf{Z}(\omega)=s\) for some outcome~\(\omega\) of~\(\mathbf{X}\).
Then there are~\(s\) pairs~\(\{(v_i,w_i)\colon i\in[s]\}\) such that each of the \(d\)-sets (other than \(U,U'\)) obtained as a subset of any~\(d+1\) consecutive vertices from \(u_1,\dots,u_d,v_i,w_i,u'_1,\dots,u'_d\) are not in~\(\mathbf{X}\).
Proceeding to the right, first with the addition of~\(v_i\), notice that each added vertex introduces~\(d\) such new \(d\)-sets, except for~\(u'_d\) which adds only \(d-1\), since we exclude \(U'\).
Thus, each of these \(d(d+1)+d-1=d^2+2d-1\) sets was not chosen for~\(\mathbf{X}\) in~\(\omega\).
Let~\(I_{(v_i,w_i)}\) be the set of these~\(d^2+2d-1\) \(d\)-sets, put \(I\coloneqq\bigcup_{i\in[s]}I_{(v_i,w_i)}\), and for each \(d\)-set \(Y\in I\), put \(c_Y\coloneqq|\{i\in[s]\colon Y\in I_{(v_i,w_i)}\}|\).
One checks easily that \(\sum_{Y\in I}c_Y=(d^2+2d-1)s\).
Notice that each \(d\)-set \(Y\in I\) has at least one and at most two vertices not among~\(U\cup U'\).
If \(|Y\setminus(U \cup U')|=1\) then there are at most~\(2n\) of the~\(i\in[s]\) such that \(Y\in I_{(v_i,w_i)}\) (pick the remaining non-\((U,U')\) vertex then order the pair of such vertices), and if \(|Y\setminus(U\cup U')|=2\) then there are at most~\(2\) such~\(i\).
It is now easy to check\COMMENT{If you pick some \(I'\subseteq I\) and now all we know is that each \(d\)-set among \(I\setminus I'\) still is not put in~\(\mathbf{X}\), then the lost valid pairs \((v_i,w_i)\) clearly can be at most the sum over all \(Y\in I'\) of the involvement of that~\(Y\), as required. \(Y\) does not have any effect beyond this.} that~\(Z\) is \((2d^2,2n)\)-observable (say)\COMMENT{\(d^2+2d-1\leq d^2+2d\leq d^2 + d^2\) since \(2\leq d\)} with respect to \(\Omega^*=\emptyset\).

Observe that \(\expn{\mathbf{Z}}=(n-2d)(n-2d-1)(1-p)^{d^2+2d-1}\).
We seek to apply Theorem~\ref{theorem:lintal} with \(t\coloneqq (1/2)n^{3/2}\log^{3/2} n\),~\(r\coloneqq2d^2\), and~\(2n\) playing the role of ``\(d\)''.
To that end, notice that \(128\cdot2d^2\cdot2n<t/2\) and \(96\sqrt{2d^2 \cdot 2n \expn{\mathbf{Z}}}<t/2\).\COMMENT{Clearly \(\expn{\mathbf{Z}}\leq n^2\)}
We thus obtain that
\[
\prob{|\mathbf{Z}-\expn{\mathbf{Z}}|>t}\leq 4\exp\left(-\frac{n^3\log^3 n}{128d^2n(5n^2)}\right)\leq\exp\left(-\frac{\log^2 n}{2}\right),
\]
so the claim follows from the union bound.\COMMENT{UB over at most \(n^{2d}\) events, because you can pick each vertex of \(U,U'\) in turn to determine the orderings. Clearly \(\expn{\mathbf{Z}}=(1\pm 3d/n)^2n^2(1-p)^{d^2+2d-1}=(1\pm 9d/n)...=(1\pm (1/2)n^{-1/2}\log^2 n)...\). And then \(\pm (1/2)n^{3/2}\log^{3/2} n = n^2(1-p)^{d^2+2d-1} \cdot \pm(1/2)n^{-1/2}\log^{3/2} n/(1-p)^{d^2+2d-1}=n^2 (1-p)^{d^2+2d-1}\cdot\pm(1/2)n^{-1/2}\log^{2}n\)}
\newline\noindent\ref{GetintoAi}: Fix \(v\in[n]\) and \(i\in[d]\).
Without loss of generality, say \(i=1\).
Fix distinct \(a_2,a_3,\dots,a_d\in[n]\setminus\{v\}\).
Let~\(\mathbf{Z}\) denote the number of acceptable choices for~\(a_1\).
We apply Lemma~\ref{lemma:permutations}.
To that end, observe that~\(\mathbf{Z}\) is determined by the outcomes \(A\in\mathbf{X}\) for those \(d\)-sets~\(A\) having exactly~\(d-1\) vertices in common with \(\{a_2,a_3,\dots,a_d,v\}\), and the permutations~\(\mathbf{f}(A')\) of those~\(A'\) among these~\(A\) having \(v\notin A'\).
Fix some outcome~\(\omega\) of all of these trials and permutations, and let \(Z=\mathbf{Z}(\omega)\).
Changing the outcome (of putting \(A\in\mathbf{X}\)) for any of the above~\(A\) changes~\(Z\) by at most~\(1\), as the corresponding vertex~\(a_1\) is uniquely determined as the only vertex in \(A\setminus\{a_2,a_3,\dots,a_d,v\}\), and~\(a_1\) either counts (one) towards~\(\mathbf{Z}\) or it does not.
Swapping two co-ordinates of any one of the above permutations also changes~\(Z\) by at most~\(1\), as the corresponding~\(a_1\) is again uniquely identified as the only non-\(\{a_2,\dots,a_d,v\}\) vertex in the domain of~\(\mathbf{f}\).
If \(Z=s\), then there are~\(s\) vertices \(a_1^{(1)},a_1^{(2)},\dots,a_1^{(s)}\notin\{a_2,a_3,\dots,a_d,v\}\) such that, for each \(i\in[s]\), \(A^{(i)}\coloneqq\{a_1^{(1)},a_2,\dots,a_d\}\in\mathbf{X}\), \(\mathbf{f}(A^{(i)})=(a^{(i)}_1,a_2,\dots,a_d)\), and each \(d\)-set \(B^{(i)}_1,B^{(i)}_2,\dots,B^{(i)}_{d-1}\) of \(\{a^{(i)}_1,a_2,\dots,a_d,v\}\), excluding~\(\{a_2,a_3,\dots,a_d,v\}\) and~\(A^{(i)}\), is not in~\(\mathbf{X}\).
For each \(i\in[s]\), this is \(1+1+(d-1)=d+1\) ``witnesses'' for the fact that~\(a^{(i)}_1\) counts towards~\(Z=\mathbf{Z}(\omega)\), and in any outcome in which each \(A^{(i)}\in\mathbf{X}\), each~\(\mathbf{f}(A^{(i)})\) is ordered as above, and each~\(B^{(i)}_j\notin\mathbf{X}\), we have \(\mathbf{Z}\geq s\), as required.

We seek to apply Lemma~\ref{lemma:permutations} with \(c=1\), \(r=d+1\), and \(t=(1/2)n^{1/2}\log^{3/2}n\).
To that end, notice that \(\expn{\mathbf{Z}}=(n-d)p(1-p)^{d-1}/d!\leq n\), so \(128\cdot 1\sqrt{(d+1)\expn{\mathbf{Z}}}<t/2\), and clearly \(512(d+1)\cdot1^2<t/2\).
Thus,
\[
\prob{|\mathbf{Z}-\expn{\mathbf{Z}}|>\frac{n^{1/2}\log^{3/2}n}{2}}\leq4\exp\left(-\frac{n\log^3 n}{128(d+1)\cdot5n}\right)\leq4\exp\left(-\frac{\log^2 n}{2}\right),
\]
so the claim follows from the union bound.\COMMENT{Union bound: at most~\(dn^d\) events. Then we have \(\mathbf{Z}=\frac{(n-d)p(1-p)^{d-1}}{d!}\pm\frac{n^{1/2}\log^{3/2}n}{2} = \frac{np(1-p)^{d-1}}{d!}\left(1\pm\frac{d}{n}\pm\frac{n^{-1/2}\log^{3/2}n d!}{2p(1-p)^{d-1}}\right)=\frac{np(1-p)^{d-1}}{d!}\left(1\pm\frac{1}{2}n^{-1/2}\log^{2}n\pm\frac{1}{2}n^{-1/2}\log^2 n\right)\)}
\end{proof}
We now have all the tools we need to prove Theorem~\ref{theorem:SimplicialComplex}.
\lateproof{Theorem~\ref{theorem:SimplicialComplex}}
Insert a new integer~\(\ell\) such that \(1/n\ll1/\ell\ll\delta\ll1/d\leq1/2\).
Put \(p\coloneqq(1-n^{-1/d + \delta/2})/(d(\ell+d))\) and \(\eps\coloneqq n^{-1/2}\log^2 n\).
By Lemma~\ref{BNslicing} (with~\(\delta/2\) playing the role of~\(\delta\)), there exists a set \(X=\{A_1,A_2,\dots,A_m\}\subseteq\binom{[n]}{d}\) (with this ordering of~\(X\) being arbitrary) and a set of permutations \(\{f(A_i)\colon i\in[m]\}\) satisfying~\ref{Arightsize}--\ref{GetintoAi}.
For each \(i\in[m]\), denote \(\overrightarrow{A}_i=(a_i^{(1)},a_i^{(2)},\dots,a_i^{(d)})\coloneqq f(A_i)\).
Define \(\cX\coloneqq\{(\overrightarrow{A}_i,\overrightarrow{A}_{i+1})\colon i\in[m-1]\}\cup\{(\overrightarrow{A}_m,\overrightarrow{A}_{1})\}\).
For each \(\nu=(\overrightarrow{A}_i,\overrightarrow{A}_{i+1})\in\cX\) (indices modulo~\(m\), here and throughout) and each sequence~\(\overrightarrow{S}\) of~\(\ell\) distinct vertices \((v_1,v_2,\dots,v_{\ell})\in[n]\setminus(A_i\cup A_{i+1})\), write \(a_i^{(j)}=v_{-d+j}\) and \(a_{i+1}^{(j)}=v_{\ell+j}\) for each \(j\in[d]\).
We refer to $v_1, \dots, v_\ell$ as the ``internal vertices'' and $v_{-d+1}, \dots, v_0, v_{\ell + 1}, \dots, v_{\ell + d}$ as the ``external vertices''.
We often drop the arrow from~\(\overrightarrow{S}\) when clear from context, still meaning an ordered sequence.\COMMENT{It adds vertical space to the line every time I write it. Also we only ever use \(S\) in the context of an ordered sequence. We use \(A_i\) and \(\overrightarrow{A}_{i}\) somewhat frequently and need to be accurate about both contexts so we don't do any dropping arrows there.}
For \(2\leq j\leq\ell+d-1\), we define~\(\cU_j(\nu,S)\) to be the collection of \(d\)-sets of (not necessarily consecutive) vertices among the~\(d+1\) vertices \(v_{j-d},\dots,v_d\), so \(|\cU_j(\nu,S)|=d+1\) for these~\(j\).
We define~\(\cU_1(\nu,S)\) and~\(\cU_{\ell+d}(\nu,S)\) analogously, though we exclude the sets~\(A_i\) and~\(A_{i+1}\) respectively.
For \(j\in[\ell+d]\) we say that a \(d\)-set~\(U\) is \textit{generated in position}~\(j\) by~\((\nu,S)\) if~\(U\in\cU_j(\nu,S)\setminus(\bigcup_{t\in[j-1]}\cU_t(\nu,S))\), and denote the set of such~\(U\) by~\(\cU^{*}_j(\nu,S)\), so that \(\cU^*(\nu,S)\coloneqq\bigcup_{j\in[\ell+d]}\cU^{*}_j(\nu,S)=\cU(\nu,S)\).
We say that~\(S\) is \(\nu\)-\textit{good} if \(\cU^*(\nu,S)\cap X=\emptyset\).
We sometimes drop the~\((\nu,S)\)-notation if~\(\nu,S\) are fixed and clear from context.

Construct an auxiliary bipartite hypergraph~\(\cH\) with parts~\(\cX\) and \(Y\coloneqq\binom{[n]}{d}\setminus X\), and for each \(\nu=(\overrightarrow{A}_i,\overrightarrow{A}_{i+1})\in\cX\) and each \(\nu\)-good sequence \(S=(v_1,\dots,v_{\ell})\) of vertices from~\([n]\setminus(A_i\cup A_{i+1})\), put the edge \(\{\nu\}\cup\cU^{*}(\nu,S)\) in~\(\cH\).\COMMENT{It turns out this is not a multihypergraph. From a given edge you can identify \(v_1\) as the only vertex appearing with all \((d-1)\)-sets of~\(A_i\) (since~\(v_2\) fails to appear with \(\{a_i^{(1)},\dots,a_i^{(d-1)}\}\)), then~\(v_2\) as the only vertex appearing with all \((d-1)\)-sets of~\(\{a_i^{(2)},\dots,a_i^{(d)},v_1\}\), and so on. It doesn't really matter since the result still applies for multis and we are already able to uniquely equate each edge (or edge copy if you're imagining it as a multi) with the sequence \(S\).}
Note that each edge of~$\cH$ is determined by a sequence $(v_{-d+1}, \dots, v_0, v_1, \dots, v_\ell, v_{\ell+1}, \dots, v_{\ell + d})$.
Observe that~\(\cH\) is \(d(\ell+d)\)-uniform.\COMMENT{I suppose we can say in a slightly shorter way that it is \(d(\ell+d)\)-bounded, which suffices for applying our bipartite theorem. You'd still need to justify that though, and I think the following uniform argument helps to increase comfort with this construction. Let me know what you think}
Indeed, for any~\((\nu,S)\), there are~\(d\) sets in~\(\cU^*_1\) since we exclude~\(A_i\).
Then for each \(2\leq j\leq \ell\), there are~\(d\) sets in~\(\cU^*_j\), since we exclude~\(\{v_{j-d},\dots,v_{j-1}\}\) from~\(\cU^*_j\) because it is generated in position~\(j-1\), and each set including~\(v_j\) is new, as \(v_j\notin A_i\cup\{v_1,\dots,v_{j-1}\}\).
The sets~\(\cU^*_j\) for \(\ell+1\leq j\leq\ell+d\) require some extra care, because we may have \(A_i\cap A_{i+1}\neq\emptyset\).
However, we do still have \(|\cU^*_j|=d\) for \(\ell+1\leq j<\ell+d\) since each~\(j\) in turn introduces a new vertex that the vertices \(\{v_{\ell-d+1},\dots,v_{\ell}\}\) have not ``seen'' before in a \(d\)-set\COMMENT{Because, in particular, they are far enough away from~\(A_i\) that they never saw the \(A_i\)-vertices. All earlier spawned sets containing any \(A_i\) vertices do not contain any \(Z\) vertices, as these appear much later in the path.}, and \(|\cU^*_{\ell+d}|=d-1\), since we exclude both \(\{v_{\ell},\dots,v_{\ell+d-1}\}\) and~\(A_{i+1}\).

\begin{claim}\label{claimGettheSC}
If~\(\cH\) has an \(\cX\)-perfect matching, then there is a strongly connected simplicial complex~\(\cK\) on~\([n]\) whose dual graph~\(G(\cK)\) is a cycle of length at least \((1-n^{-1/d+\delta})n^d/(dd!)\).
\end{claim}
\claimproof
Let~\(\cM\) be such a matching, let~\(e_i\) be the \(\cM\)-edge containing~\(\nu_i=(\overrightarrow{A}_i,\overrightarrow{A}_{i+1})\), and let~\(\overrightarrow{S}_i\) be the \(\nu_i\)-good sequence of~\(\ell\) vertices  corresponding to~\(e_i\).
Then we have that \(\overrightarrow{A}_1\overrightarrow{S}_1\overrightarrow{A}_2\overrightarrow{S}_2\dots\overrightarrow{A}_m\overrightarrow{S}_m\overrightarrow{A}_1\) is a walk~\(W\coloneqq v_1v_2\dots v_{m(\ell+d)}v_1\dots v_d\) on~\([n]\).\COMMENT{(not necessarily visiting every vertex). Also for each \(i\in[m]\) we visit the \(d\) vertices of~\(A_i\) then the \(\ell\) external vertices from~\(S_i\), so that's \(m(\ell+d)\) vertices before we return to the first vertex of~\(A_1\)}
Construct a \(d\)-dimensional simplicial complex~\(\cK\) by putting \(F(\cK)\coloneqq\{F_i\colon i\in[m(\ell+d)]\}\), where \(F_i\coloneqq\{v_{i},v_{i+1},\dots,v_{i+d}\}\), indices modulo~\(m(\ell+d)\) (and including all non-empty subsets of these sets in~\(\cK\)).
Clearly \(|F_i\cap F_{i+1}|=d\) and \(|F_{m(\ell+d)}\cap F_1|=d\), so~\(G(\cK)\) contains a cyclical walk~\(W'\) on all of its facets, and thus~\(\cK\) is strongly connected.
Notice that no facet appears more than once on~\(W'\); indeed, if \(F=F_i\) for some~\(i\in[m(\ell+d)]\), then~\(F_i\) appears as~\(d+1\) consecutive vertices (in the obvious order) on the walk~\(\overrightarrow{A}_j\overrightarrow{S}_j\overrightarrow{A}_{j+1}\) where \(i\in\{(j-1)(\ell+d)+1,(j-1)(\ell+d)+2,\dots,j(\ell+d)\}\).
Thus, all non-\(X\) \(d\)-subsets of~\(F_i\) (of which there are at least \(d\geq 2\)) are elements of~\(Y\) covered by the \(\cM\)-edge~\(e_j\).
Since~\(\cM\) is a matching, it follows that no facet can arise from different \(e_j, e_{j'}\in\cM\).
Moreover, clearly no facet can arise twice from within the same~\(e_j\) (i.e.\ as~\(d+1\) consecutive vertices from~\(\overrightarrow{A}_{j}\overrightarrow{S}_j\overrightarrow{A}_{j+1}\)), since the~\(\ell\) vertices of~\(\overrightarrow{S}_j\) are distinct and all such facets contain at least one \(\overrightarrow{S}_j\)-vertex.
Therefore~\(W'\) is a Hamilton cycle in~\(G(\cK)\).

We also need that \(|i-j|\geq 2\) implies \(|F_i\cap F_j|\leq d-1\), as this will ensure that~\(G(\cK)=W'\).
Suppose, then, that \(|i-j|\geq 2\) and \(|F_i\cap F_j|=d\).
We must have \(F_i\cap F_j\notin X\), since by construction of~\(\cH\), each \(X\)-set only appears as a \(d\)-subset of precisely two, consecutive facets (since an edge for~\((\nu,S)\) is only in~\(\cH\) if~\(S\) is \(\nu\)-good).
But then~\(F_i\cap F_j\) is a \(Y\)-set covered by some unique edge~\(e_t\in M\), so \(F_i, F_j\) arise from the same~\(e_t\).\COMMENT{\(F_i\cap F_j\) is generated as a \(d\)-set of \(F_i\) in particular, so is covered.}
Recall that \(|\cU_s(\nu_t,S_t)|=d+1\) for all \(s\in[\ell+d]\) and \(|\cU^*_s(\nu_t,S_t)|=d\) for all \(s\in[\ell+d-1]\) and \(|\cU^*_{\ell+d}(\nu_t,S_t)|=d-1\) since~\(A_{t+1}\) is excluded from the latter set.
From the definitions of these sets, it is clear that any \(d\)-set (including~\(F_i\cap F_j\)) in~\(\cU(\nu_t,S_t)\) is in~\(\cU_s(\nu_t,S_t)\) for either one value~\(s\), or two, and in the latter case these \(s\)-values must be consecutive.
It now follows that \(|i-j|=1\), a contradiction.
Finally, notice that
\begin{eqnarray*}
m(\ell+d) & \geq & \frac{(1-\eps)(1-n^{-1/d+\delta/2})\binom{n}{d}(\ell+d)}{d(\ell+d)} \geq \frac{(1-2n^{-1/d+\delta/2})(1-d^2 n^{-1})n^d}{dd!}\\ & \geq & (1-n^{-1/d+\delta})\frac{n^d}{dd!},
\end{eqnarray*}
as required, where we used~\ref{Arightsize} and the fact \(d\geq 2\).\COMMENT{No idea why but it just straight up wouldn't let me use ``stackrel'' in the appropriate place. We have \(\eps=n^{-1/2}\log^2 n\leq n^{-1/d+\delta/2}\) and \(\binom{n}{d}\geq (n-d)^d/d! =n^d(1-d/n)^d/d!=n^d(1-d^2/(dn))^d/d!\geq n^d(1-d^2n^{-1})/d!\) using \((1+x/n)^n\geq 1+x\) for \(x\coloneqq -d^2/n\) and \(`n'\coloneqq d\)}
\endclaimproof

It remains to find an \(\cX\)-perfect matching of~\(\cH\), for which we seek to use Theorem~\ref{thm:mainpartitetheorem}.
To that end, set \(D\coloneqq(1+\eps)^{\ell-1}(d(\ell+d)-1)p(1-p)^{d(\ell+d)-2}n^{\ell}\).
\begin{claim}\label{claimSCY}
\(\text{deg}_{\cH}(U)\leq D\) for all \(U\in Y\).
\end{claim}
\claimproof
Fix \(U\in Y\).
Note that each edge copy of~\(\cH\) is uniquely associated with some~\((\nu,S)\) (in fact, \(C_{d(\ell+d)}=1\), but this does not affect the argument and we leave this to the reader).
We have \(\text{deg}_{\cH}(U)=\sum_{j\in[\ell+d]}\text{deg}^{(j)}_{\cH}(U)\), where \(\text{deg}^{(j)}_{\cH}(U)\) denotes the number of~\((\nu,S)\) (for \(\nu\in\cX\) and a \(\nu\)-good~\(S\)) for which~\(U\) is generated in position~\(j\).

If~\(U\) is generated in position~\(1\), then \(v_1\in U\), and exactly one vertex among \(v_{-d+1},\dots,v_0\) is not in~\(U\).
There are~\(d!\) permutations~\(\overrightarrow{U}\) of~\(U\) and~\(d\) choices for which ``role''~\(v_j\) among \(v_{-d+1},\dots,v_0\) will be a non-\(U\) vertex.
By~\ref{GetintoAi}, there are at most~\((1+\eps)np(1-p)^{d-1}/d!\) choices for the vertex~\(v_j\) such that \((v_{-d+1},\dots,v_0)\) (obeying the obvious ordering from~\(\overrightarrow{U}\), skipping~\(v_j\) and ignoring~\(v_1\)) is some~\(\overrightarrow{A}_i\) (which also fixes~\(\overrightarrow{A}_{i+1}\)), and all \(d\)-sets of \(\{v_{-d+1},\dots,v_1\}\) other than~\(A_i\) and~\(U\) (which is already known to not be in~\(X\) since \(U\in Y\)) are not in~\(X\).
Now, by~\ref{radiation}, there are at most~\((1+\eps)n(1-p)^d\) choices for the vertex~\(v_2\) such that there are no \(X\)-sets containing~\(v_2\) among \(\{v_{-d+2},\dots,v_1,v_2\}\); that is~\(v_2\) ``generates'' no~\(X\)-sets in position~\(2\).\COMMENT{No need to take away say \(2d+\ell\) vertices here to ensure \(v_2\) isn't some previously seen vertex, as we're after an upper bound here}
Applying~\ref{radiation} repeatedly in this way, we find there are then at most~\(((1+\eps)n(1-p)^d)^{\ell-4}\) appropriate choices for \(v_3,v_4,\dots,v_{\ell-2}\).
Then we apply~\ref{connection} with~\((v_{\ell-d-1},\dots,v_{\ell-2})\) and~\(\overrightarrow{A}_{i+1}\) playing the roles of~\(\overrightarrow{U},\overrightarrow{U'}\) respectively, to deduce there are at most \((1+\eps)n^2(1-p)^{d^2+2d-1}\) choices for the pair~\((v_{\ell-1},v_{\ell})\) such that none of the generated \(d\)-sets are in~\(X\) (excluding~\(A_{i+1}\)).
Altogether, \(\text{deg}_{\cH}^{(1)}(U)\leq(1+\eps)^{\ell-1}dpn^{\ell}(1-p)^{d(\ell+d)-2}\).
One computes similarly that \(\text{deg}_{\cH}^{(\ell+d)}(U)\leq(1+\eps)^{\ell-1}(d-1)pn^{\ell}(1-p)^{d(\ell+d)-2}\) (there are~\(d-1\) rather than~\(d\) choices for which vertex among \(v_{\ell},v_{\ell+1},\dots,v_{\ell+d}\) does not appear in~\(U\), since we exclude~\(v_{\ell}\) to ensure \(U\neq A_{i+1}\), and exclude~\(v_{\ell+d}\) as this would mean~\(U\) would be generated in position~\(\ell+d-1\); \ref{radiation}--\ref{GetintoAi} can be used in either direction).\COMMENT{Finally, for \(\text{deg}_{\cH}^{(\ell+d)}(U)\), there are~\(d!\) permutations~\(\overrightarrow{U}\) of~\(U\), but only~\(d-1\) choices for the role among \((v_{\ell},\dots,v_{\ell+d})\) of a non-\(U\) vertex; the unfixed vertex cannot be~\(v_{\ell}\), as this would lead to \(\overrightarrow{U}=\overrightarrow{A}_{i+1}\), and cannot be~\(v_{\ell+d}\), as then~\(U\) would have been generated in position~\(\ell+d-1\).
Using~\ref{GetintoAi} to find this unfixed vertex, which fixes \(\overrightarrow{A}_{i+1}, \overrightarrow{A}_i\), then applying~\ref{radiation}~\(\ell-3\) times successively, moving ``backwards'', and finally~\ref{GetintoAi} to complete the ``path'', we have \(\text{deg}_{\cH}^{(\ell+d)}(U)\leq(1+\eps)^{\ell-1}(d-1)pn^{\ell}(1-p)^{d(\ell+d)-2}\).}

We now handle~\(\text{deg}_{\cH}^{(j)}(U)\) for \(2\leq j\leq \ell+d-1\).
Suppose without loss of generality that \(2\leq j\leq (\ell+d)/2\).
By~\ref{radiation}, there are at most \((1+\eps)n(1-p)^d\) choices for a vertex \(v\notin U\) such that \(U\cup\{v\}\) contains no \(X\)-sets, then \((d+1)!-d!=dd!\) permutations of \(U\cup\{v\}\) in which~\(v\) is not the final vertex (as then~\(U\) would have been generated in position~\(j-1\)).
We have thus decided the~\(d+1\) vertices \(v_{j-d},\dots,v_j\) and their order.
Then by~\ref{radiation} applied repeatedly in both directions, there are at most \(((1+\eps)n(1-p)^d)^{(j-2)+(\ell-2-j)}=((1+\eps)n(1-p)^d)^{\ell-4}\) choices for the vertices \(v_{-d+2},\dots,v_{j-d-1},v_{j+1},\dots,v_{\ell-2}\) and their order.
Then by~\ref{GetintoAi}, there are at most~\((1+\eps)np(1-p)^{d-1}/d!\) permissible choices for~\(v_{-d+1}\), which determines~\(\overrightarrow{A}_i, \overrightarrow{A}_{i+1}\), then by~\ref{connection}, there are at most~\((1+\eps)n^2(1-p)^{d^2+2d-1}\) acceptable choices for~\(v_{\ell-1},v_{\ell}\).
Thus, \(\text{deg}_{\cH}^{(j)}(U)\leq (1+\eps)^{\ell-1}dpn^{\ell}(1-p)^{d(\ell+d)-2}\) for all \(2\leq j\leq\ell+d-1\).
The claim follows.\COMMENT{\(\sum_{j=1}^{\ell+d}\text{deg}_{\cH}^{(j)}(U)\leq (1+\eps)^{\ell-1}pn^{\ell}(1-p)^{d(\ell+d)-2}(d(\ell+d)-1)\) as claimed}
\endclaimproof
Set \(D'\coloneqq(1-\eps)^{\ell}(1-p)^{d(\ell+d)-1}n^{\ell}\).
\begin{claim}\label{claimSCX}
\(\text{deg}_{\cH}(\nu)\geq D'\) for all \(\nu\in\cX\).
\end{claim}
\claimproof
Fix \(\nu=(\overrightarrow{A}_i,\overrightarrow{A}_{i+1})\in\cX\).
Applying~\ref{radiation}~\(\ell-2\) times successively from~\(A_i\) proceeding to the ``right'', there are at least \(((1-\eps)n(1-p)^d - (\ell+2d))^{\ell-2}\) choices for distinct vertices \(v_1,v_2,\dots,v_{\ell-2}\in[n]\setminus(A_i\cup A_{i+1})\) which do not lead to the generation of any \(X\)-sets.
Then by~\ref{connection}, there are at least \((1-\eps)n^2(1-p)^{d^2+2d-1}\) choices for \(v_{\ell-1}, v_{\ell}\) not generating any undesired \(X\)-sets, and we must remove at most~\(2(\ell+2d)n\) of these to ensure these vertices are ``new''.\COMMENT{Remove every pair \((v_{\ell-1},v_{\ell})\) for which~\(v_{\ell-1}\) has been seen before. That's at most~\((\ell+2d)n\). Also remove every pair such that \(v_{\ell}\) has been seen before. The pair themselves are distinct by design, so, golden.}
The claim follows.\COMMENT{\(((1-\eps)n(1-p)^d - (\ell+2d))^{\ell-2}((1-\eps)n^2(1-p)^{d^2+2d-1}-2n(\ell+2d))=((1-\eps)n(1-p)^d)^{\ell-2}(1-(\ell+2d)/((1-\eps)n(1-p)^d))^{\ell-2}(1-\eps)n^2(1-p)^{d^2+2d-1}(1-(2n(\ell+2d))/((1-\eps)n^2(1-p)^{d^2+2d-1}))\geq (1-\eps)^{\ell-1}n^{\ell}(1-p)^{d(\ell+d)-1}(1-\log n/n)^{\ell-2}(1-\log n/n)\geq (1-\eps)^{\ell}n^{\ell}(1-p)^{d(\ell+d)-1}\).}
\endclaimproof
Finally, we must study the codegree sequence of~\(\cH\) to determine the largest value~\(B\) with which we may apply Theorem~\ref{thm:mainpartitetheorem}.
This requires some delicate, but ultimately elementary\COMMENT{I'm trying to soften the fact this takes a page and a half. On one hand I don't want to underplay our achievement here, but on the other we are trying to present our main theorems as easy to apply, and the case analysis is ``elementary'' in the sense that it doesn't require any powerful probabilistic blackboxes or anything like that.} case analysis.
\begin{claim}\label{claimSCcodegs}
For each \(t\in[\ell-5d]_0\), we have \(C_{td+2}(\cH)\leq n^{\ell-t-1}\log n\).
\end{claim}
\claimproof
Fix \(t\in[\ell-5d]_0\) and $\cZ \in \binom{V(\cH)}{td+2}$, and let $\cY \coloneqq \cZ \cap Y$. We will show $\text{codeg}_{\cH}(\cZ) \leq n^{\ell - t - 1}\log n$ in two cases, depending on whether $\cZ \cap \cX = \emptyset$.

First, observe that, for any \(\nu \in \cX\) and \(\nu\)-good~\(S\), we may order all the generated \(d\)-sets by the index of their ``generation position'' \(1\leq j\leq\ell+d\), and within these subsets by, say, increasing index of the ``missing vertex''.
Thus~\((\nu,S)\) corresponds to an ordering~\(\pi\) of \(d(\ell+d)-1\) \(d\)-sets.
Clearly there are at most~\(\log n\) orderings of~\(\cY\cup\Lambda\), where~\(\Lambda\) is simply \(d(\ell+d)-1-|\cY|\) copies of the entry ``null'', say, and we can imagine such an ordering~\(\pi\) as an assignment to the sets \(\cU^*_1,\dots,\cU^{*}_{\ell+d}\), which we may assume is ``consistent'' in that, for example, any non-null sets within the same~\(\cU^{*}_{j}\) intersect on~\(d-1\) vertices. 
Thus, it suffices to show that for each such ordering $\pi$, there are at most $n^{\ell - t - 1}$ choices for $\nu \in \cX$ and $\nu$-good $S = (v_1, \dots, v_\ell)$ such that $\cZ \subseteq \{\nu\}\cup\cU^*(\nu, S)$.

Notice that we may also collect the non-\(X\) \(d\)-sets corresponding to some \(\nu\)-good~\(S\) as they ``generate'' moving along~\(S\) from~\(A_{i+1}\) to~\(A_i\); that is, into the sets \(\cU^{\dagger}_{\ell}, \cU^{\dagger}_{\ell-1},\dots,\cU^{\dagger}_{-d+1}\) where \(\cU^{\dagger}_{j}\) is the set of \(U\in Y\) for which~\(U\) appears among the \(d+1\) vertices \(\{v_j,v_{j+1},\dots,v_{j+d}\}\) and not among \(\{v_s,v_{s+1},\dots,v_{s+d}\}\) for any \(j<s\leq\ell\).
We may simultaneously imagine an ordering of~\(\cY\cup\Lambda\) as an assignment to the sets~\(\cU^{\dagger}_j\).

\textbf{Case 1:} $\cZ \cap \cX \neq \emptyset$.

In this case, since $\cH$ has bipartition $(\cX, Y)$, we may assume that \(\cZ\) contains a single element \(\nu=(\overrightarrow{A}_i,\overrightarrow{A}_{i+1})\in\cX\), with the remaining \(\cZ\)-elements \(\cY\subseteq Y\) satisfying \(|\cY|=td+1\).
Fix one of the at most~\(\log n\) orderings~\(\pi\) of~\(\cY\cup\Lambda\), with~\(\Lambda\) being~\(d(\ell+d)-1-|\cY|\) ``null'' entries, which we may associate with the assignment of all \(d\)-sets in~\(\cY\) to roles among the sets \(\cU^{*}_1,\dots,\cU^{*}_{\ell+d}\) (equivalently to roles among the sets \(\cU^{\dagger}_{\ell},\cU^{\dagger}_{\ell-1},\dots,\cU^{\dagger}_{-d+1}\))  that is ``consistent''.

Since \(|\cY|=dt+1\) and each~\(\cU^*_j\) has size at most~\(d\), by the Pigeonhole Principle, at least~\(t+1\) of the sets~\(\cU^*_j\) must be assigned some \(U\in\cY\).
Say such~\(\cU^*_j\) are \textit{occupied}.
Suppose first that there is some \(z\in[\ell+1]\) such that the sets \(\cU^*_z,\cU^*_{z+1},\dots,\cU^{*}_{z+d-1}\) are all not occupied.
If at least~\(t+1\) occupied~\(\cU^{*}_j\) have \(j\leq z-1\), then, for each such~\(j\), we learn the element~\(v_j\) from any element of~\(\cU^*_j\) (the role is decided by~\(\pi\)), and each such~\(v_j\) is ``internal'' (i.e.,\ among \(v_1,\dots,v_{\ell}\)) by assumption.
We then know at least~\(t+1\) of the internal vertices, and there are at most~\(n^{\ell-(t+1)}\) choices for the others, as desired.\COMMENT{We collect only~\(v_j\) from each~\(\cU^{*}_j\) to be sure our collected vertices are distinct}
The case that \emph{all}\COMMENT{It's slightly awkward because you can't use the same logic with ``at least~\(t+1\) of the occupied \(\cU^*_j\) have large \(j\)''; because of the way \(\cU^*\) and \(\cU^{\dagger}\) are defined, at least \(t+1\) occupied \(\cU^*_j\) can lead to only \(t\) occupied \(\cU^{\dagger}_j\); for instance some \(\cU^*_s\) may contain the element \(\{v_{s-d+1},\dots,v_{s}\}\) and \(\cU^*_{s+1}\) the element \(\{v_{s-d+1},v_{s-d+3},v_{s-d+4},\dots,v_{s+1}\}\), but these are both elements of~\(\cU^{\dagger}_{s-d+1}\). We save the day by instead assuming \emph{all} occupied buckets have large~\(j\) then applying PP to the sets~\(\cU^{\dagger}_s\).} occupied~\(\cU^*_j\) have \(z+d\leq j\leq\ell+d\) is handled analogously by considering instead the sets~\(\cU^{\dagger}_s\) (any \(d\)-set in~\(\cU^{\dagger}_s\) contains~\(v_s\), and any element of~\(\cU^*_j\) for \(j\geq z+d\) is in~\(\cU^{\dagger}_{s}\) for \(s\geq z \geq 1\); if no~\(\cU^*_j\) for \(j\leq z+d-1\) are occupied then all~\(dt+1\) elements of~\(\cY\) are distributed among \(\cU^{\dagger}_z,\dots,\cU^{\dagger}_{\ell}\) so apply the Pigeonhole Principle here).
Thus we may assume that \(1\leq q\leq t\) of the occupied~\(\cU^{*}_j\) have \(j\leq z-1\), and at least $t + 1 - q$ of the occupied~\(\cU^{*}_j\) have \(j \geq z + d\).
Then one proceeds similarly to deduce~\(q\) internal vertices among \(v_1, \dots, v_{z-1}\) from analysis of these occupied sets.
Further, at least \(dt+1-qd=d(t-q)+1\) elements of~\(\cY\) are distributed among \(\cU^{\dagger}_z,\dots,\cU^{\dagger}_{\ell}\) (here using that \(\cU^*_z,\dots,\cU^*_{z+d-1}\) are not occupied), so by the Pigeonhole Principle, at least \(t+1-q\) of these sets are occupied, from which we can learn~\(t+1-q\) additional internal vertices among \(v_{z}, \dots, v_\ell\) (in particular, sets in~\(\cU^*_{z-1}\) and~\(\cU^{\dagger}_z\) cannot intersect), and there are at most~\(n^{\ell-(t+1)}\) choices for the rest, as desired.

Now suppose instead there is no \(z\in[\ell+1]\) such that all~\(d\) of the sets \(\cU^*_z,\dots,\cU^*_{z+d-1}\) are not occupied.
Notice that if at least~\(t+1+d\) of the sets~\(\cU^{*}_1,\dots,\cU^{*}_{\ell+d}\) are occupied, then at least~\(t+1\) of the sets \(\cU^{*}_1,\dots,\cU^{*}_{\ell}\) are occupied, and one proceeds as before.
So we may assume that between~\(t+1\) and~\(t+d\) of the sets~\(\cU^{*}_j\) are occupied\COMMENT{So \(t\) is large, at least \(\lflr \ell/d\rflr -d\)}, in which case it follows that at most~\(d+1\) of these sets have only one \(\cY\)-element; indeed, if~\(d+2\) have only one \(\cY\)-element, then there are \(dt+1-(d+2)=dt-d-1\) remaining \(\cY\)-elements to distribute among at most~\(t-2\) sets, each with size at most $d$, which is impossible since \(d(t-2)=dt-2d<dt-d-1\).
Say a~\(\cU^*_j\) with at least two \(\cY\)-elements is \emph{super-occupied}, and notice that if~\(\cU^*_j\) is super-occupied, then the union of the non-null elements of~\(\cU^{*}_j\) yields all~\(d+1\) vertices~\(v_{j-d},\dots,v_j\) (and their roles are determined by~\(\pi\)).
If~\(\cU^{*}_j\) is occupied but not super-occupied, then we learn only~\(d\) of the~\(d+1\) vertices \(v_{j-d},\dots,v_j\).
Since there is no \(z\in[\ell+1]\) such that all~\(d\) of the sets \(\cU^*_z,\dots,\cU^*_{z+d-1}\) are not occupied, the union over all ``occupied indices'' \(j\in[\ell+d]\) of the set~\(\{j-d,j-d+1,\dots,j\}\) must include all of~\([\ell]\), and since, for all but at most~\(d+1\) of these~\(j\), we learn all the vertices~\(v_{j-d},\dots,v_j\) from~\(\cY\) (and all but one such vertex is learned for the non-super-occupied~\(j\)), we deduce all but at most~\(d+1\) of the internal vertices. Hence, there are at most \(n^{d+1}\) choices for the remaining internal vertices, and since $t \leq \ell - 5d$, we have at most \(n^{d+1} \leq n^{\ell-t-1}\) choices for the rest, as desired.

\textbf{Case 2: $\cZ \cap \cX = \emptyset$}

In this case, since $\cH$ has bipartition $(\cX, Y)$, we may assume \(\cZ\subseteq Y\) with \(|\cZ|=td+2\).
Again, there are at most~\(\log n\) orderings~\(\pi\) of~\(\cZ\cup\Lambda\), with~\(\Lambda\) being~\(d(\ell+d)-1-|\cZ|\) ``null'' entries, which we may associate with the assignment of all \(d\)-sets in~\(\cZ\) to roles among the sets \(\cU^{*}_1,\dots,\cU^{*}_{\ell+d}\) (equivalently to roles among the sets \(\cU^{\dagger}_{\ell},\cU^{\dagger}_{\ell-1},\dots,\cU^{\dagger}_{-d+1}\)) and we consider some ``consistent'' such~\(\pi\).
Let~\(s^*\) be the smallest~\(j\) such that~\(\cU^*_j\) is occupied.

First, assume that~\(\pi\) does not yield any external vertices (i.e., among either \(v_{-d+1},\dots,v_0\) or \(v_{\ell+1},\dots,v_{\ell+d}\)).
By the Pigeonhole Principle, there are at least~\(t+1\) occupied~\(\cU^*_j\).
If there are exactly~\(t+1\), then each must be super-occupied\COMMENT{If even 1 has only one element, then there's \(dt+1\) elements left to distribute among \(t\) remaining sets, impossible}, so we learn~\(d+1\) internal vertices from~\(\cU^*_{s^*}\), and at least the~\(t\) internal vertices~\(v_{j'}\) where~\(j'\) ranges over all other indices~\(s\) such that~\(\cU^*_s\) is occupied.
Then there are at most~\(n^d\) choices for the pair~\((A_i,A_{i+1})\), which determines the ``external'' vertices, and at most~\(n^{\ell-(t+d+1)}\) choices for the remaining internal vertices, which is altogether at most $n^{\ell - t - 1}$, as desired.
If instead there are at least~\(t+2\) occupied~\(\cU^*_j\), then we may only learn~\(d\) internal vertices from~\(\cU^{*}_{s^*}\), but we learn an additional~\(t+1\) vertices~\(v_{j'}\) from the other occupied~\(\cU^*_{j'}\), which is again sufficient.\COMMENT{So again \(t+1+d\) internal vertices known, all definitely internal due to the case we're in. Then \(n^d\) choices for the external vertices, so boom.}

Now suppose~\(\pi(\cZ\cup\Lambda)\) does determine some external vertices.
By symmetry\COMMENT{i.e. consider the sets~\(\cU^{\dagger}\) instead if necessary}, we may assume that~\(\pi(\cZ\cup\Lambda)\) determines \(1\leq j^*\leq d\) of the vertices \(v_{-d+1},\dots,v_0\). 
Again, it is useful to distinguish two cases:
If there is some \(z\in[\ell+1]\) such that the sets \(\cU^*_z,\dots, \cU^*_{z+d-1}\) are all not occupied, and~\(|\cU^*_{s^*}\cap\cZ|=1\), then we learn at least\COMMENT{If this set has only one element, then we learn \(d\) vertices from it. We must have either~\(j^{*}\) of them being external or\(j^{*}-1\) being external since any \(d\)-set only skips one vertex. The latter is possible if, say, \(j^*=2\), \(d=2\), the lone \(d\)-set in~\(\cU^*_{1}\) yields \(v_1\) and the penultimate vertex of~\(A_i\), and \(\cU^*_2\) is super-occupied so we learn the final vertex of \(A_i\) from \(\cU^*_2\), so we learn \(j^*=2\) of these external vertices, only one of which from \(\cU^*_1\). But the point is this would only help us, and clearly since at most~\(j^*\) vertices of a \(d\)-set in \(\cU^*_1\) can be external (as \(j^*\) is the total determined external vertices on the left of the path by hypothesis), we must learn at least \(d-j^*\) internal vertices from \(\cU^*_1\).} \(d-j^*\) internal vertices from~\(\cU^*_{s^*}\), and there must be at least~\(t+1\) other occupied sets. Then we proceed separately from ``both sides'' as in the \(\cZ\cap\cX\neq\emptyset\) case to learn~\(t+1\) additional internal vertices from these sets, and there are~\(n^{d-j^*}\) choices for the remaining vertices among \(v_{-d+1},\dots,v_0\), which fixes~\((A_i,A_{i+1})\), and~\(n^{\ell-(d-j^*+t+1)}\) choices for the remaining internal vertices.
If~\(\cU^{*}_{s^*}\) is super-occupied then there may be only~\(t+1\) occupied~\(\cU^*_{j'}\) in total, but we learn~\(d+1-j^*\) internal vertices from~\(\cU^*_{s^*}\).
If instead there is no such~\(z\) and at least say~\(t+2+d\) sets are occupied, then at least~\(t+2\) of these are among \(\cU^*_1,\dots,\cU^*_{\ell}\), from which we learn at least \((d-j^*)+(t+1)\) internal vertices\COMMENT{At least \(d-j^*\) from the first, then one from each of the others since they're all internal so we can uniquely pick their ``last'' vertex and be sure these are all different}, then there are \(n^{d-j^*}\) choices for the remaining vertices of~\(\overrightarrow{A}_i\), fixing~\(\overrightarrow{A}_{i+1}\), and \(n^{\ell-(t+1+d-j^*)}\) choices for the remaining internal vertices.
Finally if there is no such~\(z\) and exactly~\(t+j\) of the sets~\(\cU^*_j\) are occupied for \(1\leq j\leq d+1\), then at most~\(j+1\leq d+2\) are not super-occupied\COMMENT{Suppose exactly \(t+j\) sets are occupied and suppose for a contradiction that \(j+2\) of these have only one element. Then there are \(dt+2-(j+2)=dt-j\) elements remaining to be distributed among \((t+j)-(j+2)=t-2\) sets, but they have total capacity \(d(t-2)=dt-2d<dt-j\) since \(j\leq d+1<2d\), so, contradiction}, so we learn all but at most~\(d+2\) of the internal vertices from~\(\cZ\), so there are at most \(n^{d-j^*+d+2}\leq n^{5d-1}\leq n^{\ell-(t+1)}\) (say) choices for the other vertices.
We omit the full details as they are similar to the previous cases.
\endclaimproof
By Claim~\ref{claimSCcodegs}, if we set \(D_j\coloneqq n^{\ell-(t+1)}\log n\) for all \(t\in[\ell-5d-1]_0\) and \(td+2\leq j<(t+1)d+2\), and \(D_j\coloneqq n^{5d-1}\log n\) for \((\ell-5d)d+2\leq j\leq d(\ell+d)\), we obtain \(C_j(\cH)\leq D_j\) for all \(2\leq j\leq d(\ell+d)\).
Notice that \(\sqrt{D/D_2}\geq \sqrt{n}/\log n\), and that the value \(j\in\{4,5,\dots,k+1\}\) minimizing~\((D/D_j)^{1/(j-1)}\) is clearly either~\(j=d(\ell+d)\), or among the values~\(j=td+1\) for \(t\in[\ell-5d]\).\COMMENT{The ends of the clusters}
For such~\(t\) and \(j=td+1\) we have \(D_j=n^{\ell-t}\) and \((D/D_j)^{1/(j-1)}\geq (n^{t}/\log^2 n)^{1/(td)}\geq n^{1/d}/\log n\).
Moreover, we have \((D/D_{d(\ell+d)})^{1/(d(\ell+d)-1)}\geq n^{(\ell-5d+1)/(d(\ell+d)-1)}/\log n\geq n^{\ell/(d\ell+d^2) - 5d/(d\ell+d^2)}/\log n\geq n^{1/d - 6/\ell}/\log n\geq n^{1/d - \delta/4}\).\COMMENT{\(n^{-5d/(d\ell+d^2)}\geq n^{-5d/(d\ell)}=n^{-5\ell}\) and \(\ell/(d\ell+d^2)=1/d-1/(\ell+d)\geq 1/d-1/\ell\)}
Using \(d\geq 2\), we see that the upper bound for~\(B\) in Theorem~\ref{thm:mainpartitetheorem} is thus at least \(n^{1/d-\delta/4}\), so we may set~\(B\coloneqq n^{1/d-\delta/4}\,\,(\geq1)\).
Observe that
\begin{eqnarray*}
\frac{D'}{D} & = & \frac{(1-\eps)^{\ell}(1-p)}{(1+\eps)^{\ell-1}(d(\ell+d)-1)p}\geq\frac{(1-\eps)^{2\ell}\left(1-\frac{1}{d(\ell+d)}\right)}{(d(\ell+d)-1)\frac{1-n^{-1/d+\delta/2}}{d(\ell+d)}}\\ & \geq & (1-2\ell\eps)(1+n^{-1/d+\delta/2})\geq1+n^{-1/d+\delta/3}\geq1+B^{-1+\gamma}\log^K D,
\end{eqnarray*}
for appropriate \(1/n\ll1/K\ll\gamma\ll1/\ell\).
Thus, Theorem~\ref{thm:mainpartitetheorem} yields an \(\cX\)-perfect matching~\(\cM\) of~\(\cH\), which, by Claim~\ref{claimGettheSC}, completes the proof.
\endproof
\subsection{Matchings and chromatic index of Steiner systems}\label{Section:Steinermatchcolours}
In this subsection we prove Theorems~\ref{theorem:Steinercolorsmatchings} and~\ref{theorem:PartiteSteinercolourmatchings}, which both follow from simple applications of Theorems~\ref{theorem:maintheorem} and~\ref{theorem:maincolourtheorem}.
\lateproof{Theorem~\ref{theorem:Steinercolorsmatchings}}
Clearly we may assume \(1/n\ll1/A\ll\delta\ll1/r<1/t\leq 1/2\) for some new constant~\(A\).\COMMENT{Proving the theorem for small~\(\delta\) proves it for large \(\delta\). We need \(\delta\ll1/r\) since \(r\) parameterizes the uniformity and we need \(\gamma\ll1/k\) in the main theorems, and will use \(\gamma\coloneqq\delta/2\).}
Fix a \((t,r,n)\)-Steiner system~\(S\).
It is easy to check that~\(S\) is \(D\coloneqq\left(\binom{n-1}{t-1}/\binom{r-1}{t-1}\right)\)-regular, \(r\)-uniform, and satisfies \(C_j(S)\leq\binom{n-j}{t-j}/\binom{r-j}{t-j}\eqqcolon D_j\) for all \(2\leq j\leq t\), and \(C_j(S)=1\eqqcolon D_j\) for all \(t+1\leq j\leq r\).
Set \(\eps\coloneqq 1/\sqrt{n}\) (say).
We have \(\sqrt{D/D_2}=\sqrt{(n-1)/(r-1)}\geq\sqrt{n/2r}\) and \(1/\eps\geq\sqrt{n/2r}\).
We must also lower bound the values~\((D/D_j)^{1/(j-1)}\) for \(4\leq j\leq r\) (if there are such values, i.e.\ if \(r\geq 4\)).
For those~\(j\) among these having \(4\leq j\leq t\) (if there are any such), we have \((D/D_j)^{1/(j-1)} = ((n-1)!(r-j)!/(n-j)!(r-1)!)^{1/(j-1)}\geq ((n/2r)^{j-1})^{1/(j-1)}=n/2r > \sqrt{n/2r}\).
Since \(D_j=1\) for all \(t+1\leq j\leq r\), clearly the minimum \((D/D_j)^{1/(j-1)}\) among these~\(j\) is attained by \(j=r\), and we have \((D/D_{r})^{1/(r-1)} = ((n-1)!(r-t)!/(n-t)!(r-1)!)^{1/(r-1)}\geq(n/2r)^{(t-1)/(r-1)}\).
If \((t-1)/(r-1)<1/2\) (i.e.\ \(t<(r+1)/2\), which, in particular, can only occur if \(r\geq 4\))\COMMENT{Point being, you only have to take this term into consideration if \(r\geq 4\) so there are relevant values~\(j\) to bound \((D/D_j)^{1/(j-1)}\). But, in the case \(r=3\) (and so \(t=2\)), we're using the other bound, so we're good.} then we put \(B\coloneqq (n/2r)^{(t-1)/(r-1)}\) and \(\gamma\coloneqq\delta/2\) in Theorem~\ref{theorem:maintheorem} to obtain that \(U(S)\leq n^{1-(t-1)/(r-1)+\delta}=n^{(r-t)/(r-1)+\delta}\), as claimed, and if instead \((t-1)/(r-1)\geq 1/2\) (i.e.\ \(t\geq(r+1)/2\)), then we may put \(B\coloneqq\sqrt{n/2r}\) and \(\gamma\coloneqq\delta/2\) in Theorem~\ref{theorem:maintheorem} to obtain that \(U(S)\leq n^{1/2+\delta}\), as claimed.

To apply Theorem~\ref{theorem:maincolourtheorem}, we must also examine~\(D^{1/r}\).\COMMENT{No need for case analysis as to whether there are terms \((D/D_j)^{1/(j-1)}\) to bound this time. We always have to look at \(D^{1/r}\), and \(D^{1/r}\) is always worse than all the other \((D/D_j)^{1/(j-1)}\) in this setting.}
By the above, clearly \(D^{1/r}\geq (n/2r)^{(t-1)/r}\).\COMMENT{We had already showed \(D^{1/(r-1)}\geq(n/2r)^{(t-1)/(r-1)}\) so just raise this to the \((r-1)/r\).}
Clearly the bottleneck for~\(B\) for applying Theorem~\ref{theorem:maincolourtheorem} is\COMMENT{Because it was previously either \(\sqrt{n/2r}\) or \(D^{1/(r-1)}\), but now \(D^{1/r}<D^{1/(r-1)}\) is always worse than \(D^{1/(r-1)}\)} either~\(\sqrt{n/2r}\) or~\((n/2r)^{(t-1)/r}\) depending on whether \((t-1)/r<1/2\), so now Theorem~\ref{theorem:maincolourtheorem} applied in each of these regimes of \(t,r\) finishes the proof, as above.
%
%
\endproof
The proof of Theorem~\ref{theorem:PartiteSteinercolourmatchings} involves very similar computations, so we are more brief.
\lateproof{Theorem~\ref{theorem:PartiteSteinercolourmatchings}}
Clearly we may assume \(1/n\ll1/A\ll\delta\ll1/r<1/t\leq 1/2\) for some new constant~\(A\).
Fix a partite \((t,r,n)\)-Steiner system~\(S\).
It is easy to check that~\(S\) is \(D\coloneqq (n^{t-1})\)-regular, \(r\)-uniform, and satisfies \(C_j(S)\leq n^{t-j}\eqqcolon D_j\) for \(2\leq j\leq t\) and \(C_{j}(S)=1\eqqcolon D_j\) for \(t+1\leq j\leq r\). Set \(\eps\coloneqq1/\sqrt{n}\) (say).
We have \(\sqrt{D/D_2}=\sqrt{n}=1/\eps\).
For \(4\leq j\leq t\) (if there are such), we have \((D/D_j)^{1/(j-1)}=n>\sqrt{n}\) and for \(r\geq 4\) we must consider \((D/D_{r})^{1/(r-1)}=n^{(t-1)/(r-1)}\).
The value of~\(B\) with which to apply Theorem~\ref{theorem:maintheorem} is thus either~\(n^{(t-1)/(r-1)}\) if \(t<(r+1)/2\) (which occurs only if \(r\geq 4\)), or~\(\sqrt{n}\) otherwise.
For applying Theorem~\ref{theorem:maincolourtheorem}, the limiting expression for~\(B\) will be either~\(D^{1/r}=n^{(t-1)/r}\) if \(t<r/2 +1\), or~\(\sqrt{n}\) otherwise.
\endproof
\subsection{Partial Steiner Systems}\label{Section:Steiner}
In this subsection we use Theorem~\ref{theorem:maintheorem} to present a short proof of Theorem~\ref{theorem:Steiner}.
\lateproof{Theorem~\ref{theorem:Steiner}}
Clearly we may assume \(\delta\ll1/r\)\COMMENT{since proving the theorem for small \(\delta\) proves it for large \(\delta\)}.
Insert a new constant~\(A\) into the hierarchy such that \(1/n\ll1/A\ll\delta\ll1/r<1/t\leq 1/2\).
Consider the \(n\)-element ``ground set''~\([n]\).
Construct the hypergraph~\(H_n\) as follows:
Put \(V(H_n)\coloneqq\binom{[n]}{t}\), and for each \(r\)-set \(L\in\binom{[n]}{r}\), put an edge~\(e_L=\binom{L}{t}\) in~\(H_n\).
Clearly~\(H_n\) is \(\binom{r}{t}\)-uniform and \((n',D,\eps)\)-regular, where \(n'=\binom{n}{t}\), \(D=\binom{n-t}{r-t}\), and \(\eps=1/\sqrt{n}\) (say).\COMMENT{In fact \(H_n\) is perfectly \(\binom{n-t}{r-t}\)-regular}
Notice that, for each \(i=0,1,\dots,r-t-1\), any set~\(U\) of \(\binom{t+i}{t}+1\) vertices of~\(H_n\) must collectively cover at least \(t+i+1\) elements of the ground set, and thus have \(\text{codeg}_{H_n}(U)\leq \binom{n-(t+i+1)}{r-(t+i+1)}\).\COMMENT{Here using \(1/n\ll1/r,1/t\) so that the binomial coefficient \(\binom{n-(t+i+1)}{r-(t+i+1)}\) is strictly decreasing as \(i\) increases over \([r-t-1]_0\)}
For each \(i\in[r-t-1]_0\) and \(j\in\{\binom{t+i}{t}+1, \binom{t+i}{t}+2\dots, \binom{t+i+1}{t}\}\), set \(D_j\coloneqq\binom{n-t-i-1}{r-t-i-1}\).
This defines~\(D_j\) uniquely for all \(j\in[\binom{r}{t}]\setminus\{1\}\), and for all such~\(j\), by the above discussion (and using the fact \(C_{j+1}(H)\leq C_j(H)\) for any hypergraph~\(H\)), we have \(C_j(H)\leq D_j\).

We seek to apply Theorem~\ref{theorem:maintheorem}, for which we need to compute a lower bound for the minimum of \(\sqrt{D/D_2}\), \(\min_{j\in[\binom{r}{t}]\setminus[3]}(D/D_j)^{1/(j-1)}\), \(1/\eps\).
To that end, observe \(\sqrt{D/D_2}=\sqrt{(n-t)/(r-t)}\geq\sqrt{n/2r}\) and clearly \(1/\eps=\sqrt{n}\geq \sqrt{n/2r}\).
The only values \(r>t\geq 2\) yielding \(\binom{r}{t}\leq 3\) (so that there are no other codegree ratios we need to find a minimum of) are \(r=3\), \(t=2\), so in this case we may take \(B=\sqrt{n/2r}\) and \(\gamma\coloneqq\delta/2\) to see that there exists a matching of~\(H_n\) covering all but at most \(\binom{n}{t}(\sqrt{n/2r})^{-1+\delta/2}\log^A (n^{r})\leq \binom{n}{t}n^{-1/2+\delta/2}\) vertices of~\(H_n\), which corresponds to a partial Steiner system~\(S_{p}(t,r,n)\) of size at least \(\lflr\binom{n}{t}/\binom{r}{t}\rflr-\lcl \binom{n}{t}n^{-1/2+\delta/2}/\binom{r}{t}\rcl \geq(1-n^{-1/2 +\delta})\binom{n}{t}/\binom{r}{t} = (1-n^{-(r-t)/(\binom{r}{t}-1) +\delta})\binom{n}{t}/\binom{r}{t}\), as required.\COMMENT{As in the Steiner matching/colouring proofs, you could just compute all bounds and then later argue that the one arising from there being more terms \((D/D_j)^{1/(j-1)}\) to consider doesn't end up being the bottleneck numerically for the case where there are no such terms to consider, but it's a little less immediate to justify that this time around, and just doing them separately as I have done here doesn't really add much space.}

We may now assume that \(r\geq 4\), which implies \(\binom{r}{t}\geq 4\), so that there are terms \((D/D_j)^{1/(j-1)}\) to find a lower bound for.
In particular, since for each \(i\in[r-t-1]_0\), all~\(j\) such that \(\binom{t+i}{t}+1\leq j\leq \binom{t+i+1}{t}\) receive the same value of~\(D_j\), it suffices to find a lower bound for all the ratios corresponding to the largest~\(j\) in each ``cluster'' of equal codegrees, so it suffices to lower bound \(\min_{i\in[r-t-1]_0}(D/\binom{n-t-i-1}{r-t-i-1})^{1/(\binom{t+i+1}{t}-1)}\).
Substituting \(D=\binom{n-t}{r-t}\), each of these terms can be bounded below by \(((n-t-i)^{i+1}/(r-t)^{i+1})^{1/(\binom{t+i+1}{t}-1)}\geq (n/2r)^{(i+1)/(\binom{t+i+1}{t}-1)}\).
Setting \(f(i)\coloneqq (i+1)/(\binom{t+i+1}{t}-1)\), we claim that \((f(i))_{i\in[r-t-1]_0}\) is monotone decreasing.
Indeed,
\begin{eqnarray*}
f(i+1) & = & f(i)\cdot\frac{i+2}{i+1}\cdot\frac{\binom{t+i+1}{t}-1}{\binom{t+i+2}{t}-1}\\ & = & f(i)\left(1+\frac{1}{i+1}\right)\left(\frac{\left(\binom{t+i+2}{t}-1\right)\frac{i+2}{t+i+2}+\frac{i+2}{t+i+2}-1}{\binom{t+i+2}{t}-1}\right)\\ & \leq & f(i)\left(1+\frac{1}{i+1}\right)\frac{i+2}{t+i+2} = f(i)\left(\frac{i+2+\frac{i+2}{i+1}}{t+i+2}\right),
\end{eqnarray*}
which, together with the facts \((i+2)/(i+1)\leq 2\) (attained uniquely by \(i=0\)) and \(t\geq 2\), yields \(f(i+1)\leq f(i)\) as claimed.
Therefore, each of the terms \((D/D_j)^{1/(j-1)}\) can be bounded below by \((n/2r)^{f(r-t-1)}=(n/2r)^{(r-t)/(\binom{r}{t}-1)}\).
We claim that \((r-t)/(\binom{r}{t}-1)\leq 1/2\).
Indeed, if \(r=t+1\), then \(\binom{r}{t}-1=t\geq 2= 2(r-t)\), as claimed, and otherwise \(\binom{r}{t}-1\geq\binom{r}{2}-1\), which is at least \(2r-4\) since \(r\geq 3\),\COMMENT{I just plotted \(y=r^2 /2 -r/2 -2r+3\) on Desmos} and \(2r-4\geq 2(r-t)\).
We conclude that we may apply Theorem~\ref{theorem:maintheorem} with \(B\coloneqq (n/2r)^{(r-t)/(\binom{r}{t}-1)}\) and \(\gamma\coloneqq\delta/2\), which again yields the desired result.
\endproof
\section{The Nibble}\label{Section:Nibble}
In this section, we begin the proof of Theorem~\ref{theorem:maintheorem} by performing the probabilistic analysis that shows that an almost-regular hypergraph admits a small matching (and a very small set of ``waste'' vertices) such that the leftover hypergraph is also almost-regular and well-distributed with respect to some well-behaved set of weight functions.
Moreover, if the input hypergraph admits some set \(J^{*}\subseteq [k]\setminus\{1\}\) such that the codegree ratios~\(D_j/D_{j+1}\) are large for \(j\in J^{*}\), then we show there is a matching as above such that each~\(D_j\)  for \(j\in J^{*}\) has degraded in the leftover hypergraph, in line with expectations.
Specifically, the aim is to prove the following lemma, which we refer to throughout the paper as the `Nibble Lemma'.\COMMENT{In the case \(k=1\) (graphs), the statement forces you to have \(J^{*}\) empty, and all you have is \(D\) and \(D_2\), and the latter has to satisfy \(C_2(H)\leq D_2\). \ref{hypothesis:codegratio} holds vacuously since there are no such \(j\). The conclusion at the end holds vacuously since there are no such \(j\).}\COMMENT{No need to worry in the Nibble Lemma about having enough space between one \(D_j\) we mean to degrade and a non-moving \(D_{j+1}\). There is enough space because we're only decreasing \(D_j\) by a \(1-o(1)\) factor and there's polylog between them by hypothesis, and the Nibble Lemma on its own doesn't care about ensuring you can apply the Nibble Lemma afterwards anyway. Such concerns will arise in the future sections.} 
\begin{lemma}[Nibble Lemma]\label{lemma:nibble}
Suppose $1/D, \theta\ll 1/k\leq 1$, and let~$H$ be an $(n,D,\eps)$-regular, $(k+1)$-uniform hypergraph.
Suppose that \(D_2, D_3, \dots, D_{k+1}\) are numbers satisfying \(C_{j}(H)\leq D_j\), and consider a (permissibly empty) set \(J^{*}\subseteq [k]\setminus\{1\}\).
Suppose further that:
\begin{enumerate}[(N1), topsep = 6pt]
\item $\eps^{2}\theta\frac{D}{D_{2}}\geq\log^{5}D$;\COMMENT{I find it a useful sanity check to remember this is a lower bound on \(\eps\), so for example Lemma~\ref{lemma:concanalysis}\ref{concvertexdegree} is not saying with high probability a vertex degree is exactly \(\text{deg}(v)(1-p^{*})^k\) in the case of exactly vertex-regular hypergraphs. You're still forced to pick some value of \(\eps\geq\sqrt{D_2/D}\geq n^{-1/2}\) in this case.} \label{hypothesis:degratio}
\item $\theta^{2}\frac{D_{j}}{D_{j+1}}\geq\log^{5}D$ for all $j\in J^{*}$; \label{hypothesis:codegratio}
\item $\eps\leq\theta$. \label{hypothesis:epsleqtheta}
\end{enumerate}
Then there is a matching~$M$ of~$H$ and a set $W\subseteq V(H)$ of size $|W|\leq 10\eps\theta n$, together with numbers~$n'$ and~$D'$ satisfying $n(1-\theta-2\theta^{3/2})\leq n'\leq n(1-\theta+2\theta^{3/2})$ and $D(1-k\theta-\theta^{3/2})\leq D' \leq D(1-k\theta +\theta^{3/2})$, such that the hypergraph~$H'$ induced by~$V(H)\setminus (V(M)\cup W)$ is $(n',D',\eps')$-regular, where $\eps'\coloneqq\eps(1+\theta)$.
Further, $C_{j}(H')\leq D_{j}(1-(k-j+1)\theta+\theta^{3/2})$ for all $j\in J^{*}$.

If, in addition,~\(\cT\) is a family of functions \(\tau\colon V(H)\rightarrow \mathbb{R}_{\geq0}\) such that
\begin{enumerate}[(NP1), topsep = 6pt]
\item \(\eps\theta\tau(V(H))\geq\tau_{\text{max}}\cdot\log^5 D\) for all \(\tau\in\cT\), where \(\tau_{\text{max}}\coloneqq\max_{v\in V(H)}\tau(v)\);\label{pseudhyp:lowerbound}
\item \(\text{Supp}(\tau)\coloneqq\{v\in V(H)\colon\tau(v)>0\}\) satisfies \(|\text{Supp}(\tau)|\leq D^{\log^{5/4}D}\) for all \(\tau\in\cT\);\label{pseudhyp:supportsize}
\item Each vertex of~\(H\) is in~\(\text{Supp}(\tau)\) for at most~\(D^{\log^{5/4}D}\) of the functions \(\tau\in\cT\),\label{pseudhyp:maxinvolvement} 
\end{enumerate}
then there exist~\(M, W\) as above additionally satisfying \(\tau(V(H))(1-\theta-\theta^{3/2})\leq \tau(V(H)\setminus (V(M)\cup W))\leq\tau(V(H))(1-\theta+\theta^{3/2})\) and \(\tau(W)\leq 10\eps\theta\tau(V(H))\) for all \(\tau\in\cT\).
\end{lemma}
It suffices to prove that Lemma~\ref{lemma:nibble} holds with the inclusion of some arbitrary~\(\cT\) satisfying~\ref{pseudhyp:lowerbound}--\ref{pseudhyp:maxinvolvement}, as one then obtains the lemma without~\(\cT\) by setting \(\cT=\emptyset\).
Suppose that~$H$ satisfies the hypotheses of Lemma~\ref{lemma:nibble} (including~\ref{pseudhyp:lowerbound}--\ref{pseudhyp:maxinvolvement}).
Firstly, we remark that we may suppose that $1/n,\eps\ll 1/k$.
Indeed,~\ref{hypothesis:epsleqtheta} and the assumption $\theta\ll1/k$ show that~$\eps$ can be as small as we desire with respect to~$k$.
Moreover, for any $v\in V(H)$, we have 
\begin{equation}\label{eq:ngeqdd2}
D/2\leq D(1-\eps)\leq \text{deg}(v)\leq\sum_{w\neq v}\text{codeg}(\{v,w\})\leq nD_{2},
\end{equation}
whence (using~\ref{hypothesis:degratio}) we have $n\geq D/(2D_{2})\geq\eps^{2}\theta D/(2D_{2})\geq(\log^{5}D)/2$, so that $1/n\ll1/k$ as claimed, using the assumption $1/D\ll1/k$.
We also remark that we may assume \(D_j\geq 1\) for all \(j\in[k+1]\setminus\{1\}\), since the hypergraph must have at least some edge\COMMENT{\(1/D\ll1\) and \(\eps<1/2\)}, so any \(j\)-subset of this edge witnesses that \(D_j\geq C_j(H)\geq 1\).

We first prove Lemma~\ref{lemma:nibble} for hypergraphs satisfying \(C_{k+1}(H)=1\) (see Sections~\ref{section:description}--\ref{section:concanalysis}).
We then discuss the minor changes needed to prove Lemma~\ref{lemma:nibble} in the case of multihypergraphs (having \(C_{k+1}(H)>1\)) in Section~\ref{section:multi}. We now discuss the random procedure (the `random nibble') that we will apply to~$H$.
The aim of this section will then be to show that the random nibble applied to~$H$ produces a matching $M\subseteq E(H)$ and a `waste set' $W\subseteq V(H)$ satisfying the conclusions of the Nibble Lemma, with positive probability.
\subsection{Description of the Nibble}\label{section:description}
Set $p\coloneqq \theta/\delta(H)$.
Construct a random set of edges $\mathbf{X}\subseteq E(H)$ by selecting each edge to be in~$\mathbf{X}$ independently with probability~$p$.
We say that an edge~$e\in E(H)$ is \emph{isolated} if $e\in\mathbf{X}$ and~$e$ does not intersect any other $e'\in\mathbf{X}\setminus\{e\}$.
We define~$\mathbf{M}$ to be the set of isolated edges.
It is clear that~$\mathbf{M}$ is a matching.

For each $v\in V(H)$, notice that\COMMENT{Since~$\mathbf{M}$ is a matching, the events $\{e\in\mathbf{M}\}_{e\ni v}$ are pairwise disjoint.}
\[
p(v)\coloneqq\prob{v\in V(\mathbf{M})} = \sum_{e\ni v}\prob{e\in\mathbf{M}}=\sum_{e\ni v}p(1-p)^{f(e)},
\]
where~$f(e)$ is the number of non-$e$ edges intersecting~$e$.
Set~$p^{*}$ to be the maximum~$p(v)$ over all $v\in V(H)$, and for each $v\in V(H)$ define~$w(v)\coloneqq(p^{*}-p(v))/(1-p(v))$. It is clear that $0\leq w(v)<1$ for each $v\in V(H)$.\COMMENT{Observe that $p(v)\leq \Delta(H)p\leq(1+\eps)\theta/(1-\eps)<1$ so that~$w(v)$ is defined \textcolor{blue}{(Here used that, say $\eps\leq 1/3$, $\theta <1/2$)}, and we have $1-p(v)>0$ since $p(v)<1$, and we have $p^{*}-p(v)\geq 0$ by definition of~$p^{*}$, so $w(v)\geq 0$ for all $v\in V(H)$. Further, $p^{*}<1$ (since $p(v)<1$ for all $v$), whence $p^{*}-p(v)<1-p(v)$, so $w(v)<1$.}
Thus we may define a random set $\mathbf{W}\subseteq V(H)$ by selecting each $v\in V(H)$ to be in~$\mathbf{W}$ independently with probability~$w(v)$.
Define~$\mathbf{H'}$ to be the hypergraph induced by $V(H)\setminus(V(\mathbf{M})\cup\mathbf{W})$, define $\mathbf{n'}\coloneqq |V(\mathbf{H'})|$, and notice that for each $v\in V(H)$ we have
\begin{equation}\label{eq:pstar}
\prob{v\in V(\mathbf{H'})}=(1-p(v))(1-w(v))=1-p^{*}.
\end{equation}

We say a vertex \(v\in V(H)\) \textit{survives the nibble} if \(v\in V(\mathbf{H'})\), or equivalently \(v\notin V(\mathbf{M})\cup\mathbf{W}\).
Ensuring that each vertex is equally likely to survive the nibble is the main purpose of throwing some vertices into the `waste set'~$\mathbf{W}$.
We write~$(\Omega_{v},\Sigma_{v},\mathbb{P}_{v})$ and $(\Omega_{e},\Sigma_{e},\mathbb{P}_{e})$ for the probability space corresponding to the random decision to put \(v\in\mathbf{W}\) and \(e\in\mathbf{X}\), respectively; that is, with outcomes~\(\omega^{1}_{v}, \omega^{0}_{v}\) and \(\omega^{1}_{e}, \omega^{0}_{e}\) satisfying \(\mathbb{P}_{v}\left[\omega^{1}_{v}\right]=w(v)\), \(\mathbb{P}_{v}\left[\omega^{0}_{v}\right]=1-w(v)\) and \(\mathbb{P}_{e}\left[\omega^{1}_{e}\right]=p\), \(\mathbb{P}_{e}\left[\omega^{0}_{e}\right]=1-p\).

For any vertex $v\in V(H)$, we define \(\textbf{deg}\mathbf{'}(v)\coloneqq \sum_{e\in\partial_{H}(v)}\mathbf{I}_{e\setminus\{v\}}\), where~\(\mathbf{I}_{e\setminus\{v\}}\) is the indicator random variable for the event that all vertices \(u\in e\setminus\{v\}\) survive the nibble.
Note that if~$v$ survives the nibble, then~$\textbf{deg}\mathbf{'}(v)$ is simply the degree of~$v$ in~$\mathbf{H'}$.
Similarly, for $j\in J^{*}$ and $Z\in\binom{V(H)}{j}$, define $\textbf{codeg}\mathbf{'}(Z)\coloneqq\sum_{e\in\partial_{H}(Z)}\mathbf{I}_{e\setminus Z}$, with the obvious definition of~\(\mathbf{I}_{e\setminus Z}\).\COMMENT{There are no such \(j\) and therefore no such \(Z\) when \(k=1\) (graphs), since the statement of the lemma forced \(J^{*}\) to be empty.}
We identify the `random nibble' with the pair~$(\mathbf{X},\mathbf{W})$ (using all parameters as defined above).
It is clear that the choices of~$\mathbf{X}$ and~$\mathbf{W}$ are determined by the outcomes~$\{\omega_{v}\}_{v\in V(H)}$ and~$\{\omega_{e}\}_{e\in E(H)}$, which in turn determines all of $\mathbf{M}$, $\mathbf{H'}$, $\mathbf{n'}$, $\{\textbf{deg}\mathbf{'}(v)\}_{v\in V(H)}$, $\{\textbf{codeg}\mathbf{'}(Z)\}_{Z\subseteq V(H), |Z|=j}$ for all $j\in J^{*}$, and~\(\{\tau(V(\mathbf{H'})), \tau(\mathbf{W})\}_{\tau\in\cT}\).

It is useful at this stage to collect some inequalities about our parameters that  we will use throughout the proof.
Clearly we have\COMMENT{$(1+\eps)/(1-\eps)=1 +2\eps/(1-\eps) \leq 1+3\eps$ \textcolor{blue}{using say $\eps\leq 1/3$.}} \(\Delta(H)/\delta(H)\leq 1+3\eps.\)
With~$f(e)$ defined as above, we have for any $e\in E(H)$ that $f(e)\leq (k+1)\Delta(H)$, and $f(e)\geq (k+1)(\delta(H)-1)-\binom{k+1}{2}D_{2}\geq (k+1)D(1-3\eps/2)$\COMMENT{$(k+1)(\delta(H)-1)-\binom{k+1}{2}D_{2}=(k+1)(\delta(H)-1-kD_{2}/2)\geq (k+1)((1-\eps)D-1-kD_{2}/2)=(k+1)D(1-\eps-1/D-kD_{2}/2D)$. It now suffices to notice that $kD_{2}/2D\leq \eps/8$ and $1/D\leq\eps/4$; Indeed, $kD_{2}/2D\leq k\eps^{2}\theta/2\log^{5}D\leq k\eps^{2}\theta/2=\eps(k\eps\theta/2)\leq\eps/8$ \textcolor{blue}{using say $\theta\leq1/2$ and $\eps\leq 1/2k$ and $\log^{5}D\geq 1$} and hypothesis~\ref{hypothesis:degratio}. Then $1/D = (2/kD_{2})\cdot(kD_{2}/2D)\leq 2\cdot kD_{2}/2D\leq\eps/4$, using $k,D_{2}\geq 1$.}, where we have used hypothesis~\ref{hypothesis:degratio}.
Thus\COMMENT{\((1+3\theta)\theta\exp(-(1-3\eps/2)(k+1)\theta)\leq (1+3\theta)\theta(1-(1-3\eps/2)(k+1)\theta+(1-3\eps/2)^2 (k+1)^2\theta^2)\leq (1+3\theta)\theta(1+(k+1)^2\theta^2)=\theta+(k+1)^2\theta^3 + 3\theta^2 +3(k+1)^2\theta^4\leq \theta+\theta^{7/4}\). Interestingly we can do better here if you don't use \(\eps\leq\theta\) too early and don't just discard the negative \(O(\theta)\) term, but even then the expression ends up approx \(\theta+3\eps\theta - (k+1)\theta^2\) which could be \(<\theta\) if \(\eps\ll\theta\), but could, if \(k=1\) and \(\eps=\theta\), still be \(\theta +\theta^2\leq\theta+\theta^{7/4}\), so the given bound isn't necessarily always that wasteful},
\begin{eqnarray}\label{eq:pstarabove}
p^{*} & \leq & \Delta(H)p(1-p)^{(1-3\eps/2)(k+1)D}\leq (1+3\eps)\theta\left(1-p\right)^{(1-3\eps/2)(k+1)D}\nonumber\\ & \stackrel{\ref{hypothesis:epsleqtheta}}{\leq} & (1+3\theta)\theta\exp(-(1-3\eps/2)(k+1)pD) \leq (1+3\theta)\theta(1+(k+1)^2 \theta^2)\leq\theta+\theta^{7/4},
\end{eqnarray}
where we have used \(pD\geq\theta\).
Similarly\COMMENT{$\theta(1-p)^{(1+\eps)(k+1)D}=\theta\left(1-\frac{\theta(1+\eps)(k+1)D}{\delta(H)(1+\eps)(k+1)D}\right)^{(1+\eps)(k+1)D}\geq\theta(1-\theta(1+\eps)(k+1)D/\delta(H))\geq\theta(1-\theta(1+\eps)(k+1)/(1-\eps))=\theta-\theta^{2}(1+\eps)(k+1)/(1-\eps)$. Here we used $(1+x/n)^{n}\geq (1+x)$ for $n\geq 1, |x|\leq n$ with $n\coloneqq (1+\eps)(k+1)D$, $x\coloneqq -\theta (1+\eps)(k+1)D/\delta(H)$ and note $\theta/\delta(H)\leq\theta/(1-\eps)D \leq 1/10$ say, \textcolor{blue}{using $\theta\leq 1/2$, $\eps\leq 1/2$, $D\geq 10$}. Then note that $\theta^{2}(1+\eps)(k+1)/(1-\eps)\leq 3(k+1)\theta^{2}$ \textcolor{blue}{(using $\eps\leq 1/2$)} which is at most $\theta^{7/4}$ \textcolor{blue}{provided $\theta\leq 1/3^{4}(k+1)^{4}$}.},
\begin{equation}\label{eq:pstarbelow}
p^{*}\geq\min_{u\in V(H)}p(u)\geq \theta\left(1-p\right)^{(1+\eps)(k+1)D}\geq\theta(1-(1+\eps)(k+1)pD)\geq \theta-\theta^{7/4}.
\end{equation}
Finally, we claim
\begin{equation}\label{eq:wvabove}
w(v)\leq 8\eps\theta \hspace{5mm}\text{for}\,\text{all}\,v\in V(H).
\end{equation}
Indeed, for any $v\in V(H)$ we have\COMMENT{\textcolor{blue}{Using $p(v)\leq p^{*}\leq \theta+\theta^{7/4}\leq 1/2$ in the first inequality.}} $w(v)\coloneqq (p^{*}-p(v))/(1-p(v))\leq 2(p^{*}-\min_{u\in V(H)}p(u))\leq 2\theta(1-p)^{(1-3\eps/2)(k+1)D}(1+3\eps-(1-p)^{5\eps(k+1)D/2})$ by~(\ref{eq:pstarabove}) and~(\ref{eq:pstarbelow}).
Since\COMMENT{$(1-\theta/\delta(H))^{5\eps(k+1)D/2}\geq (1-3\theta/2\delta(H) +9\theta^{2}/4\delta(H)^{2})^{5\eps(k+1)D/2}\geq\exp(-(3\theta/2\delta(H))\cdot(5\eps(k+1)D/2))\geq 1- 5\eps\theta(k+1)D/\delta(H)\geq 1-5\eps\theta(k+1)/(1-\eps)\geq 1-10\eps\theta(k+1)\geq 1-\eps$. First inequality required $\theta/2\delta(H)\geq 9\theta^{2}/4\delta(H)^{2}$ which follows by \textcolor{blue}{$\theta<1/2$ and $\delta(H)\geq (1-\eps)D\geq (1/2)\cdot 10=5$}. Also used $e^{x}\leq 1+x +x^{2}$ for $x\coloneqq -3\theta/2\delta(H)<0<1.79$. Final inequality used \textcolor{blue}{$\theta\leq 1/10(k+1)$.}} $(1-p)^{5\eps (k+1)D/2}\geq 1-10(k+1)\eps\theta\geq 1-\eps$ and $(1-p)^{(1-3\eps/2)(k+1)D}\leq 1$, it follows that $w(v)\leq 2\theta(1+3\eps -(1-\eps))=8\eps\theta$, as claimed.
\subsection{Proof modulo concentration}\label{section:modulo}
We now state a lemma which encapsulates the concentration analysis for the random nibble, and shows that each of the associated ``bad events'' occur with very low probability individually.
We defer the proof of this lemma to Section~\ref{section:concanalysis}.
We first define each of the ``bad events''.
Let~\(A_1\) be the complement of the event \(\{n(1-\theta-2\theta^{3/2})\leq\mathbf{n'}\leq n(1-\theta+2\theta^{3/2})\}\), and let~\(A_2\) be the event~\(\{|\mathbf{W}|>10\eps\theta n\}\).
For each \(v\in V(H)\), let~\(A_v\) be the complement of the event \(\{\text{deg}(v)((1-p^{*})^k -\eps\theta/2)\leq\textbf{deg}\mathbf{'}(v)\leq \text{deg}(v)((1-p^{*})^k +\eps\theta/2)\}\), and for each \(j\in J^{*}\), \(Z\in\binom{V(H)}{j}\), let~\(A_Z\) denote the event \(\{\textbf{codeg}\mathbf{'}(Z)>D_j(1-(k-j+1)\theta + \theta^{3/2})\}\).
Finally, for each \(\tau\in\cT\), let~\(A^{(1)}_{\tau}\) denote the complement of the event \(\{\tau(V(H))(1-\theta-\theta^{3/2})\leq \tau(V(\mathbf{H'}))\leq \tau(V(H))(1-\theta+\theta^{3/2})\}\), and let~\(A^{(2)}_{\tau}\) denote the event \(\{\tau(\mathbf{W})>10\eps\theta\tau(V(H))\}\).
\begin{lemma}\label{lemma:concanalysis}
Let~\(H\) satisfy the hypotheses of Lemma~\ref{lemma:nibble} and apply a random nibble~\((\mathbf{X},\mathbf{W})\) to~\(H\).
Then:
\begin{enumerate}[label=\upshape(\roman*)]
\item \(\prob{A_1}\leq 2\exp(-\theta^{2}n/8(k+1)^3)\);\label{concleftover}
\item \(\prob{A_2}\leq 2\exp(-\eps\theta n/6)\);\label{concwaste}
\item \(\prob{A_v} \leq D^{-4\log^{5/4}D}\) for all \(v\in V(H)\);\label{concvertexdegree}
\item \(\prob{A_Z}\leq D^{-4\log^{5/4}D}\) for all \(j\in J^{*}\), \(Z\in\binom{V(H)}{j}\);\label{conccodegree}
\item \(\prob{A^{(1)}_{\tau}}\leq D^{-4\log^{5/4}D}\) for all \(\tau\in\cT\);\label{conctestpseudo}
\item \(\prob{A^{(2)}_{\tau}}\leq D^{-4\log^{5/4}D}\) for all \(\tau\in\cT\).\label{conctestwaste}
\end{enumerate}
\end{lemma}
Assuming Lemma~\ref{lemma:concanalysis}, we are now ready to prove the Nibble Lemma.
We need only show that an outcome of the random nibble in which none of the above bad events occur satisfies the conclusions of Lemma~\ref{lemma:nibble} (in particular that~\(\mathbf{H'}\) is \((\mathbf{n'},D',\eps')\)-regular for well-chosen~\(D'\))\COMMENT{As the rest is obvious from the definitions of the bad events, also \(\eps'\) is defined in the statement of Lemma~\ref{lemma:nibble}}, and that there exists such an outcome.
The latter task requires the most attention, and we will use the Lov\'{a}sz Local Lemma (Lemma~\ref{lemma:LLL}) for this purpose.
Instead of using a dependency digraph with a vertex for each of the events \(A_1, A_2, A_v, A_Z, A^{(1)}_{\tau}, A^{(2)}_{\tau}\), it will be useful for us (in the sense that it enables us to not make any further assumptions about the relationship between~\(n\) and \(D\)) to collect vertices~\(v\) and tuples~\(Z\) into a smaller number of sets~\(S\), and construct dependency-digraph vertices corresponding to the event that a fixed~\(S\) contains any~\(v\) or~\(Z\) for which~\(A_v\) or~\(A_Z\) occur.
This approach is similar to the way in which Vu~\cite{V00} collects his analogous failure probabilities using `important' and `perfect' sets (see the proof of~\cite[Lemma 5.2]{V00}).
\lateproof{Lemma~\ref{lemma:nibble}}
Apply a random nibble~$(\mathbf{X},\mathbf{W})$ to~$H$.
If \(J^{*}=\emptyset\), then set \(s=1\) and \(\mathbb{S}\coloneqq \{\{v\}\colon v\in V(H)\}\) (\(=\binom{V(H)}{s}\)).
If, instead, \(J^{*}\neq\emptyset\), then set \(s\coloneqq \max_{j\in J^{*}}j\) (\(\geq 2\) by construction of~\(J^{*}\)) and set \(\mathbb{S}\coloneqq\{S\in\binom{V(H)}{s}\colon \text{codeg}_{H}(S)>0\}\). 
For $S\in\mathbb{S}$ (whether~\(J^{*}\) is empty or not), we define the event $A^{*}_S\coloneqq\left(\bigcup_{v\in S}A_{v}\right)\cup\left(\bigcup_{j\in J^{*}}\bigcup_{Z\in\binom{S}{j}}A_{Z}\right)$\COMMENT{If \(J^{*}=\emptyset\) then the latter intersection involves no events, which I clarified in the notation section means taking an intersection with the sample space. Basically if we're only tracking vertex degrees then \(A^{*}_S\) occurring just means \(A_v\) occurs, where \(S=\{v\}\).}.
Since every vertex of~$H$ and all $Z\in\binom{V(H)}{j}$ such that $j\in J^{*}$ and $\text{codeg}(Z) >0$ are contained in an \(\mathbb{S}\)-set\COMMENT{Fix $v\in V(H)$. If \(J^{*}=\emptyset\), then clearly \(S=\{v\}\in\mathbb{S}\) contains \(v\). If \(J^{*}\neq\emptyset\), then we know $\text{deg}(v)\geq\delta(H)\geq (1-\eps)D>0$ \textcolor{blue}{(say $\eps<1/2$, $D\geq 10$)} so that there is an edge $e\ni v$ in~$H$. Arbitrarily choose $v\in S\subseteq e$ with $|S|=s < k+1$. Clearly $S$ is an \(\mathbb{S}\)-set containing~$v$. Similarly if $j\in J^{*}$ and \(Z\in\binom{V(H)}{j}\) satisfies $\text{codeg}(Z)>0$ then there is $e\supseteq Z$ in~$H$ so any $Z\subseteq S\subseteq e$ with $|S|=s$ is an \(\mathbb{S}\)-set containing~$Z$ (recall \(j\leq s\leq k\) by definition of \(s\)).}, it follows that\COMMENT{If \(J^{*}=\emptyset\) then we're only checking \(\bigcup_{v\in V(H)}A_v= \bigcup_{S\in\mathbb{S}}A^{*}_S\), but this is obvious by the definition of \(\mathbb{S}\) and \(A^{*}_S\) in the case \(J^{*}=\emptyset\).
So suppose that \(J^{*}\neq\emptyset\).
Clearly all \(Z\) for which \(\text{codeg}_H(Z)=0\) still have \(\textbf{codeg}\mathbf{'}(Z)=0\leq D_j(1-(k-j+1)\theta+\theta^{3/2})\) so such \(Z\) cannot be ``bad'' (i.e. the event \(A_Z\) does not occur).
Now, if some vertex is bad, then the \(S\) containing that vertex is bad.
If a tuple \(Z\) is bad, then it must have \(\text{codeg}_H(Z)>0\) by the above, and therefore must be in some important \(S\), which is thus bad.
Conversely, if some important tuple \(S\) is bad, then there is some bad vertex or bad tuple in it.} $\left(\bigcup_{v\in V(H)}A_{v}\right)\cup\left(\bigcup_{j\in J^{*}}\bigcup_{Z\in\binom{V(H)}{j}}A_{Z}\right)=\bigcup_{S\in\mathbb{S}}A^{*}_{S}$ (whether~\(J^{*}\) is empty or not).

Define $D'\coloneqq D(1-p^{*})^{k}$.
Using~(\ref{eq:pstarabove}) and~(\ref{eq:pstarbelow}), we observe that\COMMENT{$D(1-p^{*})^{k}\geq D(1-(\theta+\theta^{7/4}))^{k}\geq D(1-(\theta+\theta^{17/10})+(\theta+\theta^{17/10})^{2})^{k}\geq D\exp(-k(\theta+\theta^{17/10}))\geq D(1-k\theta-k\theta^{17/10})\geq D(1-k\theta-\theta^{3/2})$. Second inequality uses $-(\theta+\theta^{7/4})\geq -(\theta+\theta^{17/10})+(\theta+\theta^{17/10})^{2}$, which is equivalent to $\theta^{17/10}\geq\theta^{7/4}+\theta^{2}+2\theta^{27/10}+\theta^{34/10}$, so it suffices to have $\theta^{1/20}+\theta^{3/10}+2\theta+\theta^{17/10}\leq 1$, \textcolor{blue}{which in particular holds if $\theta\leq 1/1000$.} Third inequality uses $-(\theta+\theta^{17/10})<0<1.79$. Final inequality uses \textcolor{blue}{$\theta\leq 1/k^{5}$}. Similarly $D(1-p^{*})^{k}\leq D\exp(-p^{*}k)\leq D(1-p^{*}k+(p^{*})^{2}k^{2})\leq D(1-(\theta-\theta^{7/4})k+(\theta+\theta^{7/4})^{2}k^{2})\leq D(1-k\theta+\theta^{3/2})$, where the final inequality uses $k\theta^{7/4}+k^{2}\theta^{2}+2\theta^{11/4}+\theta^{14/4}\leq\theta^{3/2}$, which holds if in particular \textcolor{blue}{$\theta\leq 1/5^{4}k^{8}$. Also used $k\geq 1$.}}
\begin{equation}\label{eq:boundd'}
D(1-k\theta-\theta^{3/2})\leq D'\leq D(1-k\theta+\theta^{3/2}).
\end{equation}
With this definition of~$D'$, and with $\eps'\coloneqq\eps(1+\theta)$, we claim that if the event $\bigcap_{S\in \mathbb{S}}\overline{A^{*}_{S}}$ occurs, then~$\mathbf{H'}$ is~$(\mathbf{n'},D',\eps')$-regular.
Indeed, suppose $\bigcap_{S\in \mathbb{S}}\overline{A^{*}_{S}}$ occurs.
In particular, this means \(\bigcap_{v\in V(H)}\overline{A_v}\) occurs.
Recall that if $v\in V(H)$ survives the nibble then~$\textbf{deg}\mathbf{'}(v)$ is precisely the degree of~$v$ in~$\mathbf{H'}$.
It suffices to fix $v\in V(H)$ and check that $(1-\eps')D'\leq\textbf{deg}\mathbf{'}(v)\leq (1+\eps')D'$.
To that end, notice first that \((1-p^*)^k\geq (1-\theta-\theta^{7/4})^k \geq (1-3\theta+9\theta^2)^k\geq \exp(-3\theta k)\geq 1-3\theta k\geq 2/3\) by~(\ref{eq:pstarabove}).
Now fix \(v\in V(H)\).
Since~\(\overline{A_v}\) occurs, we have:
\begin{eqnarray*}
\textbf{deg}\mathbf{'}(v) & \leq & \text{deg}(v)((1-p^{*})^k + \eps\theta/2) \leq \frac{D'}{(1-p^{*})^k}(1+\eps)((1-p^{*})^k + \eps\theta/2) \\ & = & D'(1+\eps)\left(1+\frac{\eps\theta}{2(1-p^{*})^k}\right) \leq D'(1+\eps)\left(1+\frac{3\eps\theta}{4}\right)\leq D'(1+\eps').
\end{eqnarray*}
The lower bound is similar\COMMENT{$\textbf{deg}\mathbf{'}(v)\geq D(1-\eps)((1-p^{*})^{k}-\eps\theta/2)=D'(1-\eps)(1-\eps\theta/2(1-p^{*})^{k})\geq D'(1-\eps)(1-3\eps\theta/4)=D'(1-\eps-3\eps\theta/4+3\eps^{2}\theta/4)\geq D'(1-\eps-3\eps\theta/4)\geq D'(1-\eps-\eps\theta)$.}, so we omit it.
From the above discussion and~(\ref{eq:boundd'}), notice that any outcome of the nibble in the event \(\overline{A_1}\cap\overline{A_2}\cap\bigcap_{S\in\mathbb{S}}\overline{A^{*}_S}\cap\bigcap_{\tau\in\cT}\overline{A^{(1)}_{\tau}}\cap\bigcap_{\tau\in\cT}\overline{A^{(2)}_{\tau}}\) satisfies the conclusions of Lemma~\ref{lemma:nibble}.
It remains only to show that with positive probability, none of the above bad events \(A_1, A_2, A^{*}_S, A^{(1)}_{\tau}, A^{(2)}_{\tau}\) occur, for which we use Lemma~\ref{lemma:LLL}.
It will be useful to have the following bound on~\(|\mathbb{S}|\).\COMMENT{Indeed, the fact this turns out to be \(\Theta(nD)\) is the whole point of collecting the \(v\) and \(Z\) into these \(S\)}
\begin{claim}\label{claim:sizeofs}
$|\mathbb{S}|\leq 2nD(k+1)^{k}$.
\end{claim}
\claimproof
If \(J^{*}=\emptyset\) then \(|\mathbb{S}|=n\).
If \(J^{*}\neq\emptyset\), note that for each \(S\in\mathbb{S}\) there is an edge \(e_S\supseteq S\).
Since\COMMENT{This whole argument is unaffected in the multihypergraph case. For each \(S\in\mathbb{S}\) arbitrarily pick one (specific edge copy) \(e_S\supseteq S\).} counting these~\(e_S\) over all \(S\in\mathbb{S}\) counts each \(e\in E(H)\) at most~\(\binom{k+1}{s}\) times, we have \(|\mathbb{S}|\leq e(H)\binom{k+1}{s} \leq e(H)(k+1)^k\), from which the claim follows\COMMENT{\(e(H)=\sum_{v\in V(H)}\text{deg}(v)/(k+1)\leq (1+\eps)nD/(k+1)\leq 2nD\). Of course, you could replace \((k+1)^k\) in the claim with \((k+1)^{k-1}\), but I felt this may look awkward for the \(k=1\) case (in truth, if \(k=1\) then the statement of the Nibble Lemma forces \(J^{*}=\emptyset\)), and this lazy upper bound looks clean and does the job}.
\endclaimproof
Construct an auxiliary digraph~\(G\) with \(V(G)\coloneqq\{a_1,a_2\}\cup\{a_S\colon S\in\mathbb{S}\}\cup\{a^{(1)}_{\tau}, a^{(2)}_{\tau}\colon \tau\in\cT\}\) (these vertices corresponding in the obvious way to the events \(A_1, A_2, A^{*}_S, A^{(1)}_{\tau}, A^{(2)}_{\tau}\)), and put an arc in both directions between a pair of \(G\)-vertices if the corresponding events are such that there is some trial \(e\in\mathbf{X}\) or \(v\in\mathbf{W}\) affecting both.
Clearly~\(G\) is a dependency digraph.\COMMENT{If a \(G\)-vertex \(a\) sends no arcs to some set~\(U\) of \(G\)-vertices, then for each \(u\in U\), the event corresponding to \(a\) and the event corresponding to \(u\) are determined by a disjoint set of trials. Thus the set of trials determining \(a\)'s event is disjoint from the entire set of trials determining any of the events for \(u\in U\), whence clearly \(a\)'s event is mutually independent of all the events for \(u\in U\).}
Notice that~\(a_1\) and~\(a_2\) send and receive arcs from all other \(G\)-vertices.
For \(S\in\mathbb{S}\), define \(d_G(S,\mathbb{S})\) (respectively \(d_G(S,\cT)\)) to be the number of arcs in~\(G\) with tail~\(a_S\) and head~\(a_{S'}\) for some \(S'\in\mathbb{S}\) (respectively head~\(a^{(j)}_{\tau}\) for some \(\tau\in\cT, j\in\{1,2\}\)), and define \(\Delta_{G}(\mathbb{S},\mathbb{S})\coloneqq\max_{S\in\mathbb{S}}d_G(S,\mathbb{S})\) and \(\Delta_G(\mathbb{S},\cT)\coloneqq\max_{S\in\mathbb{S}}d_G(S,\cT)\).
Similarly, for \(\tau\in\cT\) and \(j\in\{1,2\}\), define \(d_G((\tau,j),\mathbb{S})\) (respectively~\(d_G((\tau,j),\cT)\)) to be the number of arcs in~\(G\) with tail~\(a^{(j)}_{\tau}\) and head~\(a_S\) for some \(S\in\mathbb{S}\) (respectively head~\(a^{(i)}_{\tau'}\) for some \(i\in\{1,2\}\), \(\tau'\in\cT\)), and define \(\Delta_G(\cT,\mathbb{S})\coloneqq\max_{\tau\in\cT, j\in\{1,2\}}d_G((\tau,j),\mathbb{S})\) and \(\Delta_G(\cT,\cT)\coloneqq\max_{\tau\in\cT, j\in\{1,2\}}d_G((\tau,j),\cT)\).
\begin{claim}\label{claim:dependencygraph}
We have \(\Delta_{G}(\mathbb{S},\mathbb{S})\leq D^7\), \(\Delta_{G}(\mathbb{S},\cT)\leq D^{2\log^{5/4}D}\), \(\Delta_{G}(\cT,\mathbb{S})\leq D^{2\log^{5/4}D}\), and \(\Delta_{G}(\cT,\cT)\leq D^{3\log^{5/4}D}\).
\end{claim}
\claimproof
Suppose first that \(S\in\mathbb{S}\) and some vertex \(v\in V(H)\) is such that the outcome~\(\omega_v\) of the choice to put \(v\in\mathbf{W}\) or not affects (is not independent of)~\(A^{*}_S\).
Then either \(v\in S\) or \(v\in e\) for some edge \(e\in E(H)\) intersecting~\(S\).
If the outcome~\(\omega_e\) of the choice to put \(e\in\mathbf{X}\) or not affects~\(A^{*}_S\), then~\(\omega_e\) affects \(\textbf{deg}\mathbf{'}(v)\) for some \(v\in S\),\COMMENT{Or, it affects \(\textbf{codeg}\mathbf{'}(Z)\) for some \(Z\subseteq S\), but in this case it clearly affects \(\textbf{deg}\mathbf{'}(v)\) for every \(v\in Z\subseteq S\).} which implies that either \(e\ni v\), or there is some edge \(f\in E(H)\) for which \(v\in f\) and \(f\cap e\neq\emptyset\), or there are edges \(f_1, f_2\in E(H)\) such that \(v\in f_1, f_1\cap f_2\neq \emptyset, f_2\cap e\neq \emptyset\).
If \(\tau\in \cT\) and~\(\omega_v\) affects~\(A^{(1)}_{\tau}\) or~\(A^{(2)}_{\tau}\), then \(v\in \text{Supp}(\tau)\).
Finally, if~\(\omega_e\) affects~\(A^{(1)}_{\tau}\), then either~\(e\) intersects~\(\text{Supp}(\tau)\), or there is an edge \(f\in E(H)\) such that \(f\cap \text{Supp}(\tau)\neq\emptyset\) and \(f\cap e\neq\emptyset\).\COMMENT{(And edge outcomes have no impact on \(\tau(\mathbf{W})\), i.e. the event~\(A^{(2)}_{\tau}\))}

Fix distinct \(S,S'\in\mathbb{S}\).
If there is some~\(v\in V(H)\) such that the outcome~\(\omega_v\) of the choice to put \(v\in\mathbf{W}\) or not affects both~\(A^{*}_S\) and~\(A^{*}_{S'}\), then by the above there is some \(v\in S\), \(v'\in S'\) such that \(\text{dist}_{H}(v,v')\leq 2\leq 5\).
If instead the trial \(e\in\mathbf{X}\) affects both, then there is some \(v\in S\), \(v'\in S'\) such that \(\text{dist}_{H}(v,v')\leq 5\).
Note there are at most~\(k\) choices for a vertex \(v\in S\), at most\COMMENT{1 at distance 0. At most \(\Delta(H)k\leq 2DK\) at distance 1. At most \((2Dk)^2\) at distance 2. ... At most \(1+2DK + \dots + (2Dk)^5 = (2Dk)^5(1+1/(2Dk) + 1/(2Dk)^2 + \dots + 1/(2Dk)^5)\leq (2Dk)^5(1+1/2+1/4+\dots+ 1/32)\leq 2\cdot(2Dk)^5\) at distance at most 5.} \(2\cdot(2kD)^5\) vertices~\(v'\in V(H)\) at distance at most~\(5\) from~\(v\) in~\(H\), then at most~\(2D\) choices for an edge of~\(H\) containing~\(v'\), and at most~\(2^{k+1}\) choices of an \(s\)-subset~\(S'\) of this edge.
We deduce that \(\Delta_{G}(\mathbb{S},\mathbb{S})\leq D^7\), as claimed. 
One checks similarly that if~\(A^{*}_S, A^{(j)}_{\tau}\) are not independent for \(S\in\mathbb{S}\), \(\tau\in\cT\), \(j\in\{1,2\}\), then there are vertices \(v\in S\), \(u\in \text{Supp}(\tau)\) such that \(\text{dist}_{H}(v,u)\leq 4\)\COMMENT{Indeed if \(j=2\) then it must be a~\(\omega_v\) affecting both, which implies \(v\in \text{Supp}(\tau)\) and (\(v\in S\) or \(v\in e\) for some \(e\) intersecting \(S\)), so distance 4 is overkill}, and that if \(A^{(i)}_{\tau}, A^{(j)}_{\tau'}\) are not independent for \(\tau,\tau'\in\cT\), \(i,j\in\{1,2\}\), then there are vertices \(u\in \text{Supp}(\tau)\), \(u'\in \text{Supp}(\tau')\) such that \(\text{dist}_{H}(u,u')\leq 3\).\COMMENT{Again if either \(i\) or \(j=2\) then it must be a \(\omega_v\) affecting both, which then actually has to be in both, so there exist \(u\in \text{Supp}(\tau)\), \(u'\in \text{Supp}(\tau')\) such that \(\text{dist}_{H}(u,u')=0\leq 3\).}
To analyze \(\Delta_{G}(\mathbb{S},\cT)\), note there are at most~\(k\) choices for a vertex \(v\in S\), at most \(2\cdot(2kD)^4\) vertices~\(u\in V(H)\) at distance at most~\(4\) from~\(v\) in~\(H\), and by~\ref{pseudhyp:maxinvolvement}, there are at most~\(D^{\log^{5/4}D}\) functions \(\tau\in\cT\) with~\(u\in\text{Supp}(\tau)\).
Counting \(G\)-arcs from~\(a_S\) to both~\(a_{\tau}^{(1)}\) and~\(a_{\tau}^{(2)}\), we deduce that\COMMENT{Just count the number of \(\tau\) for which there is a pair of vertices \(v,u\) satisfying the distance requirement, then for all such \(\tau\) we put arcs to both \(a_{\tau}^{(1)}\) and \(a_{\tau}^{(2)}\) so just double it. Clearly the upper bound provided still works.} \(\Delta_{G}(\mathbb{S},\cT)\leq D^{2\log^{5/4}D}\).
Now fix instead \(\tau\in\cT\).
By~\ref{pseudhyp:supportsize}, there are at most~\(D^{\log^{5/4}D}\) vertices \(u\in \text{Supp}(\tau)\).
There are at most \(2\cdot (2kD)^4\) choices for a vertex \(v\in V(H)\) at distance at most~\(4\) from~\(u\) in~\(H\), then at most~\(2D\) choices for an edge of~\(H\) containing~\(v\), and finally at most~\(2^{k+1}\) choices of an \(s\)-subset~\(S\) of this edge.
We deduce that \(\Delta_{G}(\cT,\mathbb{S})\leq D^{2\log^{5/4}D}\).
One proves similarly, using~\ref{pseudhyp:supportsize} and~\ref{pseudhyp:maxinvolvement}, that \(\Delta_{G}(\cT,\cT)\leq D^{3\log^{5/4}D}\).
We omit the details.\COMMENT{By~\ref{pseudhyp:supportsize}, there are at most~\(D^{\log^{5/4}D}\) choices for a vertex \(u\in \text{Supp}(\tau)\), at most \(2\cdot (2kD)^3\) choices for a vertex \(u'\in V(H)\) at distance at most \(3\) from \(u\). Then by~\ref{pseudhyp:maxinvolvement}, there are at most~\(D^{\log^{5/4}D}\) functions \(\tau'\in\cT\) with \(u'\in\text{Supp}(\tau')\). Now double it to account for edges to both \(a^{(1)}_{\tau'}\) and \(a^{(2)}_{\tau'}\)}
\endclaimproof
Set \(x_1, x_2 \coloneqq 1/2\), for each \(S\in\mathbb{S}\) set \(x_S\coloneqq(1/2) D^{-3\log^{5/4}D}\), and for each \(\tau\in\cT\) and \(j\in\{1,2\}\) set \(x^{(j)}_{\tau}\coloneqq (1/2)D^{-3\log^{5/4}D}\).
The following will be useful:
\begin{equation}\label{eq:implyoldN3}
\theta^2 n \stackrel{\ref{hypothesis:epsleqtheta}}{\geq}\eps\theta n\stackrel{(\ref{eq:ngeqdd2})}{\geq} \eps\theta \frac{D}{2D_{2}}\geq\eps^{2}\theta\frac{D}{2D_{2}}\stackrel{\ref{hypothesis:degratio}}{\geq}\frac{\log^{5}D}{2}\geq 32(k+1)^{3}.
\end{equation}
For each \(i\in[2]\) set \(f(i)\coloneqq x_i\left(\prod_{r\in([2]\setminus\{i\})\cup\mathbb{S}}(1-x_r)\right)\left(\prod_{j\in[2], \tau\in\cT}(1-x^{(j)}_{\tau})\right)\).
Then for each such~\(i\) we have
\begin{eqnarray*}
f(i) & = & \frac{1}{4}(1-(1/2)D^{-3\log^{5/4}D})^{|\mathbb{S}|+2|\cT|}\geq\frac{1}{4}\exp(-D^{-3\log^{5/4}D}(2nD(k+1)^k + 2nD^{\log^{5/4}D}))\\ & \geq & \frac{1}{4}\exp(-3nD^{-2\log^{5/4}D}),
\end{eqnarray*}
where we used Claim~\ref{claim:sizeofs} and~\ref{pseudhyp:maxinvolvement}.
Notice that the inequality \((1/4)\exp(-3nD^{-2\log^{5/4}D})\geq 2\exp(-\theta^2 n/8(k+1)^3)\) is equivalent to \(3nD^{-2\log^{5/4}D}\leq\theta^2 n/(8(k+1)^3)-\log 8\).
The right side of the latter inequality is at least \(\theta^2 n/(100(k+1)^3)\) by~(\ref{eq:implyoldN3}),\COMMENT{\(\log 8\leq 3\leq \theta^2 n/(10(k+1)^3)\) then \((\theta^2 n/(k+1)^3) \cdot (1/8 - 1/10)\geq \theta^2 n/(100(k+1)^3)\).} so the inequality holds since \(\theta\geq1/D\) by~\ref{hypothesis:degratio}.
Thus \(f(1)\geq2\exp(-\theta^2 n/8(k+1)^3)\geq\prob{A_1}\) by Lemma~\ref{lemma:concanalysis}\ref{concleftover}.
One similarly\COMMENT{\(\frac{1}{4}\exp(-3nD^{-2\log^{5/4}D})\geq 2\exp(-\eps\theta n/6) \Leftrightarrow 3nD^{-2\log^{5/4}D}\leq \eps\theta n/6 - \log 8\), but by~(\ref{eq:implyoldN3}) we have \(\eps\theta n/10 \geq 3(k+1)^3 \geq 3\geq \log 8\) so \(\eps\theta n/6 - \log 8 \geq \eps\theta n/20\), so it suffices if \(3nD^{-2\log^{5/4}D}\leq \eps\theta n/20\), but this is now clear since the \(n\)'s cancel and \(\eps,\theta\geq 1/D\). Interestingly, this is where it became important that we collected the \(v\) and \(Z\) into the sets \(S\in\mathbb{S}\), since the means the power of \(n\) on the LHS stays at 1, so the \(n\)'s cancel. Without doing this, the expression on the left could look more like \(n^k D^{-2\log^{5/4}D}\), and then we'd need some density assumption as the \(n\)'s don't cancel and this would rearrange to some \(D\geq g(n)\).} uses \(\eps,\theta\geq 1/D\),~(\ref{eq:implyoldN3}), and Lemma~\ref{lemma:concanalysis}\ref{concwaste} to show that \(f(2)\geq2\exp(-\eps\theta n/6)\geq\prob{A_2}\).

For \(S\in\mathbb{S}\), Lemma~\ref{lemma:concanalysis}\ref{concvertexdegree}--\ref{conccodegree} and the union bound yield that \(\prob{A^{*}_S}\leq 2^{k+1}D^{-4\log^{5/4}D}\).
Define \(N_{1}(S)\coloneqq \{i\in\{1,2\}\cup\mathbb{S}\colon a_S a_i\in E(G)\}\) and \(N_{2}(S)\coloneqq \{(\tau,j)\in\cT\times\{1,2\}\colon a_S a^{(j)}_{\tau}\in E(G)\}\).
Define \(f(S)\coloneqq x_S \left(\prod_{i\in N_{1}(S)}(1-x_i)\right)\left(\prod_{(\tau,j)\in N_{2}(S)}(1-x^{(j)}_{\tau})\right)\).
Then Claim~\ref{claim:dependencygraph} implies that
\begin{eqnarray*}
f(S) & \geq & \frac{1}{8}D^{-3\log^{5/4}D}(1-(1/2)D^{-3\log^{5/4}D})^{2D^{2\log^{5/4}D}}\geq\frac{1}{8}D^{-3\log^{5/4}D}\exp(-2D^{-\log^{5/4}D})\\ & \geq & \frac{1}{16}D^{-3\log^{5/4}D}\geq\prob{A^{*}_S}.
\end{eqnarray*}
Finally, fix \(\tau\in\cT\) and \(j\in\{1,2\}\).
Define \(N_{1}(\tau,j)\coloneqq \{i\in\{1,2\}\cup\mathbb{S}\colon a^{(j)}_{\tau} a_i\in E(G)\}\) and \(N_{2}(\tau,j)\coloneqq\{(\tau',j')\in\cT\times\{1,2\}\colon a^{(j)}_{\tau} a^{(j')}_{\tau'}\in E(G)\}\).
It remains to analyze \(f(\tau,j)\coloneqq x^{(j)}_{\tau} \left(\prod_{i\in N_{1}(\tau,j)}(1-x_i)\right)\left(\prod_{(\tau',j')\in N_{2}(\tau,j)}(1-x^{(j')}_{\tau'})\right)\).
By Claim~\ref{claim:dependencygraph}, we have
\[
f(\tau,j)\geq\frac{1}{8}D^{-3\log^{5/4}D}(1-(1/2)D^{-3\log^{5/4}D})^{2D^{3\log^{5/4}D}}\geq\frac{1}{8}D^{-3\log^{5/4}D}\exp(-2),
\]
which is evidently at least \(D^{-4\log^{5/4}D}\geq\prob{A_{\tau}^{(1)}}\) by Lemma~\ref{lemma:concanalysis}\ref{conctestpseudo}.
By Lemma~\ref{lemma:concanalysis}\ref{conctestwaste}, we also have \(f(\tau,2)\geq(1/8e^2)D^{-3\log^{5/4}D}\geq D^{-4\log^{5/4}D}\geq \prob{A_{\tau}^{(2)}}\).
Lemma~\ref{lemma:LLL} now completes the proof.
\endproof
It remains to prove Lemma~\ref{lemma:concanalysis}.
In Section~\ref{section:expanalysis}, we compute the expectations of the random variables~\(\textbf{deg}\mathbf{'}(v)\) and~\(\textbf{codeg}\mathbf{'}(Z)\) (see Lemma~\ref{lemma:expectations}), before using Lemmas~\ref{chernoff},~\ref{lemma:martingale},~\ref{lemma:expectations} and Theorem~\ref{theorem:lintal} to prove Lemma~\ref{lemma:concanalysis} in Section~\ref{section:concanalysis}.
\subsection{Expectation analysis}\label{section:expanalysis}
The aim of this subsection is to prove the following lemma, which will be used in the proofs of Lemma~\ref{lemma:concanalysis}\ref{concvertexdegree} and~\ref{conccodegree} in Section~\ref{section:concanalysis}.\COMMENT{Part (ii) of Lemma~\ref{lemma:expectations} is vacuous when \(J^{*}=\emptyset\) (which includes the case of graphs), there are no such \(j,Z\). Part (i) is also very easy in the case of \(k=1\) (graphs), the expectation is just equal to \(\text{deg}(v)(1-p^{*})\).}
\begin{lemma}\label{lemma:expectations}
Let~$H$ satisfy the hypotheses of Lemma~\ref{lemma:nibble}, and apply a random nibble~$(\mathbf{X},\mathbf{W})$ to~$H$.
Then:
\begin{enumerate}[label=\upshape(\roman*)]
\item\label{lemma:degreeexpectation} For all $v\in V(H)$, we have
\[
\expn{\textbf{deg}\mathbf{'}(v)}=\text{deg}(v)((1-p^{*})^{k}\pm 2^{k+3}\eps^{2}\theta^{2}).
\]
\item\label{lemma:codegreeexpectation} For all \(j\in J^{*}\) and \(Z\in\binom{V(H)}{j}\), we have\COMMENT{In fact this holds for any \(j\). It's just we can't concentrate for those \(j\notin J^{*}\).}
\[
\expn{\textbf{codeg}\mathbf{'}(Z)}\leq D_{j}(1-\theta(k+1-j)+\theta^{5/3}).
\]
\end{enumerate}
\end{lemma}
To prove Lemma~\ref{lemma:expectations}, we will use the following lemma, which essentially says that the probability a small set of vertices are all not covered by the matching~\(\mathbf{M}\) is close to what it would be if the events~\(\{u\in V(\mathbf{M})\}\) were independent.
\begin{lemma}\label{lemma:nomatchedvertices}
Let~\(H\) satisfy the hypotheses of Lemma~\ref{lemma:nibble}, and apply a random nibble~\((\mathbf{X},\mathbf{W})\) to~\(H\).
For any $1\leq j \leq k$ and \(Z\in\binom{V(H)}{j}\), we have
\[
\prob{ Z\cap V(\mathbf{M})=\emptyset}=\left(\prod_{u\in Z}(1-p(u))\right)\pm 2^{k+3}\eps^{2}\theta^{2}.
\]
\end{lemma}
Lemma~\ref{lemma:nomatchedvertices} quickly yields Lemma~\ref{lemma:expectations}:
\lateproof{Lemma~\ref{lemma:expectations}}
\newline\ref{lemma:degreeexpectation}: Fix \(v\in V(H)\).
We have
\begin{eqnarray*}
\expn{\textbf{deg}\mathbf{'}(v)} & = & \sum_{e\ni v}\prob{(e\setminus\{v\})\cap V(\mathbf{M})=\emptyset}\cdot\prob{(e\setminus\{v\})\cap\mathbf{W}=\emptyset} \\ & \stackrel{\text{(Lemma~\ref{lemma:nomatchedvertices})}}{=} & \sum_{e\ni v}\left(\left(\left(\prod_{u\in e\setminus\{v\}}(1-p(u))\right)\pm 2^{k+3}\eps^{2}\theta^{2}\right)\cdot\prod_{u\in e\setminus\{v\}}(1-w(u))\right) \\ & = & \sum_{e\ni v}\left(\left(\prod_{u\in e\setminus\{v\}}(1-p(u))(1-w(u))\right)\pm 2^{k+3}\eps^{2}\theta^{2}\right) \\ & \stackrel{\text{(\ref{eq:pstar})}}{=} & \text{deg}(v)\left((1-p^{*})^{k}\pm 2^{k+3}\eps^{2}\theta^{2}\right).
\end{eqnarray*}
\ref{lemma:codegreeexpectation}: Fix \(j\in J^{*}\) and \(Z\in \binom{V(H)}{j}\).
As in the proof of~\ref{lemma:degreeexpectation}, we can apply Lemma~\ref{lemma:nomatchedvertices} (with \(e\setminus Z\) playing the role of~\(Z\))\COMMENT{\(j\in J^{*} \implies 2\leq j \leq k \implies 1 \leq |e\setminus Z| \leq k-1\)} to obtain\COMMENT{Here using \(x\coloneqq -p^{*}\), \(n\coloneqq k+1-j\) in the inequality \((1+x)^n\leq e^{xn}\leq 1 + xn + x^2n^2\), which requires \(xn<1.79\). We can use \(\eps\leq 1\) and see the smallest non-leading power of \(\theta\) is \(\theta^{7/4}\). Altogether using \(\theta\ll1/k\) the error inside the bracket is at most \(\theta^{5/3}\).} 
\begin{eqnarray*}
\expn{\textbf{codeg}\mathbf{'}(Z)} & = & \text{codeg}(Z)\left((1-p^{*})^{k+1-j} \pm 2^{k+3}\eps^{2}\theta^{2}\right) \\ & \leq & D_{j}(1-(k+1-j)(\theta-\theta^{7/4}) + (k+1-j)^{2}(\theta+\theta^{7/4})^{2} + 2^{k+3}\eps^{2}\theta^{2}) \\ & \leq & D_{j}(1-(k+1-j)\theta + \theta^{5/3}),
\end{eqnarray*}
where we used~(\ref{eq:pstarabove}) and~(\ref{eq:pstarbelow}).
\endproof
To prove Lemma~\ref{lemma:nomatchedvertices}, we adapt an argument of Alon, Kim, and Spencer~\cite[Theorem 2.1, Claim 1]{AKS97}, who also analyzed a `non-wasteful' nibble (see Section~\ref{section:sketch} of the current paper), but in the setting of linear hypergraphs (having \(C_2(H)=1\)).
The lack of restrictions on the codegrees in our setting presents an added obstacle.
\lateproof{Lemma~\ref{lemma:nomatchedvertices}}
We begin by finding bounds for the probability that all of a small set of pairwise-disjoint edges are selected for~$\mathbf{M}$.
\begin{claim}\label{claim:intproduct}
Let $y\in[k]$ and suppose that the edges $Y\coloneqq\{e_{1},e_{2},\dots,e_{y}\}\subseteq E(H)$ are pairwise-disjoint. Then
\[
\prod_{e\in Y}\prob{e\in\mathbf{M}}\leq\prob{Y\subseteq\mathbf{M}}\leq(1+\eps^{2}\theta^{2})\prod_{e\in Y}\prob{e\in\mathbf{M}}.
\]
\end{claim}
\claimproof
For a set of edges~$R\subseteq E(H)$, define \(F(R)\coloneqq\{f\in E(H) \colon f\cap(\bigcup_{r\in R}r)\neq\emptyset, f\notin R\}\).
Note that
\[
\prob{Y\subseteq\mathbf{M}}=\left(\prod_{e\in Y}\prob{e\in\mathbf{X}}\right)\cdot\left(\prod_{f\in F(Y)}\prob{f\notin \mathbf{X}}\right),
\]
whereas
\[
\prod_{e\in Y}\prob{e\in\mathbf{M}}=\left(\prod_{e\in Y}\prob{e\in\mathbf{X}}\right)\cdot\left(\prod_{e\in Y}\prod_{f\in F(\{e\})}\prob{f\notin\mathbf{X}}\right).
\]
It follows that $\prob{Y\subseteq\mathbf{M}}=\left(\prod_{e\in Y}\prob{e\in\mathbf{M}}\right)\cdot (1-p)^{-\sum_{f\in F(Y)}(\text{mult}(f)-1)}$, where~$\text{mult}(f)$ is the number of $e\in Y$ for which $f\in F(\{e\})$.
The lower bound of the claim now follows from the fact that $\text{mult}(f)\geq 1$ for all $f\in F(Y)$.
To obtain the upper bound, observe that $\sum_{f\in F(Y)}(\text{mult}(f)-1)\leq\binom{y}{2}(k+1)^{2}D_{2}$.
Indeed, the right side is an upper bound for \(A\coloneqq\sum \text{codeg}(P)\), where the sum is over unordered pairs~\(P\) of vertices from different \(Y\)-edges.
An edge \(f\in F(Y)\) with \(\text{mult}(f)=1\) doesn't contribute to~\(A\) nor to \(B\coloneqq\sum_{f\in F(Y)}(\text{mult}(f)-1)\).
If \(\text{mult}(f)=\ell\geq 2\) then~\(f\) contributes~\(\ell-1\) to~\(B\) and at least\COMMENT{Exactly \(\binom{\ell}{2}\) if \(f\) intersects each its \(\ell\) (pairwise-disjoint) \(Y\)-edges in 1 vertex. More if it intersects any of these \(\ell\) edges in more than 1 vertex.} \(\binom{\ell}{2}\geq\ell-1\) to~\(A\).
By hypothesis~\ref{hypothesis:degratio} we have\COMMENT{The first inequality uses \(1+x+x^2\geq e^x\) for \(x\coloneqq -3p/2<1.79\) (say), \textcolor{blue}{and uses $\delta(H)\geq (1-\eps)D\geq D/2$ by $\eps\leq 1/2$}. The second inequality uses $e^{x}\leq 1+x+x^{2}$ for $x<1.79$, which holds for $x\coloneqq 2\binom{y}{2}(k+1)^{2}\theta\cdot\eps^{2}\theta/\log^{5}D$ \textcolor{blue}{by, say, $\eps,\theta\leq 1$ and $\log^{5}D\geq 2\binom{k}{2}(k+1)^{2}$}. Then $1+2\binom{y}{2}(k+1)^{2}\eps^{2}\theta^{2}/\log^{5}D + 4\binom{y}{2}^{2}(k+1)^{4}\eps^{4}\theta^{4}/\log^{10}D \leq 1+\eps^{2}\theta^{2}$, provided \textcolor{blue}{say, $\log^{5}D\geq 4\binom{k}{2}(k+1)^{2}$ and $\eps,\theta\leq 1$, $\log^{10}D\geq 8\binom{k}{2}^{2}(k+1)^{4}$}.} $(1-p)^{-\binom{y}{2}(k+1)^{2}D_{2}}\leq\exp(2\binom{y}{2}(k+1)^{2}\theta D_{2}/D)\leq1+\eps^{2}\theta^{2}$, completing the proof of the claim.
\endclaimproof
For a set $A\subseteq V(H)$, we define an $A$-\textit{cover of size}~$j$ to be a set of~$j$ pairwise-disjoint edges~$\{e_{1},e_{2},\dots,e_{j}\}\subseteq E(H)$ such that $e_{i}\cap A\neq\emptyset$ for all $i\in[j]$ and $A\setminus\bigcup_{i\in[j]}e_{i}=\emptyset$.
Write~$\cC(A,j)$ for the set of $A$-covers of size~$j$.
\begin{claim}\label{claim:acovers}
Let $A\subseteq V(H)$ have size $|A|=r$ for $r\in[k]$.
Then $|\cC(A,j)|\leq\eps^{2}\theta D^{j}$ for each $j\in [r-1]$.
\end{claim}
\claimproof
The claim holds vacuously if $r=1$, so suppose $2\leq r\leq k$ and fix $j\in[r-1]$.
We describe a process that can output any $A$-cover of size~$j$, then count the number of ways to run the process.
Firstly select a pair~$B_{0}$ of vertices of~$A$, then partition the remainder of~$A$ into sets $B_{1},B_{2},\dots,B_{j}$, where only~$B_{1}$ is permitted to be empty.
Finally, select~$j$ pairwise-disjoint edges $e_{1},e_{2},\dots,e_{j}$ of~$H$ such that $B_{0}\cup B_{1}\subseteq e_{1}$ and $B_{i}\subseteq e_{i}$ for $2\leq i\leq j$ (if such edges exist, otherwise say the process fails).
Since $j<r$, for any $A$-cover of size~$j$ there must be at least some pair of vertices in~$A$ contained within one edge of the $A$-cover, with~$B_{0}$ playing the role of this pair in the above process.
It follows that for any $A$-cover of size~$j$, there is at least one way to run the process to obtain that $A$-cover.
Further, any successful outcome $\{e_{1},\dots,e_{j}\}$ of the process is an $A$-cover of size~$j$.
When running the process, there are~$\binom{r}{2}\leq\binom{k}{2}$ ways to select~$B_{0}$, at most~$(k+1)^{2k}$ ways\COMMENT{Any valid partition $(B_{1},\dots,B_{j})$ of $A\setminus B_{0}$ can be obtained (possibly more than once) by the following procedure: Apply a permutation to $A\setminus B_{0}$. Then choose $B_{1}$ to be the first $t_{1}$ vertices of $A\setminus B_{0}$ under this permutation, where $0\leq t_{1}\leq r-2$. Then choose $B_{2}$ to be the first $t_{2}$ vertices of $A\setminus (B_{0}\cup B_{1})$, where $1\leq t_{2}\leq r-2-t_{1}$. Etc. The procedure can be performed in at most $(r-2)!(r+1)^{j}\leq k!(k+1)^{k}\leq k^{k}(k+1)^{k}\leq (k+1)^{2k}$ ways.} to perform the partitioning of~$A\setminus B_{0}$, at most~$D_{2}$ ways to select~$e_{1}$, and\COMMENT{\textcolor{blue}{Obviously $\eps\leq 1$.}} at most~$\Delta(H)\leq 2D$ ways to choose~$e_{i}$, for each $2\leq i\leq j$.
By~\ref{hypothesis:degratio}, we deduce\COMMENT{Consider a bipartite graph with parts $\cC(A,j)$ and $\cP$, where $\cP$ is the set of all ways to successfully run the process (that is, to make all the choices $B_{0},\dots,B_{j},e_{1}\dots,e_{j}$), with $CP$ an edge (for $C\in\cC(A,j)$, $P\in\cP$) iff $C=\{e_{1},\dots,e_{j}\}$. By double-counting the edges we have $|\cC(A,j)|\leq |\cP|(\max_{P\in\cP}d(P))/(\min_{C\in\cC(A,j)}d(C))$. In the main text we argued that $d(P)=1$ for all $P\in\cP$ and $d(C)\geq 1$ for all $C\in\cC(A,j)$, whence $|\cC(A,j)|\leq|\cP|$.} that\COMMENT{\textcolor{blue}{Used $\log^{5}D\geq 2^{k}(k+1)^{2k+2}$.}} $|\cC(A,j)|\leq\binom{k}{2}(k+1)^{2k}D_{2}\cdot 2^{j-1}D^{j-1}\leq 2^{k}(k+1)^{2k+2}(D_{2}/D)\cdot D^{j}\leq \eps^{2}\theta D^{j}$, completing the proof of the claim.
\endclaimproof
Next we use Claims~\ref{claim:intproduct} and~\ref{claim:acovers} to show that the probability that all of a small set of vertices are in $V(\mathbf{M})$ is close to what it would be if the events~$\{u\in V(\mathbf{M})\}$ were independent for $u\in V(H)$.
Recall that~$p(u)\coloneqq\prob{u\in V(\mathbf{M})}$.
\begin{claim}\label{claim:nearlyindep}
Let $A\subseteq V(H)$ have size $|A|=r$, where $r\in[k]$.
Then
\[
\prob{A\subseteq V(\mathbf{M})}=\left(\prod_{u\in A}p(u)\right)\pm 5\eps^{2}\theta^{2}.
\]
\end{claim}
\claimproof
The statement is clear if $r=1$, so suppose $r,k\geq 2$.
We show the upper bound first.
By Claim~\ref{claim:intproduct} we have
\begin{equation}\label{eq:mostupperbound}
\prob{A\subseteq V(\mathbf{M})} = \sum_{j=1}^{r}\sum_{B\in\cC(A,j)}\prob{B\subseteq \mathbf{M}}\leq (1+\eps^{2}\theta^{2})\sum_{j=1}^{r}\sum_{B\in\cC(A,j)}\prod_{e\in B}\prob{e\in\mathbf{M}}.
\end{equation}
Let $A=\{u_{1},u_{2},\dots,u_{r}\}$, define $\partial A\coloneqq \partial(u_{1})\times\partial(u_{2})\times\dots\times\partial(u_{r})$ and let~$\partial^{*} A$ be the set of those tuples $(e_{1},\dots,e_{r})\in\partial A$ for which $e_{1},\dots,e_{r}$ are pairwise-disjoint.
We can split the double sum in~(\ref{eq:mostupperbound}) into the sum over \(A\)-covers of size~\(r\), and the sum over all smaller \(A\)-covers.
By Claim~\ref{claim:acovers}, the latter sum is at most\COMMENT{For the final inequality, the \(j\) term is at most \(\eps^2\theta \cdot2\theta^j\), so the largest term is the \(j=1\) term \(2\eps^2\theta^2\) and all the others are asymptotically smaller, using \(r\leq k\) and \textcolor{blue}{\(\theta\ll 1/k\)}.} \(\sum_{j=1}^{r-1}|\cC(A,j)|p^j \leq \sum_{j=1}^{r-1}\eps^{2}\theta D^{j}p^{j} \leq 3\eps^2\theta^2\).
Further,
\begin{eqnarray}\label{eq:sumprodprodsum}
\sum_{B\in\cC(A,r)}\prod_{e\in B}\prob{e\in\mathbf{M}} & = & \sum_{(e_1,e_2,\dots,e_r)\in\partial^{*}A}\prod_{i\in[r]}\prob{e_i \in\mathbf{M}} \leq \sum_{(e_1,e_2,\dots,e_r)\in\partial A}\prod_{i\in[r]}\prob{e_i \in\mathbf{M}} \nonumber \\ & = & \prod_{u\in A}\sum_{e\ni u}\prob{e\in\mathbf{M}}=\prod_{u\in A}\prob{u\in V(\mathbf{M})},
\end{eqnarray}
where the penultimate equality can be most easily seen by `expanding the brackets' in the expression \(\prod_{u\in A}\sum_{e\ni u}\prob{e\in\mathbf{M}}\).
Altogether, we have
\[
\prob{A\subseteq V(\mathbf{M})} \leq (1+\eps^{2}\theta^{2})\left(\left(\prod_{u\in A}\prob{u\in V(\mathbf{M})}\right) + 3\eps^{2}\theta^{2}\right),
\]
and the upper bound of the claim follows.

To see the lower bound holds, one can apply Claim~\ref{claim:intproduct} as in~(\ref{eq:mostupperbound}) (and simply ignore the terms \(j<r\)) to obtain
\begin{equation}\label{eq:mainproblow}
\prob{A\subseteq V(\mathbf{M})}\geq \sum_{(e_{1},\dots,e_{r})\in\partial^{*} A}\prod_{i=1}^{r}\prob{e_{i}\in\mathbf{M}}.
\end{equation}
By~(\ref{eq:sumprodprodsum}) and~(\ref{eq:mainproblow}), all that remains\COMMENT{The RHS of~(\ref{eq:mainproblow}) equals $\sum_{(e_{1},\dots,e_{r})\in\partial A}\prod_{i=1}^{r}\prob{e_{i}\in\mathbf{M}}-\sum_{(e_{1},\dots,e_{r})\in\partial A\setminus \partial^{*}A}\prod_{i=1}^{r}\prob{e_{i}\in\mathbf{M}}$, which by~(\ref{eq:sumprodprodsum}) is equal to $\left(\prod_{u\in A}p(u)\right)- \sum_{(e_{1},\dots,e_{r})\in\partial A\setminus \partial^{*}A}\prod_{i=1}^{r}\prob{e_{i}\in\mathbf{M}}$. Compare this with the claim.} is to check $\sum_{(e_{1},\dots,e_{r})\in\partial A\setminus\partial^{*}A}\prod_{i=1}^{r}\prob{e_{i}\in\mathbf{M}}\leq5\eps^{2}\theta^{2}$.
To that end, note that since any $r$-tuple in~$\partial A\setminus\partial^{*}A$ must contain a pair of intersecting edges, it follows\COMMENT{There are~$\binom{r}{2}\leq k^{2}/2$ ways to choose a distinct pair $i,j\in[r]$ such that $e_{i}$ and $e_{j}$ will intersect. Choose one of the two, say $i$. Perhaps $e_{i}$ contains $u_{j}$, in which case there are at most $D_{2}$ choices for $e_{i}$ and $\Delta(H)^{r-1}$ choices for the other edges. Otherwise there are at most $\Delta(H)$ choices for an $e_{i}$ that does not contain $u_{j}$. Then there are $k+1$ choices of a vertex $z\in e_{i}$ such that $e_{j}$ will contain both $u_{j}$ and $z$ (and we have ensured that $z$ is not $u_{j}$). So there are now at most $D_{2}$ choices for $e_{j}$ and $\Delta(H)^{r-2}$ choices for the remaining edges. Crunching the numbers, $(k^{2}/2)\cdot 2(D_{2}\Delta(H)^{r-1}+\Delta(H)(k+1)D_{2}\Delta(H)^{r-2})=k^{2}(k+2)D_{2}\Delta(H)^{r-1}\leq 2k^{2}(k+2)D_{2}D^{r-1}$ by \textcolor{blue}{$(1+\eps)^{r-1}\leq (1+\eps)^{k}\leq2$}.} that $|\partial A\setminus\partial^{*}A|\leq2k^{2}(k+2)D_{2}D^{r-1}$, which, together with~\ref{hypothesis:degratio} and the fact that $\prob{e\in\mathbf{M}}\leq p$ for all $e\in E(H)$, finishes\COMMENT{$\sum_{(e_{1},\dots,e_{r})\in\partial A\setminus\partial^{*}A}\prod_{i=1}^{r}\prob{e_{i}\in\mathbf{M}}\leq |\partial A\setminus\partial^{*}A|p^{r}\leq 2k^{2}(k+2)D_{2}D^{r-1}p^{r}=2k^{2}(k+2)(D_{2}/D)D^{r}p^{r}\leq 4k^{2}(k+2)\eps^{2}\theta^{1+r}/\log^{5}D$, where we \textcolor{blue}{used $(1-\eps)^{r}\geq 1/2$}. Then \textcolor{blue}{use $\log^{5}D\geq 4k^{2}(k+2)$, $\theta\leq 1$}, $r\geq 1$.} the proof of the claim.
\endclaimproof
We can now use Inclusion-Exclusion to finish the proof of the lemma.
To that end, fix \(1\leq j\leq k\) and fix \(Z\in\binom{V(H)}{j}\).
Let~$\pr^{*}$ denote the measure for the probability space in which the vertices~$u$ of~$H$ are chosen independently for a set~$\mathbf{Y}$ with probability $\probstar{u\in\mathbf{Y}}\coloneqq\prob{u\in V(\mathbf{M})}$.
For $r\in[j]$ define $\cA_{r}\coloneqq\binom{Z}{r}$.
Note that\COMMENT{From line 2 to 3, the number of times we must count the error $\pm5\eps^{2}\theta^{2}$ is equal to the number of subsets of a \(j\)-set, which is \(2^j \leq 2^k\), and clearly $2^{k}\cdot 5\eps^{2}\theta^{2}\leq 2^{k+3}\eps^{2}\theta^{2}$.}
\begin{eqnarray*}
\prob{Z\cap V(\mathbf{M})=\emptyset} & \stackrel{\text{(I.-E.)}}{=} & 1-\sum_{r=1}^{j}(-1)^{r+1}\sum_{A\in\cA_{r}}\prob{A\subseteq V(\mathbf{M})} \\ & \stackrel{\text{(Claim~\ref{claim:nearlyindep})}}{=} & 1-\sum_{r=1}^{j}(-1)^{r+1}\sum_{A\in\cA_{r}}\left(\left(\prod_{u\in A}p(u)\right)\pm 5\eps^{2}\theta^{2}\right) \\ & = & \left(1-\sum_{r=1}^{j}(-1)^{r+1}\sum_{A\in\cA_{r}}\probstar{A\subseteq\mathbf{Y}}\right)\pm 2^{k+3}\eps^{2}\theta^{2} \\ & \stackrel{\text{(I.-E.)}}{=} & \probstar{Z\cap\mathbf{Y}=\emptyset}\pm 2^{k+3}\eps^{2}\theta^{2},
\end{eqnarray*}
which is $\left(\prod_{u\in Z}(1-p(u))\right)\pm 2^{k+3}\eps^{2}\theta^{2}$, as required.
\endproof
\subsection{Concentration analysis}\label{section:concanalysis}
In this subsection, we use the concentration inequalities from Section~\ref{section:concineqs} to prove each part of Lemma~\ref{lemma:concanalysis} in turn.
We begin by using Lemma~\ref{lemma:martingale} to analyze the concentration of \(\mathbf{n'}=|V(\mathbf{H'})|\), the size of the leftover vertex set.
\lateproof{Lemma~\ref{lemma:concanalysis}\ref{concleftover}}
Note that by~(\ref{eq:pstar}) we have $\expn{\mathbf{n'}}=n(1-p^{*})$, whence by~(\ref{eq:pstarabove}) and~(\ref{eq:pstarbelow}) we have
\begin{equation}\label{eq:expectm'}
n(1-\theta-\theta^{7/4})\leq\expn{\mathbf{n'}}\leq n(1-\theta+\theta^{7/4}).
\end{equation}
We now seek to apply Lemma~\ref{lemma:martingale} to show that~$\mathbf{n'}$ is tightly concentrated around its expectation.
To that end, note that~$\mathbf{n'}$ is a function of the mutually independent indicator random variables $\{\mathbf{t}_{u}\colon u\in V(H)\}\cup\{\mathbf{t}_{e}\colon e\in E(H)\}$, where~\(\mathbf{t}_u\) (respectively~\(\mathbf{t}_e\)) is the indicator for the event~\(\{u\in\mathbf{W}\}\) (\(\{e\in\mathbf{X}\}\)).
It is clear that each~$\mathbf{t}_{u}$ affects~$\mathbf{n'}$ by at most~$1$.
We claim that each~$\mathbf{t}_{e}$ affects~$\mathbf{n'}$ by at most~$(k+1)^{2}$.
Indeed, fix $e\in E(H)$ and consider two nibbles (assignments of values to all~$\mathbf{t}_{u}$ and all~$\mathbf{t}_{f}$ for $u\in V(H)$, $f\in E(H)$), differing only in the value of~$\mathbf{t}_{e}$.
Let~$\text{Nibble}_{0}$ and~$\text{Nibble}_{1}$ be the nibbles in which $\mathbf{t}_{e}=0$ and $\mathbf{t}_{e}=1$ respectively, set~$n_{0}$ to be the value of $|V(H)\setminus(V(\mathbf{M})\cup\mathbf{W})|$ in~$\text{Nibble}_{0}$ and set~$n_{1}$ to be $|V(H)\setminus(V(\mathbf{M})\cup\mathbf{W})|$ in~$\text{Nibble}_{1}$.
If~$e$ is isolated in~$\text{Nibble}_{1}$ then $n_{1}=n_{0}-b$, where $0\leq b\leq k+1$ is the number of vertices~$u$ in~$e$ for which $\mathbf{t}_{u}=0$.
If~$e$ is not isolated in~$\text{Nibble}_{1}$ then $n_{1}=n_{0}+b'$, where~$0\leq b' \leq (k+1)^{2}$ is the number of vertices~$u$ in edges~$e'$ intersecting~$e$ for which~$e'$ is isolated in~$\text{Nibble}_{0}$ and $\mathbf{t}_{u}=0$ (each vertex of~$e$ can only contribute at most one such isolated~$e'$).
It follows that~$\mathbf{t}_{e}$ affects~$\mathbf{n'}$ by at most~$(k+1)^{2}$, as claimed.
Note that $\sum_{u\in V(H)}w(u)(1-w(u))\cdot1^{2} +\sum_{e\in E(H)}p(1-p)(k+1)^{4} \leq 8\eps\theta n+e(H)p(k+1)^{4}\leq 8\eps\theta n+\frac{n\Delta(H)}{k+1}\cdot\frac{\theta}{\delta(H)}(k+1)^{4}\leq 8\eps\theta n +(1+3\eps)(k+1)^{3}n\theta \leq 2(k+1)^{3}n\theta$, where\COMMENT{\textcolor{blue}{Used say $\eps\leq 1/6$} so that $(1+3\eps)(k+1)^{3}n\theta\leq \frac{3}{2}(k+1)^{3}n\theta$. Also we have $8\eps\theta n\leq \frac{1}{2}(k+1)^{3}n\theta$ \textcolor{blue}{provided, say $\eps\leq (k+1)^{3}/16$}.} we have used~(\ref{eq:wvabove}).
We apply\COMMENT{We need $\alpha<2\sigma/C$. With our definitions of these parameters, this requires $\theta\sqrt{n/2(k+1)^{3}}<2\sqrt{2n\theta/(k+1)}$, which is equivalent to $\theta<16(k+1)^{2}$, so it's fine.} Lemma~\ref{lemma:martingale} with $C\coloneqq(k+1)^{2}$, $\sigma\coloneqq\sqrt{2(k+1)^{3}n\theta}$, and $\alpha\coloneqq\theta\sqrt{\frac{n}{2(k+1)^{3}}}$, obtaining
\begin{equation}\label{eq:concentratem'}
\prob{|\mathbf{n'}-\expn{\mathbf{n'}}|>n\theta^{3/2}}\leq2\exp\left(-\frac{\theta^{2}n}{8(k+1)^{3}}\right).
\end{equation}
Combining~(\ref{eq:expectm'}) and~(\ref{eq:concentratem'}), we obtain\COMMENT{Here simply using $n\theta^{7/4}\leq n\theta^{3/2}$ which holds for any \textcolor{blue}{$0<\theta<1$}.} $\prob{A_1}\leq2\exp(-\theta^{2}n/8(k+1)^{3})$, as required.
\endproof
Next, we prove Lemma~\ref{lemma:concanalysis}\ref{concwaste}, which simply says it is unlikely that~\(\mathbf{W}\) is too large.
\lateproof{Lemma~\ref{lemma:concanalysis}\ref{concwaste}}
Recall that vertices $v\in V(H)$ are chosen independently for~$\mathbf{W}$ with probability~$w(v)$.
Using~(\ref{eq:wvabove}), we obtain $\prob{A_2}\leq \prob{\text{Bin}(n,8\eps\theta)\geq10\eps\theta n}\) from a simple coupling argument\COMMENT{Suppose that $W'$ is constructed by selecting every vertex $v\in V(H)$ independently for $W'$ with probability $8\eps\theta$. We couple $W$ and $W'$ as follows: For every vertex $v\in V(H)$, couple the random variables $\mathds{1}_{v\in W}$ and $\mathds{1}_{v\in W'}$ by ensuring that if a vertex is in~$W$, then it is in~$W'$ (possible by~(\ref{eq:wvabove}), and we may couple independently like this because of the mutual independence of~$\{\mathds{1}_{v\in W}\}_{v\in V(H)}$ and~$\{\mathds{1}_{v\in W'}\}_{v\in V(H)}$). Then clearly $|W|\leq |W'|$ in every outcome (of all assignments of $\mathds{1}_{v\in W}$ and $\mathds{1}_{v\in W'}$), and $|W'|\sim\text{Bin}(n,8\eps\theta)$.}.
It follows by Lemma~\ref{chernoff}\ref{lemma:chernoffbigexp}\COMMENT{Expectation is $8\eps\theta n$, set the $\eps$ in statement of Lemma~\ref{chernoff}\ref{lemma:chernoffbigexp} to be $1/4$.} that $\prob{A_2}\leq 2\exp(-\eps\theta n/6)$, as required.
\endproof
The proofs of Lemma~\ref{lemma:concanalysis}\ref{concvertexdegree}--\ref{conctestwaste} all use Theorem~\ref{theorem:lintal} as the central tool.
We describe our approach now.
The four proofs are somewhat similar, so we focus our explanation on the concentration of~\(\textbf{deg}\mathbf{'}(v)\).
For the remainder of the section, we say that a vertex \(u\in V(H)\) is \(\mathbf{X}\)-\textit{covered} if \(u\in e\) for some \(\mathbf{X}\)-edge~\(e\), or equivalently \(u\in V(\mathbf{X})\).
We say \(u\in V(H)\) is \textit{resurrected} if \(u\in e\) for some non-isolated \(\mathbf{X}\)-edge~\(e\) (recall an \(\mathbf{X}\)-edge~\(e\) is isolated if it does not intersect any other \(\mathbf{X}\)-edge~\(e'\)).

Unfortunately, the random variable~\(\textbf{deg}\mathbf{'}(v)\) is not directly suitable for an application of a Talagrand-type concentration inequality; indeed, it will take too many `witnesses' (too large a value of~\(r\) in Theorem~\ref{theorem:lintal}) to certify that an edge~\(e\) containing~\(v\) survives the nibble (in particular, one might need to see that every edge~\(f\) intersecting~\(e\setminus\{v\}\) is not an \(\mathbf{X}\)-edge).
The random variable \((\text{deg}(v) - \textbf{deg}\mathbf{'}(v))\) is also not suitable; to certify an edge \(e\ni v\) does not survive the nibble, it does not suffice to show that some edge~\(f\) intersecting~\(e\setminus\{v\}\) was put in~\(\mathbf{X}\), since we would still need to know that none of the other edges intersecting~\(f\) were put in~\(\mathbf{X}\) (otherwise any vertex in \(f\cap e\) is resurrected and may survive the nibble).
However, we will show that~\(\textbf{deg}\mathbf{'}(v)\) can be written as the combination of (constantly many) random variables we call~\(\mathbf{f}(a,b,c)\), each of which is \((r,d)\)-observable for suitably small~\(r\).
More specifically, we set~\(\mathbf{f}(a,b,c)\) to be the number of edges \(e\ni v\) for which at least~\(a\) elements of~\(e\setminus\{v\}\) are \(\mathbf{X}\)-covered, at least~\(b\) elements of~\(e\setminus\{v\}\) are resurrected, and at least~\(c\) elements of~\(e\setminus\{v\}\) are put in~\(\mathbf{W}\).
Further, we will define the set~\(\Omega^{*}\) of exceptional outcomes in such a way that ensures no edge \(g\in E(H)\) can interfere with~\(\mathbf{f}(a,b,c)\) too much in any outcome outside of~\(\Omega^{*}\), which is needed in order to apply Theorem~\ref{theorem:lintal} with an acceptably small value of~\(d\) (we discuss this more concretely after the proof of Lemma~\ref{lemma:concanalysis}\ref{concvertexdegree}).
Once \((r,d)\)-observability of each~\(\mathbf{f}(a,b,c)\) has been established for appropriate~\(r,d\), we show that~\ref{hypothesis:degratio} ensures that Theorem~\ref{theorem:lintal} can be applied to obtain error on the order~\(\eps\theta D\).~\ref{hypothesis:codegratio} (respectively~\ref{pseudhyp:lowerbound}) performs the corresponding role in the proof of Lemma~\ref{lemma:concanalysis}\ref{conccodegree} (respectively Lemma~\ref{lemma:concanalysis}\ref{conctestpseudo}--\ref{conctestwaste}).

Delcourt and Postle~\cite{DP24} acknowledge this strategy of splitting a random variable which is not suitable for a Talagrand-type concentration inequality into a combination of random variables which are, and even have a version of Theorem~\ref{theorem:lintal} which `builds in' this combination (specifically,~\cite[Theorem 4.5]{DP24}).
For an earlier example of this kind of approach, see~\cite[Lemma 5(b)]{MR00}.
\lateproof{Lemma~\ref{lemma:concanalysis}\ref{concvertexdegree}}
Clearly we can identify the random nibble~\((\mathbf{X},\mathbf{W})\) with the product space~\((\Omega,\Sigma,\mathbb{P})\) of the spaces \(((\Omega_{x},\Sigma_{x},\mathbb{P}_{x}))_{x\in V(H)\cup E(H)}\).
Fix \(v\in V(H)\).
Define \(N^{(3)}(v)\coloneqq\{u\in V(H)\colon\text{dist}_{H}(v,u)\leq 3\}\).
We define the set~\(\Omega^{*}\subseteq\Omega\) of exceptional outcomes as follows:
\[
\Omega^{*}\coloneqq\{\omega\in\Omega\colon \exists u\in N^{(3)}(v)\,\, \text{such that}\, \,|\partial(u)\cap\mathbf{X}|>\log^{5/2}D\}.
\]
In particular, a simple application\COMMENT{Fix \(u\in V(H)\). Note \(|\partial(u)\cap\mathbf{X}|\sim\text{Bin}(\text{deg}(u),p)\). We have \(\mu\coloneqq \expn{|\partial(u)\cap\mathbf{X}|} = \frac{\theta\text{deg}(u)}{\delta(H)}(\leq 2\theta\leq1)\). Apply Lemma~\ref{chernoff}\ref{lemma:chernoffsmallexp} with \(\beta\coloneqq\frac{\log^{5/2}D}{\mu}(\geq\log^{5/2}D>1)\) to obtain
\begin{eqnarray*}
\prob{|\partial(u)\cap\mathbf{X}|>\log^{5/2}D} & \leq & \left(\frac{\delta(H)\log^{5/2}D}{e\theta\text{deg}(u)}\right)^{-\log^{5/2}D}=\left(\frac{e\theta\text{deg}(u)}{\delta(H)\log^{5/2}D}\right)^{\log^{5/2}D}\leq e^{-\log^{5/2}D} \\ & = & D^{-\log^{3/2} D},
\end{eqnarray*}
where we used \(e\theta\text{deg}(u)/(\delta(H)\log^{5/2}D) \leq 1/e\), which follows from (for example) \(\theta\leq 1/(2e)\) and \(\log^{5/2}D\geq e\).} of Lemma~\ref{chernoff}\ref{lemma:chernoffsmallexp} shows that\COMMENT{1 vertex at distance 0 from \(v\), at most \(\Delta(H)k\leq 2kD\) at distance 1, \((2kD)^2\) at distance 2, \((2kD)^3\) at distance 3. \(1+2kD+(2kD)^2+(2kD)^3=(2kD)^3(1+1/(2kD)+1/(2kD)^2+1/(2kD)^3)\leq 8k^3D^3(1+1/2+1/4+1/8)\leq16k^3D^3\)}
\begin{equation}\label{eq:theexceptionaloutcomes}
\prob{\Omega^{*}}\leq |N^{(3)}(v)|D^{-\log^{3/2} D} \leq 16k^3D^{3-\log^{3/2} D}\leq D^{-\frac{1}{2}\log^{3/2} D}.
\end{equation}
For integers \(0\leq b\leq a\leq k\) and \(0\leq c\leq k\), we define~\(\mathbf{f}(a,b,c)\) to be the number of edges \(e\ni v\) such that at least~\(a\) of the vertices~\(e\setminus\{v\}\) are \(\mathbf{X}\)-covered, at least~\(b\) of the vertices~\(e\setminus\{v\}\) are resurrected, and at least~\(c\) of the vertices~\(e\setminus\{v\}\) are put in~\(\mathbf{W}\).
Additionally, set \(\mathbf{f}(k+1,b,c)=0\).
Notice that, for each \(a\in[k]\), the number of edges \(e\ni v\) such that exactly~\(a\) of the vertices~\(e\setminus\{v\}\) are \(\mathbf{X}\)-covered is \(\mathbf{f}(a,0,0)-\mathbf{f}(a+1,0,0)\).
Moreover, the number of these edges that satisfy the additional property that none of the vertices~\(e\setminus\{v\}\) are put in~\(\mathbf{W}\) is \((\mathbf{f}(a,0,0)-\mathbf{f}(a+1,0,0))-(\mathbf{f}(a,0,1)-\mathbf{f}(a+1,0,1))\).
Similarly, the number of edges \(e\ni v\) such that exactly~\(a\) of the vertices~\(e\setminus\{v\}\) are \(\mathbf{X}\)-covered and all of them are resurrected is \((\mathbf{f}(a,a,0)-\mathbf{f}(a+1,a,0))\), and the number of such edges which further satisfy \((e\setminus\{v\})\cap\mathbf{W}=\emptyset\) is \((\mathbf{f}(a,a,0)-\mathbf{f}(a+1,a,0))-(\mathbf{f}(a,a,1)-\mathbf{f}(a+1,a,1))\).
It follows that \(\textbf{deg}\mathbf{'}(v)=\text{deg}(v)-\mathbf{f}(0,0,1)-\sum_{a=1}^{k}\mathbf{N}(a)\), where\COMMENT{The top line terms do a bunch of cancelling out when we sum over~\(a\). The bottom line terms don't. Seems simpler to not bother commenting on this? (Just concentrate each \(\mathbf{f}(a,b,c)\) and we'll get at most the desired total error on \(\textbf{deg}\mathbf{'}(v)\).) Even if one \(\mathbf{f}(a,b,c)\) appears with a coefficient \(\geq 2\) (this doesn't actually happen), you could just replace each instance of it with the (expectation \(\pm\) error) we get from Theorem~\ref{theorem:lintal}, and there are at most \(8k+1\) total instances.}
\begin{eqnarray}\label{eq:Na}
\mathbf{N}(a) = &  & ((\mathbf{f}(a,0,0)-\mathbf{f}(a+1,0,0))-(\mathbf{f}(a,0,1)-\mathbf{f}(a+1,0,1)))\nonumber \\ & - & ((\mathbf{f}(a,a,0)-\mathbf{f}(a+1,a,0))-(\mathbf{f}(a,a,1)-\mathbf{f}(a+1,a,1))),
\end{eqnarray}
since each~\(\mathbf{N}(a)\) counts the number of edges \(e\ni v\) in which none of the vertices~\(e\setminus\{v\}\) are put in~\(\mathbf{W}\), exactly~\(a\) of the vertices~\(e\setminus\{v\}\) are \(\mathbf{X}\)-covered, and not all of those~\(a\) vertices are resurrected, whence at least some vertex of~\(e\setminus\{v\}\) is in an isolated \(\mathbf{X}\)-edge, and so does not survive the nibble.
Let the random variables~\(\mathbf{f}(a,b,c)\) which appear in the above expansion of~\(\textbf{deg}\mathbf{'}(v)\) (excluding the constant random variables \(\mathbf{f}(k+1,b,c)=0\)) be called the \textit{important counters}.

Now fix an important counter\COMMENT{It doesn't matter until we bound the expectation later that \(a,b,c\) is such that~\(\mathbf{f}(a,b,c)\) is an important counter, but it's clean to do it this way and not have to separately appeal later to which choices of \(a,b,c\) we care about.} \(\mathbf{f}(a,b,c)\eqqcolon\mathbf{f}\), and notice that \(a\geq b\).
We claim that~\(\mathbf{f}\) is \((a+b+c,2(k+1)^2D_{2}\log^{5/2}D)\)-observable with respect to~\(\Omega^{*}\).
Indeed, fix an outcome \(\omega\in\Omega\setminus\Omega^{*}\) and let~\(s\) be the value of~\(\mathbf{f}\) in~\(\omega\).
By the definition of~\(\mathbf{f}\), there are distinct edges \(\{e_1,e_2,\dots,e_s\}\subseteq\partial_{H}(v)\) such that each~\(e_i\) contains a set \(U_i=\{u_{i,1},u_{i,2},\dots,u_{i,b}\}\) of resurrected non-\(v\) vertices (there are at least this many, but this subset is enough for finding `witnesses'), a set \(U^{'}_i=\{u'_{i,1},u'_{i,2},\dots,u'_{i,a-b}\}\subseteq e_i \setminus (U_i \cup\{v\})\) of \(\mathbf{X}\)-covered (and not necessarily resurrected) vertices, and a set \(W_i = \{w_{i,1},w_{i,2},\dots,w_{i,c}\}\subseteq e_i\setminus\{v\}\) of vertices that are each put in~\(\mathbf{W}\) in~\(\omega\).
For each \(u_{i,j}\in U_i\), there is an \(\mathbf{X}\)-edge~\(f_{i,j}\ni u_{i,j}\) and a non-\(f_{i,j}\) \(\mathbf{X}\)-edge~\(g_{i,j}\) intersecting~\(f_{i,j}\), which witnesses the resurrection of~\(u_{i,j}\).
Let~\(F^1_{i}=\{f_{i,j}\colon j\in[b]\}\), allowing~\(F^1_i\) to be a multiset\COMMENT{Allowing it to be a multiset only really serves two purposes in the end, given my later choice of how to define \(c_e\): Firstly, mentioning it could help the reader to understand the complexities of what's going on - a single \(\mathbf{X}\)-edge can bear witness to the fact several vertices of \(e\setminus\{v\}\) are \(\mathbf{X}\)-covered, if it happens to contain all those vertices, and secondly, it helps make it clear and tidy that \(\sum_{j\in[s]}(|W_j|+|F_j^1|+|F_j^2|+|G_j|)=s(c+b+(a-b)+b)=(a+b+c)s\) when we come to count witnesses a bit later.}, and similarly construct the multiset \(G_i=\{g_{i,j}\colon j\in[b]\}\) and the multiset \(F^2_i=\{f'_{i,j}\colon j\in[a-b]\}\), where~\(f'_{i,j}\ni u'_{i,j}\) is an \(\mathbf{X}\)-edge which witnesses the fact that~\(u'_{i,j}\) is \(\mathbf{X}\)-covered.
Set \(Y_i\coloneqq F_i^1\cup F_i^2 \cup G_i\), ignoring multiple appearances beyond the first, so that~\(Y_i\) is a set.
Put \(I\coloneqq\{w\in V(H)\colon \exists i\in[s]\,\,\text{such that}\,\,w\in W_i\} \cup\{e\in E(H)\colon \exists i\in[s]\,\,\text{such that}\,\,e\in Y_i\}\).
For each vertex \(w\in I\), set~\(c_w\) to be the number\COMMENT{Mildly unfortunate notation to have the \(c\) in \(f(a,b,c)\) and also the \(c_i\) for the Talagrand vector. Do we mind this?} of \(i\in[s]\) for which \(w\in W_i\), and similarly, for an edge \(e\in I\) set~\(c_e\) to be the number of \(i\in[s]\) such that \(e\in Y_i\).\COMMENT{I went back and forth on whether to define the \(c_e\) this way or counting multiplicity. I began typing the latter and decided the former was cleaner. Here's the start of the multiplicity version: for an edge \(e\in I\) set~\(\lambda_e^{i,j}\) to be number of times~\(e\) appears in the multiset~\(F_i^j\) for \(i\in[s], j\in[2]\), and set \(\lambda_e^{(j)}=\sum_{i\in[s]}\lambda_e^{i,j}\).
Similarly, set~\(\lambda_{e}^{i,3}\) to be the number of times~\(e\) appears in~\(G_i\), and set \(\lambda_e^{(3)}=\sum_{i\in{s}}\lambda_e^{i,3}\).
Set \(c_e\coloneqq\sum_{j\in[3]}\lambda_e^{(j)}\).
We will show that \((I,(c_i\colon i\in I))\) is an \((a+b+c,2(k+1)^2 D_2\log^{5/2}D)\)-certificate for \(\mathbf{f}(a,b,c), \omega, s\), and~\(\Omega^{*}\).
To that end, notice that \(\sum_{i\in I}c_i=\sum_{j\in[s]}(|W_j|+|F_j^1|+|F_j^2|+|G_j|)=s(c+b+(a-b)+b)=(a+b+c)s\).
...Now when bounding the \(c_e\) you'll have to count the \(i\in[s]\) for which \(e\in F_i^1\), then multiply by \(k\) (it could be the chosen edge for all \(u\in e_i\setminus\{v\}\)), do the same for~\(F_i^2\), and so forth. When doing the \(I'\subseteq I\) business, it's clear this probably overcounts the number of edges counting towards \(\mathbf{f}(a,b,c)\) we lose, and I think the non-multiplicity version was cleaner even there.}

We will show that \((I,(c_i\colon i\in I))\) is an \((a+b+c,2(k+1)^2 D_2\log^{5/2}D)\)-certificate for \(\mathbf{f}, \omega, s\), and~\(\Omega^{*}\).
To that end, notice that \(\sum_{i\in I}c_i=\sum_{j\in[s]}(|W_j|+|Y_j|)\leq\sum_{j\in[s]}(|W_j|+|F_j^1|+|F_j^2|+|G_j|)=s(c+b+(a-b)+b)=(a+b+c)s\).
For a vertex \(w\in I\), clearly\COMMENT{If a vertex appears in \(W_i\) then it only appears once in that \(W_i\) since \(W_i\) is not a multiset. Clearly \(w\) can only be in \(W_i\) for those edges containing both \(v\) and \(w\) (and \(w\) cannot be \(v\) by construction of \(I\)). We may have that some edges containing \(v\) and \(w\) aren't counting towards \(\mathbf{f}(a,b,c)\) though, since they may not have at least a \(\mathbf{X}\)-covered vertices etc etc, so we can't say \(c_w=\text{codeg}(\{v,w\}\), but we can say \(c_w\leq \text{codeg}(\{v,w\})\), which suffices.} \(c_w\leq\text{codeg}(\{v,w\})\leq D_2\).
For an edge \(e\in I\), we can only have \(e\in F_{i}^{1}\cup F_{i}^{2}\) if~\(e\) intersects~\(e_i\setminus\{v\}\).
We deduce\COMMENT{Pick a non-\(v\) vertex (since \(e\) must intersect \(e_i\setminus\{v\}\)) of~\(e\) and generate all edges containing \(v\) and this vertex.} there are at most~\((k+1)D_2\) such~\(i\).
For an edge \(e\in I\) to be in~\(G_i\), we must have that~\(e\) intersects some \(\mathbf{X}\)-edge \(f\in F_i^1\), which in turn must intersect \(e_i\setminus\{v\}\).
In particular, only edges~\(e\) entirely comprised of vertices~\(u\in N^{(3)}(v)\) can possibly be in any~\(G_i\), since the sequence \(e_i, f, e\) of edges each intersect the next, with \(v\in e_i\).
So, suppose \(e\in \cup_{i\in[s]}G_i\).
Then since \(e\subseteq N^{(3)}(v)\) and \(\omega\in\Omega\setminus\Omega^{*}\), there are at most \((k+1)\log^{5/2}D\) \(\mathbf{X}\)-edges~\(f\) intersecting~\(e\), and each of these can only be in at most~\((k+1)D_2\) of the sets~\(F_i^1\), as above.
We deduce that~\(e\) is in at most~\((k+1)^2 D_2\log^{5/2}D\) of the sets~\(G_i\), and we deduce that \(\max_{i\in I}c_i\leq (k+1)D_2 + (k+1)^2D_2\log^{5/2} D\leq 2(k+1)^{2}D_2\log^{5/2}D\).

Finally, suppose \(I'\subseteq I\).
We are required to show that \(\mathbf{f}\geq s-\sum_{i\in I'}c_i\) holds for all outcomes \(\omega'\in\Omega\setminus\Omega^{*}\) aligning with~\(\omega\) on the (affirmative) choices to put \(e\in\mathbf{X}\), \(u\in\mathbf{W}\) for all \(e,u\in I\setminus I'\).
By the definition of~\(\mathbf{f}\), it clearly suffices to show this for the outcome~\(\omega''\) in which edges~\(e\) (respectively vertices~\(u\)) are chosen for~\(\mathbf{X}\)~(\(\mathbf{W}\)) if and only if \(e\in I\setminus I'\) (\(u\in I\setminus I'\)).\COMMENT{Since~\(\omega\in\Omega\setminus\Omega^{*}\), so is this~\(\omega''\), because the \(\mathbf{X}\)-edges in~\(\omega''\) are a subset of those in~\(\omega\). Also \(\mathbf{f}(a,b,c)\) is a "monotone increasing" random variable in the sense that \(\mathbf{f}(\omega_1)\leq\mathbf{f}(\omega_2)\) if the set of edges and vertices having \(\mathds{1}(e\in\mathbf{X})=1\) and \(\mathds{1}(u\in\mathbf{W})=1\) in~\(\omega_1\) is a subset of such edges and vertices in~\(\omega_2\), so this \(\omega''\) is a worst case, from which all others we need to check follow.}
In particular, from the definition of~\(\mathbf{f}\) and our construction of~\(I\), it is clear that \(\mathbf{f}=s\) still holds in the outcome~\(\omega^{*}\) in which edges~\(e\) and vertices~\(u\) are put in~\(\mathbf{X}\) (respectively~\(\mathbf{W}\)) if and only if they are in~\(I\).
Further (starting from~\(\omega^{*}\)), when a vertex \(u\in I'\) is taken out of~\(\mathbf{W}\) (respectively an edge \(e\in I'\) taken out of~\(\mathbf{X}\)), \(\mathbf{f}\) can only decrease by at most the number of \(i\in[s]\) for which \(u\in W_i\) (\(e\in Y_i\)).
We deduce that \(\mathbf{f}\geq s-\sum_{i\in I'}c_i\) in~\(\omega''\), which suffices.
We conclude that~\(\mathbf{f}\) is \((a+b+c, 2(k+1)^2 D_2\log^{5/2}D)\)-observable with respect to~\(\Omega^{*}\), as desired.

We now seek to apply Theorem~\ref{theorem:lintal} to~\(\mathbf{f}\) with \(t\coloneqq \eps\theta\text{deg}(v)/40k\), \(r\coloneqq a+b+c\), \(d\coloneqq 2(k+1)^2 D_2\log^{5/2}D\).
To that end, observe that
\begin{equation}\label{eq:tcheck1}
128rd \leq  128\cdot 6k(k+1)^2D_2\log^{5/2}D\leq \frac{D_2}{D}D\log^3 D \stackrel{\ref{hypothesis:degratio}}{\leq}\frac{\eps^2 \theta D}{\log^2 D} < \frac{\eps\theta \text{deg}(v)}{120k},
\end{equation}
and\COMMENT{This uses \(D_2 \geq 1\), which holds as discussed immediately after the statement of Lemma~\ref{lemma:nibble}.}
\begin{equation}\label{eq:tcheck2}
8\prob{\Omega^{*}}(\text{sup}\mathbf{f}) \stackrel{(\ref{eq:theexceptionaloutcomes})}{\leq} 8\Delta(H)D^{-\frac{1}{2}\log^{3/2} D} < 1 \leq \frac{D_2}{D}\cdot D \stackrel{\ref{hypothesis:degratio}}{\leq}\frac{\eps^{2}\theta D}{\log^5 D} < \frac{\eps\theta \text{deg}(v)}{120k}.
\end{equation}
We still need an upper bound for~\(\expn{\mathbf{f}}\).
Notice that, since \(\mathbf{f}=\mathbf{f}(a,b,c)\) is an important counter\COMMENT{From the expansion of \(\textbf{deg}\mathbf{'}(v)\) in~(\ref{eq:Na}) and just before, we can see that none of the important counters have \(a=c=0\).}, at least one of~\(a\) and~\(c\) is at least~\(1\).
If \(a\geq 1\), then
\[
\expn{\mathbf{f}(a,b,c)}\leq\sum_{e\ni v}\prob{|(e\setminus\{v\})\cap V(\mathbf{X})|\geq 1}\leq\sum_{e\ni v}k\Delta(H)p\leq 2k\theta D,
\]
and similarly, if \(c\geq 1\), then
\[
\expn{\mathbf{f}(a,b,c)}\leq\sum_{e\ni v}\prob{|(e\setminus\{v\})\cap \mathbf{W}|\geq 1}\stackrel{(\ref{eq:wvabove})}{\leq} 8k\eps\theta \Delta(H)\leq 2k\theta D.
\]
Therefore:\COMMENT{\(96\sqrt{rd\expn{\mathbf{f}}}\leq96\sqrt{(a+b+c)2(k+1)^2D_2(\log^{5/2}D)2k\theta D}\). All the constants can be upper bounded by \(\log^{1/4}D\), i.e. \(\sqrt{\log^{1/2}D}\). Then multiply and divide by~\(D\) inside the square root.}
\begin{equation}\label{eq:tcheck3}
96\sqrt{rd\expn{\mathbf{f}}} \leq \sqrt{\frac{D_2}{D}\theta D^2 \log^{3}D} \stackrel{\ref{hypothesis:degratio}}{\leq} D\sqrt{\frac{\eps^2 \theta^2}{\log^2 D}} < \frac{\eps\theta \text{deg}(v)}{120k}.
\end{equation}
By~(\ref{eq:tcheck1}),~(\ref{eq:tcheck2}), and~(\ref{eq:tcheck3}), we can apply Theorem~\ref{theorem:lintal} to~\(\mathbf{f}\) with the above values of \(t, r, d\), obtaining:\COMMENT{\(t=\eps\theta D/(40k) \leq 2k\theta D\) is clear, and we also have \(\expn{\mathbf{f}}\leq 2k\theta D\) from above, also of course \(\text{deg}(v)\geq D/2\), so the denominator is \(\Theta(D_2\log^{5/2}D \cdot \theta D)\), and upping the polylog to \(\log^{8/3}D\) sorts out all the constants. A \(\theta D\) cancel from top and bottom. Then by \ref{hypothesis:degratio}, the first term is \(\leq 4\exp(-\log^{7/3}D)=4D^{-\log^{4/3}D}\), and now \(4D^{-\log^{4/3}D}+4D^{-\frac{1}{2}\log^{3/2} D}\leq D^{-5\log^{5/4}D}\) is clear.}
\begin{eqnarray*}
\prob{|\mathbf{f}-\expn{\mathbf{f}}|>\frac{\eps\theta \text{deg}(v)}{40k}} & \leq & 4\exp\left(-\frac{\eps^2\theta^2 (\text{deg}(v))^2}{(40k)^{2}8rd(4\expn{\mathbf{f}}+t)}\right)+4\prob{\Omega^{*}} \\ & \leq & 4\exp\left(-\frac{\eps^2\theta D}{D_2 \log^{8/3}D}\right) +4\prob{\Omega^{*}} \leq D^{-5\log^{5/4}D}.
\end{eqnarray*}
Here, we used \(\expn{\mathbf{f}},t\leq 2k\theta D\), \ref{hypothesis:degratio}, and~(\ref{eq:theexceptionaloutcomes}).
Taking a union bound over the (at most \(8k+1\leq 9k\)) important counters yields that\COMMENT{All the errors could add up, but of course \(9k\cdot\frac{\eps\theta\text{deg}(v)}{40k}\leq \frac{\eps\theta\text{deg}(v)}{4}\), and also clearly \(9kD^{-5\log^{5/4}D}\leq D^{-4\log^{1/4}D}\) since \(9k\leq D^{\log^{5/4}D}\).}
\[
\prob{|\textbf{deg}\mathbf{'}(v)-\expn{\textbf{deg}\mathbf{'}(v)}| > \frac{\eps\theta\text{deg}(v)}{4}}\leq D^{-4\log^{5/4}D},
\]
which, together\COMMENT{\(\textbf{deg}\mathbf{'}(v)\) is at most \(\eps\theta\text{deg}(v)/4\) away from \(\expn{\textbf{deg}\mathbf{'}(v)}\), which is in turn at most \(\text{deg}(v)2^{k+3}\eps^2 \theta^2\) away from \(\text{deg}(v)(1-p^{*})^k\). So \(\textbf{deg}\mathbf{'}(v)\) is at most \(\text{deg}(v)(\eps\theta/4 + 2^{k+3}\eps^2\theta^2)\) away from \(\text{deg}(v)(1-p^{*})^k\), and \(\eps\theta/4 + 2^{k+3}\eps^2\theta^2 \leq \eps\theta/2\) clearly.} with Lemma~\ref{lemma:expectations}\ref{lemma:degreeexpectation}, finishes the proof.
\endproof
Notice the choice~\(\Omega^{*}\) of exceptional outcomes ensures that in any \(\omega\notin\Omega^*\), any \(g\in E(H)\) can only witness the resurrection of an \((e\setminus\{v\})\)-vertex for at most a little more than~\(D_2\) of the~\(D\) edges containing~\(v\).
With \(\Omega^{*}=\emptyset\), we have that for any \(e\ni v\), the edge~\(g\) potentially intersects as many as~\((k+1)D_2\) edges containing any \(u\in e\setminus\{v\}\), and so can be a resurrection witness for potentially any of the~\(D\) edges containing~\(v\), which clearly yields a problematic value of~\(d\).
Thus, the ``exceptional outcomes'' aspect of Theorem~\ref{theorem:lintal} seems to be essential here.

The proof of Lemma~\ref{lemma:concanalysis}\ref{conccodegree} is very similar to the proof of Lemma~\ref{lemma:concanalysis}\ref{concvertexdegree}, but for completeness, we still include the necessary changes and calculations:
\lateproof{Lemma~\ref{lemma:concanalysis}\ref{conccodegree}}
Fix \(j\in J^{*}\) and \(Z\in \binom{V(H)}{j}\).
We use the same~\(\Omega^{*}\) as in the proof of Lemma~\ref{lemma:concanalysis}\ref{concvertexdegree}, except for replacing~\(N^{(3)}(v)\) with~\(N^{(3)}(Z)\) defined in the obvious way, so we still have \(\prob{\Omega^{*}}\leq D^{-(1/2)\log^{3/2}D}\).
\(\mathbf{f}(a,b,c)\) is defined analogously (just replace the instances of~\(e\setminus\{v\}\) with~\(e\setminus Z\)).
Then we have \(\textbf{codeg}\mathbf{'}(Z)=\text{codeg}(Z)-\mathbf{f}(0,0,1) - \sum_{a=1}^{k+1-j}\mathbf{N}(a)\), with~\(\mathbf{N}(a)\) as defined in~(\ref{eq:Na}).
We can use the same witness set~\(I\) (just replace all instances of~\(v\) with~\(Z\))\COMMENT{Here implicitly defining \(W_i\) and \(Y_i\) and \(F_i^1\) and so on} and the same vector \((c_i\colon i\in I)\) (specifically~\(c_w\) (respectively~\(c_e\)) is the number of \(i\in[s]\) for which \(w\in W_i\) (\(e\in Y_i\))).
Clearly we still have at most~\((a+b+c)s\) witnesses as before, so can use \(r=a+b+c\).

For a vertex~\(w\), we now have\COMMENT{No vertex of \(Z\) can ever be a witness, since all vertex witness arise from some \(e_i\setminus Z\), so \(Z\cup\{w\}\) is genuinely a set of size \(j+1\).} \(c_w\leq \text{codeg}(Z\cup\{w\})\leq D_{j+1}\).
For an edge \(e\in I\), we can only have \(e\in F_{i}^{1}\cup F_{i}^{2}\) if~\(e\) intersects~\(e_i\setminus Z\), so there are at most~\((k+1)D_{j+1}\) such~\(i\).
Again, if \(e\in\cup_{i\in[s]}G_i\), then \(e\subseteq N^{(3)}(Z)\), so \(\omega\in\Omega\setminus\Omega^{*}\) implies there are at most \((k+1)\log^{5/2}D\) \(\mathbf{X}\)-edges intersecting~\(e\), each of which can only be in at most \((k+1)D_{j+1}\) of the sets~\(F_i^1\), so \(\max_{i\in I}c_i\leq (k+1)D_{j+1}+(k+1)^2 D_{j+1}\log^{5/2}D\leq 2(k+1)^2 D_{j+1}\log^{5/2}D\), so we can set~\(d\) to be this number.
The proof that~\(\mathbf{f}\geq s - \sum_{i\in I'}c_i\) in any non-exceptional outcome aligning with~\(\omega\) on the witnesses in~\(I\setminus I'\) proceeds unchanged.

We seek to apply Theorem~\ref{theorem:lintal} to~\(\mathbf{f}\) with \(t\coloneqq D_j\theta^{3/2}/(40k)\), \(r\coloneqq a+b+c\), \(d\coloneqq 2(k+1)^2 D_{j+1}\log^{5/2}D\).
To that end, observe that
\begin{equation}\label{eq:codegtcheck1}
128rd \leq D_{j+1}\log^{3}D = \frac{D_{j+1}}{D_j}\cdot D_j\log^{3}D \stackrel{\ref{hypothesis:codegratio}}{\leq} \frac{D_j \theta^2}{\log^2 D} < \frac{D_j \theta^{3/2}}{120k},
\end{equation}
and:\COMMENT{This uses \(D_{j+1}\geq 1\), which holds as discussed immediately after the statement of Lemma~\ref{lemma:nibble}.}
\begin{equation}\label{eq:codegtcheck2}
8\prob{\Omega^{*}}(\text{sup}\mathbf{f}) \leq 8\Delta(H)D^{-\frac{1}{2}\log^{3/2} D} < 1 \leq \frac{D_{j+1}}{D_j}\cdot D_j \stackrel{\ref{hypothesis:codegratio}}{\leq}\frac{D_j \theta^2}{\log^5 D} < \frac{D_j \theta^{3/2}}{120k}.
\end{equation}
Again, since~\(\mathbf{f}\) is an important counter, we have \(a+c\geq 1\), and so, similarly to the proof of Lemma~\ref{lemma:concanalysis}\ref{concvertexdegree}, we compute that \(\expn{\mathbf{f}}\leq\max\{\sum_{e\supseteq Z}(k+1-j)\Delta(H)p, \sum_{e\supseteq Z}8(k+1-j)\eps\theta\}\leq 2k\theta D_j\), so:\COMMENT{\(96\sqrt{rd\expn{\mathbf{f}}} \leq 96\sqrt{(a+b+c)2(k+1)^2 D_{j+1}\log^{5/2}D \cdot 2k\theta D_j}\). All the constants can be bounded above by \(\sqrt{\log^{1/2}D}\). Multiply and divide by \(D_j\) inside the square root.}
\begin{equation}\label{eq:codegtcheck3}
96\sqrt{rd\expn{\mathbf{f}}} \leq \sqrt{\frac{D_{j+1}}{D_j}\theta D_j^2 \log^{3}D} \stackrel{\ref{hypothesis:codegratio}}{\leq} D_j \sqrt{\frac{\theta^{3}}{\log^2 D}} < \frac{D_j \theta^{3/2}}{120k}.
\end{equation}
By~(\ref{eq:codegtcheck1}),~(\ref{eq:codegtcheck2}), and~(\ref{eq:codegtcheck3}), we can apply Theorem~\ref{theorem:lintal} to~\(\mathbf{f}\) with the above values of \(t, r, d\), obtaining:\COMMENT{\(t=\theta^{3/2}D_j/(40k)\leq 2k\theta D_j\) is clear, and \(\expn{\mathbf{f}}\leq 2k\theta D_j\) from above, so the denominator is \(O(D_{j+1}\log^{5/2}D \cdot \theta D_j)\), and upping the polylog to \(\log^{8/3}D\) sorts out all the constants. A \(\theta D_j\) cancel from top and bottom. Then by \ref{hypothesis:codegratio}, the first term is \(\leq 4\exp(-\log^{7/3}D)\), and the rest continues as before.}
\begin{eqnarray*}
\prob{|\mathbf{f}-\expn{\mathbf{f}}| > \frac{D_j \theta^{3/2}}{40k}} & \leq & 4\exp\left(-\frac{\theta^3 D_j^2}{(40k)^2 8rd(4\expn{\mathbf{f}}+t)}\right)+4\prob{\Omega^{*}} \\ & \leq & 4\exp\left(-\frac{\theta^{2}D_j}{D_{j+1}\log^{8/3}D}\right)+4\prob{\Omega^{*}} \leq D^{-5\log^{5/4}D}.
\end{eqnarray*}
Here, we used \(\expn{\mathbf{f}}, t \leq 2k\theta D_j\), \ref{hypothesis:codegratio}, and \(\prob{\Omega^{*}}\leq D^{-(1/2)\log^{3/2}D}\).
Taking a union bound over the (at most \(8(k+1-j)+1\leq 9k\)) important counters yields that\COMMENT{All the errors could add up, but of course \(9k\cdot \frac{D_j \theta^{3/2}}{40k}\leq \frac{D_j \theta^{3/2}}{4}\), and also clearly \(9kD^{-5\log^{5/4}D}\leq D^{-4\log^{5/4}D}\).}
\[
\prob{|\textbf{codeg}\mathbf{'}(Z)-\expn{\textbf{codeg}\mathbf{'}(Z)}|>\frac{D_j \theta^{3/2}}{4}} \leq D^{-4\log^{5/4}D},
\]
which, together\COMMENT{\(\textbf{codeg}\mathbf{'}(Z)\) is at most \(D_j \theta^{3/2}/4\) above \(\expn{\textbf{codeg}\mathbf{'}(Z)}\), which is in turn at most \(D_j\theta^{5/3}\) above \(D_j(1-\theta(k+1-j))\). So \(\textbf{codeg}\mathbf{'}(Z)\) is at most \(D_j(\theta^{3/2}/4 + \theta^{5/3})\) above \(D_j(1-\theta(k+1-j))\), and \(\theta^{3/2}/4 + \theta^{5/3} \leq \theta^{3/2}\) clearly.} with Lemma~\ref{lemma:expectations}\ref{lemma:codegreeexpectation}, finishes the proof.
\endproof
Finally, we prove Lemma~\ref{lemma:concanalysis}\ref{conctestpseudo}--\ref{conctestwaste} using a similar strategy.
\lateproof{Lemma~\ref{lemma:concanalysis}\ref{conctestpseudo}--\ref{conctestwaste}}
Fix \(\tau\in\cT\).
If \(\text{Supp}(\tau)=\emptyset\), i.e.\ \(\tau(v)=0\) for all \(v\in V(H)\), then \(\tau(V(\mathbf{H'}))=\tau(V(H))=0\) and \(\tau(\mathbf{W})=0\) deterministically\COMMENT{so the bad event doesn't happen}, so assume that \(\text{Supp}(\tau)\neq\emptyset\), whence \(\tau_{\text{max}}>0\).\COMMENT{We will later divide by it. Also once you know the support isn't empty,~\ref{pseudhyp:lowerbound} ensures \(\tau\) is a thoroughly non-trivial function, it must have many vertices in the support} 
By~(\ref{eq:pstar}),~(\ref{eq:pstarabove}), and~(\ref{eq:pstarbelow}), we have\COMMENT{\(\expn{\tau(V(\mathbf{H'}))}=\expn{\sum_{v\in V(H)}\tau(v)\mathds{1}_{v\in V(\mathbf{H'})}}=\sum_{v\in V(H)}\tau(v)\expn{\mathds{1}_{v\in V(\mathbf{H'})}}=\sum_{v\in V(H)}\tau(v)\prob{v\in V(\mathbf{H'})}=\sum_{v\in V(H)}\tau(v)(1-p^{*})=\tau(V(H))(1-p^{*})\)}
\begin{equation}\label{eq:expntau'}
\tau(V(H))(1-\theta-\theta^{7/4})\leq \tau(V(H))(1-p^{*})=\expn{\tau(V(\mathbf{H'}))}\leq \tau(V(H))(1-\theta+\theta^{7/4}).
\end{equation}
For \(u\in V(H)\), set \(N^{(2)}(u)\coloneqq\{v\in V(H)\colon \text{dist}_{H}(u,v)\leq 2\}\), and set \(\Omega^{*}\coloneqq\{\omega\in\Omega\colon\exists u\in \text{Supp}(\tau), v\in N^{(2)}(u)\,\,\text{such that}\,\,|\partial(v)\cap\mathbf{X}|>\log^{5/2}D\}\).
Then by~\ref{pseudhyp:supportsize} and a simple application of Lemma~\ref{chernoff}\ref{lemma:chernoffsmallexp}, we have \(\prob{\Omega^{*}}\leq|\text{Supp}(\tau)|\cdot 2(2kD)^2 D^{-\log^{3/2}D}\leq D^{-(1/2)\log^{3/2}D}\).
For integers \(0\leq b\leq a \leq 1\), \(0\leq c\leq 1\), define~\(\mathbf{f}(a,b,c)\coloneqq\tau(\mathbf{U}_{(a,b,c)})\), where~\(\mathbf{U}_{(a,b,c)}\) is the subset of~\(\text{Supp}(\tau)\) formed of vertices \(u\in\text{Supp}(\tau)\) for which at least~\(a\) elements of~\(\{u\}\) are \(\mathbf{X}\)-covered, at least~\(b\) elements of~\(\{u\}\) are resurrected, and at least~\(c\) elements of~\(\{u\}\) are put in~\(\mathbf{W}\).
Notice that the total ``weight'' (\(\tau\)-value) of all \(\mathbf{X}\)-covered but not resurrected vertices \(u\in\text{Supp}(\tau)\) is~\(\mathbf{f}(1,0,0)-\mathbf{f}(1,1,0)\), and that the total weight of all such vertices which are also put in~\(\mathbf{W}\) is \(\mathbf{f}(1,0,1)-\mathbf{f}(1,1,1)\).
It follows that \(\tau(V(\mathbf{H'}))=\tau(V(H))-\mathbf{f}(0,0,1)-((\mathbf{f}(1,0,0)-\mathbf{f}(1,1,0))-(\mathbf{f}(1,0,1)-\mathbf{f}(1,1,1)))\).\COMMENT{Begin with the total weight. Remove the weight of all wasted vertices. Then remove the weight of all vertices which are (not wasted, but are \(\mathbf{X}\)-covered, and not resurrected). The latter bracket removes the weight of precisely those vertices in the support that are killed solely by the edge choices.}
Let these \(\mathbf{f}(a,b,c)\) be called the \textit{important counters}.

Fix an important counter~\(\mathbf{f}(a,b,c)\eqqcolon\mathbf{f}\) and notice that \(a\geq b\).
Fix an outcome \(\omega\in\Omega\setminus\Omega^{*}\) and let \(s=\mathbf{f}(\omega)\).
By definition of~\(\mathbf{f}\) and~\(\text{Supp}(\tau)\), there are~\(z\) (say) distinct vertices \(u_1, u_2, \dots, u_z\in \text{Supp}(\tau)\) such that \(\sum_{i\in[z]}\tau(u_i)=s\) and, for all \(i\in[z]\), at least~\(b\) elements of~\(\{u_i\}\) are resurrected, at least~\(a-b\) elements of~\(\{u_i\}\) are \(\mathbf{X}\)-covered (and not necessarily resurrected), and at least~\(c\) elements of~\(\{u_i\}\) are put in~\(\mathbf{W}\).
If \(b=1\), there is an \(\mathbf{X}\)-edge \(f_i\ni u_i\) and a non-\(f_i\) \(\mathbf{X}\)-edge~\(g_i\) intersecting~\(f_i\), which witnesses the resurrection of~\(u_i\).
If \(a-b=1\), there is an \(\mathbf{X}\)-edge \(h_i\ni u_i\).
If \(c=1\), then each \(u_i\in\mathbf{W}\).
If \(c=0\), then put \(I\coloneqq\cF\coloneqq\{e\in E(H)\colon \exists i\in[z]\,\text{such that}\,e\in\{f_i,g_i,h_i\}\}\), and if \(c=1\), then put \(I\coloneqq \cF\cup\{u_i\colon i\in[z]\}\).
For a vertex \(w\in I\), put \(c_w\coloneqq \tau(w)\).
For an edge \(e\in I\), put \(\cI_e^{(1)}\coloneqq\{i\in[z]\colon e=f_i\}\), put \(\cI_e^{(2)}\coloneqq\{i\in[z]\colon e=g_i\}\), and put \(\cI_e^{(3)}\coloneqq\{i\in[z]\colon e=h_i\}\).
Then set \(c_e\coloneqq\sum_{j\in[3]}\sum_{i\in\cI_e^{(j)}}\tau(u_i)\).
One checks that \(\sum_{i\in I}c_i=(a+b+c)s\).\COMMENT{Let's just run through the five important counters: If \(a=0, b=0, c=1\), then \(I\) is just the \(z\) vertices, and \(\sum_{i\in I}c_i=\sum_{i\in[z]}\tau(u_i)=s=(a+b+c)s\).
If \(a=1, b=0, c=0\), then \(I\) is just some edges each witnessing some \(u_i\) are \(\mathbf{X}\)-covered, and \(\sum_{i\in I}c_i\) is the sum over these edges, of the sum over the \(u_i\) that edge is witnessing, of \(\tau(u_i)\). Since each \(u_i\) has one such witness, that sum is \(\sum_{i\in[z]}\tau(u_i)=s=(a+b+c)s\).
If \(a=1, b=1, c=0\), then \(I\) is a collection of edges, some of which witnessing some \(u_i\) are \(\mathbf{X}\)-covered (i.e. \(e=f_i\)), some witnessing some \(u_i\) are resurrected (i.e. \(e=g_i\)), some witnessing both of these for some sets of \(i\). Since each \(u_i\) gives two witnesses (an \(f_i\) and a \(g_i\)), we have \(\sum_{e\in I}c_e=\sum_{e\in I}\sum_{i\in\cI_e^{(1)}}\tau(u_i)+\sum_{e\in I}\sum_{i\in\cI_e^{(2)}}\tau(u_i)=\sum_{i\in[z]}2\tau(u_i)=2s=(a+b+c)s\).
Cases \(a=1,b=0,c=1\) and \(a=b=c=1\) similar. In the latter case, \(b=1\) but \(a-b=0\), so we collect vertex witnesses and edge witnesses \(f_i,g_i\), but not \(h_i\).}
Further, for an edge~\(e\) to be~\(f_i\) or~\(h_i\), we must have \(u_i\in e\), so we deduce there are at most \(2(k+1)\) values \(i\in[s]\) such that \(e\in\{f_i,h_i\}\).
For~\(e\) to be~\(g_i\) for some~\(i\), all vertices of~\(e\) must be in~\(N^{(2)}(u_i)\), and therefore~\(e\) must intersect at most~\((k+1)\log^{5/2}D\) \(\mathbf{X}\)-edges (since \(\omega\in\Omega\setminus\Omega^{*}\)), each of which can be~\(f_j\) for at most~\(k+1\) values \(j\in[s]\).\COMMENT{Since we're looking at distance-2 edges here rather than distance-3 edges as in the case of concentrating the vertex degrees, you could get away with \(\Omega^*=\emptyset\) and you'd get that like \((k+1)D_2\) out of the \(D\) edges containing \(u_i\) could possibly have \(e\) playing the role of \(g_i\). But this isn't good enough. Through the action of~\(\Omega^*\) we've reduced that to more like \(\log^{5/2}D\), which is good enough.}
We deduce that \(c_i\leq 2(k+1)^2 \tau_{\text{max}}\log^{5/2}D\) for all \(i\in I\).
It is easy to show that \(\mathbf{f}\geq s-\sum_{i\in I'}c_i\) holds for all \(I'\subseteq I\) and outcomes \(\omega'\in\Omega\setminus\Omega^{*}\) aligning with~\(\omega\) on the (affirmative) choices to put \(e\in\mathbf{X}\), \(u\in\mathbf{W}\) for all \(e,u\in I\setminus I'\).\COMMENT{If I drop a vertex witness, I lose at most \(\tau_{\text{max}}\) weight. If I drop an edge witness, I lose at most \(\tau_{\text{max}}\) times the number of \(u_i\) that edge is any kind of witness for, so total \(\leq 2(k+1)^2\tau_{\text{max}}\log^{5/2}D\).}
We conclude that~\(\mathbf{f}\) is \((a+b+c,2(k+1)^2 \tau_{\text{max}} \log^{5/2}D)\)-observable with respect to~\(\Omega^{*}\).

We seek to apply Theorem~\ref{theorem:lintal} with \(t\coloneqq \tau(V(H))\theta^{3/2}/10\), \(r\coloneqq a+b+c\), and \(d\coloneqq 2(k+1)^2 \tau_{\text{max}}\log^{5/2}D\).
To that end, observe that
\begin{equation}\label{eq:Utcheck1}
128rd\leq \tau_{\text{max}}\log^3 D \stackrel{\ref{pseudhyp:lowerbound}, \ref{hypothesis:epsleqtheta}}{\leq}\frac{\tau(V(H))\theta^2}{\log^2 D} < \frac{\tau(V(H))\theta^{3/2}}{30}.
\end{equation}
Further,
\begin{eqnarray}\label{eq:Utcheck2}
8\prob{\Omega^*}(\text{sup}\mathbf{f}) & \stackrel{\ref{pseudhyp:supportsize}}{\leq} & 8D^{-(1/2)\log^{3/2}D}D^{\log^{5/4} D}\tau_{\text{max}}\leq \tau_{\text{max}}\log^5 D \stackrel{\ref{pseudhyp:lowerbound},\ref{hypothesis:epsleqtheta}}{\leq}\theta^2 \tau(V(H))\nonumber\\ & < & \frac{\tau(V(H))\theta^{3/2}}{30}.
\end{eqnarray}
Since~\(\mathbf{f}=\mathbf{f}(a,b,c)\) is an important counter, we have \(a+c\geq 1\).
If \(a\geq 1\), then \(\expn{\mathbf{f}}\leq\sum_{u\in \text{Supp}(\tau)}\tau(u)\prob{u\in V(\mathbf{X})}\leq \tau(V(H))\Delta(H)p\leq 2\theta \tau(V(H))\).
If \(c\geq 1\), then we have \(\expn{\mathbf{f}}\leq\sum_{u\in \text{Supp}(\tau)}\tau(u)\prob{u\in\mathbf{W}}\leq 8\eps\theta \tau(V(H))\leq 2\theta\tau(V(H))\) by~(\ref{eq:wvabove}).
Therefore:
\begin{equation}\label{eq:Utcheck3}
96\sqrt{rd\expn{\mathbf{f}}}\leq \sqrt{\theta\tau_{\text{max}}\tau(V(H))\log^3 D}<\frac{\tau(V(H))\theta^{3/2}}{30},
\end{equation}
since the last inequality is equivalent to \(\sqrt{\tau(V(H))}\theta > 30\sqrt{\tau_{\text{max}}}\log^{3/2} D\), which holds\COMMENT{Since both sides of \(\sqrt{\tau(V(H))}\theta > 30\sqrt{\tau_{\text{max}}}\log^{3/2} D\) are positive, it is equivalent to the inequality obtained by squaring both sides \(\tau(V(H))\theta^2>900\tau_{\text{max}}\log^{3}D\), but the latter is clear since \(\tau(V(H))\theta^2\stackrel{\ref{hypothesis:epsleqtheta}}{\geq}\tau(V(H))\eps\theta\stackrel{\ref{pseudhyp:lowerbound}}{\geq}\tau_{\text{max}}\log^5 D \geq 900\tau_{\text{max}}\log^3 D\).} by~\ref{hypothesis:epsleqtheta} and~\ref{pseudhyp:lowerbound}.
By~(\ref{eq:Utcheck1}),~(\ref{eq:Utcheck2}), and~(\ref{eq:Utcheck3}), we can apply Theorem~\ref{theorem:lintal} to~\(\mathbf{f}\) with the above values of \(t,r,d\), obtaining:\COMMENT{Use \(\expn{\mathbf{f}}, t\leq 2\theta\tau(V(H))\) to bound the denominator above by an order \(d\theta(\tau(V(H)))=O(\tau_{\text{max}}\tau(V(H))\theta\log^{5/2}D)\) term, so collecting the constants into a \(\log^{8/3-5/2}D\) term, so the denominator is bounded above by \(\tau_{\text{max}}\tau(V(H))\theta\log^{8/3}D\), then \(\tau(V(H))\theta\) cancels yielding the second line. Then \(\tau(V(H))\theta^2 \geq\tau(V(H))\eps\theta\geq\tau_{\text{max}}\log^5 D\) so we're now looking at \(4\exp(-\log^{7/3}D)+4D^{-(1/2)\log^{3/2}D}\) and now it's easy}
\begin{eqnarray*}
\prob{|\mathbf{f}-\expn{\mathbf{f}}| > \frac{\tau(V(H))\theta^{3/2}}{10}} & \leq & 4\exp\left(-\frac{(\tau(V(H)))^2\theta^3}{800rd(4\expn{\mathbf{f}}+t)}\right) + 4\prob{\Omega^{*}} \\ & \leq & 4\exp\left(-\frac{\tau(V(H))\theta^2}{\tau_{\text{max}}\log^{8/3}D}\right)+4\prob{\Omega^{*}} \leq D^{-5\log^{5/4}D}.
\end{eqnarray*}
Here, we used \(\expn{\mathbf{f}}, t \leq 2\theta\tau(V(H))\),~\ref{hypothesis:epsleqtheta},~\ref{pseudhyp:lowerbound}, \(\tau_{\text{max}}>0\) and \(\prob{\Omega^{*}}\leq D^{-(1/2)\log^{3/2}D}\).
Taking a union bound over the five important counters yields that
\[
\prob{|\tau(V(\mathbf{H'}))-\expn{\tau(V(\mathbf{H'}))}|>\frac{\tau(V(H))\theta^{3/2}}{2}}\leq D^{-4\log^{5/4}D},
\]
which, together with~(\ref{eq:expntau'}), finishes the proof of Lemma~\ref{lemma:concanalysis}\ref{conctestpseudo}.

For Lemma~\ref{lemma:concanalysis}\ref{conctestwaste}, notice that \(\tau(\mathbf{W})=\mathbf{f}(0,0,1)\) and recall that \(\expn{\mathbf{f}}\leq8\eps\theta\tau(V(H))\).
Since it suffices to take \(I=\{u_i\colon i\in[z]\}\) with \(c_w\coloneqq\tau(w)\leq\tau_{\text{max}}\) for \(w\in I\), it is clear that~\(\mathbf{f}(0,0,1)\) is \((1,\tau_{\text{max}})\)-observable with respect to~\(\Omega^{*}\coloneqq\emptyset\).
Set \(t\coloneqq2\eps\theta\tau(V(H))\), \(r\coloneqq 1\), and \(d\coloneqq\tau_{\text{max}}\), whence \(128rd=128\tau_{\text{max}}\leq2\eps\theta\tau(V(H))/3\) by~\ref{pseudhyp:lowerbound}, we have \(8\prob{\Omega^{*}}\text{sup}\mathbf{f}=0\), and \(96\sqrt{rd\expn{\mathbf{f}}}\leq96\sqrt{8\eps\theta\tau_{\text{max}}\tau(V(H))}\leq2\eps\theta\tau(V(H))/3\) since the latter inequality is equivalent to \(\sqrt{\eps\theta\tau(V(H))}\geq144\sqrt{8\tau_{\text{max}}}\), which clearly holds by~\ref{pseudhyp:lowerbound}.
Applying Theorem~\ref{theorem:lintal} with these \(t,r,d\), we obtain
\begin{eqnarray*}
\prob{|\tau(\mathbf{W})-\expn{\tau(\mathbf{W})}| > 2\eps\theta\tau(V(H))} & \leq & 4\exp\left(-\frac{4\eps^2 \theta^2 (\tau(V(H)))^2}{8\tau_{\text{max}}\cdot40\eps\theta\tau(V(H))}\right)\\ & = & 4\exp\left(-\frac{\eps\theta\tau(V(H))}{80\tau_{\text{max}}}\right)\leq D^{-4\log^{5/4}D},
\end{eqnarray*}
where we used \(\expn{\tau(\mathbf{W})}, t\leq 8\eps\theta\tau(V(H))\),~\ref{pseudhyp:lowerbound}, and \(\tau_{\text{max}}>0\).
Then \(\expn{\tau(\mathbf{W})}\leq8\eps\theta \tau(V(H))\) finishes the proof.
\endproof
\subsection{Multihypergraph Nibble}\label{section:multi}
In this short subsection, we discuss the minor definitional tweaks needed to prove Lemma~\ref{lemma:nibble} in the multihypergraph setting.

Recall from Section~\ref{section:notation} that copies of an edge~\(e\) are seen as distinct elements of~\(E(H)\).\COMMENT{And the reader hopefully then also recalls that \(\partial(v)\) and \(\partial(Z)\) can be multisets, and that \(\text{deg}(v)\) and \(\text{codeg}(Z)\) count copies, and \(D_2, D_3\) etc are therefore each upper bounds for the codegree counts including copies.}
As such, we consider them separately (and independently) for inclusion in~\(\mathbf{X}\).
In particular, if~\(e\) and~\(e'\) are copies of the same edge that are both (separately) put in~\(\mathbf{X}\), then they are not isolated.
The rest of Sections~\ref{section:description} and~\ref{section:modulo} proceed as written.

In the proof of Lemma~\ref{lemma:nomatchedvertices}, Claim~\ref{claim:intproduct}, we use the same definition of~\(F(R)\), noting that~\(F(R)\) may now contain copies.
Then the rest of the proof of Claim~\ref{claim:intproduct} proceeds as written, applying the same logic to every copy of~\(f\) in~\(F(Y)\).
Claims~\ref{claim:acovers} and~\ref{claim:nearlyindep} also do not need to be changed, with the understanding that \(A\)-covers using different copies of the same edge are seen as distinct \(A\)-covers, and similarly~\(\partial A\) and~\(\partial^{*}A\) distinguish copies as distinct.

In the proof of Lemma~\ref{lemma:concanalysis}\ref{concleftover}, the same logic still shows that~\(\mathbf{t}_e\) affects~\(\mathbf{n'}\) by at most~\((k+1)^2\), and the rest of that proof and the proof of Lemma~\ref{lemma:concanalysis}\ref{concwaste} clearly proceed without change.
In the proof of Lemma~\ref{lemma:concanalysis}\ref{concvertexdegree} (respectively Lemma~\ref{lemma:concanalysis}\ref{conccodegree}), we need to clarify that~\(\mathbf{f}(a,b,c)\) counts all copies of edges~\(e\ni v\) (\(e\supseteq Z\)) such that at least~\(a\) of the vertices~\(e\setminus\{v\}\) (\(e\setminus Z\)) are \(\mathbf{X}\)-covered, and so on.
In particular, if an edge \(e\ni v\) (\(e\supseteq Z\)) counts towards~\(\mathbf{f}(a,b,c)\), then so do all copies of~\(e\).
The edges \(e_1, e_2, \dots, e_s\) witnessing that \(\mathbf{f}(a,b,c)=s\) may contain a number of copies, but this does not impact our collection of witnesses, nor the Lipschitz constants~\(c_i\); for example a vertex \(u\in V(H)\) still has \(c_u\leq\text{codeg}(\{u,v\})\leq D_2\) (respectively \(c_u \leq\text{codeg}(\{u\}\cup Z)\leq D_{j+1}\)), since the formulation of~\(D_2\) and~\(D_{j+1}\) include the counting of all copies.
It is clear that the remaining checks for~\(\mathbf{f}(a,b,c)\) being observable proceed without change, and the final parts of both proofs are `number-crunching', unaffected by the multihypergraph setting.
Similarly, the proof of Lemma~\ref{lemma:concanalysis}\ref{conctestpseudo}--\ref{conctestwaste} does not need any changes.

\section{The Chomp}\label{section:chomp}
This section is devoted to the proof of the following lemma, which we refer to frequently as the ``Chomp Lemma''.
\begin{lemma}[Chomp Lemma]\label{lemma:chomp}
Suppose \(1/D, 1/x \ll 1/k \leq 1\), and let~\(H\) be a \((k+1)\)-uniform, \((n,D,\eps)\)-regular hypergraph.
Suppose that \(D_2, D_3, \dots, D_{k+1}\) are numbers satisfying \(C_j(H)\leq D_j\), and consider a (permissibly empty) set \(J^{*}\subseteq [k]\setminus\{1\}\).
Suppose further that:
\begin{enumerate}[(C1), topsep = 6pt]
\item \(\frac{\eps^2 D}{D_2}\geq\log^{8}D\) if \(2\in J^{*}\) or \(k=1\);\label{chomphyp:trackingdegratio}
\item \(\frac{\eps^{2}D}{x^{k-2}D_2}\geq\log^8 D\) if \(2\notin J^{*}\) and \(k\geq 2\);\label{chomphyp:nontrackingdegratio}
\item \(\frac{D_j}{xD_{j+1}}\geq\log^{10} D\) if \(j, j+1\in J^{*}\);\label{chomphyp:trackingcodegratio}
\item \(\frac{D_j}{x^{k-j+1}D_{j+1}} \geq \log^9 D\) if \(j\in J^{*}\), \(j+1\notin J^{*}\);\label{chomphyp:nontrackingcodegratio}
\item \(\eps\leq\frac{1}{x\log^2 D}\).\label{chomphyp:eps}
\end{enumerate}
Then there is a matching~\(M\) of~\(H\) and a set \(W\subseteq V(H)\) of size \(|W|\leq 10e^{2}\eps n\log x\), together with numbers~\(n'\) and~\(D'\) satisfying \(\frac{n}{e^{3}x}\leq n'\leq\frac{e^{3}n}{x}\) and \(\frac{D}{e^{2}x^k} \leq D'\leq \frac{e^{2}D}{x^k}\), such that the hypergraph \(H'=H[V(H)\setminus(V(M)\cup W)]\) is \((n', D', \eps')\)-regular, where \(\eps'\coloneqq \eps x\).
Further, \(C_{j}(H')\leq \frac{e^{2}D_j}{x^{k-j+1}}\) for all \(j\in J^{*}\) (and \(C_j(H')\leq D_j\) for all \(j\in[k+1]\setminus(J^{*}\cup\{1\})\)).

If, in addition,~\(\cT\) is a family of functions \(\tau\colon V(H)\rightarrow\mathbb{R}_{\geq0}\) and there is some \(\delta>0\) satisfying \(1/D,1/x\ll\delta<1\) such that the following conditions hold:
\begin{enumerate}[(CP1), topsep = 6pt]
\item \(\eps\tau(V(H))\geq \tau_{\text{max}}\cdot e^3 \log^7 D\) for all \(\tau\in\cT\), where \(\tau_{\text{max}}\coloneqq\max_{v\in V(H)}\tau(v)\);\label{chomppseudhyplowerbound}
\item \(\text{Supp}(\tau)\coloneqq\{v\in V(H)\colon\tau(v)>0\}\) satisfies \(|\text{Supp}(\tau)|\leq D^{\log^{6/5}D}\) for all \(\tau\in\cT\);\label{chomppseudhypsupportsize}
\item Each vertex of~\(H\) is in~\(\text{Supp}(\tau)\) for at most~\(D^{\log^{6/5}D}\) of the functions \(\tau\in\cT\);\label{chomppseudhypmaxinvolvement}
\item \(x^k \leq D^{1-\delta}\),\label{chomppseudhyppolyremainder}
\end{enumerate}
then there exist~\(M, W\) as above additionally satisfying \(\frac{\tau(V(H))}{e^2 x}\leq \tau(V(H'))\leq \frac{e^2 \tau(V(H))}{x}\) and \(\tau(W)\leq 10e^2\eps\tau(V(H))\log x\) for all \(\tau\in\cT\).
\end{lemma}
We use the name ``Chomp'' because Lemma~\ref{lemma:chomp} finds a much larger matching than the Nibble (reducing the leftover to size~\(o(n)\) rather than \(n-o(n)\)), but will itself be iterated many times (with ``small''~\(x\)) in the proof of Theorem~\ref{theorem:maintheorem} in Section~\ref{section:exhausting}.
As discussed in Section~\ref{section:sketch}, the strategy for proving the Chomp Lemma is simply to iterate the Nibble Lemma with \(\theta\coloneqq1/\log^2 D\), \(\lflr \log^2 D \log x\rflr\) times, where~\(x\) is the amount of ``progress'' we make in the current application of the Chomp (in the sense that the leftover size is reduced by a factor~\(x\)).
We use the same set~\(J^{*}\) for all Nibbles within a Chomp (see Section~\ref{section:sketch} for an explanation of the role of~\(J^{*}\)).\COMMENT{It occurs to me that if we found some way to omit the Chomp entirely and instead MCWA with nibbles, we could maybe keep the final error always logarithmic. A bit late in the process to check this fully, though, and I think the current presentation is probably simpler}
The Chomp Lemma implies\COMMENT{Don't even need to say ``up to the error term'', as we get a specific polylog error, whereas Vu claims only error \(\log^C\)} Theorem~\ref{theorem:V00} (we leave the full details up to the reader, but take \(D_j\coloneqq D_s\) for \(j\geq s\), take \(J^{*}\coloneqq[s-1]\setminus\{1\}\), and set \(\eps\approx1/x\)),\COMMENT{Let \(H, x\) satisfy the hypotheses of Theorem~\ref{theorem:V00}. Put \(J^{*}=[s-1]\setminus\{1\}\) and use the same values of \(D_j\) for \(2\leq j\leq s\) and \(D_j\coloneqq D_s\) for \(j>s\). Since \(H\) is \((n,D,1/x)\)-regular, it is \((n,D,\log^4 D/x)\)-regular.
Let \(x'\coloneqq x/\log^{10}D\) and \(\eps'=\log^4D /x\).
We seek to apply Chomp with these values \(\eps', x'\).
The assumption \(\eps\leq 1/x\) now implies \(\eps'=\log^4 D/x = \log^4 D/(x'\log^{10}D)=1/(x'\log^6 D)\leq 1/(x'\log^2 D)\) so~\ref{chomphyp:eps} is satisfied.
The assumption \(x^2\leq D/D_2\) of Vu implies that \((\eps')^2 D/D_2=D\log^8 D/(x^2 D_2)\geq \log^8 D\) so~\ref{chomphyp:trackingdegratio} holds whether \(2\in J^{*}\) or \(k=1\) or otherwise.
We need to check if~\ref{chomphyp:nontrackingdegratio} holds if \(2\notin J^{*}\) and \(k\geq 2\).
But if \(2\notin J^{*}\) then \(s=2\), whence Vu condition~\ref{condition:finalratio} says \(x^{k}\leq D/D_2\).
Then \((\eps')^2D/((x')^{k-2}D_2)=D\log^8 D/(x^2 (x')^{k-2} D_2)\geq D\log^8 D/(x^{k}D_2)\geq \log^8 D\) so~\ref{chomphyp:nontrackingdegratio} holds even without needing \(k\geq 2\).
For all those \(j\in J^{*}\) for which \(j+1\in J^{*}\) we have \(x^2\leq D_j/D_{j+1}\) from Vu~\ref{condition:otherratios}.
Then \(D_j/(x'D_{j+1}) > D_j/((x')^2 D_{j+1})\) (this assumes \(x'\coloneqq x/\log^{10}D >1\), but if \(x\leq\log^{10}D\) then Vu concludes nothing due to the arbitrary polylog, so we may assume this).
Then \(D_j/((x')^2 D_{j+1})=D_j\log^{20}D/(x^2 D_{j+1})\geq\log^{20}D\) so all good.
Finally the only \(j\in J^{*}\) having \(j+1\notin J^{*}\) is \(j=s-1\) for which Vu hypothesizes \(x^{k-s+2}\leq D_{s-1}/D_s\), so that in particular \(D_{s-1}/((x')^{k-s+2}D_{s}) = D_{s-1}\log^{10(k-s+2)}D/(x^{k-s+2}D_s)\geq \log^{10}D\), where we used \(s\leq k+1\). Then Chomp says that in particular, there's a matching covering all but at most \(e^3 n/x' + 10e^2\eps' n\log x=e^3 n\log^{10}D/x + 10e^2n\log x\log^4 D/x\) vertices, so clearly there exists \(A\) as in Vu, using \(x\leq D\). Indeed \(A=11\) suffices and we could do better} and is stronger in the following four ways:
\begin{enumerate}[label=\upshape(\roman*)]
\item Any input~\(\eps\) at least as large as roughly~\(\sqrt{D_2/D}\) (assuming~\(x\) is small or \(2\in J^{*}\), say)\COMMENT{Because of~\ref{chomphyp:nontrackingdegratio}. But I've just mentioned beforehand how we will always go on to apply Chomp with small \(x\)} degrades by the expected amount \(\eps\rightarrow\eps x\) (and the other parameters e.g.\ \(D, D_j\) degrade appropriately also), potentially allowing for many future Chomps.
Theorem~\ref{theorem:V00} makes no conclusion on the degradation of any parameters other than~\(n\), though the expected behaviour does follow from Vu's proof for all parameters other than, crucially,~\(\eps\), as one must set \(\eps\approx1/x\) at the start of Vu's proof, losing any tighter control on the degrees;
\item One is free to choose \(J^{*}\subseteq[k]\setminus\{1\}\) to tolerate/ignore any ``clustering'' of codegrees \(D_j\approx D_{j+1}\) and still proceed, showing those~\(D_j\) having \(D_j\gg D_{j+1}\) still degrade as expected.
This is essential in our handling of clustered codegrees in Section~\ref{section:exhausting};
\item \ref{chomphyp:trackingcodegratio} is a much weaker constraint on~\(x\) than Theorem~\ref{theorem:V00}\ref{condition:otherratios} if the codegree ratios are at least polynomial in~\(D\).
Therefore, if some intermediate ratio \(D_j/D_{j+1}\) for \(2\leq j<s-1\) turns out to be the (unique, say)\COMMENT{The application in the intro where we show Vu can get leftover \(n^{3/4}\) for the triangle factors has these ratios as a bottleneck, but so is \(D/D_2\). The Chomp wouldn't do any better than Vu because of \(D/D_2\) here (though it \textit{can} do better if you adjusted \(D_2=n^{3/2}\rightarrow n^{4/3}\), which is only a better setup because of the improved constraint)} bottleneck for~\(x\) in an application of Theorem~\ref{theorem:V00}, the Chomp Lemma gives\COMMENT{I went for the committal ``gives'', rather than ``should give''. In light of the fact Chomp implies Vu and the assumption that an intermediate ratio is uniquely the bottleneck, this is formally true and not too difficult to show, up to the polylogarithms - but even there, Vu states his result with an arbitrarily large polylogarithm error; for us, the logs appear in the hypotheses rather than the conclusion, but they yield clearly defined values on the log powers, so I think committing to ``larger'' is reasonable} a larger matching;\label{logerror?} 
\item The Chomp Lemma ensures the leftover is well-distributed with respect to a pre-chosen family~\(\cT\) of weight functions.
\end{enumerate}
We remark that if \(\cT=\emptyset\), then the Chomp Lemma yields logarithmic error.\COMMENT{\ref{chomppseudhyppolyremainder} gives polynomial error in some (but not all, to be fair) circumstances}
Our main theorem (Theorem~\ref{theorem:maintheorem}) is later proven in Section~\ref{section:exhausting} by iterating the Chomp Lemma an arbitrarily large number of times, which leads to subpolynomial error in the main theorem.
This arises due to the generality of the main theorem; some codegree sequences can be so poorly behaved that it is difficult to circumnavigate the need to apply the Chomp Lemma so many times, making such small progress each time (see Section~\ref{section:sketch} for a more in-depth sketch).
However, we expect that in some applications with simple codegree sequences\COMMENT{Likely including at least the triangle factors one} and no weight functions to track, one could obtain logarithmic error by applying the Chomp Lemma, say at most~\(k\) times, fully collapsing some ratio~\(D_j/D_{j+1}\) each time.
For this purpose, the optimization detailed in~\ref{logerror?} above would be essential.\COMMENT{Otherwise you're forever halving (the power of) the space available in the intermediate ratios, so you'd need to go logarithmically many times to get to the right leading term, by which time the errors grow out of control}
\lateproof{Lemma~\ref{lemma:chomp}}
Since~\(H\) has at least some edge\COMMENT{We're assuming \(1/D\ll 1\)}, we have \(D_j\geq C_j(H)\geq 1\) for all \(j\in[k+1]\setminus\{1\}\).
Set \(\theta\coloneqq 1/\log^2 D\) and \(T\coloneqq \lfloor \frac{1}{\theta}\log x\rfloor ( = \lfloor \log^2 D \log x\rfloor)\).
Firstly we claim that
\begin{equation}\label{eq:chomp1}
x^k \leq \frac{D}{\log^{12}D}
\end{equation}
holds.
Clearly~(\ref{eq:chomp1}) holds if~\ref{chomppseudhyppolyremainder} holds, but we show that \ref{chomphyp:trackingdegratio}--\ref{chomphyp:eps} suffice.\COMMENT{It is desirable to obtain the first part of the lemma without the inclusion of~\(\cT\) in its own right, as the first part yields logarithmic error in all cases, whereas~\ref{chomppseudhyppolyremainder} means one can do no better than polynomial error in some cases. The case in which \(x\) is so large that \ref{chomppseudhyppolyremainder} fails but \ref{chomphyp:trackingdegratio}--\ref{chomphyp:eps} hold is the case in which, annoyingly, we can't ensure that each vertex of \(H_i\) in at most \((D^{(i)})^{\log^{5/4}(D^{(i)})}\) of the supports. Indeed, every surviving vertex is always in as many of the supports as it always was, i.e. possibly up to \(D^{\log^{6/5}D}\) as per the hypothesis of Chomp, so one now needs \(D^{\log^{6/5}D}\leq (D^{(i)})^{\log^{5/4}(D^{(i)})}\). But if \(D^{(i)}\) is really small compared to \(D\), say \(D^{(i)}=D^{(T)}\approx D/x^k=\) polylog D, then one doesn't have this inequality at all. We do have it if \(D/x^k\) is still polynomial in \(D\), which is precisely the inspiration for~\ref{chomppseudhyppolyremainder}. This is no concern to us in the main theorem since we only aim for polynomial error there anyway, but it did make the Chomp slightly awkward since the Chomp can otherwise get logarithmic error.} 
Indeed, if \(k=1\), then\COMMENT{Used \(D_2 \geq 1\).}
\[
D \geq\frac{D}{D_2} \stackrel{\ref{chomphyp:trackingdegratio}}{\geq} \frac{\log^8 D}{\eps^2} \stackrel{\ref{chomphyp:eps}}{\geq} x^2 \log^{12}D > x^k\log^{12}D.
\]
If \(2\notin J^{*}\) and \(k\geq 2\), then
\[
D\geq\frac{D}{D_2} \stackrel{\ref{chomphyp:nontrackingdegratio}}{\geq} \frac{x^{k-2}\log^8 D}{\eps^2}\stackrel{\ref{chomphyp:eps}}{\geq} x^k \log^{12}D.
\]
Finally, if \(2\in J^{*}\),\COMMENT{which, in particular, means \(k\geq 2\) by definition of \(J^{*}\)} then let the elements of~\(J^{*}\) be \(j_1=2, j_2, j_3, \dots, j_{\ell}\) in increasing order, and let~\(r\) be the smallest \(i\in[\ell]\) such that \(j_i + 1\notin J^{*}\) (such~\(r\) must exist since~\(J^{*}\) is non-empty\COMMENT{It contains \(2\)} and \(k+1\notin J^{*}\) by construction).
Then\COMMENT{Used \(D_{j_{r}+1}\geq 1\). If \(r=1\) (i.e. \(2\in J^{*}, 3\notin J^{*}\)) then there's no intermediate ratios where we're tracking both codegrees, so we never apply \ref{chomphyp:trackingcodegratio}. Mathematically, \(j_r - 1=j_1 - 1 = 2-1=1\), so \(\prod_{i=2}^{j_r - 1}\frac{D_{i}}{D_{i+1}}\) is an empty product, which we identify with 1, and \(x^{j_r -2} = x^{j_1 -2} = x^{2-2}=1\), so that's consistent. All the polylog factors from applying \ref{chomphyp:trackingcodegratio} and \ref{chomphyp:nontrackingcodegratio} have been bounded below by 1; they're not needed and since as just discussed there may be no actual application of \ref{chomphyp:trackingcodegratio}, there may be no such polylogs, so not using them was necessary as well as clean. (Though the single application of \ref{chomphyp:nontrackingcodegratio} does occur, and yields polylogs that I bounded below by 1.)} 
\begin{eqnarray*}
D & \geq & \frac{D}{D_{j_{r}+1}}=\frac{D}{D_2}\prod_{i=2}^{j_r}\frac{D_i}{D_{i+1}} \stackrel{\ref{chomphyp:trackingdegratio}}{\geq} \frac{\log^8 D}{\eps^2}\cdot\frac{D_{j_{r}}}{D_{j_{r}+1}}\prod_{i=2}^{j_{r}-1}\frac{D_i}{D_{i+1}} \\ & \stackrel{\ref{chomphyp:trackingcodegratio}, \ref{chomphyp:nontrackingcodegratio}}{\geq} & \frac{\log^8 D}{\eps^2}\cdot x^{k-j_{r}+1}x^{j_{r}-2} \stackrel{\ref{chomphyp:eps}}{\geq} x^2 \log^{12}D \cdot x^{k-1} > x^{k}\log^{12}D.
\end{eqnarray*}
We deduce that~(\ref{eq:chomp1}) holds, as claimed.
Notice\COMMENT{\(D^{1/\log D} = \exp(\log (D^{1/\log D}))=\exp((1/\log D)\cdot\log D)=\exp(1)=e\)} also that\COMMENT{This is what lets us bound our errors by multiplicative constants (powers of \(e\)). I originally tried the whole proof with \(\theta=1/\log D\) and found I couldn't get small enough error, but forcing \(\theta\) to be just a bit smaller at \(\theta=1/\log^2 D\) means our error stays under control, basically because of (\ref{eq:chomp2}). It makes some sense to me that we can't just take \(\theta\) to be as large as we like without the errors becoming uncontrolled. Taking this in the other direction, we can get multiplicative error \(1+o(1)\) in the expressions for \(n', D'\) in the output of Chomp, if we put say \(\theta=1/\log^4 D\).}
\begin{equation}\label{eq:chomp2}
x^{\sqrt{\theta}} \stackrel{(\ref{eq:chomp1})}{\leq} D^{\sqrt{\theta}} = D^{1/\log D} = e.
\end{equation}
For \(i\in[T]_0\) and \(j\in J^{*}\), we set:

\begin{minipage}{.5\linewidth}
 \begin{eqnarray*}
    n_i^{(-)} &\coloneqq& n(1-\theta -2\theta^{3/2})^i,\\
    D^{(i,-)} &\coloneqq& D(1-k\theta - \theta^{3/2})^i,\\
    \eps_{(i)} &\coloneqq& \eps(1+\theta)^i,
  \end{eqnarray*}
\end{minipage}%
\begin{minipage}{.5\linewidth}
  \begin{eqnarray*}
    n_i^{(+)} &\coloneqq& n(1-\theta + 2\theta^{3/2})^i,\\
    D^{(i,+)} &\coloneqq& D(1-k\theta + \theta^{3/2})^i,\\
    D_j^{(i)} &\coloneqq& D_j(1-(k-j+1)\theta + \theta^{3/2})^i,
  \end{eqnarray*}
\end{minipage}
Observe that \((n_i^{(-)}), (n_i^{(+)}), (D^{(i,-)}), (D^{(i,+)}), (D_j^{(i)})\) are strictly decreasing sequences\COMMENT{Since \(\theta^{3/2}\ll\theta\) (because \(1/D\ll 1\) and \(\theta=1/\log^2 D\)), we subtract more than we add in the parameters on the right of the display. \(j\leq k\) for \(j\in J^{*}\) so we do subtract at least 1 \(\theta\) each time in the definition of \(D_j^{(i)}\).} and \((\eps_{(i)})\) is strictly increasing.
We now collect some bounds on our compounding error terms that will be useful.
\begin{equation}
\begin{split}
1-k\theta-\theta^{3/2} &\geq 1-(k\theta+2\theta^{3/2})+(k\theta+2\theta^{3/2})^2 \geq\exp(-(k\theta+2\theta^{3/2})),\label{eq:chomperror}\\
1-\theta-2\theta^{3/2} &\geq 1-(\theta+3\theta^{3/2})+(\theta+3\theta^{3/2})^2 \geq\exp(-(\theta+3\theta^{3/2})),\\
1+\theta &\geq 1 + \theta(1-\theta) + \theta^{2}(1-\theta)^2\geq\exp(\theta-\theta^2),\\
1 - (k-j+1)\theta + \theta^{3/2} & \geq 1- (k-j+1)\theta + (k-j+1)^2 \theta^2 \geq\exp(-\theta(k-j+1)), 
\end{split}
\end{equation}
where \(j\in J^{*}\) (so \(j\leq k\)).
\begin{claim}\label{claim:nibblechomp}
Suppose \(i\in[T-1]_0\) and~\(H_i\) is an \((n_i, D^{(i)}, \eps_{(i)})\)-regular, \((k+1)\)-uniform hypergraph for some numbers \(n_i^{(-)}\leq n_i \leq n_i^{(+)}\) and \(D^{(i,-)} \leq D^{(i)} \leq D^{(i,+)}\), further satisfying \(C_j(H_i)\leq D_j^{(i)}\) for all \(j\in J^{*}\) and \(C_j(H_i)\leq D_j\) for all \(j\in[k+1]\setminus(J^{*}\cup\{1\})\).
Then there is a matching~\(M_i\) of~\(H_i\) and a set \(W_i\subseteq V(H_i)\) of size \(|W_i|\leq 10\eps_{(i)}\theta n_i\), together with numbers \(n_{i+1}^{(-)}\leq n_{i+1} \leq n_{i+1}^{(+)}\) and \(D^{(i+1,-)}\leq D^{(i+1)} \leq D^{(i+1,+)}\) such that \(H_{i+1}\coloneqq H_i[V(H_i)\setminus(V(M_i)\cup W_i)]\) is \((n_{i+1}, D^{(i+1)}, \eps_{(i+1)})\)-regular, and further satisfies \(C_j(H_{i+1})\leq D_j^{(i+1)}\) for all \(j\in J^{*}\) (and \(C_j(H_{i+1})\leq D_j\) for all \(j\in [k+1]\setminus(J^{*}\cup\{1\})\)).

If, in addition,~\(\cT_i\) is a family of functions \(\tau_i\colon V(H_i)\rightarrow\mathbb{R}_{\geq0}\) such that \(\eps_{(i)}\theta\tau_i(V(H_i))\geq(\max_{v\in V(H_i)}\tau_i(v))\cdot \log^5(D^{(i)})\) for all \(\tau_i\in\cT_i\), and \(|\{v\in V(H_i)\colon\tau_i(v)>0\}|\leq (D^{(i)})^{\log^{5/4}(D^{(i)})}\) for all \(\tau_i\in\cT_i\), and each vertex of~\(H_i\) is in \(\{v\in V(H_i)\colon\tau_i(v)>0\}\) for at most~\((D^{(i)})^{\log^{5/4}(D^{(i)})}\) of the functions \(\tau_i\in\cT_i\), then there exist~\(M_i, W_i\) as above additionally satisfying \(\tau_i(V(H_i))(1-\theta-\theta^{3/2})\leq \tau_i(V(H_{i+1}))\leq \tau_i(V(H_i))(1-\theta+\theta^{3/2})\) and \(\tau_i(W_i)\leq 10\eps_{(i)}\theta \tau_i(V(H_i))\) for all \(\tau\in\cT_i\).
\end{claim}
\claimproof
We seek to apply Lemma~\ref{lemma:nibble} to~\(H_i\) with \(n_i, D^{(i)}, \eps_{(i)}, \theta=1/\log^2 D, D_j^{(i)}\,\, (j\in J^{*}), D_j\,\, (j\notin J^{*}), \cT_i\) playing the roles of \(n, D, \eps, \theta, D_j,\cT\) respectively.
We first check that we may assume \(1/D^{(i)}, \theta \ll 1/k\).
Since \(\theta=1/\log^2 D\), we may clearly assume \(\theta\ll 1/k\), using the assumption \(1/D\ll1/k\) of Lemma~\ref{lemma:chomp}.
Further,\COMMENT{Here using the earlier observation that \((D^{(i,-)})\) is a (strictly) decreasing sequence. Also used \(T-1 \leq T= \lfloor\frac{1}{\theta}\log x\rfloor \leq \frac{1}{\theta}\log x\).}
\begin{eqnarray}
D^{(i)} & \geq & D^{(i,-)}\geq D^{(T,-)} \stackrel{(\ref{eq:chomperror})}{\geq} D\exp(-T(k\theta+2\theta^{3/2}))\nonumber\\ & \geq & D\exp(-(\log x)(k+2\sqrt{\theta})) \stackrel{(\ref{eq:chomp2})}{\geq} \frac{D}{e^2 x^k} \stackrel{(\ref{eq:chomp1})}{\geq} \log^{11} D,\label{eq:chompDilarge}
\end{eqnarray}
so we may also assume \(1/D^{(i)}\ll1/k\).
We now check \ref{hypothesis:degratio}--\ref{hypothesis:epsleqtheta}.
To that end, if \(2\in J^{*}\), notice that\COMMENT{The first inequality uses \(D^{(i)}\geq D^{(i,-)}\). Used (\ref{eq:chomperror}) for the numerator bounds. For the denominator bound, used simply \(1+z\leq e^z\), i.e. \(1+ (-(k-1)\theta + \theta^{3/2})\leq \exp(-(k-1)\theta +\theta^{3/2}).\) Then \(2i\theta - 2i\theta^2 -ik\theta -2i\theta^{3/2}+i(k-1)\theta-i\theta^{3/2}=i\theta(2-k+k-1)-3i\theta^{3/2}-2i\theta^2 = i\theta(1 -3\theta^{1/2}-2\theta)\geq i\theta/2\geq 0\).}\COMMENT{At the end also used \(D^{(i)}\leq D^{(i,+)}\leq D^{(0,+)}\) since \((D^{(i,+)})\) is (strictly) decreasing, and \(D^{(0,+)}=D\).}
\begin{eqnarray*}
\frac{\eps_{(i)}^2\theta D^{(i)}}{D_2^{(i)}} & \geq & \frac{\eps^2 (1+\theta)^{2i}D(1-k\theta-\theta^{3/2})^i}{(\log^2 D)\cdot D_2(1-(k-1)\theta +\theta^{3/2})^i} \\ & \stackrel{(\ref{eq:chomperror})}{\geq} & \frac{\eps^2 D\exp(2i\theta(1-\theta)-i(k\theta+2\theta^{3/2}))}{(\log^2 D)\cdot D_2\exp(-i(k-1)\theta + i\theta^{3/2})} \geq \frac{\eps^2 D}{D_2\log^2 D} \stackrel{\ref{chomphyp:trackingdegratio}}{\geq} \log^5 D \\ & \geq & \log^5 D^{(i)}.
\end{eqnarray*}
As an aside, notice the excess~\(\exp(i\theta)\) term, which indicates the left side is minimum at the start of the iteration process and grows throughout; this is the expected behaviour since by the end of the process we expect roughly \(\eps\rightarrow\eps x\), \(D\rightarrow D/x^k\), and \(D_2\rightarrow D_2/x^{k-1}\) since \(2\in J^{*}\), so \(\eps^2 D/D_2\rightarrow \eps^2 Dx/D_2 \approx \eps^2 D\exp(T\theta)/D_2\).

If instead \(2\notin J^{*}\), notice that
\begin{eqnarray}\label{eq:chompcheckN1}
\frac{\eps_{(i)}^2 \theta D^{(i)}}{D_2} & \geq & \frac{\eps^2 (1+\theta)^{2i}D(1-k\theta-\theta^{3/2})^i}{D_2\log^2 D} \stackrel{(\ref{eq:chomperror})}{\geq} \frac{\eps^2 D}{D_2\log^2 D}\exp(2i\theta - 2i\theta^2 -ik\theta - 2i\theta^{3/2})\nonumber\\ & \geq & \frac{\eps^2 D}{D_2\log^2 D}\exp(i\theta(2-k)-4(T-1)\theta^{3/2})\nonumber\\ & \geq & \frac{\eps^2 D}{D_2\log^2 D}\exp(i\theta(2-k)-4\sqrt{\theta}\log x) \stackrel{(\ref{eq:chomp2})}{\geq}\frac{\eps^2 De^{i\theta(2-k)}}{e^4 D_2\log^2 D}.
\end{eqnarray}
If \(k=1\), then the right-hand side of~(\ref{eq:chompcheckN1}) is minimized when \(i=0\), so it suffices to see that \(\eps^2 D/(e^4 D_2 \log^2 D) \geq \log^5 D \geq \log^5 D^{(i)}\) by~\ref{chomphyp:trackingdegratio}.
If instead \(k\geq 2\), then the right-hand side of~(\ref{eq:chompcheckN1}) is minimized when \(i=T-1\), so the right-hand side of~(\ref{eq:chompcheckN1}) is at least\COMMENT{\(2-k=-(k-2)\)}
\[
\frac{\eps^2 D\exp(-(T-1)\theta(k-2))}{D_2\log^3 D} \geq \frac{\eps^2 D \exp(-(k-2)\log x)}{D_2\log^3 D} \stackrel{\ref{chomphyp:nontrackingdegratio}}{\geq} \log^5 D \geq\log^5 D^{(i)}.
\]
We deduce that~\ref{hypothesis:degratio} holds.
To check~\ref{hypothesis:codegratio}, fix \(j\in J^{*}\).
If \(j+1\in J^{*}\), notice that
\begin{eqnarray*}
\frac{\theta^2 D_j^{(i)}}{D_{j+1}^{(i)}} & = & \frac{D_j (1-(k-j+1)\theta + \theta^{3/2})^i}{(\log^4 D)\cdot D_{j+1}(1-(k-j)\theta + \theta^{3/2})^i} \stackrel{(\ref{eq:chomperror})}{\geq} \frac{D_j\exp(-i\theta(k-j+1))}{D_{j+1}\exp(-i\theta(k-j)+i\theta^{3/2})\log^4 D}\\ & = & \frac{D_j \exp(-i\theta-i\theta^{3/2})}{D_{j+1}\log^4 D} \geq \frac{D_j\exp(-(T-1)\theta(1+\sqrt{\theta}))}{D_{j+1}\log^4 D} \geq \frac{D_j e^{-(1+\sqrt{\theta})\log x}}{D_{j+1}\log^4 D}\\ & \stackrel{(\ref{eq:chomp2})}{\geq} & \frac{D_j}{exD_{j+1}\log^4 D} \stackrel{\ref{chomphyp:trackingcodegratio}}{\geq} \log^5 D \geq\log^5 D^{(i)}.
\end{eqnarray*}
If instead \(j+1\notin J^{*}\), then notice that
\begin{eqnarray*}
\frac{\theta^2 D_j^{(i)}}{D_{j+1}} & = & \frac{D_j (1-(k-j+1)\theta +\theta^{3/2})^i}{(\log^4 D)\cdot D_{j+1}} \stackrel{(\ref{eq:chomperror})}{\geq} \frac{D_j \exp(-i\theta(k-j+1))}{D_{j+1}\log^4 D}\\ & \geq & \frac{D_j \exp(-T\theta(k-j+1))}{D_{j+1}\log^4 D} \geq \frac{D_j}{x^{k-j+1}D_{j+1}\log^4 D} \stackrel{\ref{chomphyp:nontrackingcodegratio}}{\geq} \log^5 D \geq\log^5 D^{(i)}.
\end{eqnarray*}
We deduce that~\ref{hypothesis:codegratio} holds.
Finally, notice
\begin{equation}\label{eq:chompcheckN3}
\eps_{(i)}=\eps(1+\theta)^i \leq \eps e^{\theta T} \leq \eps x \stackrel{\ref{chomphyp:eps}}{\leq}\frac{1}{\log^2 D} = \theta,
\end{equation}
so~\ref{hypothesis:epsleqtheta} holds.
Further, if we have some~\(\cT_i\) as in the claim statement, then clearly~\ref{pseudhyp:lowerbound}--\ref{pseudhyp:maxinvolvement} are satisfied.
Now, Lemma~\ref{lemma:nibble} applied to~\(H_i\) yields the claim.\COMMENT{The size of the waste set is a direct translation. Lemma~\ref{lemma:nibble} gives \(n_{i+1}\leq n_i (1-\theta+2\theta^{3/2})\leq n_{i}^{(+)}(1-\theta +2\theta^{3/2})=n_{i+1}^{(+)}\). The other bounds in the output of Lemma~\ref{lemma:nibble} are the same, it's clear from the way I've defined \(n_{i}^{(-)}, D^{(i,-)}, \eps_{(i)}\) etc that the output just feeds into the parameters for timestep \(i+1\).}
\endclaimproof
We now finish the proof of Lemma~\ref{lemma:chomp} in the case we do not have a family~\(\cT\) of weight functions as in the second part of the statement, and return to the case in which there is such a~\(\cT\) afterward.
Notice \(H_0\coloneqq H\) is \((n_0, D^{(0)}, \eps_{(0)})\)-regular and \((k+1)\)-uniform, where \(n_0^{(-)}=n_0=n_0^{(+)}\) and \(D^{(0,-)}=D^{(0)}=D^{(0,+)}\), and further satisfies \(C_j(H_0)\leq D_j=D_j^{(0)}\) for \(j\in J^{*}\) and \(C_j(H_0)\leq D_j\) for \(j\in[k+1]\setminus (J^{*}\cup\{1\})\).
Applying Claim~\ref{claim:nibblechomp},~\(T\) times successively beginning with~\(H_0\), yields hypergraphs~\((H_i)_{i\in[T]_0}\), matchings~\((M_i)_{i\in[T-1]_0}\), and sets~\((W_i)_{i\in[T-1]_0}\).
Set \(H'\coloneqq H_T\), \(M\coloneqq\bigcup_{i\in[T-1]_0}M_i\), and \(W\coloneqq\bigcup_{i\in[T-1]_0}W_i\).
It is clear by construction that \(H'=H[V(H)\setminus(V(M)\cup W)]\).
We have
\begin{eqnarray}\label{eq:chompWsmall}
|W| & = & \sum_{i\in[T-1]_0}|W_i|\leq\sum_{i\in[T-1]_0}10\eps_{(i)}\theta n^{(i,+)}=\sum_{i\in[T-1]_0}10\eps (1+\theta)^i \theta n(1-\theta +2\theta^{3/2})^i\nonumber\\ & \leq & 10\eps\theta n\sum_{i\in[T-1]_0}e^{2i\theta^{3/2}} \leq 10\eps\theta nTe^{2T\theta^{3/2}} \leq 10\eps n\log x \cdot e^{2(\log x)\sqrt{\theta}} \stackrel{(\ref{eq:chomp2})}{\leq} 10e^2 \eps n\log x,
\end{eqnarray}
as required.
Set \(n'\coloneqq n_T\), \(D'\coloneqq D^{(T)}\), and recall that \(\eps'\coloneqq \eps x\geq \eps_{(T)}\) by~(\ref{eq:chompcheckN3}), so that~\(H'\) is \((n',D',\eps')\)-regular.
For \(j\in J^{*}\), we have \(C_j(H') \leq D_j^{(T)}\leq D_j\exp(-\theta T(k-j+1)+\theta^{3/2}T)\leq D_j\exp(-\theta((\log x)/\theta -1)(k-j+1)+\sqrt{\theta}\log x) \leq  D_j\exp(-(k-j+1)\log x + 2\sqrt{\theta}\log x) \leq e^2 D_j/x^{k-j+1}\),\COMMENT{Used \(\theta(k-j+1) \leq\sqrt{\theta}\log x\). All this requires is \(\log x \geq 1\) and \(\theta\ll 1/k\), which we have.} as required, by~(\ref{eq:chomp2}). 
The calculations to show that \(n/(e^3 x)\leq n'\leq e^3 n/x\) and \(D/(e^2 x^k)\leq D'\leq e^2 D/x^k\) are similar, so we omit them.\COMMENT{\(n'\geq n_T^{(-)}\stackrel{(\ref{eq:chomperror})}{\geq} n\exp(-T(\theta + 3\theta^{3/2}))\geq n\exp(-\log x (1+3\sqrt{\theta})) \stackrel{(\ref{eq:chomp2})}{\geq}n/(e^3 x).\) Also \(n'\leq n_T^{(+)}\leq n\exp(-T\theta+2T\theta^{3/2})\leq n\exp(-((\log x)/\theta -1)\theta +2(\log x)\theta^{1/2})\leq n\exp(-\log x+3\sqrt{\theta}\log x) \stackrel{(\ref{eq:chomp2})}{\leq} e^3 n/x,\) where we used \(\theta\leq \sqrt{\theta}\log x\) as in previous comment. Also \(D' \geq D^{(T,-)}\stackrel{(\ref{eq:chomperror})}{\geq} D\exp(-T(k\theta+2\theta^{2/3}))\geq D\exp(-\log x(k+2\sqrt{\theta}))\stackrel{(\ref{eq:chomp2})}{\geq} D/(e^2 x^k)\). Finally \(D'\leq D^{(T,+)} \leq D\exp(-k\theta T +\theta^{3/2}T)\leq D\exp(-k\theta((\log x)/\theta -1) +\sqrt{\theta}\log x)\leq D\exp(-k\log x +2\sqrt{\theta}\log x) \stackrel{(\ref{eq:chomp2})}{\leq} e^2 D/x^k\).}

Suppose instead that there is additionally a family~\(\cT\) and a constant~\(\delta\) satisfying \ref{chomppseudhyplowerbound}--\ref{chomppseudhyppolyremainder}.
For each \(\tau\in\cT\) and \(i\in[T]_0\), set \(m_{\tau}^{(i,+)}=\tau(V(H))(1-\theta+\theta^{3/2})^i\) and \(m_{\tau}^{(i,-)}=\tau(V(H))(1-\theta-\theta^{3/2})^i\).
Clearly each \(\tau\in\cT\) satisfies \(m_{\tau}^{(0,-)}=\tau(V(H))=m_{\tau}^{(0,+)}\), and hypotheses \ref{chomppseudhyplowerbound}--\ref{chomppseudhypmaxinvolvement} ensure the first iteration of Claim~\ref{claim:nibblechomp} proceeds since \(D^{(0)}=D\), \(\eps_{(0)}=\eps\), \(\theta=1/\log^2 D\), and yields that each \(\tau\in\cT\) satisfies \(m_{\tau}^{(1,-)}\leq \tau(V(H_1))\leq m_{\tau}^{(1,+)}\). 
Similarly, if each of the~\(T\) iterations of Claim~\ref{claim:nibblechomp} (recursively making~\(\tau_{i+1}\) the obvious restriction of~\(\tau_{i}\) each time; that is, set \(\tau_0=\tau\) and set~\(\tau_{i+1}\) to be the restriction of~\(\tau_i\) to the domain~\(V(H_{i+1})\)) proceed, then each \(\tau=\tau_0\in\cT\) satisfies\COMMENT{\(m_{\tau}^{(T,-)}=\tau(V(H))(1-\theta-\theta^{3/2})^T=\tau(V(H))(1-(\theta+\theta^{3/2}))^T\geq\tau(V(H)) (1-(\theta+2\theta^{3/2})+(\theta+2\theta^{3/2})^2)^T\geq \tau(V(H))\exp(-T(\theta+2\theta^{3/2}))\geq\tau(V(H))\exp(-(1+2\sqrt{\theta})\log x)=\tau(V(H))/x^{1+2\sqrt{\theta}}\geq \tau(V(H))/e^2 x\). Also \(m_{\tau}^{(T,+)}=\tau(V(H))(1-\theta+\theta^{3/2})^T\leq \tau(V(H))\exp(-T(\theta-\theta^{3/2}))\leq \tau(V(H))\exp(-((\log x)/\theta -1)(\theta-\theta^{3/2}))=\tau(V(H))\exp(-\log x+\theta^{1/2}\log x + \theta -\theta^{3/2})\leq\tau(V(H))\exp(-\log x + 2\sqrt{\theta}\log x)\leq e^2 \tau(V(H))/x\).} \(\tau(V(H))/e^2 x \leq m_{\tau}^{(T,-)}\leq \tau_T(V(H_T))\leq m_{\tau}^{(T,+)}\leq e^2 \tau(V(H))/x\), as required since \(\tau_T(V(H_T))=\tau(V(H_T))\).
To show that the~\(T\) iterations of Claim~\ref{claim:nibblechomp} proceed, we still need to prove (inductively) that:
\begin{enumerate}[(CP{i}1), topsep = 6pt]
\item \(\eps_{(i)}\theta\tau_i(V(H_i))\geq(\max_{v\in V(H_i)}\tau_i(v))\cdot\log^5 (D^{(i)})\) for all \(i\in[T-1]_0, \tau_i\in\cT_i\);\label{situatedpseudsetsize}
\item \(|\{v\in V(H_i)\colon\tau_i(v)>0\}|\leq(D^{(i)})^{\log^{5/4}(D^{(i)})}\) for all \(i\in[T-1]_0, \tau_i\in\cT_i\);\label{situatedpseudfamilysize}
\item Each vertex of~\(H_i\) is in~\(\{v\in V(H_i)\colon\tau_i(v)>0\}\) for at most~\((D^{(i)})^{\log^{5/4}(D^{(i)})}\) of the functions \(\tau_i\in\cT_i\), for each \(i\in[T-1]_0\).\label{situatedpseudinvolvement}
\end{enumerate}
As already noted, the \(i=0\) case of the above holds, so fix \(i\in[T-1]\) and \(\tau\in\cT\).
Recall \(\theta=1/\log^2 D\) and clearly \(\max_{v\in V(H_i)}\tau_i(v)\leq \max_{v\in V(H)}\tau(v)\eqqcolon\tau_{\text{max}}\), so~\ref{situatedpseudsetsize} clearly holds for~\(\tau\) if \(\eps_{(i)}m_{\tau}^{(i,-)}\geq\tau_{\text{max}}\log^7 D\) holds (inductively using \(\tau(V(H_i))\geq m_{\tau}^{(i,-)}\)), but we have \(\eps_{(i)}m_{\tau}^{(i,-)}=\eps(1+\theta)^i \tau(V(H))(1-\theta-\theta^{3/2})^i \geq \eps\tau(V(H))/e^3\), so~\ref{situatedpseudsetsize} follows from~\ref{chomppseudhyplowerbound}.\COMMENT{Notice \((1+\theta)^i (1-\theta-\theta^{3/2})^i \geq (1+\theta(1-\theta)+\theta^{2}(1-\theta)^{2})^i(1-(\theta+2\theta^{3/2})+(\theta+2\theta^{3/2})^2 )^i \geq \exp(i\theta(1-\theta) - i(\theta+2\theta^{3/2}))=\exp(-2i\theta^{3/2}-i\theta^2) \geq \exp(-3T\theta^{3/2})\geq \exp(-3\sqrt{\theta}\log x)\stackrel{(\ref{eq:chomp2})}{\geq} 1/e^3\).}
Moreover, clearly \(|\{v\in V(H_i)\colon\tau_i(v)>0\}|\leq|\{v\in V(H)\colon\tau(v)>0\}|\), and each vertex of~\(H_i\) is in the same number of the sets~\(\{v\in V(H_i)\colon\tau_i(v)>0\}\) for all \(0\leq j\leq i\), so~\ref{situatedpseudfamilysize},~\ref{situatedpseudinvolvement} follow from~\ref{chomppseudhypsupportsize},~\ref{chomppseudhypmaxinvolvement} if \(D^{\log^{6/5}D}\leq(D^{(i)})^{\log^{5/4}(D^{(i)})}\) for all \(i\in[T-1]_0\), which holds since \(D^{(i)}\stackrel{(\ref{eq:chompDilarge})}{\geq} D/e^2x^k \stackrel{\ref{chomppseudhyppolyremainder}}{\geq} D^{\delta}/e^2\geq D^{\delta/2}\), and therefore \((D^{(i)})^{\log^{5/4}(D^{(i)})}\geq D^{(\delta/2)\cdot((\delta/2)^{5/4}\log^{5/4}D)}\geq D^{\log^{6/5}D}\).

It remains only to show that \(\tau(W)\leq 10e^2 \eps\tau(V(H))\log x\) for each \(\tau\in\cT\).
For any such~\(\tau\), from the outputs of each application of Claim~\ref{claim:nibblechomp}, we have
\[
\tau(W)=\sum_{i\in[T-1]_0}\tau(W_i) \leq \sum_{i\in[T-1]_0}10\eps_{(i)}\theta\tau(V(H_i))\leq\sum_{i\in[T-1]_0}10\eps(1+\theta)^i \theta \tau(V(H))(1-\theta+\theta^{3/2})^i,
\]
which is at most~\(10e^2 \eps \tau(V(H))\log x\) as claimed\COMMENT{Indeed the \(e^2\) could be replaced with \(e\) due to one fewer appearance of \(\theta^{3/2}\)}, by the same logic as in~(\ref{eq:chompWsmall}).
\endproof
\section{Exhausting the codegree sequence}\label{section:exhausting}
The aim of this section is to prove Theorem~\ref{theorem:maintheorem}, which we restate now for convenience.

\mainThm*
We note that one can obtain the \(k=1\) case of Theorem~\ref{theorem:maintheorem} via Vizing's Theorem in the case that \(H\) is a simple graph (\(D_2=1\)) and \(\cT=\emptyset\).\COMMENT{If \(\eps<1/D\) then \(G\) is \(D\)-regular so in any optimal colouring there is a colour class with at least \(e(G)/(D+1)=nD/2(D+1)\) edges, covering at least \(nD/(D+1)\) vertices, leaving at most \(n(1-D/(D+1))=n/(D+1)\leq n/D\leq n/\sqrt{D}\) vertices uncovered. But since \(1/\eps > D>\sqrt{D}\) we have \(\min(\sqrt{D},1/\eps)=\sqrt{D}\) so \(B\leq \sqrt{D}\) by hypothesis so \(n/\sqrt{D}\leq n/B\leq nB^{-1+\gamma}\log^A D\) (here used \(B\geq 1\) so \(B^{\gamma}\geq 1\)). So now suppose \(\eps\geq 1/D\), whence \((1+\eps)D+1\leq 1+2\eps D\). Now any optimal colouring has a colour class with at least \(e(G)/((1+\eps)D+1)\geq (1-\eps)Dn/2(1+2\eps)D\geq\frac{n}{2}(1-3\eps)\) edges, thus leaving at most \(3\eps n\) vertices uncovered. If \(\min(\sqrt{D},1/\eps)=1/\eps\geq B\geq 1\), then \(3\eps n\leq 3nB^{-1}\leq nB^{-1+\gamma}\log^A D\). If \(\min(\sqrt{D},1/\eps)=\sqrt{D}\geq B\geq 1\) then \(3\eps n\leq 3n/\sqrt{D}\leq 3nB^{-1}\leq nB^{-1+\gamma}\log^A D\).}
However, since it is unclear how to easily extend this result to the case of multigraphs and \(\cT\neq\emptyset\), we still provide a proof of the \(k=1\) case of Theorem~\ref{theorem:maintheorem} using our machinery (Lemma~\ref{lemma:chomp}).
It is convenient for us to separate the \(k=1\) case, as the strategy is slightly different; specifically the hypotheses of the Chomp Lemma are not the same for \(k=1\) compared to \(k\geq 2\), so a different analysis is needed, and just one application of the Chomp Lemma suffices for \(k=1\), whereas it is not always possible to avoid\COMMENT{Imagine a codegree sequence in which getting the best possible result necessitates collapsing some intermediate codegree ratios and merging the clusters, and then carrying on to make further non-trivial progress. In this case the \(x\) of the Chomp Lemma is bottlenecked by the first codegree ratio to be collapsed, and so clearly we would need multiple applications of the Chomp Lemma in this scenario. If you want good error control in the final result, I think you even need arbitrarily many Chomps, as we do - it wouldn't necessarily suffice to create a different version of MCWA which always identifies the ratio which is currently the bottleneck for \(x\) and sets \(x\) to be that then updates the clusters; there could come a time when \(D/D_2\) is the limiting ratio, and \(D_2\) is in the same cluster as say \(D_3, D_4\), and \(4\in J^{*}\). Then you'd set \(x\) to collapse \(D\) to \(D_2\), after which the process stops. However \(D_2\) didn't drop at all in this application of the Chomp since \(2\notin J^{*}\), and \(D_4\) was in the process of dropping towards \(D_5\) (but didn't reach it as it wasn't the bottleneck ratio). The true MCWA we use would have feathered down each of the \(D_2, D_3, D_4\), keeping \(D_2\) quite close to wherever \(D_4\) drops to by the end of the process, allowing \(D\) to drop non-trivially further (creating a non-trivially larger matching) by the time \(D\) drops to \(D_2\).} using the Chomp Lemma many times for \(k\geq 2\). 

To prove the \(k\geq 2\) case, the rough idea is to apply the Chomp Lemma many times, each with the ``small'' value \(x=B^{\gamma^4}\).
At each timestep, we adjust~\(J^{*}\) to include precisely those indices~\(j\) for which \(D_j/D_{j+1}\) is currently large enough to accommodate putting \(j\in J^{*}\) for a Chomp with this value of~\(x\).
We show that this causes the codegree sequence to behave as a collection of degrading and merging ``clusters''.
This behaviour and the hypothesis on~\(B\) ensure that we may continue the process for essentially~\(\gamma^{-4}\) iterations, which suffices.
For a more in-depth proof sketch, see Section~\ref{section:sketch}.
\lateproof{Theorem~\ref{theorem:maintheorem}}
It suffices to prove Theorem~\ref{theorem:maintheorem} with the inclusion of a family~\(\cT\) as in the second part of the statement as one then obtains the statement without such a~\(\cT\) by setting \(\cT=\emptyset\).
We show that \(A=10/\gamma^4\) suffices.
If \(B\leq \log^{10/\gamma^4}D\) then the theorem is trivial\COMMENT{In this case \(nB^{-1+\gamma}\log^{10/\gamma^4}D \geq nB^{-1}\log^{10/\gamma^4}D\geq n\) (here using the assumption \(B\geq 1\) so that \(B^{\gamma}\geq 1\)), so the empty matching witnesses the first part of the result. Moreover clearly \(\tau(V(H))B^{-1}\leq \tau(V(H))\) (and here is where we had to add the stipulation \(B\geq 1\)) and \(\tau(V(H))B^{-1+\gamma}\log^{A}D\geq \tau(V(H))\), so put \(H'=H\), disjoint from the empty matching \(M\) and the conclusion holds}, so we may assume that
\begin{equation}\label{eq:bisbig}
B\geq\log^{10/\gamma^4}D
\end{equation}
throughout.
We note that we may assume \(D_j\geq 1\) for each \(j\in[k+1]\setminus\{1\}\).\COMMENT{The hypergraph must have at least some edge since \(1/D\ll 1\) and \(B\geq\log D\) implies \(\eps<1/2\), and any \(j\)-subset of this edge witnesses \(D_j\geq C_j(H)\geq 1\).}
Set \(\eps^{*}\coloneqq B^{-1+10k\gamma^{3}}\), and notice that \(\eps^{*}\geq\eps\), so~\(H\) is \((n,D,\eps^{*})\)-regular.

First, assume \(k=1\), so that \(\log^{10/\gamma^4}D\stackrel{(\ref{eq:bisbig})}{\leq} B\leq\min\{\sqrt{D/D_2},1/\eps\}\).
We seek to apply Lemma~\ref{lemma:chomp} to~\(H\) with \(x\coloneqq B^{1-\gamma^2}\), which clearly satisfies \(1/x\ll 1/k=1\) by~(\ref{eq:bisbig}), and with~\(\eps^{*}\) playing the role of~\(\eps\).
We set \(J^{*}=\emptyset\), so we need only check~\ref{chomphyp:trackingdegratio},~\ref{chomphyp:eps}, and~\ref{chomppseudhyplowerbound}--\ref{chomppseudhyppolyremainder}.
To that end, notice that \((\eps^{*})^2 D/D_2 = B^{20k\gamma^3}D/(B^{2}D_2)\geq B^{20k\gamma^3}\geq\log^8 D\) by~(\ref{eq:bisbig}), so~\ref{chomphyp:trackingdegratio} holds.
Further, \(\eps^{*}x = B^{-1+10k\gamma^3 +1-\gamma^2}\leq B^{-\gamma^2 /2}\leq 1/\log^2 D\) by~(\ref{eq:bisbig}), so~\ref{chomphyp:eps} holds.
Fix \(\tau\in\cT\).
With \(\tau_{\text{max}}\coloneqq\max_{v\in V(H)}\tau(v)\), we have
\[
\eps^{*}\tau(V(H))=B^{-1+10k\gamma^3}\tau(V(H))\stackrel{\ref{mainpseud:lowerbound}}{\geq}\tau_{\text{max}}\cdot B^{10k\gamma^3}\stackrel{(\ref{eq:bisbig})}{\geq}\tau_{\text{max}}\cdot e^3 \log^7 D,
\]
so~\ref{chomppseudhyplowerbound} holds.
Further, clearly \(D^{\log D}\leq D^{\log^{6/5}D}\), so \ref{mainpseud:supportsize}--\ref{mainpseud:maxinvolvement} imply \ref{chomppseudhypsupportsize}--\ref{chomppseudhypmaxinvolvement}.
Finally notice \(x^k = B^{1-\gamma^2}\leq D^{1-\gamma^2}\), so~\ref{chomppseudhyppolyremainder} holds with \(\delta\coloneqq \gamma^2\), say.
Then Lemma~\ref{lemma:chomp} yields a matching~\(M\) that covers all but at most \(e^3 n/x + 10e^2 \eps^{*}n\log x \leq e^3 n B^{-1+\gamma^2} + 10e^2 B^{-1+10k\gamma^3}n\log D\leq nB^{-1+\gamma}\log^A D\) vertices of~\(H\), and each \(\tau\in\cT\) satisfies \(\tau(V(H)\setminus V(M))\geq \tau(V(H))/(e^3 x)=\tau(V(H))B^{-1+\gamma^2}/e^3\geq \tau(V(H))B^{-1}\) by~(\ref{eq:bisbig}).
Finally, we have \(\tau(V(H)\setminus V(M))\leq e^3 \tau(V(H))/x +10e^2\eps^{*}\tau(V(H))\log x\), which is at most~\(\tau(V(H))B^{-1+\gamma}\log^A D\).\COMMENT{Compare this to the upper bound on the total number of uncovered vertices above - the expressions are so similar I felt it was OK not to even explicitly mention it}

Now assume \(k\geq 2\).
Set \(\eta\coloneqq\gamma^4\) and \(x\coloneqq B^{\eta}\).
For all \(t\in\mathbb{N}_{0}\) and \(\tau\in\cT\), put

\begin{minipage}{.3\linewidth}
 \begin{eqnarray*}
    n_t^{(-)} &\coloneqq& \frac{n}{(e^3 x)^t},\\
    D^{(t,-)} &\coloneqq& \frac{D}{(e^2 x^k)^t},\\
    m_{\tau}^{(t,-)} &\coloneqq& \frac{\tau(V(H))}{(e^2 x)^t},
  \end{eqnarray*}
\end{minipage}%
\begin{minipage}{.7\linewidth}
  \begin{eqnarray*}
    n_t^{(+)} &\coloneqq& \left(\frac{e^3}{x}\right)^t n,\hspace{20mm}\eps_{(t)} \coloneqq \eps^{*}x^t,\\
    D^{(t,+)} &\coloneqq& \left(\frac{e^2}{x^k}\right)^t D,\\
    m_{\tau}^{(t,+)} &\coloneqq& \left(\frac{e^2}{x}\right)^t \tau(V(H)).
  \end{eqnarray*}
\end{minipage}

Clearly~\((n_t^{(-)}), (n_t^{(+)}), (D^{(t,-)}), (D^{(t,+)}), (m_{\tau}^{(t,-)}), (m_{\tau}^{(t,+)})\) are strictly decreasing\COMMENT{\(x=B^{\eta}\) is at least some polylog \(D\) by~(\ref{eq:bisbig}) so \(e^3 /x \ll 1\)} and~\((\eps_{(t)})\) is strictly increasing.
\begin{claim}\label{claim:tacticalchomp}
Suppose that \(0\leq t<\frac{1}{\eta}-\frac{1}{\gamma^2}\) and~\(H_t\) is a \((n_t, D^{(t)}, \eps_{(t)})\)-regular subhypergraph of~\(H\) for some \(n_{t}^{(-)}\leq n_t \leq n_{t}^{(+)}\) and \(D^{(t,-)}\leq D^{(t)} \leq D^{(t,+)}\), and that \(D_2^{(t)}\geq D_3^{(t)}\geq\dots\geq D_{k+1}^{(t)}\) are numbers satisfying \(C_j(H_t)\leq D_j^{(t)}\) for each \(j\in[k+1]\setminus\{1\}\).
Let \(J_{t}^{*}\subseteq[k]\setminus\{1\}\) be a (permissibly empty) set.
Suppose further that:
\begin{enumerate}[(Ct1), topsep = 6pt]
\item \(\eps_{(t)}^2 \frac{D^{(t,-)}}{D_2^{(t)}} \geq B^{2k\gamma^4}\);\label{tacchompdeghyp}
\item \(\frac{D_j^{(t)}}{D_{j+1}^{(t)}}\geq B^{2k\gamma^4}\,\,\,\text{for all}\,\,j\in J_{t}^{*}\);\label{tacchompcodeghyp}
\item \(m_{\tau}^{(t,-)}\leq \tau(V(H_t))\leq m_{\tau}^{(t,+)}\) for all \(\tau\in\cT\).\label{tacchompsetsizehyp}
\end{enumerate}
Then there is a matching \(M_t\subseteq E(H_t)\) and a set \(W_t\subseteq V(H_t)\) of size \(|W_t|\leq 10e^2 \eps_{(t)}n_t \eta\log B\), together with numbers \(n_{t+1}^{(-)}\leq n_{t+1} \leq n_{t+1}^{(+)}\) and \(D^{(t+1,-)} \leq D^{(t+1)} \leq D^{(t+1,+)}\) such that the hypergraph \(H_{t+1}\coloneqq H_t[V(H_t)\setminus(V(M_t)\cup W_t)]\) is \((n_{t+1}, D^{(t+1)}, \eps_{(t+1)})\)-regular and satisfies \(m_{\tau}^{(t+1,-)}\leq \tau(V(H_{t+1}))\leq m_{\tau}^{(t+1,+)}\) and \(\tau(W_t)\leq 10e^2\eps_{(t)}m_{\tau}^{(t,+)}\eta\log B\) for all \(\tau\in\cT\).
Further, setting \(D_{j}^{(t+1)}\coloneqq D_j^{(t)}\) for all \(j\in[k+1]\setminus(J_t^{*}\cup\{1\})\) and \(D_{j}^{(t+1)}\coloneqq D_j^{(t)}e^2/x^{k-j+1}\) for all \(j\in J_t^{*}\), we have \(C_j(H_{t+1})\leq D_j^{(t+1)}\) for all \(j\in[k+1]\setminus\{1\}\), and \(D_2^{(t+1)}\geq D_3^{(t+1)}\geq\dots\geq D_{k+1}^{(t+1)}\).
\end{claim}
\claimproof
We seek to apply Lemma~\ref{lemma:chomp} to~\(H_t\) with \(n^{(t)}, D^{(t)}, \eps_{(t)}, D_j^{(t)}, J^{*}_t\) playing the roles of \(n, D, \eps, D_j, J^{*}\) respectively, with \(\cT_t=\{\tau_t\colon\tau\in\cT\}\) playing the role of~\(\cT\) where \(\tau_t\) is the restriction of~\(\tau\) to the domain~\(V(H_t)\subseteq V(H)\), and \(\delta\coloneqq 1/2\).
We first check that we may assume \(1/D^{(t)}\ll1/k\).
To that end, notice that\COMMENT{Here, using \(t\leq 1/\gamma^4 - 1/\gamma^2\) so \(e^{2t}\leq\log D\) and \(x^{kt}=B^{k\gamma^4 t}\leq B^{k\gamma^4(1/\gamma^4 - 1/\gamma^2)}=B^{k(1-\gamma^2)}\).}
\begin{equation}\label{eq:tacchomphierarchycheck}
D^{(t)}\geq D^{(t,-)}=\frac{D}{(e^2 x^k)^t}\geq\frac{D}{(\log D)B^{k(1-\gamma^2)}}.
\end{equation}
If \(k=2\), we have\COMMENT{In the \(k=2\) case, the middle expression in the definition of \(B\) contributes no terms, so we use \(B\leq\sqrt{D/D_2}\), then we use the hypothesis \(D_2\geq D_3\)} by hypothesis \(B\leq\sqrt{D/D_2}\leq\sqrt{D/D_3}\), so that \(B^k=B^2\leq D/D_3\), whence the right side of~(\ref{eq:tacchomphierarchycheck}) is at least\COMMENT{Uses \(D_3\geq 1\), as discussed at the start of the proof.} \(D^{\gamma^2}D_3^{1-\gamma^2}/\log D \geq D^{\gamma^2}/\log D\geq D^{\gamma^2 /2}\), whence we may assume \(1/D^{(t)}\ll 1/k\), using the assumption \(1/D\ll\gamma\ll1/k\) of Theorem~\ref{theorem:maintheorem}.
If, instead, \(k\geq 3\), then we use the hypothesis \(B\leq(D/D_{k+1})^{1/k}\) to see that the right side of~(\ref{eq:tacchomphierarchycheck}) is at least \(D^{\gamma^2}D_{k+1}^{1-\gamma^2}/\log D\geq D^{\gamma^2 /2}\), which again suffices,\COMMENT{Used \(D_{k+1}\geq 1\), as discussed at the start of the proof.}\COMMENT{This is where the proof ensures/employs that we can only use at most the total available space in the codegree sequence, i.e. \(D\) is still at least a bit above \(D_{k+1}\), even at the very end.} and we deduce that
\begin{equation}\label{eq:Dtgamma}
D^{(t)}\geq D^{\gamma^2 /2}
\end{equation}
for all \(k\geq 2\).
To check that \(1/x\ll 1/k\), it suffices to notice that \(x=B^{\eta}\geq\log^{10} D\) by~(\ref{eq:bisbig}).

It remains to check \ref{chomphyp:trackingdegratio}--\ref{chomphyp:eps} and \ref{chomppseudhyplowerbound}--\ref{chomppseudhyppolyremainder}.
Since \(k\geq 2\) (and \(x\geq 1\)), we have\COMMENT{Used \(k\geq 1\) and \(D^{(t)}\leq D^{(t,+)}\leq D^{(0,+)}=D\) at the end.}
\begin{eqnarray*}
\frac{\eps_{(t)}^2 D^{(t)}}{D_2^{(t)}}\geq \frac{\eps_{(t)}^2 D^{(t,-)}}{D_2^{(t)}x^{k-2}}\geq \frac{\eps_{(t)}^2 D^{(t,-)}}{D_2^{(t)}x^k}=\frac{\eps_{(t)}^2 D^{(t,-)}}{D_2^{(t)}B^{k\gamma^4}} \stackrel{\ref{tacchompdeghyp}}{\geq} B^{k\gamma^4}\stackrel{(\ref{eq:bisbig})}{\geq}\log^{10} D\geq\log^8 D^{(t)},
\end{eqnarray*}
so~\ref{chomphyp:trackingdegratio} and~\ref{chomphyp:nontrackingdegratio} hold (whether \(2\in J_t^{*}\) or otherwise).\COMMENT{\(k\geq 2\) by assumption. Then we simply ensure the LHS of~\ref{chomphyp:trackingdegratio} and~\ref{chomphyp:nontrackingdegratio} are both at least \(\log^8 D\), whether \(2\in J^{*}\) or not. We get to do this because we do a small Chomp each time, so even \(x^k\) still fits in the space that remains between \(D\) and \(D_2\).}
Similarly, fix \(j\in J^{*}_t\),\COMMENT{(and~\ref{chomphyp:trackingcodegratio} and~\ref{chomphyp:nontrackingcodegratio} hold vacuously if there are no \(j\in J_t^{*}\))} and observe that\COMMENT{First inequality uses \(j\leq k\) by definition of \(J^{*}\), so that \(k-j+1\geq 1\). Second inequality uses \(j\geq 1\), (we actually have \(j\geq 2\)) so that \(k-j+1\leq k\).}
\[
\frac{D_j^{(t)}}{xD_{j+1}^{(t)}}\geq\frac{D_j^{(t)}}{x^{k-j+1}D_{j+1}^{(t)}}\geq\frac{D_j^{(t)}}{x^{k}D_{j+1}^{(t)}}=\frac{D_j^{(t)}}{B^{k\gamma^4}D_{j+1}^{(t)}}\stackrel{\ref{tacchompcodeghyp}}{\geq} B^{k\gamma^4}\stackrel{(\ref{eq:bisbig})}{\geq}\log^{10} D\geq\log^{10} D^{(t)},
\]
so that~\ref{chomphyp:trackingcodegratio} and~\ref{chomphyp:nontrackingcodegratio} hold (whether \(j+1\in J_t^{*}\) or otherwise).
Moreover,\COMMENT{Used \(1\leq 1/(2\gamma^3) \Leftrightarrow \gamma^3\leq1/2\). Used \(10k\gamma^3\leq\gamma^2/4 \Leftrightarrow \gamma\leq 1/40k\).}
\[
\eps_{(t)}x=\eps^{*}x^{t+1}\leq\eps^{*}x^{1/\gamma^4 - 1/(2\gamma^2)}=B^{-1+10k\gamma^3}\cdot B^{1-\gamma^2/2}\leq B^{-\gamma^2/4}\stackrel{(\ref{eq:bisbig})}{\leq} \frac{1}{\log^2 D}\leq\frac{1}{\log^2 D^{(t)}},
\]
so~\ref{chomphyp:eps} holds.
For each \(\tau\in\cT\), setting \(\tau_{\text{max}}\coloneqq\max_{v\in V(H)}\tau(v)\), we have
\begin{eqnarray*}
\eps_{(t)}\tau_t(V(H_t)) & \stackrel{\ref{tacchompsetsizehyp}}{\geq} & \eps_{(t)}m_{\tau}^{(t,-)} =\frac{\eps^{*}x^t \tau(V(H))}{e^{2t}x^t} = \frac{\tau(V(H))B^{10k\gamma^3}}{e^{2t}B}\stackrel{\ref{mainpseud:lowerbound}}{\geq}\frac{\tau_{\text{max}}\cdot B^{10k\gamma^3}}{e^{2t}}\\ & \stackrel{(\ref{eq:bisbig})}{\geq} & \tau_{\text{max}}\cdot e^3 \log^7 (D^{(t)}),
\end{eqnarray*}
so~\ref{chomppseudhyplowerbound} holds.
Observe that
\begin{equation}\label{eq:DlogDbound}
(D^{(t)})^{\log^{6/5}(D^{(t)})} \stackrel{(\ref{eq:Dtgamma})}{\geq} D^{(\gamma^2 /2)\cdot(\gamma^2 /2)^{6/5}\log^{6/5}D}\geq D^{\log D}.
\end{equation}
Since \(|\{v\in V(H_t)\colon\tau_t(v)>0\}|\leq|\{v\in V(H)\colon\tau(v)>0\}|\) for all \(\tau\in\cT\) and each \(u\in V(H)\) is in the same number of sets \(\{v\in V(H_t)\colon\tau_t(v)>0\}\) as \(\{v\in V(H_{t'})\colon\tau_{t'}(v)>0\}\) for all \(0\leq t'\leq t\), it is clear from~(\ref{eq:DlogDbound}) that \ref{mainpseud:supportsize}--\ref{mainpseud:maxinvolvement} imply \ref{chomppseudhypsupportsize}--\ref{chomppseudhypmaxinvolvement}.
Finally, \(x^k = B^{k\gamma^4}\leq (D/D_2)^{(1/2)k\gamma^4}\leq D^{(1/2)k\gamma^4}\stackrel{(\ref{eq:Dtgamma})}{\leq} (D^{(t)})^{k\gamma^2}\leq (D^{(t)})^{1/2}=(D^{(t)})^{1-\delta}\), so~\ref{chomppseudhyppolyremainder} holds.

Lemma~\ref{lemma:chomp} yields a matching~\(M_t\) and a set~\(W_t\subseteq V(H_t)\) satisfying \(|W_t|\leq 10e^2 \eps_{(t)}n_t\eta\log B\), as claimed.\COMMENT{\(\log x = \log(B^{\eta})=\eta\log B\).}
Set \(n_{t+1}, D^{(t+1)}\) to be \(n', D'\), respectively, from the conclusion of Lemma~\ref{lemma:chomp}, so that \(H_{t+1}=H_t[V(H_t)\setminus(V(M_t)\cup W_t)]\) is \((n_{t+1}, D^{(t+1)}, \eps_{(t)}x)\)-regular (and note that \(\eps_{(t)}x=\eps_{(t+1)}\)).
Then \(n_{t+1}\leq e^3 n_t/x\leq e^3 n_t^{(+)}/x = n_{t+1}^{(+)}\), and similarly \(n_{t+1}\geq n_{t+1}^{(-)}\), and \(D^{(t+1,-)}\leq D^{(t+1)}\leq D^{(t+1,+)}\), and \(m_{\tau}^{(t+1,-)}\leq \tau(V(H_{t+1}))\leq m_{\tau}^{(t+1,+)}\) for all \(\tau\in\cT\).\COMMENT{\(n_{t+1}\geq n_t/(e^3 x) \geq n_t^{(-)}/(e^3 x) = n_{t+1}^{(-)}\). \(D^{(t+1)}\geq D^{(t)}/(e^2 x^k)\geq D^{(t,-)}/(e^2 x^k)=D^{(t+1,-)}\). \(D^{(t+1)}\leq e^2 D^{(t)}/x^k \leq e^2 D^{(t,+)}/x^k = D^{(t+1,+)}\). \(\tau(V(H_{t+1}))\geq m_{\tau}^{(t,-)}/e^2 x = m_{\tau}^{(t+1,+)}\) and same for the upper bound.}
Moreover, Lemma~\ref{lemma:chomp} yields \(\tau(W_t)=\tau_t(W_t)\leq 10e^2 \eps_{(t)}m_{\tau}^{(t,+)}\eta\log B\) for all \(\tau\in\cT\), here using~\ref{tacchompsetsizehyp}.
Further, clearly \(C_j(H_{t+1})\leq D_j^{(t)}=D_{j}^{(t+1)}\) for all \(j\in[k+1]\setminus(J_t^{*}\cup\{1\})\), since \(H_{t+1}\subseteq H_t\).
For \(j\in J_t^{*}\), Lemma~\ref{lemma:chomp} yields \(C_j(H_{t+1})\leq e^2 D_j^{(t)}/x^{k-j+1}=D_{j}^{(t+1)}\).

It remains only to check that \(D_2^{(t+1)}\geq D_3^{(t+1)}\geq\dots\geq D_{k+1}^{(t+1)}\).
To that end, fix \(j\in[k]\setminus\{1\}\).
Notice that \(D_{j+1}^{(t+1)}\leq D_{j+1}^{(t)}\) (whether \(j+1\in J_t^{*}\) or otherwise)\COMMENT{They're equal if \(j+1\notin J_t^{*}\). Notice \(k+1\) is never in \(J_t^{*}\) by definition of \(J_t^{*}\), so if we're assuming \(j+1\in J_t^{*}\), this means \(j+1\leq k\), i.e. \(j\leq k-1\), so \(D_{j+1}^{(t+1)}=e^2 D_{j+1}^{(t)}/x^{k-j} \leq D_{j+1}^{(t)}\), since \(j\leq k-1\) and \(x=\omega(1)\).}.
If \(j\notin J_t^{*}\) then \(D_j^{(t+1)}=D_j^{(t)}\geq D_{j+1}^{(t)}\geq D_{j+1}^{(t+1)}\), where we used the hypothesis \(D_j^{(t)}\geq D_{j+1}^{(t)}\).
If \(j\in J_t^{*}\), then\COMMENT{The first inequality holds with equality if \(j+1\notin J_t^{*}\), and otherwise we simply use \(D_{j+1}^{(t+1)}\leq D_{j+1}^{(t)}\). Used \(e^2\geq 1\), \(k-j+1\leq k \Leftrightarrow j\geq 1\).}
\[
\frac{D_{j}^{(t+1)}}{D_{j+1}^{(t+1)}}\geq\frac{e^2 D_{j}^{(t)}}{x^{k-j+1}D_{j+1}^{(t)}}\stackrel{\ref{tacchompcodeghyp}}{\geq} \frac{e^2 B^{2k\gamma^4}}{x^{k-j+1}}\geq \frac{B^{2k\gamma^4}}{x^k}=B^{k\gamma^4}\stackrel{(\ref{eq:bisbig})}{\geq}1,
\]
which completes the proof of the claim.
\endclaimproof
We now introduce an algorithm we call MINI-CHOMP-WHERE-ABLE, or MCWA for short.
Roughly speaking, the algorithm will begin with~\(H\), and in each timestep will identify which \(j\in[k]\setminus\{1\}\) are such that \(D_j^{(t)}/D_{j+1}^{(t)}\) is large enough to include \(j\in J_t^{*}\) for a ``Chomp'' with the ``small'' value \(x=B^{\eta}\) (or more accurately, to apply Claim~\ref{claim:tacticalchomp}).
The algorithm then applies Claim~\ref{claim:tacticalchomp} to obtain~\(H_{t+1}, M_t, W_t\) satisfying the obvious bounds, and produces the values \(D_j^{(t+1)}\) accordingly (depending on which~\(j\) are in~\(J_t^{*}\)).
\vspace{3mm}

\noindent\textbf{Algorithm} (MINI-CHOMP-WHERE-ABLE)
\newline\noindent \underline{Input}: The hypergraph~\(H\eqqcolon H_0\), which is \((n_0\coloneqq n, D^{(0)}\coloneqq D, \eps_{(0)})\)-regular (and notice that \(n_0^{(-)}=n_0=n_0^{(+)}\) and \(D^{(0,-)}=D^{(0)}=D^{(0,+)}\)), and the family~\(\cT\) of functions \(\tau\colon V(H)\rightarrow\mathbb{R}_{\geq0}\) as in the statement of Theorem~\ref{theorem:maintheorem} (and notice that each \(\tau\in\cT\) satisfies \(m_{\tau}^{(0,-)}=\tau(V(H_0))= m_{\tau}^{(0,+)}\)).
Set \(t\coloneqq 0\) and \(D_j^{(0)}\coloneqq D_j\) for all \(j\in[k+1]\setminus\{1\}\), so that \(C_j(H_0)\leq D_j^{(0)}\) for all such~\(j\) and \(D_2^{(0)}\geq D_3^{(0)}\geq\dots\geq D_{k+1}^{(0)}\).
\newline\noindent \underline{Step 1}: If \(t=\lflr\frac{1}{\eta}-\frac{1}{\gamma^2}\rflr\), then output \(t^{*}\coloneqq t\) and TERMINATE.
Else, go to Step 2.
\newline\noindent \underline{Step 2}: Set \(J_t^{*}\coloneqq \{j\in[k]\setminus\{1\}\colon D_j^{(t)}/D_{j+1}^{(t)}\geq B^{2k\gamma^4}\}\).
If any of~\ref{tacchompdeghyp},~\ref{tacchompcodeghyp}, and~\ref{tacchompsetsizehyp} are not satisfied (with this choice of~\(J_t^{*}\)), then output \(t^{*}\coloneqq t\) and TERMINATE.
Else, go to Step 3.
\newline\noindent \underline{Step 3}: (Apply Claim~\ref{claim:tacticalchomp} to) Obtain a matching \(M_t\subseteq E(H_t)\) and a set \(W_t\subseteq V(H_t)\) of size \(|W_t|\leq 10e^2\eps_{(t)}n_t \eta\log B\), together with numbers~\(n_{t+1}\) and~\(D^{(t+1)}\) satisfying \(n_{t+1}^{(-)}\leq n_{t+1}\leq n_{t+1}^{(+)}\) and \(D^{(t+1,-)}\leq D^{(t+1)}\leq D^{(t+1,+)}\) such that \(H_{t+1}\coloneqq H_{t}[V(H_t)\setminus(V(M_t)\cup W_t)]\) is \((n_{t+1}, D^{(t+1)}, \eps_{(t+1)})\)-regular, further satisfying \(m_{\tau}^{(t+1,-)}\leq \tau(V(H_{t+1}))\leq m_{\tau}^{(t+1,+)}\) and \(\tau(W_t)\leq 10e^2 \eps_{(t)}m_{\tau}^{(t,+)}\eta\log B\) for all \(\tau\in\cT\) and \(C_j(H_{t+1})\leq D_j^{(t+1)}\) for all \(j\in[k+1]\setminus\{1\}\), where \(D_j^{(t+1)}\coloneqq D_j^{(t)}\) for all \(j\in[k+1]\setminus(J_t^{*}\cup\{1\}),\) and \(D_{j}^{(t+1)}\coloneqq e^2 D_j^{(t)}/x^{k-j+1}\) for all \(j\in J_t^{*}\), and \(D_2^{(t+1)}\geq D_3^{(t+1)}\geq\dots\geq D_{k+1}^{(t+1)}\).
Update \(t\coloneqq t+1\) and go to Step 1.
\vspace{3mm}

We aim to show\COMMENT{This was originally in the main body, but it's effectively proven after the observations, so I reduce it to the comments as a sanity check: In particular, with \(t=0<\lfloor1/\eta - 1/\gamma^2\rfloor\) and the given construction of~\(J_0^{*}\), it is clear that MCWA only terminates with \(t^{*}=0\) if~\ref{tacchompdeghyp} fails when \(t=0\). However, since \(B\leq\sqrt{D/D_2}\) by hypothesis, we have \(\eps_{(0)}^2 D^{(0,-)}/D_2^{(0)}=(\eps^{*})^2 D/D_2=B^{-2+20k\gamma^3}D/D_2 \geq B^{20k\gamma^3}\geq B^{2k\gamma^4}\), so that~\ref{tacchompdeghyp} holds when \(t=0\), so MCWA does not terminate in the first step.} that MCWA terminates with \(t^{*}=\lflr 1/\eta - 1/\gamma^2\rflr\), i.e.\ it never terminates due to the failure of~\(H_t\) to satisfy~\ref{tacchompdeghyp}--~\ref{tacchompsetsizehyp} (the remaining hypotheses of Claim~\ref{claim:tacticalchomp} are clearly ensured by the initial conditions in the timestep \(t=0\), and ensured by the output of timestep~\(t-1\), Step~3, for other values of~\(t\)).
To that end, notice that the construction of~\(J_t^{*}\) ensures that~\ref{tacchompcodeghyp} can never fail, and~\ref{tacchompsetsizehyp} does not fail in the timestep \(t=0\), and is then ensured for timestep~\(t+1\) by the output of timestep~\(t\), Step 3, for all \(t\leq\lflr1/\eta -1/\gamma^2 \rflr\).
Moreover,~\(\eps_{(t)}\) and~\(D^{(t,-)}\) are clearly defined functions of \(\eps^{*}, D, t, B\), so to understand whether~\ref{tacchompdeghyp} is satisfied, it remains only to understand the evolution of~\(D_2^{(t)}\), which in turn depends on the evolution of the other codegrees \(D_j^{(t)}\).
To study the evolution of these parameters, we introduce the notions of semi-stuck and super-stuck indices, active and dormant clusters, and slowpoke indices, which we define now.

For \(t\in[t^{*}]_0\), we say that \(j\in[k]\setminus\{1\}\) is a \textit{super-stuck} index (at time~\(t\)) if \(D_j^{(t)}< D_{j+1}^{(t)}B^{2k\gamma^{4}}\).
We write~\(S_t\) for the set of super-stuck indices at time~\(t\).
Note that, by construction, \(J_t^{*}=\{2,3,\dots,k\}\setminus S_t\).
We say that~\(j\in[k]\setminus\{1\}\) is \textit{semi-stuck} (at time~\(t\)) if \(D_j^{(t)}<D_{j+1}^{(t)}B^{\gamma^3}\).
Clearly, if~\(j\) is super-stuck, then it is semi-stuck.
Notice that the set~\(\{2,3,\dots,k+1\}\) is partitioned into a set~\(\cC_t\) of subsets of consecutive indices, where the largest index \(j\in C\in\cC_t\) is not semi-stuck at time~\(t\), but all other indices in~\(C\) are semi-stuck at time~\(t\).
(In particular,~\(k+1\) is not semi-stuck by definition.)
For \(t\in[t^{*}]_0\) and \(j\in\{2,3,\dots,k+1\}\), we write \(C_t(j)\) for the element of~\(\cC_t\) containing~\(j\).
We call the elements of~\(\cC_t\) \textit{clusters}.
~\(C_t(k+1)\) is the \textit{dormant} cluster, and all other clusters are \textit{active} clusters.
Indices in a dormant (respectively active) cluster are \textit{dormant} (\textit{active}) \textit{indices}.
The largest index in a cluster (not the largest codegree, which will correspond to the smallest index since we maintain \(D_2^{(t)}\geq D_3^{(t)}\geq\dots\geq D_{k+1}^{(t)}\)) is called the \textit{slowpoke} (of that cluster, at that time~\(t\)).
Write~\(Z_t\) for the set of slowpokes at time~\(t\).
For an index \(j\in\{2,3,\dots,k+1\}\), we write \(z_t(j)\) for the slowpoke of~\(C_t(j)\).
We now make five simple observations about MCWA.
\begin{observation}\label{obsclustererror}
\(D_j^{(t)}\leq D_{z_t(j)}^{(t)}B^{k\gamma^3}\) for all \(j\in\{2,3,\dots,k+1\}\) and \(t\in[t^{*}]_0\).
\end{observation}
\obsproof
The result is clear if \(j=z_t(j)\).
Otherwise, since~\(j\) and~\(z_t(j)\) are in the same cluster of~\(\cC_t\) and \(z_t(j)\geq j\), we have that the indices \(j,j+1,\dots,z_t(j)-1\) are all semi-stuck, from which the observation follows.\COMMENT{\(D_j(t)<D_{j+1}^{(t)}B^{\gamma^3}\) since \(j\) is semi-stuck by the definition of the clusters, and similarly each ratio is at most~\(B^{\gamma^3}\). There are at most \(k\) of these consecutive ratios (in fact at most \(k-1\) since the largest possible cluster is \(\{2,3,\dots,k+1\}\)).}
\endclaimproof
\begin{observation}\label{obsclustershare}
If \(t\in[t^{*}-1]_0\) and \(j+1\in C_t(j)\), then \(j+1\in C_{t+1}(j)\).
\end{observation}
\obsproof
Since \(j+1\in C_t(j)\), we deduce from the definition of~\(\cC_t\) that~\(j\) is semi-stuck, so \(D_j^{(t)}<D_{j+1}^{(t)}B^{\gamma^3}\).
If \(j,j+1\notin J_t^{*}\), then \(D_j^{(t)}=D_{j}^{(t+1)}\) and \(D_{j+1}^{(t)}=D_{j+1}^{(t+1)}\) so clearly \(D_{j}^{(t+1)}<D_{j+1}^{(t+1)}B^{\gamma^3}\), so~\(j\) is still semi-stuck at time~\(t+1\), so that \(j+1\in C_{t+1}(j)\).
If \(j\in J_t^{*}, j+1\notin J_t^{*}\), then clearly \(D_{j+1}^{(t+1)}=D_{j+1}^{(t)}\) and \(D_j^{(t+1)}<D_j^{(t)}\) so the observation follows in this case also.\COMMENT{\(D_j^{(t+1)}<D_j^{(t)}<D_{j+1}^{(t)}B^{\gamma^3}=D_{j+1}^{(t+1)}B^{\gamma^3}\).}
If \(j,j+1\in J_t^{*}\), then
\[
\frac{D_j^{(t+1)}}{D_{j+1}^{(t+1)}}=\frac{e^2 D_j^{(t)}}{x^{k-j+1}}\cdot\frac{x^{k-j}}{e^2 D_{j+1}^{(t)}}=\frac{D_j^{(t)}}{xD_{j+1}^{(t)}}<\frac{D_j^{(t)}}{D_{j+1}^{(t)}}<B^{\gamma^3},
\]
so again~\(j\) is semi-stuck at time~\(t+1\), so \(j+1\in C_t(j)\).
Finally, if \(j\notin J^{*}_t\), \(j+1\in J_t^{*}\), then~\(j\) is super-stuck at time~\(t\) by construction of~\(J_t^{*}\), so \(D_{j}^{(t)}<D_{j+1}^{(t)}B^{2k\gamma^4}\), whence
\[
\frac{D_j^{(t+1)}}{D_{j+1}^{(t+1)}}=\frac{D_j^{(t)}x^{k-j}}{e^2 D_{j+1}^{(t)}}\leq\frac{D_{j}^{(t)}x^k}{D_{j+1}^{(t)}}=\frac{D_j^{(t)}}{D_{j+1}^{(t)}}B^{k\gamma^4}<B^{\gamma^3},
\]
which completes the proof of the observation.
\endclaimproof
As an aside, we remark that Observation~\ref{obsclustershare} shows that clusters never lose elements.
Once any pair of indices \(j<j'\) are in a cluster together, (the definition of the clusters means~\(\{j,j+1,\dots,j'-1\}\) are currently semi-stuck and so) they remain always in a cluster together, so the clusters~\(\cC_t\) remain the same until some future timestep in which possibly two (consecutive) clusters (fully) ``merge'', so~\(|\cC_t|\) can only decrease over time.
\begin{observation}\label{obsalwaysslowpoke}
If \(t\in [t^{*}]_0\) and \(j\in Z_t\), then \(j\in Z_{t'}\) for all \(0\leq t'\leq t\).
\end{observation}
\obsproof
Suppose for a contradiction that \(j\notin Z_{t'}\) for some \(0\leq t'<t\).
Then~\(j+1\in C_{t'}(j)\) by the definition of the slowpokes.
Inductively applying Observation~\ref{obsclustershare}~\(t-t'\) times, we obtain \(j+1\in C_t(j)\); a contradiction.
\endclaimproof
\begin{observation}\label{obsactiveslowpokes}
If \(t\in[t^{*}]_0\) and \(j\in Z_t \setminus\{k+1\}\), then \(D_j^{(t)}=D_j(e^2/x^{k-j+1})^t\).
\end{observation}
\obsproof
By Observation~\ref{obsalwaysslowpoke}, we have \(j\in Z_{t'}\setminus\{k+1\}\) for all \(0\leq t'\leq t\), which means~\(j\) was not semi-stuck for any such~\(t'\), and thus was not super-stuck.
Since~\(J_{t'}^{*}\) is constructed to be the set of all non-super-stuck indices in~\(\{2,3,\dots,k\}\), we deduce that \(j\in J_{t'}^{*}\) for all \(0\leq t'<t\), which yields the observation. 
\endclaimproof
\begin{observation}\label{obskey}
\(D_2^{(t)}\leq B^{2k\gamma^3}\max_{j\in[k+1]\setminus\{1\}}\left(\frac{D_j}{x^{t(k-j+1)}}\right)\) for all \(t\in[t^{*}]_0\).\COMMENT{And indeed, basically is this expression, i.e. this isn't a wasteful upper bound. It takes a few more observations to prove that but I did it whilst playing with MCWA, but we don't need more than an upper bound for the proof, so it isn't typed. Essentially the initial slowpokes all eventually become \(D_j/x^{t(k-j+1)}\); those where this expression becomes smaller than the corresponding expression for slowpoke \(j'>j\) are no longer slowpokes after that time, the clusters merge. The result is that the slowpoke \(j\) for the furthest left cluster is the slowpoke that maximizes \(D_j/x^{t(k-j+1)}\) among initial slowpokes. Now, if there was a non-initial-slowpoke index \(j'\) with \(D_{j'}/(x^{t(k-j'+1)})>D_j/x^{t(k-j+1)}\) then, since it wasn't a slowpoke, it wasn't initially much more than its corresponding initial slowpoke, which in turn has value \(D_{j'}/x^{t(k-j'+1)}\) at most the maximum such value over slowpokes at time \(t\), which is now to say that even the maximum over all initial indices (slowpoke or otherwise) is within the margin of error (say \(B^{k\gamma^3}\)) of the maximum \(D_j/x^{t(k-j+1)}\) over initial slowpokes, the latter being not much less than \(D_2^{(t)}\), which demonstrates that \(D_2^{(t)}\) basically is this expression.}
\end{observation}
\obsproof
If~\(2\) is a dormant index at time~\(t\), then
\begin{eqnarray*}
D_{2}^{(t)} & \stackrel{\text{Obs}~\ref{obsclustererror}}{\leq} & D_{k+1}^{(t)}B^{k\gamma^3} = D_{k+1}B^{k\gamma^3}=B^{k\gamma^3}\left(\frac{D_{k+1}}{x^{t(k-(k+1)+1)}}\right)\\ & \leq & B^{2k\gamma^3}\max_{j\in[k+1]\setminus\{1\}}\left(\frac{D_j}{x^{t(k-j+1)}}\right),
\end{eqnarray*}
where we used the fact that~\(k+1\) is never in~\(J_{t'}^{*}\) by construction.
If, instead,~\(2\) is an active index at time~\(t\), then~\(C_t(2)\) is an active cluster, and~\(z_t(2)\) is an active slowpoke index, whence\COMMENT{\(e^{2t}\leq e^{2(1/\eta - 1/\gamma^2)}\leq\log D\) and \(B^{k\gamma^3}\stackrel{(\ref{eq:bisbig})}{\geq} \log D\).}
\begin{eqnarray*}
D_2^{(t)} & \stackrel{\text{Obs}~\ref{obsclustererror}}{\leq} & D_{z_t(2)}^{(t)}B^{k\gamma^3} \stackrel{\text{Obs}~\ref{obsactiveslowpokes}}{=} B^{k\gamma^3}D_{z_t(2)}\left(\frac{e^2}{x^{k-z_t(2)+1}}\right)^t\stackrel{(\ref{eq:bisbig})}{\leq}B^{2k\gamma^3}\left(\frac{D_{z_t(2)}}{x^{t(k-z_t(2)+1)}}\right)\\ & \leq & B^{2k\gamma^3}\max_{j\in[k+1]\setminus\{1\}}\left(\frac{D_j}{x^{t(k-j+1)}}\right),
\end{eqnarray*}
as required.
\endclaimproof
Armed with Observation~\ref{obskey}, we will now show that MCWA does not terminate due to the failure of~\ref{tacchompdeghyp}, which, together with the fact that~\ref{tacchompcodeghyp}--\ref{tacchompsetsizehyp} never fail as discussed, will ensure that MCWA terminates with \(t^{*}=\lflr1/\eta-1/\gamma^2\rflr\).
To that end, assume that \(0\leq t<1/\eta-1/\gamma^2\), set~\(r(t)\) to be the smallest~\(j\) which attains \(\max_{j\in[k+1]\setminus\{1\}}(D_j/x^{t(k-j+1)})\), and observe that
\begin{eqnarray}\label{eq:thelimitingcodegree}
\frac{\eps_{(t)}^2 D^{(t,-)}}{D_2^{(t)}} & = & (\eps^{*})^2 x^{t(2-k)}\frac{D}{e^{2t}D_2^{(t)}}\stackrel{(\ref{eq:bisbig})}{\geq} B^{-2+10k\gamma^3}x^{t(2-k)}\frac{D}{D_2^{(t)}}\nonumber\\ & \stackrel{\text{Obs}~\ref{obskey}}{\geq} & \frac{DB^{-2+8k\gamma^3}x^{t(2-k)}}{\max_{j\in[k+1]\setminus\{1\}}(D_j/x^{t(k-j+1)})} = \frac{DB^{-2+8k\gamma^3}x^{t(3-r(t))}}{D_{r(t)}}
\end{eqnarray}
If \(r(t)=2\), then the right side of~(\ref{eq:thelimitingcodegree}) is \(DB^{-2+8k\gamma^3}x^t/D_2\geq DB^{-2+8k\gamma^3}/D_2\geq B^{8k\gamma^3}\geq B^{2k\gamma^4}\), so that~\ref{tacchompdeghyp} holds, where we used the hypothesis \(B\leq\sqrt{D/D_2}\).
If \(r(t)=3\), then the right side of~(\ref{eq:thelimitingcodegree}) is \(DB^{-2+8k\gamma^3}/D_3\geq DB^{-2+8k\gamma^3}/D_2\geq B^{2k\gamma^4}\) as before, so that~\ref{tacchompdeghyp} holds, where we used the hypothesis \(D_2\geq D_3\).
Finally, if \(r(t)\geq 4\), we use \(x^t=B^{\eta t}\leq B\) to see that the right side of~(\ref{eq:thelimitingcodegree}) is \(DB^{-2+8k\gamma^3}/D_{r(t)}x^{t(r(t)-3)}\geq DB^{-2+8k\gamma^3}/D_{r(t)}B^{r(t)-3} = DB^{8k\gamma^3}/D_{r(t)}B^{r(t)-1}\geq B^{8k\gamma^3}\geq B^{2k\gamma^4}\), so that~\ref{tacchompdeghyp} holds, where we used the hypothesis that \(B\leq(D/D_{r(t)})^{1/(r(t)-1)}\) since \(r(t)\geq 4\).

We conclude that MCWA terminates with \(t^{*}=\lflr1/\eta-1/\gamma^2\rflr\), yielding hypergraphs~\((H_t)\) for \(t\in[t^{*}]_0\), matchings~\((M_t\subseteq E(H_t))\), and sets~\((W_t\subseteq V(H_t))\) for \(t\in[t^{*}-1]_0\).
Set \(H'\coloneqq H_{t^{*}}\), \(M\coloneqq\bigcup_{t}M_t\), and \(W\coloneqq\bigcup_t W_t\).
By construction, \(H'=H[V(H)\setminus(V(M)\cup W)]\), so that~\(M\) is a matching covering all but at most~\(|V(H_{t^{*}})|+|W|\) vertices of~\(H\).
Using \(B\leq\sqrt{D/D_2}\leq D\), we have\COMMENT{Clearly the penultimate expression is \(\leq nB^{-1}\log^3 D (B^{2\gamma^2 + 10k\gamma^3}) \leq nB^{-1+\gamma}\log^A D\), recalling we were using \(A=10/\gamma^4\).}
\begin{eqnarray}\label{eq:finishingup}
|V(H_{t^{*}})|+|W| & \leq & n_{t^{*}}^{(+)} +\sum_{t\in[t^{*}-1]_0}|W_t|\leq \frac{e^{2t^{*}}n}{x^{1/\eta-2/\gamma^2}} + \sum_{t\in[t^{*}-1]_0}10e^2\eps_{(t)}n_t^{(+)}\eta\log B \nonumber\\ & \leq & nB^{-1+2\gamma^2}\log D + \sum_{t\in[t^{*}-1]_0}10e^2\eps^{*}n\eta\log^2 D \nonumber\\ & \leq & nB^{-1+2\gamma^2}\log D + \eps^{*}n\log^3 D = nB^{-1+2\gamma^2}\log D + nB^{-1+10k\gamma^3}\log^3 D \nonumber\\ & \leq & nB^{-1+\gamma}\log^A D.
\end{eqnarray}
Finally, for each \(\tau\in\cT\) we have \(\tau(V(H)\setminus V(M))=\tau(V(H'))+\tau(W\setminus V(M))\leq m_{\tau}^{(t^{*},+)} + \tau(W)\leq m_{\tau}^{(t^{*},+)}+\sum_{t\in[t^{*}-1]_0}10e^2 \eps_{(t)}m_{\tau}^{(t,+)}\eta\log B\), which is at most \(\tau(V(H))B^{-1+\gamma}\log^A D\), via similar logic\COMMENT{\(m_{\tau}^{(t^{*},+)}+\sum_{t\in[t^{*}-1]_0}10e^2 \eps_{(t)}m_{\tau}^{(t,+)}\eta\log B \leq \tau(V(H))\log D/ (x^{1/\eta - 2\gamma^2}) + 10e^2 \eta\log B\sum_{t\in[t^{*}-1]_0}\eps^{*}x^t (e^{2t}/x^t)\tau(V(H)) \leq \tau(V(H))B^{-1+2\gamma^2}\log D + (\log^2 D)\eps^{*}\tau(V(H))=\tau(V(H))B^{-1+2\gamma^2}\log D + (\log^2 D)B^{-1+10k\gamma^3}\tau(V(H))\leq \tau(V(H))B^{-1+\gamma}\log^A D\)} to~(\ref{eq:finishingup}).
On the other hand, \(\tau(V(H)\setminus V(M))\geq m_{\tau}^{(t^{*},-)}\geq \tau(V(H))/x^{1/\eta - 1/\gamma^2}\log D =\tau(V(H))B^{-1+\gamma^2}/\log D\stackrel{(\ref{eq:bisbig})}{\geq} \tau(V(H))/B\), as claimed.\COMMENT{Clearly, we could say marginally better here if we cared to, but I don't think it's important}
\endproof

\providecommand{\bysame}{\leavevmode\hbox to3em{\hrulefill}\thinspace}
\providecommand{\MR}{\relax\ifhmode\unskip\space\fi MR }
\providecommand{\MRhref}[2]{%
  \href{http://www.ams.org/mathscinet-getitem?mr=#1}{#2}
}
\providecommand{\href}[2]{#2}

\end{document}